\documentclass[a4paper, 11pt, twoside]{article}
\usepackage{hyperref}

\usepackage{amsmath,amssymb}
\usepackage{amsthm}
\usepackage{mathrsfs}
\usepackage[pdftex]{graphicx}

\usepackage[normalem]{ulem}
\usepackage{url}
\usepackage{comment}
\usepackage{bm}
\usepackage{accents}

\usepackage{enumerate}
\usepackage[all]{xy}
\usepackage{mathtools}
\usepackage{graphicx}
\usepackage{overpic}
\usepackage{pdfpages}

\usepackage[margin=1in]{geometry}

\usepackage{fancyhdr}
\newcommand\shorttitle{Toward a topological description of LCH of unit conormal bundles}
\newcommand\authors{Yukihiro Okamoto}

\usepackage{tikz}
\usepackage{tikz-cd}
\usepackage{pgfplots}

\usetikzlibrary{arrows}
\usetikzlibrary{patterns}

\theoremstyle{definition}
\newtheorem{defi}{Definition}[section]
\newtheorem{rem}[defi]{Remark}
\newtheorem{ex}[defi]{Example}

\newtheorem*{notation}{Notation}

\newtheorem*{acknow}{Acknowledgements}

\theoremstyle{plain}
\newtheorem{thm}[defi]{Theorem}

\newtheorem{prop}[defi]{Proposition}
\newtheorem{lem}[defi]{Lemma}
\newtheorem{cor}[defi]{Corollary}
\newtheorem{conj}[defi]{Conjecture}

\renewcommand{\ker}{\operatorname{Ker}}
\newcommand{\coker}{\operatorname{Coker}}

\renewcommand{\Im}{\operatorname{Im}}

\newcommand{\R}{\mathbb{R}}

\newcommand{\Z}{\mathbb{Z}}

\newcommand{\pr}{\operatorname{pr}}

\newcommand{\id}{\operatorname{id}}

\newcommand{\reg}{\mathrm{reg}}
\renewcommand{\tilde}{\widetilde}
\newcommand{\dr}{\mathrm{dR}}
\newcommand{\sing}{\mathrm{sing}}
\newcommand{\ev}{\operatorname{ev}}
\newcommand{\len}{\mathrm{length}}
\renewcommand{\epsilon}{\varepsilon}
\newcommand{\con}{\operatorname{con}}
\newcommand{\fib}{\mathrm{fiber}}

\newcommand{\la}{\langle}
\newcommand{\ra}{\rangle}

\newcommand{\supp}{\operatorname{supp}}
\newcommand{\str}{\mathrm{string}}
\newcommand{\tl}{\operatorname{tl}}
\newcommand{\interior}{\mathrm{int}}
\newcommand{\Th}{\mathrm{Th}}
\newcommand{\sign}{\operatorname{sign}}
\newcommand{\hopf}{\mathrm{Hopf}}

\newcommand{\codim}{\operatorname{codim}}

\newcommand{\rest}[2]{\left.#1\right|_{#2}}
\newcommand{\ftimes}[2]{\ {}_{#1}\!\times_{#2}}

\newcommand{\sm}{\mathrm{sm}}

\usepackage{authblk}

\AtEndDocument{%
  \par
  \medskip
  \begin{tabular}{@{}l@{}}%
    Research Institute for Mathematical Sciences, Kyoto University, Kyoto 606-8502 JAPAN \\
    \textit{E-mail address}: \texttt{yukihiro@kurims.kyoto-u.ac.jp}
  \end{tabular}}

\begin{document}

\title{Toward a topological description of Legendrian contact homology of unit conormal bundles}
\author{Yukihiro Okamoto }
\date{}
\maketitle

\begin{abstract}
For a smooth compact submanifold $K$ of a Riemannian manifold $Q$, its unit conormal bundle $\Lambda_K$ is a Legendrian submanifold of the unit cotangent bundle of $Q$ with a canonical contact structure. Using pseudo-holomorphic curve techniques, the Legendrian contact homology of $\Lambda_K$ is defined when, for instance, $Q=\R^n$. In this paper, aiming at giving another description of this homology, we define a graded $\R$-algebra for any pair $(Q,K)$ with orientations from a perspective of string topology and prove its invariance under smooth isotopies of $K$.
The author conjectures that it is isomorphic to the Legendrian contact homology of $\Lambda_K$ with coefficients in $\R$ in all degrees.
This is a reformulation of a homology group, called string homology, introduced by Cieliebak, Ekholm, Latschev and Ng when the codimension of $K$ is $2$, though the coefficient is reduced from original $\Z[\pi_1(\Lambda_K)]$ to $\R$.
We compute our invariant (i) in all degrees for specific examples, and (ii) in the $0$-th degree when the normal bundle of $K$ is a trivial $2$-plane bundle.
\end{abstract}

\tableofcontents

\section{Introduction}

\noindent
\textbf{Convention.}
Throughout this paper, all manifolds are of class $C^{\infty}$ without boundary and second countable, and all submanifolds are of class $C^{\infty}$ without boundary, unless otherwise specified.

\

\noindent
\textbf{Backgrounds.}
Let $Q$ be a manifold with a Riemannian metric, and $K$ be a compact submanifold of $Q$.
For any pair $(Q,K)$, one can associate the unit cotangent bundle $UT^*Q$ of $Q$ and the unit conormal bundle  $\Lambda_K$ of $K$. It is well-known that $UT^*Q$ has a canonical contact structure and $\Lambda_K$ is a Legendrian submanifold of $UT^*Q$. 

As an invariant of Legendrian submanifolds,
the \textit{Legendrian contact homology} has been studied for pairs $(M,\Lambda)$ of a contact manifold $M$ and its compact Legendrain submanifold $\Lambda$. It is the homology of a differential graded algebra generated by Reeb chords of $\Lambda$, and was introduced by Chekanov \cite{chek} and Eliashberg \cite{eli}. 
The differential is defined by using pseudo-holomorphic curves in the symplectization of $M$. A rigorous definition was given by Ekholm, Etnyre, and Sullivan in \cite{eesR, ees} when there is a diffeomorphism from $M$ to the contactization of a Liouville manifold which preserves contact forms. As is mentioned in \cite[Section 5.1]{ees}, this included the case of $M=UT^*\R^n$. (Remark: The definition of \cite{ees} is given by pseudo-holomorphic curves in the Liouville manifold. These curves can be lifted to pseudo-holomorphic curves in the symplectization of $M$. See \cite{D-R}.)

Suppose conceptually that we have an algebraic invariant in symplectic or contact topology defined by using pseudo-holomorphic curves, and apply it to an object related to the cotangent bundle $T^*Q$. For instance, we consider the symplectic homology of $T^*Q$, or the wrapped Floer homology of the conormal bundle $L_K$ of $K$ in $T^*Q$. In such case, it is known by the following results that these invariants have another view from the topology of the loop or path space of $Q$, without using pseudo-holomorphic curves (Here we assume that $Q$ is a closed spin manifold and all homology groups have $\Z$-coefficient.):
\begin{itemize}
\item The symplectic homology $SH_*(T^*Q)$ of $T^*Q$ is isomorphic to the  singular homology of the free loop space of $Q$ \cite{a-s, abouzaid, viterbo-icm}.
\item The wrapped Floer homology $WF_*(L_K,L_K)$ of $L_K$ is isomorphic to the singular homology of the space of paths in $Q$ with end points in $K$ \cite{a-s-conormal}.
\end{itemize}

These results lead us to an expectation that if the Legendrian contact homology of a pair $(UT^*Q,\Lambda_K)$ is defined, it has another description in terms of the topology of the path space of $Q$. This expectation has already been confirmed in particular cases. When the codimension of $K$ is $2$, Cieliebak, Ekholm, Lastchev, and Ng defined in \cite{celn} a graded $\Z[\pi_1(\Lambda_K)]$-algebra, called \textit{string homology}, which is inspired by string topology of the path space of $Q$. They showed that when $Q$ is equal to $\R^3$ with the standard metric and $K$ is a knot, the $0$-th degree part of this algebra is isomorphic to the $0$-th degree part of the \textit{fully non-commutative} Legendrian contact homology of $(UT^*\R^3,\Lambda_K)$ with coefficients in $\Z[\pi_1(\Lambda_K)]$.
However, such topological descriptions have not yet been defined in higher degrees or for $K$ with $\codim K\neq 2$.

\

\noindent
\textbf{Main results.}
Let $Q$ be an oriented manifold and $K$ be its compact oriented submanifold of codimension $d\geq 1$.
The main purpose of this paper is to define a graded $\R$-algebra $H^{\str}_*(Q,K)$ and observe its basic properties.
This graded $\R$-algebra can be regarded as a reformulation of the string homology of \cite{celn}, whose coefficient is reduced from $\Z[\pi_1(\Lambda_K)]$ to $\R$. The feature of our formulation is that $H^{\str}_*(Q,K)$ is defined for $K$ of an arbitrary codimension and in all degrees, compared to the string homology defined for $K$ of codimension $2$ and generated by singular chains of degree less than or equal to $2$.
The two main differences from the string homology in its construction are the reduction of the coefficient and the substitution of singular chains by \textit{de Rham chains} explained below.

The construction of $H^{\str}_*(Q,K)$ can be  briefly summarized as follows: We first choose auxiliary data including a complete Riemannian metric on $Q$.
As a graded $\R$-vector space, it is defined to be
\begin{align}\label{intro-string-homology}
H^{\str}_*(Q,K)\coloneqq  \varinjlim_{a\to \infty}\varprojlim_{(\epsilon,\delta)\in \mathcal{T}_{a}\colon \epsilon\to 0} H^{<a}_*(\epsilon,\delta),
\end{align}
where $H^{<a}_*(\epsilon,\delta)$ for $a\in \R_{> 0}\setminus \mathcal{L}(K)$ and $(\epsilon,\delta) \in \mathcal{T}_a$ is the homology of a chain complex
\begin{align}\label{chain-cpx-intro}
\left( C^{<a}_*(\epsilon)\coloneqq \bigoplus_{m=0}^{\infty} C^{\dr}_{* - m(d-2)}(\Sigma^{a+m\epsilon}_m ,\Sigma^{0}_{m}), D_{\delta}\right).
\end{align}
An explanation for each piece of the above definition is the following:
\begin{enumerate}
\item $C^{\dr}_*(X,A)$ is the $\R$-vector space of de Rham chains defined for a pair of differentiable spaces $(X,A)$. Together with the boundary operator
\[\partial \colon C^{\dr}_*(X,A)\to C^{\dr}_{*-1}(X,A) ,\]
$(C^{\dr}_*(X,A),\partial)$ becomes a chain complex. 
de Rham chains can be used as substitutions of singular chains over $\R$. Their basic properties are summarized in Section \ref{sec-de-Rham}.
The advantage is that the fiber product of de Rham chains can be defined in a natural way.
The main references are \cite{Irie-BV, irie-pseudo}.
\item For $a\in \R_{> 0}$ and $m\in \Z_{\geq 1}$, $\Sigma^a_m$ is a differentiable space of sequences $(\gamma_1,\dots ,\gamma_m)$ of paths $\gamma_k\colon [0,T_k]\to Q$ ($k= 1,\dots ,m$) with end points in $K$. It contains all $(\gamma_k)_{k=1,\dots ,m}$ whose total length is less than $a$. For the precise definition, see Section \ref{subsec-diff-sp-of-seq-of-path}.
Exceptionally, $\Sigma^a_0$ is the one point set for $a>0$ and $\Sigma^0_0$ is the empty set.
\item $D_{\delta}$ is defined by $D_{\delta}(x)\coloneqq \partial x + \sum_{k=1}^m (-1)^{p+kd+1} f_{k,\delta} (x)$ for $x\in C^{\dr}_{p-m(d-2)}(\Sigma^{a+m\epsilon}_m,\Sigma^{0}_{m})$.  Here
\[f_{k,\delta}\colon  C^{\dr}_{*}(\Sigma^{a+m\epsilon}_m,\Sigma^{0}_{m}) \to  C^{\dr}_{*+1-d}(\Sigma^{a+(m+1)\epsilon}_{m+1}, \Sigma^{0}_{m+1}) \ (k=1,\dots ,m)\]
are operators which play the key role in our construction. The idea comes from an operation of string topology explained by three steps:
\begin{enumerate}
\item[(i)] Fix a pair of short paths $(\sigma_i\colon [0,\epsilon_i]\to N_{\epsilon})_{i=1,2}$ in a tubular neighborhood $N_{\epsilon}$ of $K$ such that $\sigma_1(\epsilon_1),\sigma_2(0)\in K$ and $\sigma_1(0)=\sigma_2(\epsilon_2)$.
\item[(ii)] For any sequence of $m$-paths $(\gamma_k)_{k=1,\dots ,m}$, we split the $k$-th path $\gamma_k\colon [0,T_k]\to Q$ at a time, say $\tau$, if the image $\gamma_k(\tau)$ coincides with $\sigma_1(0)$. 
We then concatenate $\rest{\gamma_k}{[0,\tau]}$ (resp. $\rest{\gamma_k}{[\tau,T_k]}$) with $\sigma_1$ (resp. $\sigma_2$) to get a new sequence of $(m+1)$-paths
\[(\gamma_1,\dots,\gamma_{k-1}, ( \rest{\gamma_k}{[0,\tau]} \cdot \sigma_1) , (\sigma_2 \cdot\rest{\gamma_k}{[\tau,T_k]}) ,\gamma_{k+1}, \dots ,\gamma_m).\]
\item[(iii)] We extend the procedures (i), (ii) for families (or chains) of paths parameterized over manifolds.
\end{enumerate}

For the precise definition, we need to take fiber products of chains. See Section \ref{subsec-split-concatenate} and \ref{subsec-operation}. $f_{k,\delta}$ depends on a chain $\delta\in C^{\dr}_{n-d}(S_{\epsilon})$, where $S_{\epsilon}$ for $\epsilon>0$ is a differentiable space of pairs of short paths in $N_{\epsilon}$ introduced in Section \ref{subsec-diff-sp-of-seq-of-path}.
For $a\in \R_{>0}\setminus \mathcal{L}(K)$, where $\mathcal{L}(K)$ is a closed subset of Lebesgue measure $0$, a set $\mathcal{T}_{a}$ consisting of pairs $(\epsilon,\delta)$ is defined in Definition \ref{def-class-of-data}. It is necessary to prove $D_{\delta}\circ D_{\delta}=0$ for any $(\epsilon,\delta)\in \mathcal{T}_{a}$ in Proposition \ref{prop-D-2}  to define the chain complex (\ref{chain-cpx-intro}).
\item The inverse limit in (\ref{intro-string-homology}) is defined from an inverse system
\[(\{H^{<a}_*(\epsilon,\delta)\}_{(\epsilon,\delta)\in \mathcal{T}_{a}}, \{k_{(\epsilon',\delta'),(\epsilon,\delta)} \}_{\epsilon'\leq \epsilon} ).\]
Its construction is given in Section \ref{subsec-limit}. 
To define the linear map
$k_{(\epsilon',\delta'),(\epsilon,\delta)} \colon H^{<a}_*(\epsilon',\delta') \to H^{<a}_*(\epsilon,\delta)$,
we need to factor through another homology group constructed from ``$[-1,1]$-modeled de Rham chains''. Furthermore, to check its well-definedness and a claim about composition, we need one more homology group constructed from ``$[-1,1]^2$-modeled de Rham chains''. These variants of de Rham chains are introduced in Section \ref{subsec-[-1,1]-model}.
\item The inverse limit is denoted by $H^{<a}_*(Q,K)\coloneqq \varprojlim_{\epsilon\to 0} H^{<a}_*(\epsilon,\delta)$. The direct limit in (\ref{intro-string-homology}) is defined from $(\{H^{<a}_*(Q,K)\}_{a\in \R_{>0}\setminus \mathcal{L}(K)},\{I^{a,b}\}_{a\leq b} )$, where $I^{a,b}\colon H^{<a}_*(Q,K)\to H^{<b}_*(Q,K)$ is induced by the inclusion maps $\Sigma^{a+m\epsilon}_m\to \Sigma^{b+m\epsilon}_m$ for all $m\in \Z_{\geq 0}$.
See Section \ref{subsec-string-hmgy}.
\item A graded associative product structure on $H^{\str}_*(Q,K)$ is induced by natural maps $\Sigma^a_m\times \Sigma^{a'}_{m'}\to \Sigma^{a+a'}_{m+m'}$ for all $m,m'\in \Z_{\geq 0}$. The unit comes from $1\in \R=C^{\dr}_0(\Sigma^a_0,\Sigma^0_0)$ for $a>0$. See Subsection \ref{subsubsec-product}.
\end{enumerate}

A fundamental property of $H^{\str}_*(Q,K)$ is the invariance by $C^{\infty}$ isotopies of $K$.
\begin{thm}
The unital graded $\R$-algebra $H^{\str}_*(Q,K)$ is independent up to isomorphism on auxiliary data and invariant by changing the orientation of $K$.
Moreover, it is invariant by $C^{\infty}$ isotopies of $K$. (See Proposition \ref{prop-isotopy-invariance}.)
\end{thm}
We also give non-trivial computations when $Q=\R^{2d-1}$ for $d\geq 2$. For two specific submanifolds in $\R^{2d-1}$, both of which are diffeomorphic to $S^{d-1}\sqcup S^{d-1}$, we prove that our invariant is isomorphic to the homology of a finitely generated differential graded algebra. Using this computation, we obtain the next result.
\begin{thm}\label{thm-intro-example}
For every $d\geq 2$, there are two non-isotopic oriented submanifolds $K,K'$ in $\R^{2d-1}$ of codimension $d$ such that $\Lambda_{K}$ is isotopic to $\Lambda_{K'}$ as a $C^{\infty}$ submanifold with a spin structure in $UT^*\R^{2d-1}$, while $H^{\str}_*(\R^{2d-1},K)\not\cong H^{\str}_*(\R^{2d-1},K')$. (See Corollary \ref{cor-distinguish}.)
\end{thm}
The spin structure on $\Lambda_K$ for any submanifold $K$ in a spin manifold $Q$ is explained in Proposition \ref{prop-spin}.

Another purpose of this paper is to enlighten the relation to Legendrian contact homology. The following result is non-trivial from the construction.
\begin{thm}
When the codimension of $K$ is $2$ and the normal bundle of $K$ is trivial, $H^{\str}_0(Q,K)$ is isomorphic to the cord algebra of $(Q,K)$ over $\R$. (See Corollary \ref{thm-isom-cord-string}.)
\end{thm}
If $K$ is connected, the cord algebra over $\R$ we consider in this paper is a reduction of the cord algebra over $\Z[H_1(\Lambda_K)]$ defined by Ng in \cite{ng}.
Combining with the result by Ekholm, Etnyre, Ng, and Sullivan in \cite{eens}, the cord algebra for a knot $K$ in $\R^3$ was proved to be isomorphic to the $0$-th degree part of the Legendrian contact homology of $(UT^*\R^3,\Lambda_K)$. Later, another direct proof was given in \cite{celn}.

The author makes the following more radical conjecture when $Q=\R^n$. 
\begin{conj}\label{conj-intro}
For any compact oriented submanifold $K$ in $\R^n$,
$H^{\str}_*(\R^n,K)$ is isomorphic to the Legendrian contact homology of $(UT^*\R^n,\Lambda_K)$ with coefficients in $\R$. 
\end{conj}

The Legendrian contact homology with coefficients in $\R$ is an invariant of Legendrian submanifolds with a spin structure\cite{ees-ori,ees}. If Conjecture \ref{conj-intro} is true, then our invariant can be applied to study the contact topology of $UT^*\R^n$. For instance, assuming this conjecture, Theorem \ref{thm-intro-example} would imply that $\Lambda_{K}$ is not isotopic to $\Lambda_{K'}$  as a Legendrian submanifold with a spin structure.

\

\noindent
\textbf{Organization of paper.}

In Section \ref{sec-de-Rham}, general notions of a differentiable space and its de Rham chain complex are introduced. 
In Section \ref{subsec-diff-sp-of-seq-of-path}, the differentiable spaces $\Sigma^a_m$ and $S_{\epsilon}$ are defined. Their de Rham chain complexes are observed in Section \ref{subsec-hmgy-grp}.
Through Section \ref{subsec-split-concatenate} and \ref{subsec-operation}, the operator $f_{k,\delta}$ is defined. 
In Section \ref{subsec-[-1,1]-model}, $[-1,1]$-modeled and $[-1,1]^2$-modeled de Rham chains for path spaces are introduced.
In Section \ref{subsec-def-of-chain-cpx}, we define the chain complexes (\ref{chain-cpx-intro}) and give a couple of computations.
In Section \ref{subsec-chain-cpx-by-[-1,1]}, we consider their variants using those chains in Section \ref{subsec-[-1,1]-model}. They are necessary to define the map $k_{(\epsilon',\delta'), (\epsilon,\delta)}$ in Section \ref{subsec-limit}. 
The definition of $H^{\str}_*(Q,K)$ is given in Section \ref{subsec-string-hmgy}. The independence on auxiliary data is checked in Section \ref{subsec-invariance}, from which the isotopy invariance follows immediately.
In Section \ref{sec-example}, we examine the algebraic structure of $H^{\str}_*(\R^{2d-1},K)$ when $K$ is a higher-dimensional generalization of the Hopf link or the unlink in $\R^3$.
%
In Section \ref{subsec-cord-alg-and-string-hmgy},  refering to \cite{celn}, we define the cord algebra and its another description as the $0$-th degree part of the string homology. In Section \ref{subsec-construction-of-chain-map}, we construct a graded map from the string homology to $H^{\str}_*(Q,K)$. In Section \ref{subsec-proof-of-isom}, this map is proved to be an isomorphism on the $0$-th degree part.

%
\begin{acknow}
The author would like to express his deep gratitude to his supervisor Kei Irie for spending hours of discussion and giving so much valuable suggestions and continuous encouragement.
This work was supported by JST, the establishment of university fellowships towards 
the creation of science technology innovation, Grant Number JPMJFS2123. 
Part of this work was supported by the WINGS-FMSP program at the Graduate School of Mathematical Science, the University of Tokyo.
\end{acknow}

\section{Differentiable space and de Rham chains}\label{sec-de-Rham}

In this section, the notions of differentiable spaces and de Rham chains are introduced. We also summarize results applied in the latter sections
\begin{rem}
The notion of differentiable space goes back to \cite{chen} by K.-T. Chen. The notion of de Rham chains was proposed by Fukaya in \cite{fuk}, and later, Irie gave the definition in \cite{Irie-BV, irie-pseudo}.
We mainly refer, especially about sign conventions, to \cite{irie-pseudo}. As is mentioned in \cite[Remark 4.1]{Irie-BV}, the definition of plots (elements of a differentiable structure) in this paper is different from that of \cite{chen}.
\end{rem}

\subsection{Notations and conventions}

For $m,N\in \Z_{\geq 0}$, let $\mathcal{U}_{m,N}$ be the set of oriented $m$-dimensional submanifolds of $\R^N$. We then define $\mathcal{U}\coloneqq \bigcup_{m,N\in \Z_{\geq 0}}\mathcal{U}_{m,N}$.
Let us fix a few conventions about orientations.
If we write $\R^n$ for $n\in \Z_{\geq 1}$, this means the manifold $\R^n\in \mathcal{U}_{n,n}$ whose orientation is given so that $dx_1\wedge \dots \wedge dx_n$ is a positive volume form when $(x_1,\dots ,x_n)$ is the standard coordinate of $\R^n$.
If we write $\{0\}$, this means $\{0\}\in \mathcal{U}_{0,0}$ with a positive sign assigned.

Let us think about the orientation of fiber products of oriented manifolds.
For $U,V,M\in \mathcal{U}$, suppose that there are two $C^{\infty}$ maps $f\colon U\to M$, $g\colon V\to M$. We also assume that $g$ is a submersion. (Here after, all submersions are of class $C^{\infty}$.) Then, the fiber product
\[U \ftimes{f}{g} V\coloneqq \{(u,v)\in U\times V\mid f(u)=g(v)\}\]
is a $C^{\infty}$ submanifold of $U\times V$.
In order to determine the orientation at $(u,v)\in U\ftimes{f}{g} V$, we take a right inverse $s\colon T_{g(v)}M\to T_vV$ of $(dg)_v$ (i.e. $(dg)_v\circ s =\id_{T_{g(v)}M}$). Then, there are two isomorphisms
\begin{align*}
&T_{g(v)}M\times \ker (dg)_v \to T_vV\colon (z,y) \mapsto s(z)+y , \\
&T_uU\times  \ker (dg)_v \to T_{(u,v)} (U \ftimes{f}{g} V ) \colon (x,y) \mapsto (x,y+s\circ (df)_u(x)).
\end{align*}
The orientations of $\ker (dg)_v$ and $T_{(u,v)} (U \ftimes{f}{g} V) $ are determined so that the above isomorphisms preserve orientations. Of course, when $X$ and $Y$ are oriented $\R$-vector spaces, we assign the product orientation on $X\times Y$.
In particular, when $M=\{0\}$, this gives the orientation of the product manifold $U\times V$.

For $U\in \mathcal{U}$, $\Omega_c^p(U)$ is the vector space of compactly supported $C^{\infty}$ differential $p$-forms on $U$. When $p<0$ or $p>\dim U$, we define $\Omega^p_c(U)\coloneqq 0$.
For $U,U'\in \mathcal{U}$ and a submersion $\pi \colon U'\to U$, we have an $\R$-linear map
\[\pi_!\colon \Omega^p_c(U')\to \Omega^{p-(\dim U' -\dim U)}_c(U),\]
called the \textit{integration along fibers}.
When $U'=\R^d\times \R^k$, $U=\R^k$ and $\pi (t,x)=x$ for $(t,x)\in U'$, this map is characterized by the following:
For $f\in \Omega^0_c(U')$, $1\leq i_1<\dots <i_a\leq d$ and $1\leq j_1<\dots <j_b\leq k$, if we take $\omega\coloneqq f (dt_{i_1}\wedge \dots \wedge dt_{i_a} \wedge dx_{j_1}\wedge \dots \wedge dx_{j_b})$, then for every $x\in U$,
\[\left( \pi_!(\omega) \right)_x=
\begin{cases}
0 &  \text{ if }a<d , \\
\left( \int_{\R^d}  f(\cdot,x )dt_1\wedge \dots \wedge dt_d \right) (dx_{i_1}\wedge \dots \wedge dx_{i_a})_x &\text{ if }a=d .
\end{cases}\]
For an arbitrary submersion $\pi\colon U'\to U$, $\pi_!$ is defined by taking local charts and a partition of unity on $U$.
%


\subsection{de Rham chain complex}

\subsubsection{Differentiable space}
We proceed to the definition of differentiable spaces.

\begin{defi}
Let $X$ be a set and $P_X$ be a set of pairs $(U,\varphi)$ of $U\in \mathcal{U}$ and a map $\varphi \colon U \to X$.
We say $P_X$ is a \textit{differentiable structure} on $X$ if it satisfies the following condition:
\begin{itemize}
\item For any $(U,\varphi)\in P_X$, $U'\in \mathcal{U}$ and a submersion $\pi \colon U'\to U$, the pair $(U',\varphi \circ \pi)$ is also an element of $P_X$.
\end{itemize}
We call such pair $(X,P_X)$ a \textit{differentiable space}. An element of $P_X$ is called a \textit{plot} of $(X,P_X)$.
\end{defi}

\begin{ex}\label{ex-diff}
Let $M$ be a manifold. There are two types of canonical differentiable structures on $M$:
\[\begin{cases}
P_M&\coloneqq \{(U,\varphi)\mid \varphi \colon U\to M \text{ is a $C^{\infty}$ map}\}, \\
P^{\reg}_M&\coloneqq \{(U,\varphi) \mid \varphi\colon U\to M \text{ is a submersion}\}.
\end{cases}\]
Clearly, $(M,P_M)$ and $(M,P^{\reg}_M)$ are differentiable spaces. The latter is denoted by $M^{\reg}$. We consider the differentiable structure $P_M$ for any manifold $M$, unless we declare to use $M^{\reg}$.
\end{ex}

\begin{defi}\label{def-prod-subsp} Let $(X,P_X)$, $(Y,P_Y)$ be differentiable spaces and $Z$ be a subset of $X$. Denote the projection map from $X\times Y$ to $X$ (resp. $Y$) by $\pr_X$ (resp. $\pr_Y$), and the inclusion map from $Z$ to $X$ by $\iota_Z$.
\begin{enumerate}
\item We define differentiable structures on $X\times Y$ and $Z$ by
\begin{align*}
P_{X\times Y}&\coloneqq \{(U,\varphi) \mid \pr_X\circ \varphi\in P_X \text{ and } \pr_Y\circ \varphi \in P_Y\} , \\
P_Z &\coloneqq \{(U,\varphi) \mid \iota_Z \circ \varphi \in P_X\}.
\end{align*}
\item Let $f\colon X\to Y$ be a map. We say $f$ is a \textit{smooth map} if $(U,f\circ \varphi)\in P_Y$ for any $(U,\varphi)\in P_X$.
\end{enumerate}
\end{defi}
In the case of the above definition, we simply call $(Z,P_Z)$ a \textit{subspace} of $(X,P_X)$. Note that
given a set $W$ and two maps $f \colon X\to W$ and $g\colon Y\to W$, the fiber product $X\ftimes{f}{g} Y$ becomes a differentiable space as a subspace of $(X\times Y, P_{X\times Y})$.

\subsubsection{de Rham chains}
Next, we introduce the notion of de Rham chain complex. Here after, if we say that $X$ is a differentiable space, this means that $X$ is equipped with a differentiable structure denoted by $P_X$.

Let $X$ be a differentiable space. We consider a graded $\R$-vector space
\[A_*(X)\coloneqq \bigoplus_{(U,\varphi)\in P_X}\Omega_c^{\dim U -*}(U).\]
For $(U,\varphi)\in P_X$ and $\omega\in \Omega^{\dim U-*}_c(U)$, let $(U,\varphi,\omega)$ denote the element of $A_*(X)$ such that its component for $(V,\psi)\in P_X$ is
\[(U,\varphi,\omega)_{(V,\psi)}=\begin{cases}
\omega &\text{ if }(V,\psi)=(U,\varphi) , \\
0 & \text{ if } (V,\psi) \neq (U,\varphi).
\end{cases}\]
We take a linear subspace $Z_*(X)$ of $A_*(X)$ generated by
\[\{ (U',\varphi \circ \pi , \omega)- (U, \varphi, \pi_!\omega)\mid (U,\varphi)\in P_X \text{ and }\pi \colon U'\to U \text{ is a submersion} \}.\]
Then we define a quotient vector space
\[C^{\dr}_*(X)\coloneqq A_*(X)/Z_*(X).\]
The equivalence class of $(U,\varphi,\omega)\in A_*(X)$ in $C^{\dr}_*(X)$ is denoted by $[U,\varphi,\omega]$. 
We also define an $\R$-linear map $\partial \colon C^{\dr}_*(X)\to C^{\dr}_{*-1}(X)$ of degree $(-1)$ by
\[\partial [U,\varphi,\omega] \coloneqq (-1)^{|\omega|+1} [U,\varphi,d\omega] . \]
This map is well-defined and $\partial \circ \partial =0$ holds. 
$(C^{\dr}_*(X), \partial)$ is called the \textit{de Rham chain complex} of a differentiable space $X$, and its elements are called \textit{de Rham chains} of $X$.
By taking its homology, we obtain
\[H^{\dr}_*(X) \coloneqq H_*(C^{\dr}_*(X),\partial).\]
In addition, a functoriality holds. Namely, any smooth map $f\colon X\to Y$ induces a chain map
\[f_*\colon C^{\dr}_*(X)\to C^{\dr}_*(Y)\colon [U,\varphi, \omega]\mapsto [U,f\circ \varphi, \omega].\]
\begin{rem}The following are fundamental techniques to compute de Rham chains:
\begin{enumerate}
\item For $[U,\varphi,\omega]\in C^{\dr}_*(X)$, suppose that $V\subset U$ is an open subset containing $\supp \omega$. Then $[U,\varphi,\omega]=[V,\rest{\varphi}{V}, \rest{\omega}{V}]\in C^{\dr}_*(X)$.
\item If $(\R\times U, \varphi)\in P_X$ satisfies $\varphi(s,\cdot) =\begin{cases} \varphi_0 & \text{ if }s\leq 0 , \\ \varphi_1 & \text{ if }s\geq 1, \end{cases}$ for some $(U,\varphi_0),(U,\varphi_1)\in P_X$, then
\[ \partial [\R\times U,\varphi,(-1)^{ |\omega| } \chi\times \omega] = [U,\varphi_1,\omega] - [U,\varphi_0,\omega] \in C^{\dr}_*(X)\]
for a closed form $\omega\in \Omega^{\dim U-*}_c(U)$ and $\chi \colon \R\to [0,1]$ such that $\supp \chi$ is compact and $\chi(s)=1$ for every $s\in [0,1]$.
\end{enumerate}
\end{rem}

\begin{ex}\label{ex-hmgy-mfd-reg}
Let $M$ be an oriented smooth manifold. The de Rham chain complex of $M^{\reg}$ is naturally isomorphic to $(\Omega_c^{\dim M-*}(M),d)$ through the map
\[C_p^{\dr}(M^{\reg})\to \Omega_c^{\dim M-p}(M) \colon [U,\varphi,\omega] \mapsto (-1)^{s(p)} \varphi_!\omega. \]
Here $s(p)\coloneqq (p-\dim M)(p-\dim M -1)/2$.
Hence $H^{\dr}_*(M^{\reg})$ is isomorphic to the compactly supported de Rham cohomology $H^{\dim M-*}_{c,\dr}(M)$.
\end{ex}


Let us define the de Rham chain complex for a pair of differentiable spaces. A smooth map $f\colon X\to \R$ is said to be  \textit{approximately smooth} if there exists a decreasing sequence $(f_j)_{j\in \Z_{\geq 1}}$ of smooth maps from $X$ to $\R$ such that $\lim_{j\to \infty}f_j(x) = f(x)$ for every $x\in X$. The following lemma is proved in \cite[Lemma 4.11]{Irie-BV}.
\begin{lem}\label{lem-approx-smooth-level}
For an approximately smooth function $f\colon X\to \R$, let $X^a\coloneqq f^{-1}( (-\infty,a))$ for every $a\in \R\cup \{\infty\}$. Then for $a,b\in \R\cup \{\infty\}$ with $a\leq b$, the linear map $i_*\colon C^{\dr}_*(X^a )\to C^{\dr}_*( X^b )$,
which is induced by the inclusion map $i\colon X^a\to X^b$, is injective.
\end{lem}
In the case of the above lemma, we define a quotient complex
\[C^{\dr}_*(X^b,X^a) \coloneqq C^{\dr}_*(X^b)/i_*(C^{\dr}_*(X^a)).\]
Its homology is denoted by $H^{\dr}_*(X^b,X^a)$.


Next, we define a fiber product of de Rham chains.

\begin{defi}\label{def-fib-prod-ope}
Let $(X,P_X)$ and $(Y,P_Y)$ be differentiable spaces. Suppose that we have an oriented manifold $M$ of dimension $n$ and
two smooth maps
\[f\colon (X,P_X) \to (M,P_M) ,\ g\colon (Y,P_Y)\to (M,P^{\reg}_M)=M^{\reg}\]
Then, we define a linear map
\[C^{\dr}_{p+n}(X)\otimes C^{\dr}_{q+n}(Y) \to C^{\dr}_{p+q+n}(X\ftimes{f}{g} Y)\colon x\otimes y \to x \ftimes{f}{g} y \]
by
\[x\ftimes{f}{g} y\coloneqq (-1)^{p|\eta|}[ W , \rest{(\varphi\times \psi)}{W} , \rest{ (\omega\times \eta) }{W}]\]
for $x=[U,\varphi,\omega]\in C^{\dr}_{p+n}(X)$ and $y=[V,\psi,\eta]\in C^{\dr}_{q+n}(Y)$. Here, $W \coloneqq U \ftimes{f\circ \varphi}{g\circ \psi} V$ is a fiber product over $M$.
\end{defi}

It is straightforward to check the well-definedness of $x\ftimes{f}{g} y$. It can also be checked that
\[\partial(x\ftimes{f}{g} y)=(\partial x)\ftimes{f}{g} y +(-1)^px\ftimes{f}{g} (\partial y)\]
holds for any $x\in C^{\dr}_{p+n}(X)$ and $y\in C^{\dr}_{q+n}(Y)$.
When $M=\{0\}$, we simply write $x\ftimes{f}{g} y$ by $x\times y$.

\subsubsection{Collection of results about de Rham chain complex}

In the rest of this section, let us summarize a couple of basic results about de Rham chain complex.
The first result can be compared with the computation for $M^{\reg}$ in Example \ref{ex-hmgy-mfd-reg}. Hereafter, $H^{\sing}(\cdot)$ denotes the singular homology with coefficients in $\R$.
\begin{prop}\label{prop-hmgy-of-mfd}
For every oriented manifold $M$, there exists a canonical isomorphism
\[\Psi_{M}\colon H^{\sing}_*(M)\to H^{\dr}_*(M)\]
such that for any $C^{\infty}$ map $f\colon M\to N$ between oriented manifolds, $\Psi_N\circ f_*=f_*\circ \Psi_M$ holds.
\end{prop}
For the details of the construction of $\Psi_M$, see \cite[Section 4.7]{Irie-BV}. It is the composition of a natural isomorphism between $H^{\sing}_*(M)$ and $H^{\sm}_*(M)$ (the homology of smooth singular chains in $M$) and a canonical map from $H^{\sm}_*(M)$ to $H^{\dr}_*(M)$.
For the proof that $\Psi_M$ is an isomorphism, see \cite[Section 5]{Irie-BV}.
This result can be extended to relative homology groups for $(M,N)$, where $N$ is an open submanifold of $M$ such that $N=f^{-1}((-\infty,a))$ for some approximately smooth map $f\colon M\to \R$.

Next, let $f,g\colon X\to Y$ be smooth maps between differentiable spaces. We say $f$ is \textit{homotopic} to $g$ if there exists a smooth map $H\colon \R\times X\to Y $ such that $H(t, x)=f(x)$ for $t\leq 0$ and $H(t,x)=g(x)$ for $t\geq 1$.
Then we have the following result. For the proof, see \cite[Proposition 4.7]{Irie-BV}. 
\begin{prop}\label{prop-chain-htpy}
For two smooth map $f,g \colon X\to Y$, if $f$ is homotopic to $g$, then there exists a chain homotopy $K\colon C^{\dr}_*(X)\to C^{\dr}_{*+1}(Y)$ such that $\partial K + K\partial =f_* -g_*$. In particular, $f_*=g_*\colon H^{\dr}_*(X)\to H^{\dr}_*(Y)$ holds.
\end{prop}
\begin{rem}For three smooth maps $f,g,h\colon X\to Y$ such that $f$ is homotopic to $g$ and $g$ is homotopic to $h$, we can ask whether $f$ is homotopic to $h$. In fact, if the differentiable structure $P_Y$ of $Y$ satisfies the following condition, such transitivity holds (The proof is straightforward.):
\begin{itemize}
\item For any $U\in \mathcal{U}$ and  $(U_1,\varphi_1),(U_2,\varphi_2)\in P_Y$ such that $(U_i)_{i=1,2}$ is an open cover of $U$ and $\rest{\varphi_1}{U_1\cap U_2}= \rest{\varphi_2}{U_1\cap U_2}$, $(U,\varphi)\in P_Y$ holds for $\varphi\colon U\to Y$ which maps $u\in U_i$ to $\varphi_i(u)$ ($i=1,2$).
\end{itemize}
All differentiable spaces appearing after Section \ref{sec-sp-of-paths} satisfy this condition.
However, as mentioned in \cite[Remark 4.4]{Irie-BV}, it seems difficult in general case to prove such transitivity.
\end{rem}

The last one is a result about excisions.
\begin{prop}\label{prop-excision}
Let $X$ be a differentiable space and $Y=f^{-1}((-\infty,a))\subset X$ for some approximately smooth function $f\colon X\to \R$ and $a\in \R$. Suppose there is another  approximately smooth function $g\colon X\to \R$ and $b_0\in \R$ such that $g^{-1}((b_0,\infty))\subset Y$. For every $b>b_0$, let $X^b\coloneqq g^{-1}((-\infty,b))$ and $Y^b\coloneqq (\rest{g}{Y})^{-1}((-\infty,b))$. Then, the inclusion map of pairs $i\colon (X^b,Y^b)\to (X,Y)$ induces an isomorphism 
$i_*\colon C^{\dr}_*(X^b,Y^b) \to C^{\dr}_*(X,Y)$.
\end{prop}
\begin{proof}
We first prove the assertion when $g\colon X\to \R$ is a smooth map. For $b> b_0$, choose $\delta>0$ and a smooth function $\kappa\colon \R\to [0,1]$ such that $2\delta<b-b_0$ and $\kappa(b')=\begin{cases} 1 & \text{ if } b'\leq b_0+\delta, \\ 0 & \text{ if }b'\geq b-\delta .\end{cases}$ Then we define a linear map
\[ r\colon C^{\dr}_*(X) \to C^{\dr}_*(X^b) \colon [U,\varphi,\omega] \mapsto [U^b ,\rest{\varphi}{U^b}, (\kappa \circ g\circ \varphi)\cdot \rest{\omega}{U^b}] ,\]
where $U^b\coloneqq (g\circ \varphi)^{-1} ((-\infty,b))$.
This reduces to a map  $\bar{r}\colon C^{\dr}_*(X,Y) \to C^{\dr}_*(X^b,Y^b)$. We claim that $\bar{r}$ is the inverse map of $i_*$. Indeed, for any $x=[U,\varphi,\omega]\in C^{\dr}_*(X)$, we have
\begin{align*}
 x-i_*\circ \bar{r}(x) =& [U,\varphi, \omega]- [U,\varphi, (\kappa \circ g\circ \varphi)\cdot\omega] \\
 = & [U_0 ,\rest{\varphi}{U_0}, ((1-\kappa) \circ g\circ \varphi)\cdot \rest{\omega}{U_0} ] 
\in C^{\dr}_*(Y) \text{ for }U_{0}\coloneqq (g\circ \varphi)^{-1}((b_0,\infty)) .  \end{align*}
Similarly, we can show that $x-\bar{r}\circ i_*(x)\in C^{\dr}_*(Y^b) $ for $x\in C^{\dr}_*(X^b)$.

In a general case, there exists a decreasing sequence $(g_j)_{j\geq 1}$ of smooth maps $g_j\colon X\to \R$ such that $g_j(x) \to g(x)$ ($j\to \infty$) for every $x\in X$. For $b>b_0$, let $X^b_j\coloneqq g_j^{-1}((-\infty,b))$ and $Y_j^b\coloneqq (\rest{g_j}{Y})^{-1}((-\infty,b))$. From \cite[Corollary 4.12 (i)]{Irie-BV},
$\varinjlim_{j\to \infty} C^{\dr}_*(X^b_j,Y^b_j) \to C^{\dr}_*(X^b,Y^b)$, induced by inclusion maps, is an isomorphism. We have shown that $\left( \rest{i}{(X^b_j,Y^b_j)}\right)_* \colon C^{\dr}_*(X^b_j,Y^b_j) \to C^{\dr}_*(X,Y)$ is an isomorphism for every $j\geq 1$, so $i_*$ is also an isomorphism .
\end{proof}

\section{Differentiable space of paths and operations from string topology}\label{sec-sp-of-paths}

Throughout this paper, $Q$ is a manifold of dimension $n$, and $K$ is a compact submanifold of $Q$ of codimension $d\geq 1$.
In addition, both $Q$ and $K$ are required to have fixed orientations.
The construction of $H^{\str}_*(Q,K)$ depends on the following auxiliary data:
\begin{enumerate}
\item a complete Riemannian metric $g$ on $Q$. (We write $g(v,w)=\la v,w\ra_g$ and $\sqrt{g(v,v)}=|v|_g$.)
\item a constant $C_0\geq 1$.
\item a positive real number $\epsilon_0 $ for which the map
\begin{align}\label{tubular-neighborhood}
  \{(x,v)\in (TK)^{\perp} \mid |v|_g <\epsilon_0 \} \to Q \colon (x,v) \mapsto \exp_x(v)
\end{align}
is an open embedding.
\item a $C^{\infty}$ function $\mu \colon [0,\frac{3}{2} ] \to [0,1]$ such that $\mu (t)=\begin{cases} t & \text{ near }t=0, \\ 1 & \text{ near }t=\frac{3}{2}, \end{cases}$ and $0\leq \mu'(t)\leq 1$ for every $t\in [0,\frac{3}{2}]$.
\end{enumerate}
%
The independence of $H^{\str}_*(Q,K)$ on these data up to isomorphism is proved in Section \ref{subsec-invariance}. Until then, these data are fixed, so $\la v,w\ra_g$ and $|v|_g$ are denoted by $\la v,w\ra$ and $|v|$ respectively.

We define $\mathcal{C}(K)$ to be the set of geodesics $\gamma\colon [0,T]\to Q$ with unit speed such that $\gamma(0),\gamma(T)\in K$ and $\gamma'(0)\in (T_{\gamma(0)}K)^{\perp}$, $\gamma'(T)\in (T_{(\gamma(T)}K)^{\perp}$. Such geodesics are called \textit{binormal chords} of $K$. We also define for $m\in \Z_{\geq 1}$
\begin{align*}
\mathcal{L}_m (K)&\coloneqq \left\{ \textstyle{\sum_{k=1}^m}\len \gamma_m \mid \gamma_1,\dots ,\gamma_m\in \mathcal{C}(K) \right\}, \\
\mathcal{L}(K)&\coloneqq \textstyle{\bigcup_{m=1}^{\infty} }\mathcal{L}_m (K).
\end{align*}
These are closed subsets of $\{a\in \R\mid a\geq  2\epsilon_0\}$, since $K$ is compact.
Moreover, they are null sets with respects to the Lebesgue measure. For the proof, see Lemma \ref{lem-measure0}.

\subsection{Differentiable space of paths}\label{subsec-diff-sp-of-seq-of-path}
In this section, we introduce two differentiable spaces of paths, $\Sigma^a_m$ and $S_{\epsilon}$.
Let $\Omega_K(Q)$ be the set of $C^{\infty}$ paths $\gamma\colon [0,T]\to Q$ with $T>0$ such that $\gamma(0),\gamma(T)\in K$ and $|\gamma'(t)|\leq C_0$ for any $t\in [0,T]$. For any $C^{\infty}$ path $\gamma\colon [0,T]\to Q$, its length is denoted by
\[\len \gamma \coloneqq \int_0^{T}|\gamma' (t)|dt .\]

%
For $a\in \R_{\geq 0}\cup\{\infty\}$ and $m\in \Z_{\geq 1}$, we define $\Sigma_m^a$ to be a subset of $\Omega_K(Q)^{\times m}$ which consists of $(\gamma_k\colon [0,T_k]\to Q)_{k=1,\dots, m}$ satisfying \textit{either} of the following two conditions:
\begin{itemize}
\item $\sum_{k=1}^m \len \gamma_k <a $.
\item $\min_{1\leq k\leq m} \len \gamma_k <\epsilon_0$.
\end{itemize}
The differentiable structure on $\Sigma^a_m$ is defined by
\[P_{\Sigma^a_m} \coloneqq \{(U,\varphi) \mid U\in \mathcal{U} \text{ and } \varphi \colon U \to \Sigma^a_m \text{ is smooth}\}.\]
Here, we say $\varphi$ is smooth in the following sense: If we write $\varphi(u) =(\gamma^u_k\colon [0,T^u_k] \to Q )_{k=1,\dots ,m}$ for $u\in U$, then for each $k\in \{1,\dots ,m\}$, the function $U \to \R_{>0}\colon u\mapsto T^u_k$ is of class $C^{\infty}$ and
\[\{(u,t)\in U\times \R\mid 0\leq t \leq T^u_k\} \to Q \colon (u,t)\mapsto \gamma^u_k(t)\]
is a $C^{\infty}$ map.
As an exception, let us define $\Sigma^a_0\coloneqq \begin{cases} \{*\}& \text{ if }a>0, \\ \varnothing &\text{ if }a=0,\end{cases}$ together with the differentiable structure $P_{\Sigma^a_0}\coloneqq \{(U,\varphi) \mid U\in \mathcal{U},\ \varphi\colon U\to \Sigma^a_0\}$.

We consider the de Rham chain complex $(C^{\dr}_*(\Sigma^a_m),\partial)$ for $a\in \R_{\geq 0}$ and $m\in \Z_{\geq 0}$. Lemma \ref{lem-approx-smooth-level} implies that we may think of $C^{\dr}_*(\Sigma^{a}_m)$ as a linear subspace of $C^{\dr}_*(\Sigma^b_m)$ when $a\leq b$, since
\[\begin{cases}
\Sigma^b_m \to \R\colon (\gamma_l)_{l=1,\dots ,m} \mapsto \len\gamma_k \ (k=1,\dots ,m), \\
 \R^m \to \R \colon (a_{k})_{k=1,\dots ,m}\mapsto \min_{1\leq k\leq m} a_k ,
 \end{cases}\]
are approximately smooth functions. Thus the quotient complex $(C^{\dr}_*(\Sigma^b_m,\Sigma^a_m),\partial )$ is defined.

\begin{rem}\label{rem-large-m}
When $a\leq m\epsilon_0$, the condition that $\sum_{k=1}^m\len \gamma_k<a$ implies that one of $\gamma_k$ ($k=1,\dots ,m$) has length less than $\epsilon_0$. Thus, $\Sigma^a_m=\Sigma^0_m$ if $a\leq m\epsilon_0$.
When $a=\infty$, $\Sigma^{\infty}_m=\Omega_K(Q)^{\times m}$, which will be used only in Section \ref{sec-cord-alg}.
\end{rem}

Next, we define another differentiable space of paths.
For every $\epsilon \in (0,\epsilon_0]$, the open subset in $Q$
\[N_{\epsilon}\coloneqq \{\exp_x(v) \mid x\in K, v\in (T_xK)^{\perp} \text{ and } |v| <\epsilon/2\}\]
is a tubular neighborhood of $K$ in $Q$.
Then we define a set $S_{\epsilon}$ which consists of pairs of $C^{\infty}$ paths $(\sigma_i\colon [0,\epsilon_i]\to N_{\epsilon})_{i=1,2} $ satisfying:
\begin{itemize}
\item $0< \epsilon_i \leq \epsilon/2$ for $i=1,2$.
\item $\sigma_1(\epsilon_1), \sigma_2(0)\in K $ and $\sigma_1(0)=\sigma_2(\epsilon_2)$.
\item $|\sigma'_i(t)|\leq 1 $ for $i=1,2$ and any $t\in [0,\epsilon_i ]$.
\end{itemize}
On this set, the evaluation map $\ev_0$ is defined by
\[\ev_0\colon S_{\epsilon} \to  N_{\epsilon} \colon (\sigma_1,\sigma_2) \mapsto \sigma_1(0) .\]
The differentiable structure on $S_{\epsilon}$ is defined by
\[P_{S_{\epsilon}} \coloneqq \{(V,\psi) \mid V\in \mathcal{U},\ \psi \text{ is a smooth map such that } \ev_0\circ \psi \colon V\to N_{\epsilon} \text{ is a submersion}\}\]
Here we say $\psi$ is smooth in the following sense: If we write $\psi(v) =(\sigma^v_i\colon [0,\epsilon^v_i] \to N_{\epsilon} )_{i=1,2}$ for $v\in V$, then for $i\in \{1,2\}$, the function $V \to \R_{>0}\colon v\mapsto \epsilon^v_i$ is of class $C^{\infty}$ and
\[\{(v,t)\in V\times \R\mid 0\leq t \leq \epsilon^v_i\} \to N_{\epsilon} \colon (v,t)\mapsto \sigma^v_i(t)\]
is a $C^{\infty}$ map.
Note that $\ev_0$ is a smooth map from $(S_{\epsilon},P_{S_{\epsilon}})$ to $(N_{\epsilon},P^{\reg}_{N_{\epsilon}})=N_{\epsilon}^{\reg}$ defined in Example \ref{ex-diff}.

\subsection{Homology groups}\label{subsec-hmgy-grp}

In this section, we examine the homology groups 
$H^{\dr}_*(\Sigma^b_m,\Sigma^a_m)$  and $H^{\dr}_*(S_{\epsilon})$. The main results are Proposition \ref{prop-hmgy-grp} and Proposition \ref{prop-exactness}. At the end, several additional results are proved.

\subsubsection{Finite dimensional approximation of $\Sigma^a_m$}
Let us fix $b_0\in \R_{> 0}$ and prepare several notations related to the Riemannian metric $g$.
We note that there is a compact subset of $Q$ which contains the images of all paths $\gamma\in \Omega_K(Q)$ with $ \len \gamma\leq b_0$, since $K$ is compact and $g$ is complete.
For any two points $q,q'\in Q$, let $d(q,q')$ be the distance between $q$ and $q'$. Let us also fix
$\epsilon_g>0$ such that if $q$ and $q'$ in this compact set satisfy $d(q,q')<\epsilon_g $, then there exists a unique geodesic path on $[0,1]$ of length $d(q,q')$ from $q$ to $q'$. We write this geodesic by $\overline{qq'}\colon [0,1]\to Q$.


For every $a\in [0,b_0)$ and $m\in \Z_{\geq 1}$, let $\bar{\Sigma}^a_m$ be a subspace of $\Sigma^a_m$ which consists of $(\gamma_k)_{k=1,\dots , m}$ satisfying $\sum_{k=1}^m\len \gamma_k <b_0$. 
From Proposition \ref{prop-excision} about excision, the inclusion map $\iota \colon \bar{\Sigma}^b_m \to \Sigma^b_m$ induces an isomorphism
\begin{align}\label{isom-excision}
\iota_*\colon H^{\dr}_*(\bar{\Sigma}^b_m,\bar{\Sigma}^a_m) \to H^{\dr}_*(\Sigma^b_m,\Sigma^a_m)
\end{align}
for $a,b\in [0,b_0)$ with $a\leq b$. This means that we cut out a subset
\[\{(\gamma_k)_{k=1,\dots ,m} \mid \min_{1\leq k\leq m}\len \gamma_k <\epsilon_0 \text{ and }\sum_{k=1}^m \len \gamma_k \geq b_0\}\]
to compute the homology group.

First, we approximate $\bar{\Sigma}^a_m$ for $a\in [0,b_0)$ by finite dimensional manifolds.
For every $\nu\in \Z_{\geq 1}$, let us define
\begin{align}\label{time-length-nu}
\bar{\Sigma}^a_m(\nu)\coloneqq \{ (\gamma_k\colon [0,T_k]\to Q)_{k=1,\dots ,m}\in \bar{\Sigma}^a_m \mid \max_{1\leq k\leq m}T_k< C_0^{-1}\epsilon_g \nu \} 
\end{align}
so that $\bigcup_{\nu =1}^{\infty}\bar{\Sigma}^a_m(\nu)= \bar{\Sigma}^a_m$. 
Let us also define $B_m(\nu)$ to be a submanifold of $(Q^{\times (\nu+1)})^{\times m}$ which consists of $(q^l_k)_{k=1,\dots , m}^{l=0,\dots, \nu}$ satisfying:
\begin{itemize}
\item $q^0_k,q^{\nu}_k\in K$ for every $k=1,\dots ,m$.
\item $\sum_{k=1}^m\sum_{l=0}^{\nu-1} d(q_k^l,q_k^{l+1}) <b_0 $.
\item $d( q_k^l,q_k^{l+1} )<\epsilon_g$ for every $k=1,\dots, m$ and $l=0,\dots ,\nu-1$.
\end{itemize}
We then define $B^a_m(\nu)$ for $a<b_0$ to be an open submanifold of $B_m(\nu)$ which consists of  $(q^l_k)_{k=1,\dots , m}^{l=0,\dots ,\nu}$ satisfying \textit{either} of the following two conditions:
\begin{itemize}
\item $\sum_{k=1}^m\sum_{l=0}^{\nu-1} d(q_k^l,q_k^{l+1}) <a $.
\item $\min_{1\leq k\leq m} \left( \sum_{l=0}^{\nu-1} d(q_k^l,q_k^{l+1})\right) <\epsilon_0 $.
\end{itemize}
The differentiable structures on this manifold is $P_{B^a_m(\nu)}$ defined in Example \ref{ex-diff}.
For every $\nu \in Z_{\geq 1}$, there are two maps:
$\iota_{\Sigma,\nu}\colon \bar{\Sigma}^a_m(\nu)\to \bar{\Sigma}^a_m(2\nu) $
is just the inclusion map, and
$\iota_{B,\nu}\colon B^a_m(\nu)\to B^a_m(2\nu) $
%
is an embedding of a manifold which maps $ (q^l_k)_{k=1,\dots , m}^{l=0,\dots ,\nu} \in B^a_m(\nu)$ to $(\bar{q}^{l'}_k)_{k=1,\dots , m}^{l'=0,\dots , 2\nu} \in B^a_m(2\nu)$, where
\[\bar{q}^{l'}_k= \begin{cases}
q^{l}_{k} & \text{ if }l' \text{ is even and } l'=2l, \\
\overline{q^l_kq^{l+1}_k}(\frac{1}{2}) & \text{ if } l' \text{ is odd and } l'=2l+1.
\end{cases}\]
In addition, we define  two maps
\[
f_{\nu}\colon  \bar{\Sigma}^a_m(\nu) \to B^a_m(\nu), \  g_{\nu}\colon B^a_m(\nu) \to \bar{\Sigma}^a_m(2\nu), \]
as follows: $f_{\nu}$ maps $(\gamma_k\colon [0,T_k]\to Q)_{k=1,\dots , m}\in \bar{\Sigma}^a_m(\nu)$ to $ \left( \gamma_k\left( \frac{l}{\nu}T_k\right) \right)_{k=1,\dots , m}^{l=0,\dots ,\nu} \in B^a_m(\nu)$.
Note that for $l=0,\dots,\nu-1$ and $k=1, \dots , m$,
\[\textstyle{ d \left( \gamma_k\left( \frac{l}{\nu}T_k\right) , \gamma_k\left( \frac{l+1}{\nu}T_k\right) \right)\leq  \len \rest{\gamma_k}{[\frac{l}{\nu}T_k , \frac{l+1}{\nu} T_k]} \leq \frac{T_k}{\nu} \cdot \sup_{t\in [0,T_k]} |\gamma'(t)| <\epsilon_g . }\]
On the other hand, $g_{\nu}$ maps $ (q_k^l)_{k=1,\dots , m}^{l=0,\dots , \nu} \in B^a_m(\nu)$ to
\[(\gamma_k\colon [0,\textstyle{ \frac{3 }{2} }C_0^{-1}\epsilon_g \nu]\to Q)_{k=1,\dots , m} \in \bar{\Sigma}^a_m(2\nu),\]
where, for $l=0,\dots ,\nu-1$,
\[\textstyle{ \gamma_k(t)\coloneqq  \overline{q^l_kq^{l+1}_k} \circ \chi\left( \frac{C_0}{\epsilon_g}\cdot t-\frac{3}{2} l\right) \text{ if } \frac{3}{2}C_0^{-1}\epsilon_g l \leq t\leq \frac{3}{2}C_0^{-1}\epsilon_g (l+1)  } .\]
Here $\chi\colon [0,\frac{3}{2}]\to [0,1]$ is a $C^{\infty}$ function such that
$\chi(t)=\begin{cases} 0 & \text{ near }t=0 , \\
\frac{1}{2} & \text{ near }t=\frac{3}{4}, \\
1 & \text{ near }   t = \frac{3}{2} , \\
\end{cases}$
and $0\leq \chi'(t) \leq 1$ for any $t\in [0,\frac{3}{2} ]$.
Note that $|(\gamma_k)'(t)| \leq  d(q^l_k,q^{l+1}_k)\cdot\sup_{t\in [0,3/2]} |\chi'(t)|\cdot \frac{C_0}{\epsilon_g} \leq C_0$.
The next lemma shows that $B^a_m(\nu)$ approximates $\bar{\Sigma}^a_m(\nu)$  as $\nu \to \infty$. The readers can refer to \cite[Lemma 6.3]{Irie-BV}. 
\begin{lem}\label{lem-diagram-up-to-htpy}
The following diagram commutes up to homotopy:
\[\xymatrix{ \bar{\Sigma}^a_m(\nu)  \ar[r]^{\iota_{\Sigma,\nu}}\ar[d]_{f_{\nu}} & \bar{\Sigma}^a_m(2 \nu)\ar[d]_{f_{2\nu}}  \\
B^a_m(\nu) \ar[r]^{\iota_{B,\nu}}\ar[ru]^{g_{\nu}} & B^a_m(2\nu) }\]
\end{lem}

\begin{proof} 
The lower right triangle commutes in the strict sense. For the upper left triangle, we need to show that $\iota_{\Sigma,\nu}$ is homotopic to $g_{\nu} \circ f_{\nu}$.

We abbreviate $ C_0^{-1}\cdot \epsilon_g$ by $c_0$.
For $(\gamma_k\colon [0,T_k]\to Q)_{k=1\dots ,m}\in \bar{\Sigma}^a_m(\nu)$,  let us define a path $\gamma^s_k\colon [0,c_0\nu]\to Q$ for $k=1,\dots ,m$ and $s\in [0,1]$ by
\[\gamma^s_k(t) \coloneqq \begin{cases}
\overline{ \gamma_k(\frac{l}{\nu}T_k) \gamma_k(\frac{l+s}{\nu}T_k) } (t-c_0l) & \text{ if } c_0l\leq t \leq c_0(l+s), \\
\gamma_k(  \frac{T_k}{c_0\nu} (t -c_0l) + \frac{l}{\nu}T_k) & \text{ if }  c_0(l+s)\leq t\leq c_0(l+1),
\end{cases} (l=0,\dots ,l-1)\]
Then, $\gamma^0_k$ is equal to $\gamma_k(\frac{T_k}{c_0\nu}\cdot)$ and $\gamma^1_k$ is a  broken geodesics connecting $(\gamma_k(\frac{l}{\nu}T_k))^{l=0,\dots ,\nu}$. We modify $\gamma^s_k$ to a $C^{\infty}$ path.
For instance, we take a $C^{\infty}$ function $\tilde{\chi}\colon [0,1]\times [0,\frac{3}{2}c_0\nu]\to [0,c_0\nu]$ satisfying 
$0\leq \frac{\partial}{\partial t}\tilde{\chi}(s,t) \leq 1$ and
\[\tilde{\chi}(s,t)=\begin{cases}c_0l  & \text{ on a neighborhood of } \{t=\frac{3}{2}c_0l \} , \\
\frac{2}{3} (t- \frac{1}{4}c_0) & \text{ on a neighborhood of } \{t=s+\frac{3}{2}c_0l+\frac{1}{4}c_0) \} .\end{cases}\]
Then, we define $\tilde{\gamma}^s_k \coloneqq  \gamma^s_k\circ \tilde{\chi}(s,\cdot) \colon [0,\frac{3}{2}c_0\nu]\to Q$.
If we take a $C^{\infty}$ function $\kappa\colon \R\to [0,1] $ such that $\kappa(t)=\begin{cases} 0 &\text{ if }t\leq 0, \\ 1 & \text{ if }t\geq 1, \end{cases}$
then we get a smooth map
\[H\colon \R\times \bar{\Sigma}^a_m(\nu) \to \bar{\Sigma}^a_m(2\nu) \colon (s,(\gamma_k)_{k=1,\dots ,m}) \mapsto (\tilde{\gamma}^{\kappa(s)}_k)_{k=1,\dots ,m} .\]

This gives a homotopy from $H(0,\cdot)$ to $H(1,\cdot)$. Moreover, $\iota_{\Sigma,\nu}$ is homotopic to $H(0,\cdot )$ since the paths of $\iota_{\Sigma,\nu}((\gamma_k)_{k=1,\dots ,m})$ and those of $H(0, (\gamma_k)_{k=1,\dots ,m})$ differ only by parameterizations, so the homotopy can be constructed by interpolating these parameterizations.
For the same reason, $g_{\nu}\circ f_{\nu}$ is homotopic to $H(1,\cdot)$. Therefore, $\iota_{\Sigma,\nu}$ is homotopic to $g_{\nu}\circ f_{\nu}$.
\end{proof}

From Lemma \ref{lem-diagram-up-to-htpy}, it follows that for any $a,b\in \R_{\geq 0}$ with $a\leq b< b_0$,
\begin{align*}
\varinjlim_{j\to \infty} (f_{2^j})_* &\colon H^{\dr}_*(\bar{\Sigma}^b_m,\bar{\Sigma}^a_m)= \varinjlim_{j\to \infty} H^{\dr}_*(\bar{\Sigma}^b_m(2^j),\bar{\Sigma}^a_m(2^j)) 
\to \varinjlim_{j\to \infty} H^{\dr}_*(B^b_m(2^j),B^a_m(2^j))
\end{align*}
is an isomorphism. Combining with (\ref{isom-excision}), we get an isomorphism
\begin{align}\label{cong-dr-sing-path}
\begin{split}
( \varinjlim_{j\to \infty} (f_{2^j})_* ) \circ  (\iota_*)^{-1} \colon H^{\dr}_*(\Sigma^b_m,\Sigma^a_m) \to \varinjlim_{\nu\to \infty} H^{\dr}_*(B^b_m(\nu),B^a_m(\nu)).
\end{split}
\end{align}
Furthermore, from Proposition \ref{prop-hmgy-of-mfd}, $H^{\sing}_*(B^b_m(\nu),B^a_m(\nu))\cong H^{\dr}_*(B^b_m(\nu),B^a_m(\nu))$.

\subsubsection{Computation of homology by Morse theory}

Next, we examine the homology group $H^{\sing}_*( B^b_m(\nu), B^a_m(\nu))$ in terms of Morse theory. Fix $m\in \Z_{\geq 1}$ and $\nu\in \Z_{\geq 1}$. 
For $k\in \{1,\dots ,m\}$ and $l\in \{0,\dots ,\nu-1\}$,
we set
\[h^l_k\colon B_m(\nu)\to \R\colon (q^l_k)^{l=0,\dots , \nu}_{k=1,\dots , m} \mapsto d(q^l_k,q^{l+1}_k)^2 .\]
For every $r>0$, let us introduce the following:
\begin{itemize}
\item a $C^{\infty}$ function $\sigma_r\colon  \R_{\geq 0} \to \R_{>0} \colon z\mapsto \sqrt{z+r}$.
\item a $C^{\infty}$ function
\[L_r\colon B_m(\nu)\to \R\colon \bold{q}\mapsto \sum_{k=1}^m \sum_{l=0}^{\nu-1} \sigma_r\circ h^l_k(\bold{q}) .\]
\item compact subsets of $B_m(\nu)$
\begin{align*}
Z_r&\coloneqq \{\bold{q}\in L_{r}^{-1}([0,b_0])\mid \sigma_r\circ h^l_k(\bold{q}  ) \leq \epsilon_g\text{ for every }k=1,\dots ,m \text{ and }l=0,\dots ,\nu-1 \} ,\\
Z^0_{r} &\coloneqq \{\bold{q} \in Z_r \mid \min_{1\leq k\leq m}\sum_{l=1}^{\nu-1} \sigma_r\circ h^l_k(\bold{q}) \leq \epsilon_0 \}.
\end{align*}
\end{itemize}  
The role of $\{\sigma_r\}_{r>0}$ is to approximate $\sqrt{z}$ by $C^{\infty}$ functions.
We define for every $a\in [0,b_0)$ and $r>0$
\[Z^a_r\coloneqq (\rest{L_r}{Z_r})^{-1}([0,a)) \cup Z^0_r. \]
Then, $Z^a_r\subset Z^a_{r'}$ holds if $0<r'<r$. Furthermore, $\bigcup_{r>0} Z^a_r =B^a_m(\nu)$ holds. Therefore, we have an isomorphism induced by the inclusion maps
\[  \varinjlim_{r\to 0} H^{\sing}_*(Z^b_r, Z^a_r) \to H^{\sing}_*(B^b_m(\nu),B^a_m(\nu)) \]
for $a,b\in \R_{\geq 0}$ with $a\leq b<b_0$.
In order to apply Morse theory to $L_{r}$, we need to determine its critical points.
The next lemma is fundamental. We omit the proof, but a similar result is proved, for instance in \cite{milnor}, when there is no boundary condition. 
\begin{lem}\label{lem-grad-d2}
For $k_0\in \{1,\dots,m\}$ and $l_0\in \{0,\dots ,\nu-1\}$, let $X^{l_0}_{k_0}$ be the gradient vector field of
$h^{l_0}_{k_0} $. We define $\pi^{l_0}$ to be the orthogonal projection form $\rest{TQ}{K} $ to $TK$ if $l_0=0$ or $\nu$, and otherwise $\pi^{l_0}\coloneqq \id_{TQ}$.
Then, each component of $X^{l_0}_{k_0}=(v^{l}_{k})_{k=1,\dots ,m}^{l=0,\dots ,\nu}$ 
is determined by the following: If $(k,l)\neq (k_0,l_0),(k_0,l_0+1)$, then $v^l_k=0$. Otherwise,
\begin{align*}
v^{l_0}_{k_0}= \pi^{l_0}\left( -\rest{ \frac{d}{dt}}{t=0}\overline{q^{l_0}_{k_0}q^{l_0+1}_{k_0}} \right)  ,\        v^{l_0+1}_{k_0} =  \pi^{l_0+1} \left( \rest{ \frac{d}{dt}}{t=1}\overline{q^{l_0}_{k_0}q^{l_0+1}_{k_0}}  \right).
\end{align*}
\end{lem}

\begin{prop}\label{prop-crit-pt-L} Suppose that $a \in [0,b_0) \setminus \mathcal{L}_m(K)$ and $r>0$ is sufficiently small.
Then $\bold{q} \in Z_r \setminus Z^0_r$ is a critical point of $L_{r}$ with its value in $[0,a]$
if and only if there exist $\gamma_1,\dots ,\gamma_m\in \mathcal{C}(K)$ such that
$(\gamma_k)_{k=1,\dots , m}\in \bar{\Sigma}^a_m(\nu)$ and 
$\bold{q}=f_{\nu}((\gamma_k)_{k=1,\dots , m})$.
\end{prop}

\begin{proof} 
From Lemma \ref{lem-grad-d2}, $\bold{q}=(q^l_k)^{l=0,\dots, \nu}_{k=1,\dots , m}\in B_m(\nu)$ is a critical point of $L_{r}$ if and only if the following conditions hold for every $k\in \{1,\dots ,m\}$:
\begin{align}\label{crit-point}
\begin{split}
&(\sigma_r)'(h^0_k(\bold{q}))\cdot \rest{ \frac{d}{dt}}{t=0}\overline{q^0_kq^1_k} \in (T_{q^0_k}K)^{\perp} \text{ and } (\sigma_r)'(h^{\nu-1}_k(\bold{q})) \cdot \rest{ \frac{d}{dt}}{t=1}\overline{q^{\nu-1}_kq^{\nu}_k} \in (T_{q^{\nu}_k}K)^{\perp} ,\\
&(\sigma_r)' (h^{l-1}_k(\bold{q})) \cdot \rest{ \frac{d}{dt}}{t=1}\overline{q^{l-1}_kq^l_k} = (\sigma_r)' (h^l_k(\bold{q})) \cdot
\rest{\frac{d}{dt}}{t=0}\overline{q^l_kq^{l+1}_k}  \text{ for every } l\in \{1,\dots ,\nu-1\}.
\end{split}
\end{align}
Comparing the norms of both sides in the second line of (\ref{crit-point}), we have for $l\in \{1,\dots ,\nu-1\}$
\begin{align*}
(\sigma_r)' (h^{l-1}_k(\bold{q})) \cdot \sqrt{ h^{l-1}_k(\bold{q}) }= (\sigma_r)' (h^l_k(\bold{q})) \cdot \sqrt{ h^l_k(\bold{q}) } . 
\end{align*}
Since $(\sigma_r)'(z)\sqrt{z}$ is a strictly increasing function of $z$, this equation means that $d(q^{l-1}_k,q^l_k)=d(q^l_k,q^{l+1}_k)$ for $l\in \{1,\dots ,\nu-1\}$.
%
Since $(\sigma_r)'>0$, the conditions of (\ref{crit-point}) are equivalent to the following:
\begin{align}\label{crit-point-2}
\begin{split}
&\rest{ \frac{d}{dt}}{t=0}\overline{q^0_kq^1_k} \in (T_{q^0_k}K)^{\perp} \text{ and }\rest{ \frac{d}{dt}}{t=1}\overline{q^{\nu-1}_kq^{\nu}_k} \in (T_{q^{\nu}_k}K)^{\perp} , \\
&\rest{ \frac{d}{dt}}{t=1}\overline{q^{l-1}_kq^l_k}  = \rest{\frac{d}{dt}}{t=0}\overline{q^l_kq^{l+1}_k} \text{ for every }l\in \{1,\dots ,\nu-1\}.
\end{split}
\end{align}
For every $\bold{q}\in B_m(\nu)$, let us define $T_{k,\bold{q}}\coloneqq \sum_{l=0}^{\nu-1}d(q^l_k,q^{l+1}_k)$ and a  path $\gamma_{k, \bold{q}}\colon [0,T_{k, \bold{q}}]\to Q$ determined by
\[\gamma_{k,\bold{q}}(T_{k,\bold{q}}t) \coloneqq  \overline{q^l_kq^{l+1}_k} \left( \nu t -l\right) \text{ if }\frac{l}{\nu} \leq t \leq \frac{l+1}{\nu} \text{ for }l\in \{0,\dots ,\nu-1\}. \]
Note that if $(\gamma_{k,\bold{q}})_{k=1,\dots ,m}\in \bar{\Sigma}^a_m(\nu)$, then $\bold{q}= f_{\nu}((\gamma_{k,\bold{q}})_{k=1,\dots ,m})$ holds.
We take $r>0$ so small that $T_{k, \bold{q}}>0$ holds for every $k=1,\dots ,m$ and $\bold{q}\in Z_r\setminus Z^0_r$. Then, for $\bold{q}\in Z_r \setminus Z^0_r$, the condition (\ref{crit-point-2}) is equivalent to that $\gamma_{k,\bold{q}}$
is a binormal chord of $K$ in $Q$, that is, $\gamma_{k,\bold{q}}\in \mathcal{C}(K)$.
In addition,
\[\textstyle{ \lim_{r\to 0} L_r(\bold{q})= \sum_{k=1}^m\len \gamma_{k,\bold{q}} }\]
for every $\bold{q}\in Z_r$ uniformly. Recall that $\mathcal{L}_m(K)$ is a closed subset of $\R_{>0}$.
Assuming that $a\notin \mathcal{L}_m(K)$ and $r>0$ is sufficiently small, it follows that a critical point $\bold{q}\in Z_r\setminus Z^0_r$ of $L_r$ satisfies $L_r(\bold{q})\leq a$ if and only if the binormal chords $(\gamma_{k,\bold{q}})_{k=1,\dots,m}$ satisfy $\sum_{k=1}^m\len \gamma_{k,\bold{q}}<a$. This proves the proposition.
\end{proof}
Let $Y_{r}$ be the gradient vector field of $L_{r}$.
To prove the next lemma, which is rather technical, let us prepare a few computations.
 $X^l_{k,r} \coloneqq ((\sigma_r)'\circ h^l_k) \cdot X^l_k$ is the gradient vector field of $\sigma_r\circ h^l_k$.
For every $k\in \{1,\dots ,m\}$, $l\in\{1,\dots ,\nu-2\}$ and $\bold{q}\in B_m(\nu)$, we have
\begin{align*}
\la Y_{r} ,X^{l}_{k} \ra (\bold{q}) =& \left\la  (\sigma_r)'(h^{l-1}_k(\bold{q}))\cdot \rest{ \frac{d}{dt}}{t=1}\overline{q^{l-1}_kq^l_k}  , -\rest{ \frac{d}{dt}}{t=0}\overline{q^{l}_kq^{l+1}_k}  \right\ra \\
&+ (\sigma_r)'(h^l_k(\bold{q}))\cdot \left| \rest{ \frac{d}{dt}}{t=0}\overline{q^{l}_kq^{l+1}_k} \right|^2 + (\sigma_r)'(h^l_k(\bold{q}))\cdot \left| \rest{ \frac{d}{dt}}{t=1}\overline{q^{l}_kq^{l+1}_k} \right|^2 \\
&+ \left\la - (\sigma_r)'(h^{l+1}_k(\bold{q}))\cdot \rest{ \frac{d}{dt}}{t=0}\overline{q^{l+1}_kq^l_k}  ,  \rest{ \frac{d}{dt}}{t=1}\overline{q^{l}_kq^{l+1}_k}  \right\ra.
\end{align*}
We abbreviate the increasing function $ (\sigma_r)'\sqrt{z}$ by $\tau_r$.
Then, by Cauchy-Schwarz inequality,
\begin{align*}
\la Y_{r} ,X^{l}_{k,r} \ra (\bold{q}) \geq& -\tau_r(h^{l-1}_k(\bold{q})) \cdot \tau(h^l_k(\bold{q})) + 2 (\tau_r(h^l_k(\bold{q})))^2  -\tau_r(h^{l+1}_k(\bold{q})) \cdot \tau_r(h^{l}_k(\bold{q}))  \\
=& (2\tau(h^l_k(\bold{q}))- \tau_r(h^{l-1}_k(\bold{q}))-\tau_r(h^{l+1}_k(\bold{q})) )\cdot \tau_r(h^l_k(\bold{q})) .
\end{align*}
Since $\sigma_r$ and $\tau_r$ are increasing functions, we have for $k_0\in \{1,\dots ,m \}$ and $l_0\in \{1,\dots ,\nu-2\}$,
\begin{align}\label{max-grad}
 \sigma_r\circ h^{l_0}_{k_0}(\bold{q}) = \max_{k,l} \sigma_r\circ h^l_k(\bold{q}) \Rightarrow \la -Y_r, X^{l_0}_{k_0,r} \ra(\bold{q}) \leq 0 .
 \end{align}
The same result holds when $l_0=0$ or $\nu-1$. We also note that for every $k_0\in \{1,\dots ,m\}$,
\begin{align}\label{sum-grad}
\textstyle{\la -Y_r, \sum_{l=1}^{\nu-1}X^{l}_{k_0} \ra (\bold{q}) \leq 0 }.
\end{align}

\begin{lem}\label{lem-gradient}
The trajectory of any point in $ Z_r$ (resp. $ Z^0_r$) along $-Y_{r}$ never goes outside $Z_r$ (resp. $Z^0_r$) at positive time.
\end{lem}
\begin{proof}
Suppose that $\Gamma\colon [0,T]\to B_m(\nu)$ is a trajectory along $-Y_r$. Let us consider two continuous functions $f,g\colon [0,T]\to \R$ defined by
\[ f(t)\coloneqq \max_{k,l}\sigma_r\circ h^l_k(\Gamma(t)) ,\ g(t)\coloneqq \min_{k} \sum_{l=1}^{\nu-1}\sigma_r\circ h^l_k(\Gamma(t)).\]
To prove this lemma, it suffices to show that they are decreasing functions.
Indeed, there exists a discrete subset $A\subset [0,T]$ such that $f$ and $g$ are differentiable at every $t\in [0,T]\setminus A$. (\ref{max-grad}) and (\ref{sum-grad}) imply that $f'(t)\leq 0$ and $g'(t)\leq 0$ for every $t\in [0,T]\setminus A$. Hence, $f$ and $g$ are decreasing on $[0,T]$.
\end{proof}

We apply a general result from Morse theory.
\begin{lem}\label{lem-morse-general}
Let $B$ be a manifold and $L\colon B \to \R$ be a $C^{\infty}$ function. For $a,b\in \R$ with $a\leq b$ and two compact subsets $Z,Z^0 \subset B$, suppose that there is no critical point of $L$ in $(\rest{L}{Z})^{-1}([a,b]) \setminus Z^0$ and that the trajectory of any point in $Z$ (resp. $Z^0$) along the negative gradient vector field of $L$ never goes outside $Z$ (resp. $Z^0$) at positive times. Let us define $Z^{a'}\coloneqq (\rest{L}{Z})^{-1}((-\infty,a')) \cup Z^0$ for $a'\in \{a,b\}$. Then
$H^{\sing}_*(Z^b,Z^a) =0$ holds.
\end{lem}
\begin{proof}The conditions on $a,b,Z$ and $Z^0$ show that
$Z^b$ can be deformed into $Z^a$ along the negative gradient flow of $L$. Therefore, we get a map from $(Z^b,Z^a)$ to $(Z^a,Z^a)$ which gives the inverse map of the inclusion map up to homotopy. 
\end{proof}
%
%
%
Combining the above results, we prove the first main proposition in this section.
\begin{prop}\label{prop-hmgy-grp}
If $\mathcal{L}_m(K) \cap [a,b]=\varnothing$, then $H^{\dr}_*( \Sigma^b_m ,\Sigma^a_m )=0$.
\end{prop}
\begin{proof}
Form Proposition \ref{prop-crit-pt-L} and Lemma \ref{lem-gradient}, we can apply Lemma \ref{lem-morse-general} to show that if $\mathcal{L}_m(K) \cap [a,b]=\varnothing$ and $r>0$ is sufficiently small, then $H^{\sing}_*(Z^b_r,Z^a_r )=0$ holds, and thus \[H_*^{\sing}(B^b_m(\nu),B^a_m(\nu))\cong \varinjlim_{r\to 0} H^{\sing}_*(Z^b_r,Z^a_r )=0.\]
From (\ref{cong-dr-sing-path}), it follows that $H^{\dr}_*( \Sigma^b_m ,\Sigma^a_m ) \cong \varinjlim_{j\to \infty } H_*^{\sing}(B^b_m(2^j),B^a_m(2^j))= 0$.
\end{proof}

\subsubsection{  $H^{\dr}_*(S_{\epsilon})$ and the evaluation map}

Next, we examine $H^{\dr}_*(S_{\epsilon})$ for $\epsilon\in (0, \epsilon_0]$.
Choose a Riemannian metric  $g'$ on $N_{\epsilon_0}$ for which $K$ is a totally geodesic submanifold of $N_{\epsilon_0}$.
Then, we can take a constant $C_1\geq 1$ such that
for any $q,q'\in N_{\epsilon_0/C_1}$ with $d(q,q')<\epsilon_0/C_1$, there exists a unique shortest geodesic in $N_{\epsilon_0}$ with respect to $g'$ from $q$ to $q'$. Let us write this geodesic by $\tilde{qq'}\colon [0,1]\to N_{\epsilon_0}$.
In this subsection, $\exp$ denotes the exponential map with respect to $g$.

There is a smooth map between differentiable spaces $s_{\epsilon}\colon N^{\reg}_{\epsilon} \to S_{\epsilon}$ which maps $\exp_xv\in N_{\epsilon}$ ($x\in K$, $v\in (T_xK)^{\perp}$ with $|v|<\epsilon/2$) to
\[s_{\epsilon}( \exp_xv) \coloneqq (\sigma^v_i\colon [0,\epsilon/2] \to N_{\epsilon})_{i=1,2},\] 
where $\sigma^v_1(t)=\exp_x( \frac{\epsilon-2t}{\epsilon}v)$ and $\sigma^v_2(t)=\exp_x(\frac{2t}{\epsilon}v)$ for $t\in [0,\epsilon/2]$. This satisfies $\ev_0\circ s_{\epsilon}=\id_{N_{\epsilon}}$.
For $\epsilon, \bar{\epsilon}\in (0,\epsilon_0]$ with $\epsilon\leq \bar{\epsilon}$, let $i_{\epsilon,\bar{\epsilon}} \colon S_{\epsilon}\to S_{\bar{\epsilon}}$ denote the inclusion map.
\begin{lem}\label{lem-htpy-sev}
There exists a constant $C\geq C_1$ such that for any $\epsilon \in (0,\epsilon_0/C]$, the inclusion map $i_{\epsilon,C\epsilon} \colon S_{\epsilon}\to S_{C\epsilon}$ is homotopic to $i_{\epsilon,C\epsilon}\circ s_{\epsilon}\circ \ev_0 $.

\end{lem}

\begin{proof}
We define a $C^{\infty}$ map
\[G\colon \{(q,q')\in N_{\epsilon_0/C_1}\times N_{\epsilon_0/C_1} \mid d(q,q')<\epsilon_0/C_1 \}\times [0,1]\to N_{\epsilon_0}\colon ( (q,q'),s) \mapsto \tilde{qq'}(s).\]
Then there is a constant $C\geq C_1$ so that
\[ | d \left( G( \cdot ,s)\right)_{(q,q')} (v,v')|_g\leq \frac{C}{2}\cdot  ( |v|_g + |v'|_g ) \]
for every $s\in [0,1]$ and $(v,v')\in T_qN_{\epsilon_0/C}\times T_{q'}N_{\epsilon_0/C}$ with $d(q,q')<\epsilon_0/C_1$.
For any $\epsilon\in (0,\epsilon_0/C]$ and $(\sigma_i\colon [0,\epsilon_i]\to N_{\epsilon})_{i=1,2} \in S_{\epsilon}$, we set $\bar{\epsilon}\coloneqq C\epsilon$ and define $(\sigma^{(s)}_i\colon [0,\epsilon^s_i]\to N_{\bar{\epsilon}} )_{i=1,2}\in S_{\bar{\epsilon}}$ for $s\in \R$ as follows: Take $x\in K$ and $v\in T_xK$ such that $\sigma_1(0)=\exp_xv$ and $|v|<\epsilon/2$. Then we define
\begin{align*}
\epsilon^s_i &\coloneqq \begin{cases}
(1-\kappa(s))\epsilon_i +\kappa(s)\bar{\epsilon} & \text{ if }s\leq \frac{1}{3}, \\
\bar{\epsilon} & \text{ if }\frac{1}{3}\leq s \leq \frac{2}{3}, \\
(1-\kappa(s-\frac{2}{3}))\bar{\epsilon} + \kappa(s-\frac{2}{3}) \epsilon & \text{ if }s\leq \frac{1}{3},
\end{cases} \\
\sigma^{(s)}_i (t)&\coloneqq \begin{cases}
\sigma_i(\epsilon_it/\epsilon^s_i ) & \text{ if }s\leq \frac{1}{3},\\
G ( \sigma_i ( \epsilon_it/\bar{\epsilon} ) ,\sigma^v_i (t/C) , \kappa( s-\frac{1}{3} ) ) & \text{ if }\frac{1}{3}\leq s\leq \frac{2}{3} ,\\
\sigma^v_i(\epsilon t/\epsilon_i^s  ) & \text{ if }\frac{2}{3}\leq s.
\end{cases}\end{align*}
Here, $\kappa\colon \R \to [0,1] $ is a $C^{\infty}$ function such that $\kappa(s)= \begin{cases} 0 & \text{ if } s\leq 0, \\ 1& \text{ if }s\geq \frac{1}{3} .\end{cases}$ 
Note that when $\frac{1}{3}\leq s\leq \frac{2}{3}$,
\[|(\sigma^{(s)}_i)'(t) |_g \leq \frac{C}{2} \cdot \left( \frac{\epsilon_i}{\bar{\epsilon}}\sup|\sigma'_i|_g +  \frac{1}{C} \sup |(\sigma^v_i)'|_g \right) \leq \frac{C}{2} \cdot\left( \frac{\epsilon_i}{\bar{\epsilon}} +\frac{1}{C} \right) \leq 1.\]
Now the homotopy from $i_{\epsilon,\bar{\epsilon}} $ to $i_{\epsilon,\bar{\epsilon}} \circ s_{\epsilon}\circ \ev_0$ is given by the map
\[\R\times S_{\epsilon}\to S_{\bar{\epsilon}}\colon (s,(\sigma_i)_{i=1,2}) \to (\sigma^{(s)}_i\colon [0,\epsilon^s_i] \to N_{\bar{\epsilon}})_{i=1,2}. \]
\end{proof}

\begin{prop}\label{prop-exactness}
Let $C$ be the constant of Lemma \ref{lem-htpy-sev}.
For any $\epsilon \in (0,\epsilon_0/C]$ and $x \in H^{\dr}_*(S_{\epsilon})$,
\[(\ev_0)_*(x)=0 \in H^{\dr}_*(N_{\epsilon})\Rightarrow (i_{\epsilon,C\epsilon})_*( x)= 0 \in H^{\dr}_*(S_{C\epsilon}) . \]
\end{prop}
\begin{proof}
By Lemma \ref{lem-htpy-sev}, $(i_{\epsilon,C\epsilon })_* (x) = (i_{\epsilon,C\epsilon} \circ s_{\epsilon})_* ( (\ev_0)_*(x))=0$ if $(\ev_0)_*(x)=0 \in H^{\dr}_*(N_{\epsilon})$.
\end{proof}

\subsubsection{Additional results}\label{subsubsec-additional}
Using the computations obtained in the former subsections, we prove several additional results. $b_0$ is a fixed real number as before and $a\in [0,b_0)$ which may belong to $\mathcal{L}(K)$.
\begin{lem}\label{lem-measure0}
There exists a manifold $Z$ and a $C^{\infty}$ function $f\colon Z\to \R$ such that
$\mathcal{L}_1(K)\cap [0,a)$ is contained in the set of critical values of $f$. 
\end{lem}
\begin{proof}
We use the notations in the proof of Proposition \ref{prop-crit-pt-L} for $m=1$.
For every critical point $\bold{q}\in Z_r\setminus Z^0_r$ of $L_r$, $\len \gamma_{1,\bold{q}}= f_r (L_r (\bold{q}))\in (0, f_r (b_0))$, where $f_r$ is determined by $f_r(\nu \sigma_r(l))=\nu\sqrt{l}$ for every $l\in \R_{> 0}$. We choose $r>0$ so that $a< f_r(b_0)$. Then, from the correspondence between binormal chords and critical points of $L_r$, $\mathcal{L}_m(K)\cap [0,a)$ is contained in the critical value of $f\coloneqq f_r \circ L_r$ on $Z\coloneqq Z_r\setminus Z^0_r$.
\end{proof}
This proves that  $\mathcal{L}(K)=\bigcup_{m=1}^{\infty} \mathcal{L}_m(K)$ is a null set.
Indeed, for any $m\in \Z$, $\mathcal{L}_m(K)\cap [0,a)$ is contained in the set of critical values of the $C^{\infty}$ function
$Z^{\times m} \to \R \colon (x_1,\dots ,x_m)\mapsto f(x_1)+\dots +f(x_m)$.
By Sard's theorem, it is a null set in $\R$. 
Since $b_0$ was chosen arbitrarily, it follows that $\mathcal{L}_m(K)$ is a null set.

Next, recall that the definition of $\Sigma^a_m$ depends on auxiliary data $C_0$ and $\epsilon_0$.
\begin{lem}\label{lem-epsilon0-indep} $H^{\dr}_*(\Sigma^a_m)$ does not depend on the choice of $C_0$ and $\epsilon_0$. More precisely, the following hold:
\begin{itemize}
\item 
If we write $\Sigma^a_m$ by $\Sigma^a_{m,C_0}$ to clarify the dependence on $C_0$, the inclusion map
$\Sigma^a_{m,C_0}\to \Sigma^a_{m,C'_0}$ for $C_0\leq C'_0$  induces an isomorphism on homology.
\item
If we write $\Sigma^a_m$ by $\Sigma^a_{m,\epsilon_0}$  to clarify the dependence on $\epsilon_0$,
the inclusion map $\Sigma^a_{m,\epsilon'_0}\to \Sigma^a_{m,\epsilon_0}$ for $\epsilon'_0<\epsilon_0$ induces an isomorphism on homology.
\end{itemize}
\end{lem}
\begin{proof}
We define a smooth map $\Sigma^a_{m,C'_0} \to \Sigma^a_{m,C_0} $ which maps $ (\gamma_k\colon [0,T_k]\to Q)_{k=1,\dots ,m}$ to
\[ \left( [0, C'_0T_k/C_0]\to Q\colon t\mapsto \gamma_k(C_0t/C'_0) \right)_{k=1,\dots , m} . \]
This gives the inverse map of the inclusion map up to homotopy. This proves the assertion for $C_0$.

To prove the assertion for $\epsilon_0$, let us write $Z^a_r$ by $Z^a_{r,\epsilon_0}$.
Then, there is no critical point of $\sum_{l=0}^{\nu-1} \sigma_r\circ h^l_k$ in $Z^a_{r,\epsilon_0}\setminus  Z^a_{r,\epsilon'_0}$ for every $k\in \{1,\dots ,m\}$. By deforming along the negative gradient vector field $-\sum_{l=0}^{\nu-1}X^l_{k,r}$ inductively on $k=1,2,\dots ,m$, we can see that $Z^a_{r,\epsilon'_0}$ is a deformation retract of $Z^a_{r,\epsilon_0}$. This implies that $H^{\sing}_*(Z^a_{r,\epsilon_0}, Z^a_{r,\epsilon'_0})=0$ and thus $H^{\dr}_*(\Sigma^a_{m,\epsilon_0},\Sigma^a_{m,\epsilon'_0}) =0$.
\end{proof}


For the sake of  discussions in Section \ref{sec-cord-alg}, let us fix a topology on the set $\Sigma^a_m$ for $a\in \R_{\geq 0}\cup\{\infty\}$ as follows: $\Omega_K(Q)$ becomes a topological space so that the injection
\[\Omega_K(Q) \to C^{\infty}([0,1], Q)\times \R_{>0}\colon (\gamma\colon [0,T]\to Q) \mapsto (\gamma(T^{-1} \cdot ),T)\]
is a homeomorphism onto its image when $C^{\infty}([0,1],Q)$ is equipped with $C^{\infty}$-topology. We give $\Sigma^a_m$ the restricted topology from $\Omega_K(Q)^{\times m}$. Then, we can consider singular homology groups, such as $H^{\sing}_*(\Sigma^b_m,\Sigma^a_m)$ for $a\leq b$.

Suppose that $a,b\in \R_{\geq 0}$ and $a\leq b<b_0$.
By excision theorem,  $\iota_*\colon H^{\sing}_*(\bar{\Sigma}^b_m,\bar{\Sigma}^a_m) \to H^{\sing}_*(\Sigma^b_m,\Sigma^a_m) $ is an isomorphism.
All maps in Lemma \ref{lem-diagram-up-to-htpy} are continuous and the diagram commutes up to continuous homotopy. Therefore, we have an isomorphism
\begin{align}\label{cong-sing-path}
(\varinjlim_{j\to \infty} (f_{2^j})_*) \circ (\iota_*)^{-1} \colon H^{\sing}_*(\Sigma^b_m, \Sigma^a_m) \to \varinjlim_{j\to \infty} H^{\sing}_*(B^b_m(2^j),B^a_m(2^j)).
\end{align}

\subsection{Splitting and concatenating paths}\label{subsec-split-concatenate}

For $\epsilon \in (0,\epsilon_0/(5C_0)]$, we define an open subset of $\R^2$ 
\[A_{\epsilon}\coloneqq \{(T,\tau) \mid T> 4\epsilon\text{ and } 2\epsilon<\tau < T-2\epsilon\}.\]
This becomes a differentiable space as a subspace of $(\R^2)^{\reg}$.
For $a\in \R_{\geq 0}$, $m\in \Z_{\geq 1}$ and $k\in\{1,\dots ,m\}$, there are smooth maps
\begin{align*}
\tl_k&\colon \Sigma^a_m\to \R\colon (\gamma_l\colon [0,T_l]\to Q)_{l=1,\dots , m} \mapsto T_k, \\
\pr_T&\colon A_{\epsilon}\to \R^{\reg}\colon (T,\tau) \mapsto T.
\end{align*}
%
($\tl$ stands for the time length.)
These maps define a fiber product $\Sigma^a_m \ftimes{\tl_k}{\pr_T}A_{\epsilon}$ over $\R$, and the $k$-th evaluation map $\ev_k$ is defined on it by
\[\ev_k\colon \Sigma^a_m \ftimes{\tl_k}{\pr_T}A_{\epsilon} \to Q\colon ((\gamma_l)_{l=1,\dots , m},(T,\tau)) \mapsto \gamma_k(\tau).\]
From $\ev_k$ and $\ev_0\colon S_{\epsilon}\to Q^{\reg}$, we obtain a fiber product over $Q$.
We define a map on this fiber product
\[\con_k\colon \left( \Sigma^a_m \ftimes{\tl_k}{\pr_T}A_{\epsilon}\right) \ftimes{\ev_k}{\ev_0}S_{\epsilon} \to \Sigma^{a+\epsilon}_{m+1},\]
which maps $\left( (\gamma_l)_{l=1,\dots ,m}, (T,\tau), (\sigma_i\colon [0,\epsilon_i]\to N_{\epsilon})_{i=1,2}\right)$ to $(\gamma_1,\dots,\gamma_{k-1},\tilde{\gamma_k}^1,\tilde{\gamma_k}^2,\gamma_{k+1},\dots ,\gamma_m)$, where $\tilde{\gamma_k}^i$ ($i=1,2$) are the following paths:
\begin{align}\label{concatenation}
\begin{split}
\tilde{\gamma_k}^1 &\colon [0,\tau+2\epsilon_1]\to Q\colon t\mapsto
\begin{cases}
\gamma_k(t) & \text{ if }0\leq t\leq \tau -\epsilon_1, \\
\gamma_k(\tau-\epsilon_1 +\epsilon_1\mu(\frac{t-\tau+\epsilon_1}{\epsilon_1})) & \text{ if }\tau-\epsilon_1\leq t\leq \tau+\frac{\epsilon_1}{2} , \\
\sigma_1(\epsilon_1-\epsilon_1\mu(\frac{\tau+2\epsilon_1-t}{\epsilon_1})) & \text{ if }\tau+\frac{\epsilon_1}{2} \leq t \leq  \tau+2\epsilon_1,
\end{cases} \\
\tilde{\gamma_k}^2 & \colon [0,T-\tau+2\epsilon_2]\to Q\colon t\mapsto
\begin{cases}
\sigma_2(\epsilon_2\mu(\frac{t}{\epsilon_2})) & \text{ if }0\leq t\leq \frac{3}{2}\epsilon_2 , \\
\gamma_k( \tau+\epsilon_2-\epsilon_2\mu(\frac{3\epsilon_2-t}{\epsilon_2})) & \text { if }\frac{3}{2}\epsilon_2\leq t \leq 3\epsilon_2 ,\\
\gamma_k(t+\tau-2\epsilon_2) & \text{ if }3\epsilon_2\leq t\leq T-\tau +2\epsilon_2.
\end{cases}
\end{split}
\end{align}
Here, 
$\mu\colon [0,\frac{3}{2}] \to [0,1]$ is one of fixed data we have chosen in the beginning of Section \ref{sec-sp-of-paths}.

This definition can be explained as follows (See Figure \ref{figure-concatenation}.):
We split the $k$-th path $\gamma_k \colon [0,T]\to Q$ at $\tau\in (2\epsilon, T-2\epsilon)$ where $\gamma_k(\tau)=\sigma_1(0) (=\sigma_2(\epsilon_2)) \in N_{\epsilon}$, and then concatenate $\rest{\gamma_k}{[0,\tau] }$ (resp. $\rest{\gamma_k}{[\tau,T]}$) with $\sigma_1$ (resp. $\sigma_2$). The reparameterizations via $\mu$ is necessary in order to modify them to $C^{\infty}$ paths.
Note that
\[\len \tilde{\gamma}^1_k +\len \tilde{\gamma}^2_k =\len \gamma_k +\len \sigma_1 +\len \sigma_2<\len \gamma_k +\epsilon.\]

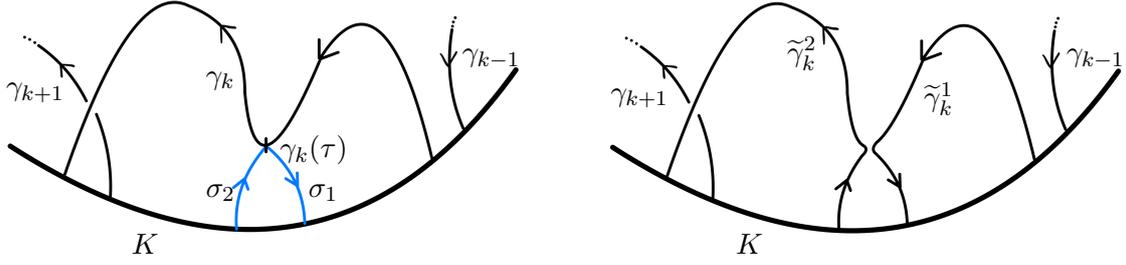
\begin{figure}
\centering
\begin{overpic}[width=15cm]{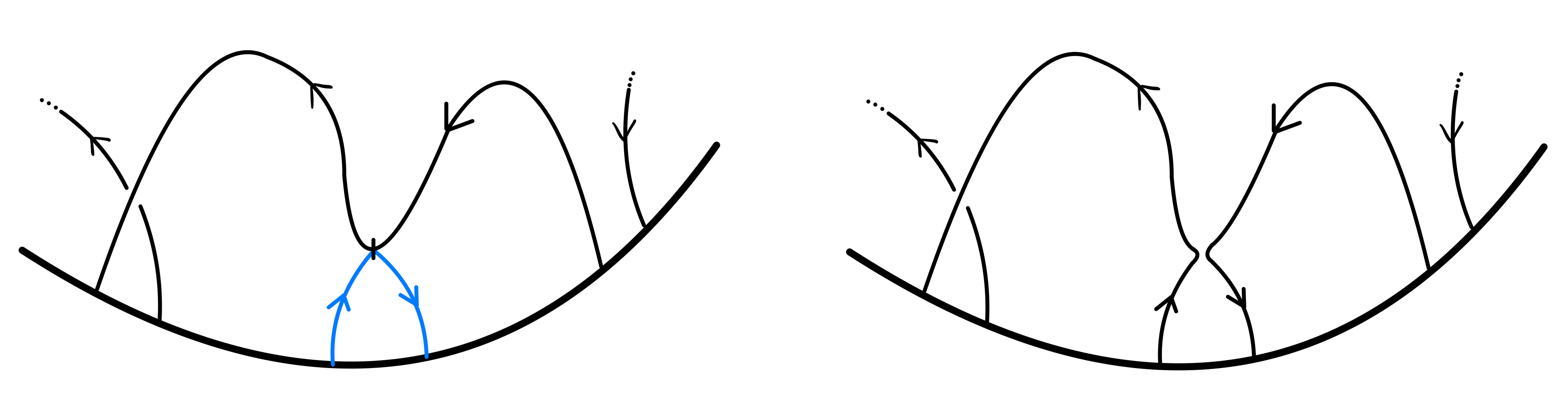}
\put(18.5,16){$\gamma_k$}
\put(1,15){$\gamma_{k+1}$}
\put(41,18){$\gamma_{k-1}$}
\put(18.5,6){$\sigma_2$}
\put(27.5,6){$\sigma_1$}
\put(12,1){$K$}
\put(65,1){$K$}
\put(69.5,18){$\tilde{\gamma}^2_k$}
\put(81.5,14){$\tilde{\gamma}^1_k$}
\put(54,14.5){$\gamma_{k+1}$}
\put(94,18){$\gamma_{k-1}$}
\put(25,9.5){$\gamma_k(\tau)$}
\end{overpic}
\caption{The process to define $\tilde{\gamma}^1_k$ and $\tilde{\gamma}^2_k$.}\label{figure-concatenation}
\end{figure}

The following lemma shows the cases where an element in the fiber product is mapped by $\con_k$ into $\Sigma^0_{m+1}$.
\begin{lem}\label{lem-short-length}
For $((\gamma_l)_{l=1,\dots , m},(T,\tau),(\sigma_i)_{i=1,2})\in \left( \Sigma^a_m \ftimes{\tl_k}{\pr_T}A_{\epsilon}\right) \ftimes{\ev_k}{\ev_0}S_{\epsilon}$, we have
\[\con_k((\gamma_l)_{l=1,\dots ,m},(T,\tau),(\sigma_i)_{i=1,2})\in \Sigma^0_{m+1}, \]
if either of the following three  condition holds:
\begin{enumerate}
\item[(i)] $(\gamma_l)_{l=1,\dots , m} \in \Sigma^0_m$.
\item[(ii)] $\tau< 4\epsilon_0/(5C_0) $ or $T- 4\epsilon_0/(5C_0)  <\tau$.
\item[(iii)] $\gamma_k$ satisfies for every $\tau'\in (\gamma_k)^{-1}(N_{\epsilon})$ that either $\rest{\gamma_k}{[0,\tau']}$ or $\rest{\gamma_k}{[\tau',T]}$ has length less than $4\epsilon_0/5$.
\end{enumerate}
\end{lem}

\begin{proof}
As in the definition of $\con_k$, let us write
\[\con_k((\gamma_l)_{l=1,\dots , m},(T,\tau),(\sigma_i)_{i=1,2})= (\gamma_1,\dots ,\tilde{\gamma_k}^1,\tilde{\gamma_k}^2,\dots ,\gamma_m).\]
Under the condition (i), $\len \gamma_{l}<\epsilon_0$ for some $l\in \{ 1,\dots ,m\}$. If $l\neq k$, the assertion is trivial. If $l=k$, either $\rest{\gamma_k}{[0,\tau]}$ or $\rest{\gamma_k}{[\tau,T]}$ has length less than $\epsilon_0/2$, and thus either $\tilde{\gamma_k}^1$ or $\tilde{\gamma_k}^2$ has length less than
\[\epsilon_0/2 + \max \{\len \sigma_i\mid i=1,2 \}, \]
and this value is smaller than $\epsilon_0$.
Under the condition (ii), if  $\tau< 4\epsilon_0/(5C_0) $ (resp. $T- 4\epsilon_0/(5C_0)  <\tau$) holds, then $\len \rest{\gamma_k}{[0,\tau]} <4\epsilon_0/5$ (resp. $\len \rest{\gamma_k}{[\tau,T]}<4\epsilon_0/5$). 
Therefore, either $\tilde{\gamma_k}^1$ or $\tilde{\gamma_k}^2$ has length smaller than $\epsilon_0$. 
The same result also holds under the condition (iii). 
\end{proof}

\subsection{Operation on de Rham chains}\label{subsec-operation}

For $\epsilon\in (0,\epsilon_0/(5C_0)]$, let us choose a $C^{\infty}$ cutoff function
$\rho_{\epsilon} \colon A_{\epsilon}\to [0,1]$ such that
\[\rho_{\epsilon}(T,\tau)=\begin{cases}
0 & \text{ if }t\leq \frac{10}{3}\epsilon \text{ or } T- \frac{10}{3}\epsilon \leq \tau , \\
1 & \text{ if } \frac{11}{3}\epsilon \leq \tau \leq T- \frac{11}{3}\epsilon.
\end{cases}\]
In particular, $\rho_{\epsilon}(T,\tau)=0$ if $T\leq  5\epsilon$.
We also choose a $C^{\infty}$ function $\chi_{\nu}\colon A_{\epsilon} \to [0,1]$ for every $\nu \in \Z_{\geq 1}$ such that
$\chi_{\nu}(T,\tau)=\begin{cases}
1 & \text{ if } T\leq \nu, \\
0 & \text{ if } T\geq \nu +1.
\end{cases}$
The support of $\chi_{\nu}\rho_{\epsilon} $ is compact, so we obtain a de Rham chain
\[\alpha_{\epsilon,\nu}\coloneqq [A_{\epsilon},\id_{A_{\epsilon}}, \chi_{\nu}\rho_{\epsilon} ]\in C^{\dr}_2(A_{\epsilon}) .\]
In addition, we define $\Sigma^a_m(\nu)\coloneqq \bigcap_{k=1}^m (\tl_k)^{-1}([0,\nu))$ which is a subspace of $\Sigma^a_m$.



For $m\in \Z_{\geq 1}$, $k\in \{1,\dots ,m\}$ and $\xi\in C^{\dr}_{q}(S_{\epsilon})$, we define a linear map
\[f_{k,\xi}\colon C^{\dr}_*(\Sigma^a_m) \to C^{\dr}_{*+1+q-n}(\Sigma^{a+\epsilon}_{m+1})\]
so that for $x\in C^{\dr}_*(\Sigma^a_m(\nu)) \subset C^{\dr}_*(\Sigma^a_m)$ ($\nu\in \Z_{\geq 1}$),
\[f_{k,\xi}(x) = (\con_k)_*\left( (x \ftimes{\tl_k}{\pr_T} \alpha_{\epsilon,\nu}) \ftimes{\ev_k}{\ev_0}\xi \right). \]
This map is well-defined since
$x\ftimes{\tl_k}{\pr_T}(\alpha_{\epsilon,\nu}-\alpha_{\epsilon,\nu'})=0$
when $x\in C^{\dr}_*(\Sigma^a_m(\nu))$ and $\nu' \geq \nu$.

Returning to Definition \ref{def-fib-prod-ope}, we can describe the de Rham chain $f_{k,\xi}(x)$ explicitly.
For $x=[U,\varphi,\omega]\in C^{\dr}_p(\Sigma^a_m)$, 
we write $\varphi (u)=(\gamma_l^u\colon [0,T^u_l]\to Q)_{l=1,\dots m}$ for $u\in U$. If we take $\nu >\sup_{u\in \supp\omega} T^u_k$, then $x\in C^{\dr}_p(\Sigma^a_m(\nu) )$. First, we have $ x \ftimes{\tl_k}{\pr_T} \alpha_{\epsilon,\nu}=[\tilde{U}_k,\tilde{\varphi}_k , \tilde{\omega}_k]$, where
\begin{align*}
\begin{split}
\tilde{U}_k &\coloneqq \{(u,\tau) \in U\times \R \mid 2\epsilon<\tau <T^u_k-2\epsilon\} , \\
\tilde{\varphi}_k &\colon \tilde{U}_k\to \Sigma^a_m \ftimes{\tl_k}{\pr_T}A_{\epsilon} \colon (u,\tau) \mapsto (\varphi(u), (T^u_k,\tau) ) , \\
\tilde{\omega}_k&\in \Omega^{*}_c(\tilde{U}) \colon \  (\tilde{\omega}_k)_{(u,\tau)}\coloneqq \rho_{\epsilon}(T^u_k,\tau)\cdot \chi_{\nu} (T^u_k,\tau) \cdot \omega_u = \rho_{\epsilon}(T^u_k,\tau)\cdot \omega_u.
\end{split}
\end{align*}
Here, $\tilde{U}_k$ is oriented as an open submanifold of $U\times \R$.
The last equality holds since $\chi_{\nu} (T^u_k,\tau)=1$ for $u\in \supp \omega$. This shows the independence of $f_{k,\xi}(x)$ on the choice of $\chi_{\nu}$.
For $\xi=[V,\psi,\eta]\in C^{\dr}_q(S_{\epsilon})$, we write $\psi(v)=(\sigma^v_i)_{i=1,2}$ for $v\in V$. Then we have $f_{k,\xi} (x)= (-1)^{r} [W_k,\Phi_k, \zeta_k]$, where
\begin{align}\label{explicit-rep}
\begin{split}
W_k &\coloneqq \{(u,\tau,v)\in U\times \R\times V \mid 2\epsilon <\tau < T^u_k-2\epsilon,\ \gamma^u_k(\tau)=\sigma^v_1(0)\} , \\
\Phi_k&\colon W_k\to \Sigma^{a+\epsilon}_{m+1}\colon (u,\tau,v)\mapsto \con_k(\varphi(u),(T^u_k,\tau), \psi(v)) , \\
\zeta_k&\in \Omega^{*}_c(W_k)\colon (\zeta_k)_{(u,\tau,v)}\coloneqq \rho_{\epsilon}(T^u_k,\tau)\cdot (\omega_u \times \eta_v ) , \\
r&\coloneqq (p+1-n)|\eta| .
\end{split}
\end{align}
Here $W_k$ is oriented as a fiber product over $Q$ of $\tilde{U}_k\to Q\colon (u,\tau) \mapsto \gamma^u_k(\tau)$
and $\ev_0\circ \psi \colon V\to Q$.

\begin{lem}\label{lem-delta-small}
For $x\in C^{\dr}_p(\Sigma^a_m)$ and $\xi\in C^{\dr}_{q}(S_{\epsilon})$,
\[\partial\circ f_{k,\xi}(x)-f_{k,\xi}\circ \partial(x) - (-1)^{p+1-n} f_{k,\partial \xi}(x) \in C^{\dr}_{p+q-n}(\Sigma^0_m). \]
\end{lem}
\begin{proof}
Using the notation of (\ref{explicit-rep}) for $x=[U,\varphi,\omega]\in C^{\dr}_p(\Sigma^a_m)$ and $\xi=[V,\psi,\eta]\in C^{\dr}_q(S_{\epsilon})$, we have
\begin{align*}
\partial\circ f_{k,\xi}(x)-f_{k,\xi}\circ \partial(x) - (-1)^{p+1-n} f_{k,\partial \xi}(x)&=(-1)^{p-1}(\con_k)_* ( (x \ftimes{\tl_k}{\pr_T} (\partial\alpha_{\epsilon,\nu}))\ftimes{\ev_k}{\ev_0} \xi ) \\
&=(-1)^{(p-n)|\eta| +1}[W_k,\Phi_k, \theta_k],
\end{align*}
where $\theta_k\in \Omega_{c}^{*}(W_k)$ is defined by $(\theta_k)_{(u,\tau,v)}=\frac{\partial \rho_{\epsilon}}{\partial \tau}(T^u_k,\tau) \cdot (\omega_u\times d\tau \times \eta_v )$. From the condition on $\rho_{\epsilon}$, the support of $\theta$ lies in an open subset 
\[\bar{W}_k\coloneqq \{(u,\tau,v)\in W_k \mid \tau <4\epsilon \text{ or } T^u_k- 4\epsilon<\tau \}.\]
Since $(\varphi(u),(T^u_k,\tau),\psi(v))\in \left( \Sigma^a_m \ftimes{\tl_k}{\pr_T}A_{\epsilon}\right) \ftimes{\ev_k}{\ev_0}S_{\epsilon}$ for $(u,\tau,v)\in \bar{W}_k$ satisfies the condition (ii) of Lemma \ref{lem-short-length}, it follows that $\Phi_k(\bar{W}_k)\subset \Sigma^0_m$.
Therefore, 
\[[W_k,\Phi_k,\theta]= [\bar{W}_k,\rest{\Phi_k}{\bar{W}_k},\rest{\theta}{\bar{W}_k}] \in C^{\dr}_{p+q-n}(\Sigma^0_m) .\]
This proves the lemma.
\end{proof}

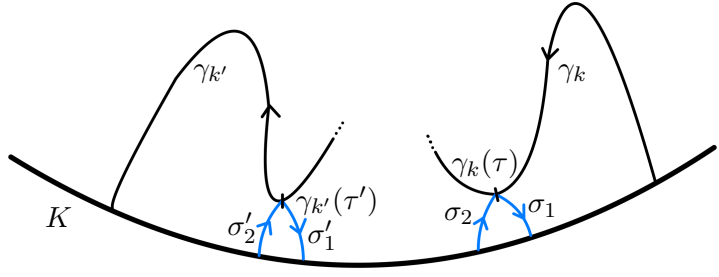
\begin{figure}
\centering
\begin{overpic}[width=10cm]{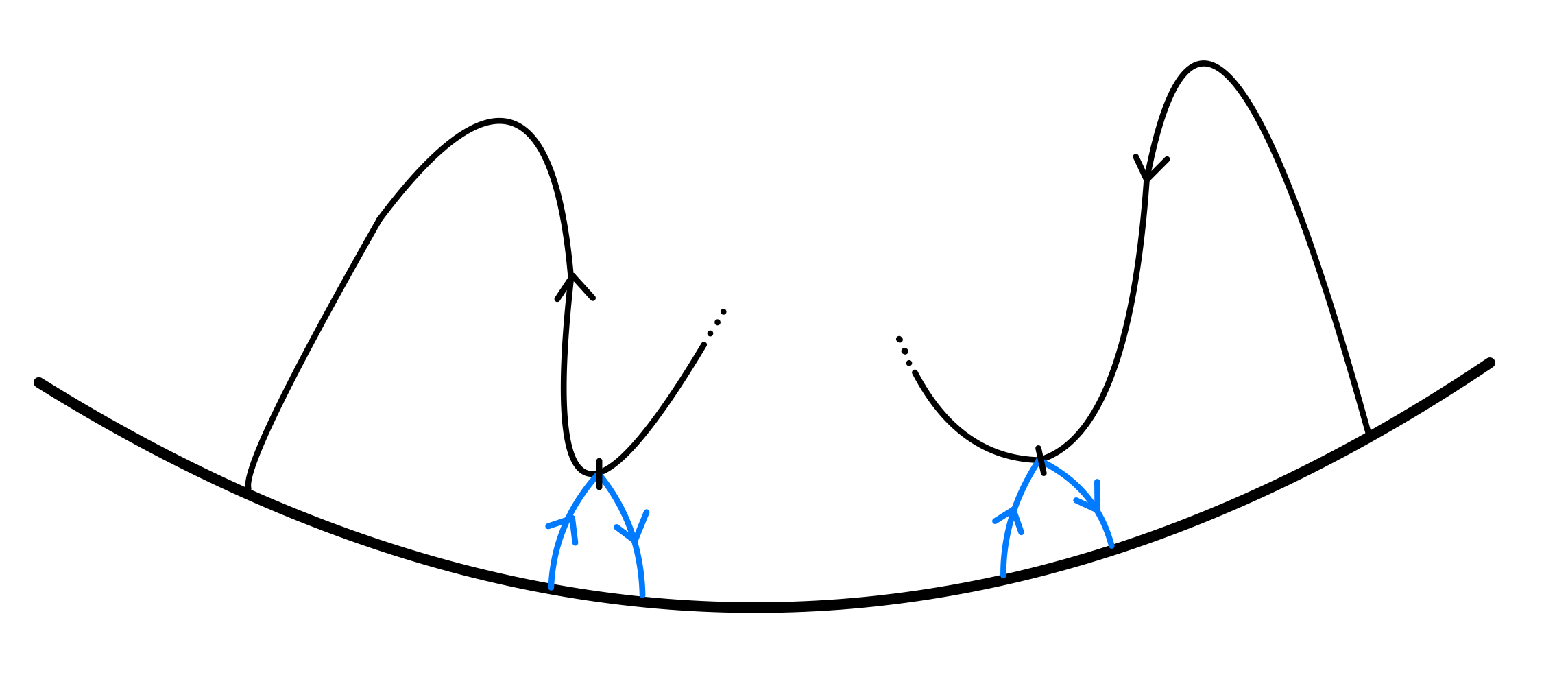}
\put(7,10){$K$}
\put(26.5,30){$\gamma_{k'}$}
\put(74.5,30){$\gamma_k$}
\put(31,9){$\sigma'_2$}
\put(41.5,8){$\sigma'_1$}
\put(40,12.5){$\gamma_{k'}(\tau')$}
\put(59.5,11){$\sigma_2$}
\put(70.5,12){$\sigma_1$}
\put(60.5,17.5){$\gamma_k(\tau)$}
\end{overpic}
\caption{The case where $(\gamma_l)_{l=1,\dots ,m}$ intersects both $(\sigma_i)_{i=1,2}$ and $(\sigma'_i)_{i=1,2}$.}\label{figure-two-path}
\end{figure}

The next lemma is crucial to define chain complexes in Section \ref{subsec-def-of-chain-cpx}. Before stating it, let us give an observation. Suppose that we have
\[\begin{cases}
(\gamma_l)_{l=1,\dots , m} \in \Sigma^a_m, \\
(\sigma_i\colon [0,\epsilon_i] \to N_{\epsilon})_{i=1,2},(\sigma'_i \colon [0,\epsilon'_i] \to N_{\epsilon})_{i=1,2} \in S_{\epsilon}, \\
k,k'\in \{1,\dots ,m\} \text{ with }k\leq k',
\end{cases}\]
and there exist two points $\tau\in (2\epsilon ,T_k-2\epsilon )$ and $\tau'\in (2\epsilon ,T_{k'}-2\epsilon)$ such that
\[\gamma_k(\tau)= \sigma_1(0),\ \gamma_{k'}(\tau')=\sigma'_1(0). \]
When $k=k'$, we additionally assume that $\tau + 2\epsilon < \tau'$. (Figure \ref{figure-two-path} describes the situation we consider.) Then, we can split $\gamma_k$ at $t=\tau$ and $\gamma_{k'}$ at $t=\tau'$, and concatenate them with $(\sigma_i)_{i=1,2}$ and $(\sigma'_i)_{i=1,2}$ respectively. Depending on which point we use first, there are two elements
\begin{align*}
\Phi &\coloneqq \con_k ((\gamma_l)_{l=1,\dots , m} , (T_k,\tau), (\sigma_i)_{i=1,2}) \in \Sigma^{a+\epsilon}_{m+1} , \\
\Phi' & \coloneqq \con_{k'} ((\gamma_l)_{l=1,\dots , m} , (T_{k'},\tau'), (\sigma'_i)_{i=1,2}) \in \Sigma^{a+\epsilon}_{m+1}.
\end{align*}
In either case, there remains another point which we have not yet used.
When $k<k'$, we can split the $(k'+1)$-th path of $\Phi$ at $t=\tau'$ and the $k$-th path of $\Phi'$ at $t=\tau$, and concatenate them with $(\sigma'_i)_{i=1,2}$ and $(\sigma_i)_{i=1,2}$ respectively.
When $k=k'$, we can split the $(k+1)$-th path of $\Phi$ at $t=\tau'-\tau+2\epsilon_2$ and the $k$-th path of $\Phi'$ at $t=\tau$, and concatenate them with  $(\sigma'_i)_{i=1,2}$ and $(\sigma_i)_{i=1,2}$ respectively.
After these two steps, we get the following equations:
\begin{align}\label{change-of-order}
\begin{cases}
\con_{k'+1} (\Phi, (T_{k'},\tau'), (\sigma'_i)_{i=1,2}) = \con_k (\Phi', (T_k,\tau), (\sigma_i)_{i=1,2})& \text{if }k<k' , \\
\con_{k+1} (\Phi, (\tilde{T}^2_k , \tau'-\tau +2\epsilon_2), (\sigma'_i)_{i=1,2}) = \con_k(\Phi', (\tilde{T}^1_k ,\tau),(\sigma_i)_{i=1,2} ) &\text{if }k=k'.
\end{cases}
\end{align}
Here, $\tilde{T}^2_k\coloneqq T_k-\tau +2\epsilon_2$ and $\tilde{T}^1_k\coloneqq \tau'+2\epsilon'_1$.
This observation leads us to the following lemma about de Rham chains.

\begin{lem}\label{lem-delta-commute}
For $x\in C^{\dr}_p(\Sigma^a_m)$, $\xi\in C^{\dr}_{q}(S_{\epsilon})$ and $k,k'\in \{1,\dots ,m\}$ with $ k\leq k'$, the following hold:
\[
\begin{cases}
f_{k'+1,\xi} \circ f_{k,\xi}(x)+ (-1)^{q-n} f_{k,\xi}\circ f_{k',\xi}(x)=0 &\text{ if }k<k', \\
f_{k'+1,\xi} \circ f_{k,\xi}(x)+ (-1)^{q-n} f_{k,\xi}\circ f_{k',\xi}(x)\in C^{\dr}_{p+2+2q-2n }(\Sigma^{0}_{m+2}) & \text{ if }k=k'.
\end{cases}\]
\end{lem}

\begin{proof}
We use the notations of (\ref{explicit-rep}) for  $x=[U,\varphi,\omega]$ and $\xi=[V,\psi,\eta]$. For short, let us abbreviate for $(T,\tau,T',\tau')\in A_{\epsilon}\times A_{\epsilon}$
\[\rho(T,\tau,T',\tau)\coloneqq \rho_{\epsilon}(T,\tau)\cdot \rho_{\epsilon}(T',\tau') . \]

\textit{Case 1.} We consider the case $k<k'$. We have
$f_{k'+1,\xi}\circ f_{k,\xi} (x)=(-1)^{s} [ W_{k,k'},\Phi_{k,k'},\zeta_{k,k'} ]$ for
\begin{align*}
W_{k,k'}& \coloneqq \{(u,\tau,v, \tau',v')\mid (u,\tau,v)\in W_k,\ 2\epsilon <\tau' <T^u_{k'}- 2\epsilon ,\ \gamma^u_{k'}(\tau')=\sigma^{v'}_1(0)\} ,\\
\Phi_{k,k'}&\colon W_{k,k'}\to \Sigma^{a+2\epsilon}_{m+2} \colon (u,\tau,v, \tau',v') \mapsto \con_{k'+1}( \Phi_k(u,\tau,v), (T^u_{k'},\tau'),\psi(v')) ,\\
\zeta_{k,k'}&\in \Omega^{*}_c(W_{k,k'}) \colon (\zeta_{k,k'})_{(u,\tau,v, \tau',v')}\coloneqq \rho (T^u_k,\tau,T^u_{k'},\tau')\cdot (\omega_u\times \eta_v\times \eta_{v'}) ,\\
s&\coloneqq (q+1-n)|\eta|,
\end{align*}
by substituting $k$ and $[U,\varphi,\omega]$ in (\ref{explicit-rep}) with $k'$ and $[W_k,\Phi_k,(-1)^r\zeta_k]$.
Similarly, $f_{k}\circ f_{k'}(x)= (-1)^s[W'_{k,k'},\Phi'_{k,k'}, \zeta'_{k,k'}]$ for
\begin{align*}
W'_{k,k'}&\coloneqq \{(u,\tau',v',\tau,v) \mid (u,\tau',v')\in W_{k'},\ 2\epsilon <\tau <T^u_k- 2\epsilon ,\ \gamma^u_k(\tau)=\sigma^{v}_1(0) \} ,\\
\Phi'_{k,k'} &\colon W'_{k,k'}\to \Sigma^{a+2\epsilon}_{m+2}\colon (u,\tau',v',\tau,v) \mapsto \con_k (\Phi_{k'}(u,\tau',v'),(T^u_k,\tau),\psi(v)) ,\\
\zeta'_{k,k'}&\in \Omega^{*}_c(W'_{k,k'}) \colon (\zeta'_{k,k'})_{(u,\tau',v'\tau,v)}=\rho(T^u_{k'},\tau',T^u_k,\tau)\cdot (\omega_u\times \eta_{v'}\times \eta_v).
\end{align*}
We define a diffeomorphism
\[h\colon W_{k,k'}\to W'_{k,k'}\colon (u,\tau,v,\tau',v') \mapsto (u,\tau',v',\tau,v) , \]
which changes the sign of orientation by $(-1)^{(1+\dim V-n)^2}$. From (\ref{change-of-order}), it follows that $\Phi'_{k,k'}\circ h=\Phi_{k,k'} $.
Moreover, $h_*( \zeta_{k,k'})= (-1)^{|\eta|^2} \zeta'_{k,k'}$ holds.
Combining these computations,
\[f_{k'+1,\xi}\circ f_{k,\xi}(x)= (-1)^{s+(1+\dim V-n+|\eta|) } [W'_{k,k'},\Phi'_{k,k'},\zeta'_{k,k'}]=(-1)^{q-n+1}f_{k,\xi}\circ f_{k',\xi} (x).\]

\textit{Case 2.} We consider the case $k=k'$. We have $f_{k+1}\circ f_k(x)= (-1)^s [ W_{k,k},\Phi_{k,k},\zeta_{k,k} ]$ for
\begin{align*}
W_{k,k}&\coloneqq \{(u,\tau,v,\tau',v') \mid (u,\tau,v)\in W_k,\ 2\epsilon <\tau' < \tilde{T}^2_k(u,\tau,v) -2\epsilon,\  \gamma^u_k(\tau'+ \tau-2\epsilon^v_2)=\sigma^{v'}_2(0)\} , \\
\Phi_{k,k}&\colon W_{k,k}\to \Sigma^{a+2\epsilon}_{m+2}\colon (u,\tau,v,\tau',v') \mapsto \con_{k+1}(\Phi_k(u,\tau,v),(\tilde{T}^2_k(u,\tau,v),\tau'),\psi(v')) , \\
\zeta_{k,k}&\in \Omega^*_c(W_{k,k}) \colon (\zeta_{k,k})_{(u,\tau,v,\tau',v')}= \rho(T^u_k,\tau,\tilde{T}^2_k(u,\tau,v),\tau')\cdot (\omega_u\times \eta_v\times \eta_{v'}) ,
\end{align*}
where $\tilde{T}^2_k(u,\tau,v)\coloneqq T^u_k-\tau+2\epsilon^v_2$.
Similarly, $f_{k,\xi}\circ f_{k,\xi}(x)=(-1)^s[W'_{k,k},\Phi'_{k,k},\zeta'_{k,k}]$ for
\begin{align*}
W'_{k,k}&\coloneqq \{(u,\tau',v',\tau,v) \mid (u,\tau',v')\in W_k,\ 2\epsilon<\tau <\tilde{T}^1_k(\tau',v')-2\epsilon ,\ \gamma^u_k(\tau)=\sigma^{v}_1(0)\} , \\
\Phi'_{k,k} &\colon W'_{k,k}\to \Sigma^{a+2\epsilon}_{m+2}\colon (u,\tau',v',\tau,v) \mapsto \con_k(\Phi_k(u,\tau',v'), (\tilde{T}^1_k(\tau',v'),\tau),\psi(v)) , \\
\zeta'_{k,k}&\in \Omega^*_c(W_{k,k}) \colon (\zeta'_{k,k})_{(u,\tau',v',\tau,v)}= \rho(T^u_k,\tau',\tilde{T}^1_k(\tau',v'),\tau)\cdot (\omega_u\times \eta_{v'}\times \eta_v) ,
\end{align*}
where $\tilde{T}^1_k(\tau',v')\coloneqq \tau'+2\epsilon^{v'}_1$.
Since $\rho(T_k,\tau,\tilde{T}^2_k(u,\tau,v),\tau')=0$ for $\tau'\leq \frac{10}{3}\epsilon$, we have $f_{k+1}\circ f_k(x)= (-1)^s [ \bar{W}_{k,k},\Phi_{k,k},\zeta_{k,k} ]$ for $\bar{W}_{k,k}\coloneqq W_{k,k}\cap \{\tau'>3\epsilon \}$.

This time, we define a map
\[h\colon \bar{W}_{k,k}\to W'_{k,k}\colon  (u,\tau,v,\tau',v')\mapsto (u,\tau'+\tau-2\epsilon^{v}_2,v', \tau,v), \]
which is an open embedding and changes the sign of orientation by $(-1)^{(1+\dim V-n)^2}$. From (\ref{change-of-order}), it follows that $\Phi'_{k,k}\circ h=\Phi_{k,k}$. Indeed, if we set $\tau'_*\coloneqq \tau'+\tau -2\epsilon^v_2$, then
\begin{align*}
\Phi'_{k,k}\circ h (u,\tau,v,\tau',v')&=\con_k(\Phi_k(u,\tau'_*,v'), (\tilde{T}^1_k(\tau'_*,v'),\tau),\psi(v)) \\
&=\con_{k+1}(\Phi_k(u,\tau,v),(\tilde{T}^1_k(u,\tau,v), \tau'_*-\tau+2\epsilon^v_2),\psi(v')) \\
&=\Phi_{k,k} (u,\tau,v,\tau',v')
\end{align*}
holds for every $(u,\tau,v,\tau',v')\in \bar{W}_{k,k}$.
Therefore, we have
\[ (-1)^{q-n+1} f_{k+1,\xi}\circ f_{k,\xi}(x) - f_{k,\xi}\circ f_{k,\xi}(x) = (-1)^s [W'_{k,k},\Phi'_{k,k},(-1)^{|\eta|^2} h_*(\zeta_{k,k})-\zeta'_{k,k} ] .\]
For $(u,\tau',v',\tau,v) \in W'_{k,k}$,
\begin{align*}
& (-1)^{|\eta|^2} ( h_*(\zeta_{k,k}) )_{(u,\tau',v',\tau,v)} - (\zeta'_{k,k})_{(u,\tau',v',\tau,v)} \\
=& \left( \rho(T^u_k,\tau,\tilde{T}^2_k(u,\tau,v), \tau'-\tau+2\epsilon^v_2 ) -\rho(T^u_k,\tau',\tilde{T}^1_k(\tau',v'),\tau)\right) \cdot (\omega_u\times \eta_{v'}\times \eta_v).
\end{align*}
If $\frac{11}{3}\epsilon \leq \tau' \leq T^u_k- \frac{11}{3}\epsilon $ and $ \frac{11}{3}\epsilon \leq \tau \leq  \tilde{T}^1_k(\tau',v')-\frac{11}{3}\epsilon $, it can be checked that $\frac{11}{3}\epsilon\leq \tau'-\tau+2\epsilon^v_2\leq \tilde{T}^2_k(u,\tau,v)-\frac{11}{3}\epsilon$ and $\frac{11}{3}\leq \tau\leq T^u_k-\frac{11}{3}\epsilon$ hold, and thus,
\[\rho(T^u_k,\tau,\tilde{T}^2_k(u,\tau,v), \tau'-\tau+2\epsilon^v_2 )  =1 = \rho(T^u_k,\tau',\tilde{T}^1_k(\tau',v'),\tau). \]
Therefore, $\supp ( (-1)^{|\eta|^2}h_*(\zeta_{k,k})-\zeta'_{k,k})$ lies in $W^1_{k,k}\cup W^2_{k,k}$, where
\begin{align*}
W^1_{k,k}&\coloneqq \{(u,\tau',v',\tau,v)\in W'_{k,k} \mid \tau' <4\epsilon \text{ or } T^u_k-4\epsilon<\tau' \}, \\
W^2_{k,k}&\coloneqq \{(u,\tau',v',\tau,v)\in W'_{k,k} \mid \tau<4\epsilon \text{ or } \tilde{T}^1_k(\tau',v') -4\epsilon <\tau \}.
\end{align*}
From  Lemma \ref{lem-short-length}, we have $\Phi_k(u,\tau',v') = \con_k( \varphi(u),(T^u_k,\tau'),\psi(v'))\in \Sigma^0_{m+1}$ for all $(u,\tau',v',\tau,v)\in W^1_{k,k}$. Then, Lemma \ref{lem-short-length} is applied again to show that
\[\Phi'_{k,k}(u,\tau',v',\tau,v)=\con_k (\Phi_k(u,\tau',v'),( \tilde{T}^1_k(\tau',v') ,\tau),\psi(v))\]
is an element of $\Sigma^0_{m+2}$ for every $(u,\tau',v',\tau,v)\in W^1_{k,k}\cup W^2_{k,k}$. Indeed, we can apply the case (i) of Lemma \ref{lem-short-length} for $(u,\tau',v',\tau,v)\in W^1_{k,k}$, and the case (ii) for $(u,\tau',v',\tau,v)\in W^2_{k,k}$.
As a consequence,
\begin{align*}
&(-1)^{q-n+1}f_{k+1,\xi}\circ f_{k,\xi}(x)-f_{k,\xi}\circ f_{k,\xi}(x) \\
=&\left[ W^1_{k,k} \cup W^2_{k,k},\rest{ \Phi'_{k,k} }{ W^1_{k,k} \cup W^2_{k,k} }, \rest{ \left(  (-1)^{|\eta|^2}h_*(\zeta_{k,k})-\zeta'_{k,k}\right) }{ W^1_{k,k} \cup W^2_{k,k} } \right] \in C^{\dr}_{p+2+2q-2n}(\Sigma^0_{m+2}) .
\end{align*}
\end{proof}

In the definition of $f_{k,\xi}$ for $\xi \in C^{\dr}_{q}(S_{\epsilon})$, there is an ambiguity about the choice of $\rho_{\epsilon} \colon A_{\epsilon} \to [0,1]$. Suppose that we choose another $\rho'_{\epsilon}$, and define $\alpha'_{\epsilon,\nu}$ $f'_{k,\xi}$ in the same way as $\alpha_{\epsilon,\nu}$ and $f_{k,\xi}$. Take an arbitrary chain $x\in C^{\dr}_p(\Sigma^a_m)$. Since $\supp (\rho'_{\epsilon}-\rho_{\epsilon})$ lies in $\{(T,\tau) \in A_{\epsilon}\mid \tau<4\epsilon \text{ or }T-4\epsilon<\tau \}$, $x\ftimes{\tl_k}{\pr_T}(\alpha_{\epsilon,\nu}-\alpha'_{\epsilon,\nu})$ is a chain in the subspace
\[\{( (\gamma_k)_{k=1,\dots , m},(T,\tau)) \in \Sigma^a_m \ftimes{\tl_k}{\pr_T}A_{\epsilon} \mid \tau<4\epsilon \text{ or }T-4\epsilon<\tau\}.\]
From Lemma \ref{lem-short-length}, we have
\[f_{k,\xi}(x)-f'_{k,\xi}(x)=(\con_k)_*((x\ftimes{\tl_k}{\pr_T}(\alpha_{\epsilon,\nu}-\alpha'_{\epsilon,\nu}) ) \ftimes{\ev_k}{\ev_0} \xi)\in C^{\dr}_{p+1+q-n}(\Sigma^0_{m+1}).\]
Therefore, for $a\in \R_{\geq 0}$, the induced map on the quotient spaces, which is denoted by the same symbol,
\[f_{k,\xi} \colon C^{\dr}_*(\Sigma^a_m,\Sigma^0_m ) \to C^{\dr}_{*+1+q-n}(\Sigma^{a+\epsilon}_{m+1},\Sigma^{0}_{m+1} ) \]
is independent on the choice of $\rho_{\epsilon}$.
For this map, the equation
\begin{align}\label{delta-small-quotient}
\partial \circ f_{k,\xi} -f_{k,\xi}\circ \partial = (-1)^{p+1-n} f_{k,\partial\xi} &\colon C^{\dr}_p(\Sigma^a_m,\Sigma^0_m)\to C^{\dr}_{p+1+q-n}(\Sigma^{a+\epsilon}_{m+1},\Sigma^{0}_{m+1}) \end{align}
follows from Lemma \ref{lem-delta-small}, and the equation for $k'\geq k$
\begin{align}\label{delta-commute-quotient}
f_{k'+1,\xi}\circ f_{k,\xi} + (-1)^{q-n} f_{k,\xi}\circ f_{k',\xi}=0 &\colon C^{\dr}_p(\Sigma^a_m,\Sigma^0_m)\to C^{\dr}_{p+2+2q-2n}(\Sigma^{a+2\epsilon}_{m+2}, \Sigma^{0}_{m+2})
\end{align}
follows from  Lemma \ref{lem-delta-commute}. When $\partial \xi=0$, (\ref{delta-small-quotient}) implies that $f_{k,\xi}$ is a chain map shifting the degree by $(1+q-n)$.

\subsection{$[-1,1]$ and $[-1,1]^2$-modeled de Rham chains}\label{subsec-[-1,1]-model}

In this section, we introduce two types of variants of de Rham chains. In this paper,
they are necessary for only four kinds of differentiable spaces: $\Sigma^a_m$, $\Sigma^a_m\times \Sigma^{a'}_{m'}$, $S_{\epsilon}$ and $(M,P_M)$ for a manifold $M$.
Throughout this section, $X$ 
represents these differentiable space.
%
\subsubsection{$[-1,1]$-modeled de Rham chains}
We introduce chains in $\R\times X$.
We define $\bar{P}_X$ as the set of tuples $(U,\varphi,(\tau_+,\tau_-))$ such that:
\begin{itemize}
\item $(U,\varphi)\in P_{\R^{\reg}\times X}$.
If $X=S_{\epsilon}$, we additionally require that $(\id_{\R}\times \ev_0)\circ \varphi \colon U\to \R\times Q$ is a submersion.
Let us write $\varphi=(\varphi_{\R},\varphi_X)\colon U\to \R\times X$ and $U_I \coloneqq \varphi_{\R}^{-1}(I)$ for any subset $I\subset \R$.
\item $\tau_+\colon U_{ \R_{\geq 1} }\to \R_{\geq 1}\times U_{\{1\}} $ and $\tau_-\colon U_{\R_{\leq -1}}\to \R_{\leq -1}\times U_{\{-1\}}$ are diffeomorphisms such that
\[\varphi\circ \tau_+^{-1}=i_{\R_{\geq 1}} \times \rest{ \varphi_X}{U_{ \{1\}}},\ \varphi\circ \tau_-^{-1}=i_{\R_{\leq -1}} \times \rest{ \varphi_X}{U_{\{ -1\}}}.\]
Here, $i_{\R_{\geq 1}}$ (resp. $i_{\R_{\leq{-1}}}$) is the inclusion map from $\R_{\geq 1}$ (resp. $\R_{\leq {-1}}$) to $\R$.
\end{itemize}

\begin{rem}When $X=S_{\epsilon}$,
the condition that $(U,\varphi)\in P_{\R^{\reg}\times S_{\epsilon}}$ implies only that the composition of $(\id_{\R}\times \ev_0)\circ \varphi \colon U\to \R\times Q$ with $\pr_{\R}$ (resp. $\pr_Q$) is a submersion to $\R$ (resp. $Q$).
The condition that $(\id_{\R}\times \ev_0)\circ \varphi$ itself is a submersion is necessary to define a fiber product of $[-1,1]$-modeled de Rham chains in the latter subsection.
\end{rem}

For $(U,\varphi,(\tau_+,\tau_-))\in \bar{P}_X$, we define a linear subspace $\Omega^p_c(U,\varphi,(\tau_+,\tau_-))$ of $\Omega^p(U)$ which consists of $p$-forms $\omega$ on $U$ such that $\supp \omega\cap U_{[-1,1]}$ is compact, $(\tau^{-1}_+)^*\omega = 1\times \rest{ \omega}{U_{\{1\}}}$ and $(\tau^{-1}_-)^*\omega =1\times \rest{\omega}{U_{\{-1\}}}$.
We consider a graded $\R$-vector space
\[\bar{A}_*(X)\coloneqq \bigoplus_{(U,\varphi,(\tau_+,\tau_-))\in \bar{P}_X} \Omega_c^{\dim U- 1-*}(U,\varphi,(\tau_+,\tau_-)).\]
For $\bold{U}=(U,\varphi,(\tau_+,\tau_-))\in \bar{P}_X$ and $\omega \in \Omega_c^p(U,\varphi,(\tau_+,\tau_-))$, let $(U,\varphi,(\tau_+,\tau_-),\omega)$ denote the element of $\bar{A}_*(X)$ such that its component for $\bold{V}\in \bar{P}_X$ is
\[(U,\varphi, (\tau_+,\tau_-),\omega)_{\bold{V} }=
\begin{cases}\omega & \text{ if } \bold{V}=\bold{U}, \\
 0 & \text{ if }\bold{V}\neq \bold{U} .
\end{cases}\]
We take a linear subspace $\bar{Z}_*(X)$ of $\bar{A}_*(X)$ generated by vectors
\[(V,\varphi\circ\pi,(\sigma_+,\sigma_-),\omega )- (U,\varphi,(\tau_+,\tau_-),\pi_!\omega)\] 
for any submersion $\pi\colon V\to U$ such that
\[(\id_{\R_{\geq 1}}\times \rest{\pi}{V_{\{1\}}} ) \circ \sigma_+=\tau_+\circ \pi ,\ (\id_{\R_{\leq -1}}\times \rest{\pi}{V_{\{-1\}}} ) \circ\sigma_-=\tau_-\circ \pi .\]
We define a quotient vector space
\[\bar{C}^{\dr}_*(X)\coloneqq \bar{A}_*(X)/\bar{Z}_*(X),\]
whose elements we call $\mathit{[-1,1]}$\textit{-modeled de Rham chains}.
$[U,\varphi,(\tau_+,\tau_-),\omega]$ denotes the equivalence class of $(U,\varphi,(\tau_+,\tau_-),\omega)$.
We define a degree $(-1)$ linear map
$\partial \colon \bar{C}^{\dr}_*(X) \to \bar{C}^{\dr}_{*-1}(X)$ by
\[\partial [U,\varphi,(\tau_+,\tau_-),\omega]\coloneqq  (-1)^{|\omega|+1}[U,\varphi,(\tau_+,\tau_-),d\omega].\]
Obviously $\partial \circ \partial =0$ holds and we obtain a chain complex $( \bar{C}^{\dr}_*(X),\partial)$. Its homology is denoted by $\bar{H}^{\dr}_*(X)$.

Naturally, there are three chain maps:
\begin{align}\label{bar-i}
\bar{i}\colon C^{\dr}_*(X) \to \bar{C}^{\dr}_*(X)\colon  [V,\psi,\omega] \mapsto (-1)^{\dim V}[\R\times V, \id_{\R}\times \psi , (\id_{\R_{\geq 1}\times V}, \id_{\R_{\leq -1}\times V} ), 1\times \omega] 
\end{align}
and
\begin{align}\label{e-pm}
\begin{split}
e_+&\colon \bar{C}^{\dr}_*(X)\to C^{\dr}_*(X) \colon[U,\varphi,(\tau_+,\tau_-),\omega] \mapsto  (-1)^{\dim U -1} [U_{\{1\}},\rest{\varphi_X}{U_{\{1\}}}, \rest{\omega}{U_{\{1\}}}] , \\
e_-&\colon \bar{C}^{\dr}_*(X)\to C^{\dr}_*(X) \colon [U,\varphi,(\tau_+,\tau_-),\omega] \mapsto  (-1)^{\dim U -1} [U_{\{-1\}},\rest{\varphi_X}{U_{\{-1\}}}, \rest{\omega}{U_{\{-1\}}}] .
\end{split}
\end{align}
Here, $U_{\{1\}}$ and $U_{\{-1\}}$ are oriented so that $\tau_+$ and $\tau_-$ preserve orientations.
Clearly, $e_+\circ \bar{i}=e_-\circ \bar{i}=\id_{C^{\dr}_*(X)}$. For $\bar{i}\circ e_+$ and $\bar{i}\circ e_-$, the next result holds.
\begin{lem}\label{lem-rest-isom-bar}
$\bar{i}\circ e_+$ and $\bar{i}\circ e_-$ are chain homotopic to the identity map $\id_{\bar{C}^{\dr}_*(X) }.$
\end{lem}
\begin{proof}
This assertion is essentially proved in \cite[Lemma 4.8]{irie-pseudo}. We should notice that the result in the reference is proved for a specific differentiable space $\mathscr{L}_{k+1}(a)$ (a differentiable space of marked Moore loops in a manifold). However, even for a differentiable space $X$ considered in this section, we can extend the definition of a chain
\[K ( [U,\varphi,(\tau_+,\tau_-),\omega] )\coloneqq (-1)^{|\omega|+1}[\R \times U,\bar{\varphi},(\bar{\tau}_+,\bar{\tau}_-),\bar{\omega}]  \in \bar{C}^{\dr}_{*+1}(X)\]
for any $ [U,\varphi,(\tau_+,\tau_-),\omega] \in \bar{C}^{\dr}_*(X)$, which appears in the proof of \cite[Lemma 4.8]{irie-pseudo}. Then $K\colon \bar{C}^{\dr}_*(X) \to \bar{C}^{\dr}_{*+1}(X)$ gives a chain homotopy from $\id_{\bar{C}^{\dr}_*(X)}$ to $\bar{i} \circ e_+$. The proof for $\bar{i} \circ e_-$ is completely parallel.
\end{proof}

\subsubsection{$[-1,1]^2$-modeled de Rham chains}
Let us define the smooth map
\[\iota\colon (\R^2)^{\reg} \times X \to (\R^2)^{\reg}\times X \colon ((r_1,r_2),x)\mapsto ((r_2,r_1),x).\]
We often use the coordinate $(r_1,r_2)$ of $\R^2$ to denote its subsets, for instance $\R_{\geq 1}\times \R=\{r_1\geq 1\}$.

We introduce chains in $\R^2\times X$. 
We define $\hat{P}_X$ as the set of tuples $(U,\varphi,(\tau^1_+,\tau^1_-),(\tau^2_+,\tau^2_-))$ such that:
\begin{itemize}
\item $(U,\varphi)\in P_{(\R^2)^{\reg}\times X }$. If $X=S_{\epsilon}$, we additionally require that $(\id_{\R^2}\times \ev_0)\circ \varphi \colon U\to \R^2\times Q$ is a submersion. Let us write $\varphi= ((\varphi^1_{\R}, \varphi^2_{\R}),\varphi_X)\colon U\to \R^2 \times X$ and $U_D\coloneqq \{u\in U \mid (\varphi^1_{\R}(u),\varphi^2_{\R}(u)) \in D\}$ for any subset $D\subset \R^2$.
\item $\tau^j_+,\tau^j_-$ ($j=1,2$) are diffeomorphisms such that 
\begin{align*}
&\varphi\circ (\tau^1_+)^{-1}=i_{\R_{\geq 1}} \times \rest{ (\varphi^2_{\R} \times \varphi_X)}{U_{ \{r_1=1\}}},  \ \varphi\circ (\tau^1_-)^{-1}=i_{\R_{\leq -1}} \times \rest{ (\varphi^2_{\R} \times \varphi_X)}{U_{ \{r_1=-1\}}} , \\
&\iota\circ \varphi\circ (\tau^2_+)^{-1}=i_{\R_{\geq 1}} \times \rest{ (\varphi^1_{\R} \times \varphi_X)}{U_{ \{r_2=1\}}}, \ \iota \circ\varphi\circ (\tau^2_-)^{-1}=i_{\R_{\leq -1}} \times \rest{ (\varphi^1_{\R} \times \varphi_X)}{U_{ \{r_2=-1\}}}.
\end{align*}
\end{itemize}

For $(U,\varphi,(\tau^1_+,\tau^1_-),(\tau^2_+,\tau^2_-))\in \hat{P}_X$, we define a linear subspace
$\Omega_c^p(U,\varphi,(\tau^1_+,\tau^1_-),(\tau^2_+,\tau^2_-)) $ of $\Omega^p(U)$ which consists of $p$-forms $\omega$ on $U$ such that $\supp \omega \cap U_{[-1,1]\times[-1,1]}$ is compact and
\begin{align*}
((\tau^j_+)^{-1})^*\omega =1\times \rest{\omega}{U_{\{r_j=1\}} },\ ((\tau^j_-)^{-1})^*\omega =1\times \rest{\omega}{U_{r_j=-1}}\ (j=1,2).
\end{align*}
We consider the graded $\R$-vector space
\[\hat{A}_*(X)\coloneqq \bigoplus_{ (U,\varphi,(\tau^1_+,\tau^1_-),(\tau^2_+,\tau^2_-))\in \hat{P}_X } \Omega_c^{\dim U-2-*}(U,\varphi,(\tau^1_+,\tau^1_-),(\tau^2_+,\tau^2_-)) . \]
For $\bold{U}=(U,\varphi,(\tau^1_+,\tau^1_-),(\tau^2_+,\tau^2_-)) \in \hat{P}_X$ and $\omega \in \Omega^p_c(U,\varphi,(\tau^1_+,\tau^1_-),(\tau^2_+,\tau^2_-))$, let
\[(U,\varphi,(\tau^1_+,\tau^1_-),(\tau^2_+,\tau^2_-),\omega)\]
denote the element of $\hat{A}_*(X)$ such that its component for $\bold{V}\in \hat{P}_X$ is
\[(U,\varphi,(\tau^1_+,\tau^1_-),(\tau^2_+,\tau^2_-),\omega)_{\bold{V}} =\begin{cases}
\omega & \text{ if }\bold{V}=\bold{U}, \\
0 & \text{ if }\bold{V}\neq \bold{U}.
\end{cases} \]
We take a linear subspace $\hat{Z}_*(X)$ of $\hat{A}_*(X)$ generated by vectors
\[ (V,\varphi\circ \pi ,(\sigma^1_+,\sigma^1_-),(\sigma^2_+,\sigma^2_-),\omega) - (U,\varphi,(\tau^1_+,\tau^1_-),(\tau^2_+,\tau^2_-),\pi_!\omega)\]
for any  submersion $\pi \colon V\to U$
such that for $j=1,2$
\begin{align*}
(\id_{\R_{\geq 1}} \times \rest{\pi}{V_{\{r_j=1\} }}) \circ \sigma^j_+ = \tau^j_+\circ \pi,\ (\id_{\R_{\leq -1}} \times \rest{\pi}{V_{\{r_j=-1\} }}) \circ \sigma^j_- = \tau^j_-\circ \pi .
\end{align*}
Now we define a quotient vector space
\[\hat{C}^{\dr}_*(X)\coloneqq \hat{A}_*(X)/\hat{Z}_*(X),\]
whose elements we call $\mathit{[-1,1]^2}$\textit{- modeled de Rham chains}.
$[U,\varphi,(\tau^1_+,\tau^1_-),(\tau^2_+,\tau^2_-),\omega]$ denotes the equivalence class of $(U,\varphi,(\tau^1_+,\tau^1_-),(\tau^2_+,\tau^2_-),\omega)$. We define a degree $(-1)$ linear map
$\partial \colon \hat{C}^{\dr}_*(X) \to \hat{C}^{\dr}_{*-1}(X)$ by
\[\partial [U,\varphi,(\tau^1_+,\tau^1_-),(\tau^2_+,\tau^2_-),\omega] \coloneqq (-1)^{|\omega|+1} [U,\varphi,(\tau^1_+,\tau^1_-),(\tau^2_+,\tau^2_-),d\omega] .\]
Obviously $\partial \circ \partial =0$ holds and we obtain a chain complex $(\hat{C}^{\dr}_*(X),\partial )$. Its homology is denoted by $\hat{H}^{\dr}_*(X)$.

Naturally, there are six chain maps
\begin{align}\label{e-j-pm}
\begin{split}
\hat{i}^1,\hat{i}^2&\colon \bar{C}^{\dr}_*(X) \to \hat{C}^{\dr}_*(X) , \\
e^1_+,e^2_+, e^1_-,e^2_-&\colon \hat{C}^{\dr}_*(X)\to \bar{C}^{\dr}_*(X) ,
\end{split}
\end{align}
define as follows: $\hat{i}^1$ and $\hat{i}^2$ map $x= [V,\psi,(\tau_+,\tau_-),\omega]\in \bar{C}^{\dr}_*(X) $ to
\begin{align*}
\hat{i}^1 x &\coloneqq (-1)^{\dim V -1} [\R\times V, \id_{\R}\times \psi ,(\id_{\R_{\geq 1}\times V}, \id_{\R_{\leq -1}\times V}) , (\hat{\tau}_+,\hat{\tau}_- ), 1\times \omega] ,\\
\hat{i}^2 x &\coloneqq (-1)^{\dim V} [\R\times V,  \iota\circ (\id_{\R}\times \psi) , (\hat{\tau}_+,\hat{\tau}_- ), (\id_{\R_{\geq 1}\times V}, \id_{\R_{\leq -1}\times V} ), 1\times \omega] ,
\end{align*}
where $\hat{\tau}_+\colon \R\times V_{\R_{\geq 1}}\to \R_{\geq 1}\times (\R\times V_{\{1\}})$ and $\hat{\tau}_-\colon \R\times V_{\R_{\leq -1}}\to \R_{\leq -1}\times (\R\times V_{\{ -1\}})$ are determined by
\begin{align*}
&\hat{\tau}_+(r' ,\tau_+^{-1}(r,u_+)) =(r, (r',u_+)) \text{ for } r'\in \R \text{ and } (r,u_+)\in \R_{\geq 1}\times V_{\{1\}}, \\
&\hat{\tau}_-(r' , \tau_-^{-1}(r,u_-)) =(r, (r',u_-)) \text{ for } r'\in \R \text{ and } (r,u_-)\in \R_{\leq -1}\times V_{\{-1\}}.
\end{align*}
In addition, $e^j_+$ and $e^j_-$ ($j=1,2$) map $y =  [U,\varphi,(\tau^1_+,\tau^1_-),(\tau^2_+,\tau^2_-),\omega] \in \hat{C}^{\dr}_*(X)$ to
\begin{align*}
\begin{split}
e^1_+ y &\coloneqq  (-1)^{\dim U} [U_{\{1\}\times \R}, \rest{(\varphi^2_{\R} ,\varphi_X)}{U_{\{1\}\times \R}}, ( \rest{ \tau^2_+}{U_{\{1\}\times \R_{\geq 1}}}, \rest{\tau^2_-}{U_{\{1\}\times \R_{\leq -1}}} ), \rest{\omega}{U_{\{1\}\times \R}} ] ,\\
e^2_+y &\coloneqq (-1)^{\dim U -1} [U_{\R\times\{1\} }, \rest{(\varphi^1_{\R} ,\varphi_X)}{ U_{\R\times\{1\} } },( \rest{ \tau^1_+}{U_{\R_{\geq 1}\times\{1\} }}, \rest{\tau^1_-}{U_{\R_{\leq -1}\times\{1\} } } ), \rest{\omega}{ U_{\R\times\{1\} }} ] ,\\
e^1_- y &\coloneqq  (-1)^{\dim U} [U_{\{-1\}\times \R}, \rest{(\varphi^2_{\R} ,\varphi_X)}{U_{\{-1\}\times \R}}, ( \rest{ \tau^2_+}{U_{\{-1\}\times \R_{\geq 1}}}, \rest{\tau^2_-}{U_{\{-1\}\times \R_{\leq -1}}} ), \rest{\omega}{U_{\{-1\}\times \R}} ] ,\\
e^2_-y &\coloneqq (-1)^{\dim U-1 } [U_{\R\times\{-1\} }, \rest{(\varphi^1_{\R} ,\varphi_X)}{ U_{\R\times\{-1\} } },( \rest{ \tau^1_+}{U_{\R_{\geq 1}\times\{-1\} }}, \rest{\tau^1_-}{U_{\R_{\leq -1}\times\{-1\} } } ), \rest{\omega}{ U_{\R\times\{-1\} }} ].
\end{split}
\end{align*}
Here, the orientations of $U_{D}$ for $D=\{r_j=1\}, \{r_j=-1\}$ ($j=1,2$) are determined so that $\tau^j_+$ and $\tau^j_-$ preserve orientations.
The signs are chosen so that
\begin{align}\label{e-and-e-j}
e_+\circ e^1_+=e_+\circ e^2_+,\ e_-\circ e^1_+=e_+\circ e^2_-,\ e_+\circ e^1_-=e_-\circ e^2_+,\ e_-\circ e^1_-=e_-\circ e^2_-.
\end{align}
hold.
Clearly, $e^j_+\circ \hat{i}^j=e^j_-\circ \hat{i}^j=\id_{\bar{C}^{\dr}_*(X)}$ for $j=1,2$. For $\hat{i}^j\circ e^j_+$ and $\hat{i}^j\circ e^j_-$ ($j=1,2$), the next result holds
\begin{lem}\label{lem-rest-isom-hat}
$\hat{i}^j\circ e^j_+$ and $\hat{i}^j\circ e^j_-$ ($j=1,2$) are chain homotopic to $\id_{\hat{C}^{\dr}_*(X) }.$
\end{lem}
\begin{proof}We omit the detailed proof as Lemma \ref{lem-rest-isom-bar}.
For any $x= [U,\varphi, (\tau^1_+,\tau^1_-), (\tau^2_+,\tau^2_-),\omega] \in \hat{C}^{\dr}_*(X)$, let us define diffeomorphisms for $j=1,2$
\[\tilde{\tau}^j_+\colon \R\times U_{\{r_j\geq 1\}}\to \R_{\geq 1}\times (\R\times U_{\{r_j=1\}}),\ 
\tilde{\tau}^j_-\colon \R\times U_{\{r_j\leq-1\}}\to \R_{\leq -1}\times (\R\times U_{\{r_j=-1\}}),
\]
so that $(\tilde{\tau}^j_+)^{-1}(r_j, (r,u))= (r, (\tau^j_+)^{-1}(r_j,u))$ and $(\tilde{\tau}^j_-)^{-1}(r_j, (r,u))= (r, (\tau^j_-)^{-1}(r_j,u))$.
Referring to the proof of \cite[Lemma 4.8]{irie-pseudo}, we can find $\bar{\varphi}^j$, $\bar{\tau}^j_{\pm}$, $\bar{\omega}^j$ ($j=1,2$) to define 
\begin{align*}
K^1 (x) &\coloneqq (-1)^{|\omega|+1}[\R\times U, \bar{\varphi}^1, (\bar{\tau}^1_+, \bar{\tau}^1_-), (\tilde{\tau}^2_+, \tilde{\tau}^2_-), \bar{\omega}^1 ] , \\
K^2 (x) &\coloneqq (-1)^{|\omega|+1} [\R\times U, \bar{\varphi}^2, (\tilde{\tau}^1_+, \tilde{\tau}^1_-),  (\bar{\tau}^2_+, \bar{\tau}^2_-),\bar{\omega}^2] ,
\end{align*}
so that $K^j\colon \hat{C}^{\dr}_*(X) \to \hat{C}^{\dr}_{*+1}(X)$ is a chain homotopy from $\id_{ \hat{C}^{\dr}_*(X)}$ to $\hat{i}^j\circ e^j_+$ for $j=1,2$. The proof for $\hat{i}^j\circ e^j_-$ is completely parallel. 
\end{proof}

\subsubsection{Collection of analogies with ordinary de Rham chains}
As above, $X$ and $Y$ are chosen from one of the following differentiable space: $\Sigma^a_m$, $\Sigma^a_m\times \Sigma^{a'}_{m'}$, $S_{\epsilon}$ and $(M,P_M)$ for a manifold $M$.
Let $f\colon X \to Y$ be a smooth map. If $Y=S_{\epsilon}$, we require that $X=S_{\epsilon'}$ and $\ev_0\circ f=\ev_0$. Then, $f$ induces chain maps
\begin{align*}
f_*&\colon \bar{C}^{\dr}_*(X)\to \bar{C}^{\dr}_*(Y)\colon [U,\varphi,(\tau_+,\tau_-),\omega]\mapsto  [U,f\circ \varphi,(\tau_+,\tau_-),\omega], \\
f_*&\colon \hat{C}^{\dr}_*(X)\to \hat{C}^{\dr}_*(Y)\colon [U,\varphi,(\tau^1_+,\tau^1_-),(\tau^2_+,\tau^2_-),\omega]\mapsto  [U,f\circ \varphi,(\tau^1_+,\tau^1_-),(\tau^2_+,\tau^2_-),\omega].
\end{align*}
If $X=\Sigma^a_m$ and $Y=\Sigma^b_m$ for $a\leq b$ and $f$ is the inclusion map,
then we claim that the above maps are injective. This can be proved as Lemma \ref{lem-approx-smooth-level}, so we omit the proof. As a consequence, we can define $\bar{C}^{\dr}_*(\Sigma^b_m,\Sigma^a_m)$ and  $\hat{C}^{\dr}_*(\Sigma^b_m,\Sigma^a_m)$ as quotient complexes.

Next, let $(X,Y)=(\Sigma^a_m,\Sigma^{a'}_{m'})$ or $(\{0\},S_{\epsilon})$. We identify $\{0\}\times S_{\epsilon}$ with $S_{\epsilon}$. Then, a cross product
$x\times y \in \bar{C}^{\dr}_{p+q}(X\times Y)$
is defined for $x=[U,\varphi,(\tau_+,\tau_-), \omega]\in \bar{C}^{\dr}_p(X)$ and $y=[V,\psi,(\sigma_+,\sigma_-),\eta]\in \bar{C}^{\dr}_q(Y)$ by
\[x\times y\coloneqq (-1)^{p|\eta|} [W, \tilde{\varphi}, (\tilde{\tau}_+,\tilde{\tau}_-) , \omega\times \eta ].\]
Here, $W\coloneqq U\ftimes{\varphi_{\R} }{\psi_{\R}} V$ is a fiber product over $\R$ and $\tilde{\varphi},\tilde{\tau}_+,\tilde{\tau}_-$ are determined by
\begin{align*}
&\tilde{\varphi}\colon  W\to \R\times (X\times Y)\colon (u,v)\mapsto (\varphi_{\R}(u),\varphi_X(u),\psi_Y(v)),\\
&\tilde{\tau}_+(u,v)= (r,(u_+,v_+))\ \text{ for }(u,v)=((\tau_+)^{-1}(r,u_+),(\sigma_+)^{-1}(r,v_+))\in W_{\R_{\geq 1}}, \\
&\tilde{\tau}_-(u,v)= (r,(u_-,v_-))\ \text{ for }(u,v)=((\tau_-)^{-1}(r,u_-),(\sigma_-)^{-1}(r,v_-))\in W_{\R_{\leq -1}}.
\end{align*}

Similarly, a cross product $x\times y\in \hat{C}^{\dr}_{p+q}(X\times Y)$ 
is defined for $x=[U,\varphi,(\tau^1_+,\tau^1_-), (\tau^2_+,\tau^2_-), \omega]\in \hat{C}^{\dr}_p(X)$ and $y=[V,\psi,(\sigma^1_+,\sigma^1_-),(\sigma^2_+,\sigma^2_-), \eta]\in \hat{C}^{\dr}_q(Y)$ by
\[x\times y\coloneqq (-1)^{p|\eta|} [W, \tilde{\varphi}, (\tilde{\tau}^1_+,\tilde{\tau}^1_-) , (\tilde{\tau}^2_+,\tilde{\tau}^2_-) , \omega\times \eta ].\]
Here, $W\coloneqq U\ftimes{ (\varphi^1_{\R},\varphi^2_{\R}) }{(\psi^1_{\R}, \psi^2_{\R})} V$ is a fiber product over $\R^2$ and $\tilde{\varphi},\tilde{\tau}^j_+,\tilde{\tau}^j_-$ ($j=1,2$) are determined by
\begin{align*}
&\tilde{\varphi}\colon  W\to \R\times (X\times Y)\colon (u,v)\mapsto (\varphi_{\R}(u),\varphi_X(u),\psi_Y(v)),\\
&\tilde{\tau}^j_+(u,v)= (r,(u_+,v_+))\ \text{ for }(u,v)=((\tau^j_+)^{-1}(r,u_+),(\sigma^j_+)^{-1}(r,v_+))\in W_{\{r_j=1\}}, \\
&\tilde{\tau}^j_-(u,v)= (r,(u_-,v_-))\ \text{ for }(u,v)=((\tau^j_-)^{-1}(r,u_-),(\sigma^j_-)^{-1}(r,v_-))\in W_{\{r_j=-1\}}.
\end{align*}

The next result is analogous to Proposition \ref{prop-hmgy-grp} and Proposition \ref{prop-exactness}. It follows immediately from the fact that $(e_+)_*\colon \bar{H}^{\dr}_*(X)\to H^{\dr}_*(X)$ is an isomorphism (see Lemma \ref{lem-rest-isom-bar}.).
\begin{prop}\label{prop-analogy}
Let $a,b\in \R_{>0}$ with $a\leq b$ and $\epsilon \in (0,\epsilon_0/C]$ for the constant $C$ of Lemma \ref{lem-htpy-sev}. Then, the following hold:
\begin{itemize}
\item If $\mathcal{L}_m(K)\cap [a,b]=\varnothing$, then $\bar{H}^{\dr}_*(\Sigma^b_m,\Sigma^a_m)=0$.
\item For any $x\in \bar{H}^{\dr}_*(S_{\epsilon})$,
\[(\ev_0)_*\circ (e_+)_*(x)=0\in H^{\dr}_*(N_{\epsilon}) \Rightarrow (i_{\epsilon,C\epsilon})_*(x)=0\in \bar{H}^{\dr}_*(S_{C \epsilon})\]
\end{itemize}
\end{prop}

\subsubsection{Operations on $[-1,1]$ and $[-1,1]^2$-modeled de Rham chains}

In the rest of this section, let us define operators corresponding to $f_{k,\xi}$. We rather refer to the explicit description (\ref{explicit-rep}) of $f_{k,\xi}(x)$ than its original definition using fiber products of chains.
Let $\epsilon\in (0,\epsilon_0/(5C_0)]$ and $\rho_{\epsilon}\colon A_{\epsilon}\to [0,1]$ be the $C^{\infty}$ function we have chosen in the beginning of Section \ref{subsec-operation}

For $k=1,\dots ,m$ and $\bar{\xi} \in \bar{C}^{\dr}_q(S_{\epsilon})$, we define a linear map
\begin{align*}
\bar{f}_{k,\bar{\xi}}\colon \bar{C}^{\dr}_*(\Sigma^a_m) \to \bar{C}^{\dr}_{*+1+q-n}(\Sigma^{a+\epsilon}_{m+1})
\end{align*}
as follows: Let
\begin{align*}
x=[U,\varphi,(\tau_+,\tau_-),\omega] \in \bar{C}^{\dr}_p(\Sigma^a_m),\ 
\bar{\xi}=[V,\psi,(\sigma_+,\sigma_-),\eta] \in \bar{C}^{\dr}_q(S_{\epsilon}),
\end{align*}
and denote
\begin{align*}
\varphi(u) &=(\varphi_{\R} (u),\varphi_{\Sigma}(u))= (\varphi_{\R}(u), (\gamma^u_l\colon [0,T^u_l]\to Q)_{l=1,\dots ,m}) \in \R\times \Sigma^a_m, \\
\psi(v)&=(\psi_{\R}(v),\psi_S(v))= (\psi_{\R} (v), (\sigma^v_i)_{i=1,2}) \in \R\times S_{\epsilon},
\end{align*}
for every $u\in U$ and $v\in V$.
Then we define $\bar{f}_{k,\bar{\xi}}(x) \coloneqq (-1)^s[W_k,\Phi_k,(\tilde{\tau}_+,\tilde{\tau}_-), \zeta_k]$, where
\begin{align}\label{explicit-rep-bar}
\begin{split}
W_k &\coloneqq \{(u,\tau,v)\in U\times \R\times V \mid 2\epsilon <\tau < T^u_k-2\epsilon,\ (\varphi_{\R} (u),\gamma^u_k(\tau)) = (\psi_{\R}(u),\sigma^v_1(0)) \} , \\
\Phi_k&\colon W_k\to \R\times \Sigma^{a+\epsilon}_{m+1}\colon (u,\tau,v)\mapsto (\varphi_{\R}(u), \con_k(\varphi_{\Sigma}(u),(T^u_k,\tau), \psi_S(v)) ) , \\
\zeta_k&\in \Omega^{*}_c(W_k)\colon (\zeta_k)_{(u,\tau,v)}\coloneqq \rho_{\epsilon}(T^u_k,\tau)\cdot (\omega_u \times \eta_v ) , \\
s&\coloneqq (p+1-n)|\eta| +n +1 ,
\end{split}
\end{align}
and
\begin{align*}
\tilde{\tau}_+ (u,\tau,v) \coloneqq (r,(u_+,\tau,v_+)) \in  \R\times (W_k)_{\{r=1\}}
\end{align*}
for $(u,\tau,v)= (\tau_+^{-1}(r,u_+),\tau, \sigma^{-1}_+(r,v_+)) \in (W_k)_{\{r\geq 1\}}$, and 
\[\tilde{\tau}_- (u,\tau,v) \coloneqq (r,(u_-,\tau,v_-)) \in  \R\times (W_k)_{\{r=-1\}} \]
for $(u,\tau,v)= (\tau_-^{-1}(r,u_-),\tau, \sigma^{-1}_-(r,v_-)) \in (W_k)_{\{r\leq-1\}}$.
Here, $W_k$ is oriented as a fiber product over $\R\times Q$ of the map
\[ \{(u,\tau) \in U\times \R\mid 2\epsilon<\tau<T^u_k-2\epsilon\} \to \R\times Q\colon (u,\tau)\to (\varphi_{\R}(u), \gamma^u_k(\tau)) \]
and the submersion $(\id_{\R}\times \ev_0)\circ \psi\colon V\to \R\times Q$.
It can be checked that
$\bar{i}$ and $e_+,e_-$ intertwine this operator and $f_{k,\xi}$ for $\xi\in C^{\dr}_q(S_{\epsilon})$. Namely,
\[ \bar{i}\circ f_{k,\xi} =f_{k,\bar{i}\xi}\circ \bar{i},\  e_+ \circ \bar{f}_{k,\bar{\xi}} = f_{k,e_+\bar{\xi}} \circ e_+ ,\ e_- \circ \bar{f}_{k,\bar{\xi}} = f_{k,e_-\bar{\xi}} \circ e_-. \]

Similar results as $f_{k,\xi}$ are the following:  $\bar{f}_{k,\bar{\xi}}$ induces a linear map
\[ \bar{f}_{k,\bar{\xi}} \colon \bar{C}^{\dr}_*(\Sigma^a_m,\Sigma^0_m) \to \bar{C}^{\dr}_{*+1+q-n}(\Sigma^{a+\epsilon}_{m+1},\Sigma^0_{m+1}) ,\]
which is independent on $\rho_{\epsilon}$. The next equations are variants of  (\ref{delta-small-quotient}) and (\ref{delta-commute-quotient}), and they follow from similar computations as Lemma \ref{lem-delta-small} and Lemma \ref{lem-delta-commute}, so we omit the proof.
\begin{prop}\label{prop-f-bar-equations}
For $k'\geq k$, the following equations hold:
\begin{align*}
&\partial \circ  \bar{f}_{k,\bar{\xi}} -  \bar{f}_{k,\bar{\xi}}\circ \partial =(-1)^{p+1-n} \bar{f}_{k,\partial \bar{\xi}} \colon \bar{C}^{\dr}_p(\Sigma^a_m,\Sigma^0_m) \to \bar{C}^{\dr}_{p+1+q-n}(\Sigma^{a+\epsilon}_{m+1},\Sigma^0_{m+1}), \\
&\bar{f}_{k'+1,\bar{\xi}} \circ \bar{f}_{k,\bar{\xi}}  + (-1)^{q-n} \bar{f}_{k,\bar{\xi}} \circ \bar{f}_{k',\bar{\xi}} =0\colon  \bar{C}^{\dr}_p(\Sigma^a_m,\Sigma^0_m) \to \bar{C}^{\dr}_{p+2+2q-2n}(\Sigma^{a+2\epsilon}_{m+2},\Sigma^0_{m+2}).
\end{align*}
\end{prop}

Next, for $k=1,\dots ,m$ and $\hat{\xi}\in \hat{C}^{\dr}_q(S_{\epsilon})$, we define a linear map
\begin{align*}
\hat{f}_{k,\hat{\xi}}\colon \bar{C}^{\dr}_*(\Sigma^a_m) \to \bar{C}^{\dr}_{*+1+q-n}(\Sigma^{a+\epsilon}_{m+1})
\end{align*}
as follows: Let
\begin{align*}
x=[U,\varphi,(\tau^1_+,\tau^1_-), (\tau^2_+.\tau^2_-),\omega] \in \hat{C}^{\dr}_p(\Sigma^a_m),\\
\hat{\xi}=[V,\psi,(\sigma^1_+,\sigma^1_-),(\sigma^2_+,\sigma^2_-),\eta] \in \hat{C}^{\dr}_q(S_{\epsilon}),
\end{align*}
and denote
\begin{align*}
\varphi(u) &=(\varphi_{\R^2} (u),\varphi_{\Sigma}(u))= ((\varphi_{\R}^1(u),\varphi_{\R}^2(u)), (\gamma^u_l\colon [0,T^u_l]\to Q)_{l=1,\dots ,m}) \in \R^2 \times \Sigma^a_m, \\
\psi(v)&=(\psi_{\R^2}(v),\psi_S(v))= ((\psi_{\R}^1(v),\psi_{\R}^2(v)), (\sigma^v_i)_{i=1,2}) \in \R^2 \times S_{\epsilon},
\end{align*}
for every $u\in U$ and $v\in V$.
Then we define $\hat{f}_{k,\hat{\xi}}(x) \coloneqq (-1)^s[W_k,\Phi_k,(\tilde{\tau}^1_+,\tilde{\tau}^1_-), (\tilde{\tau}^2_+,\tilde{\tau}^2_-),  \zeta_k]$, where
\begin{align*}
W_k &\coloneqq \{(u,\tau,v)\in U\times \R\times V \mid 2\epsilon <\tau < T^u_k-2\epsilon,\ (\varphi_{\R^2} (u),\gamma^u_k(\tau)) = (\psi_{\R^2}(u),\sigma^v_1(0)) \} , \\
\Phi_k&\colon W_k\to \R^2\times \Sigma^{a+\epsilon}_{m+1}\colon (u,\tau,v)\mapsto (\varphi_{\R^2}(u), \con_k(\varphi_{\Sigma}(u),(T^u_k,\tau), \psi_S(v)) ) , \\
\zeta_k&\in \Omega^{*}_c(W_k)\colon (\zeta_k)_{(u,\tau,v)}\coloneqq \rho_{\epsilon}(T^u_k,\tau)\cdot (\omega_u \times \eta_v ) , \\
s&\coloneqq (p+1-n)|\eta| ,
\end{align*}
and for $j=1,2$,
\begin{align*}
\tilde{\tau}^j_+ (u,\tau,v) \coloneqq (r_j,(u^j_+,\tau,v^j_+)) \in  \R_{\geq 1}\times (W_k)_{\{r_j=1\}}
\end{align*}
for $(u,\tau,v)= ((\tau^j_+)^{-1}(r_j,u^j_+),\tau, (\sigma^j)^{-1}_+(r_j,v^j_+)) \in (W_k)_{\{r_j\geq 1\}}$ and 
\begin{align*}
\tilde{\tau}^j_- (u,\tau,v) \coloneqq (r_j,(u^j_-,\tau,v^j_-)) \in  \R_{\leq -1}\times (W_k)_{\{r_j=-1\}}
\end{align*}
for $(u,\tau,v)= ((\tau^j_-)^{-1}(r_j,u^j_-),\tau, (\sigma^j)^{-1}_-(r_j,v^j_-)) \in (W_k)_{\{r_j\leq -1\}}$.
Here, $W_k$ is oriented as a fiber product over $\R^2\times Q$ of the map
\[ \{(u,\tau) \in U\times \R\mid 2\epsilon<\tau<T^u_k-2\epsilon\} \to \R^2\times Q\colon (u,\tau)\to (\varphi_{\R^2}(u), \gamma^u_k(\tau)) \]
and the submersion $(\id_{\R^2}\times \ev_0)\circ \psi\colon V\to \R^2\times Q$.
It can be checked that
$\hat{i}^j$ and $e^j_+,e^j_-$ ($j=1,2$) intertwine this operator and $\bar{f}_{k,\bar{\xi}}$ for $\bar{\xi}\in \bar{C}^{\dr}_q(S_{\epsilon})$.

Similar results as $f_{k,\xi}$ are the following:  $\hat{f}_{k,\hat{\xi}}$ induces a linear map
\[ \hat{f}_{k,\hat{\xi}} \colon \hat{C}^{\dr}_*(\Sigma^a_m,\Sigma^0_m) \to \hat{C}^{\dr}_{*+1+q-n}(\Sigma^{a+\epsilon}_{m+1},\Sigma^0_{m+1}) ,\]
which is independent on $\rho_{\epsilon}$. The next equations are variants of  (\ref{delta-small-quotient}) and (\ref{delta-commute-quotient}), and they follow from similar computations as Lemma \ref{lem-delta-small} and Lemma \ref{lem-delta-commute}, so we omit the proof.
\begin{prop}\label{prop-f-hat-equations}
For $k'\geq k$, the following equations hold:
\begin{align*}
&\partial \circ  \hat{f}_{k,\hat{\xi}} -  \hat{f}_{k,\hat{\xi}}\circ \partial =(-1)^{p+1-n} \hat{f}_{k,\partial \hat{\xi}} \colon \hat{C}^{\dr}_p(\Sigma^a_m,\Sigma^0_m) \to \hat{C}^{\dr}_{p+1+q-n}(\Sigma^{a+\epsilon}_{m+1},\Sigma^0_{m+1}), \\
&\hat{f}_{k'+1,\hat{\xi}} \circ \hat{f}_{k,\hat{\xi}}  + (-1)^{q-n} \hat{f}_{k,\hat{\xi}} \circ \hat{f}_{k',\hat{\xi}} =0\colon  \hat{C}^{\dr}_p(\Sigma^a_m,\Sigma^0_m) \to \hat{C}^{\dr}_{p+2+2q-2n}(\Sigma^{a+2\epsilon}_{m+2},\Sigma^0_{m+2}).
\end{align*}
\end{prop}

\section{Construction of $H^{\str}_*(Q,K)$}\label{sec-string-hmgy}

\subsection{Definition of chain complex}\label{subsec-def-of-chain-cpx}

For $a\in \R_{> 0}$ and $\epsilon\in (0, \epsilon_0/(5C_0)]$, we define a graded $\R$-vector space
\[C^{<a}_*(\epsilon) \coloneqq \bigoplus_{m=0}^{\infty}C^{\dr}_{*-m(d-2)} (\Sigma^{a+m\epsilon}_m,\Sigma^{0}_m). \]
If $m\geq \frac{2a}{\epsilon_0}$, then $a+m\epsilon\leq m(\frac{1}{2}\epsilon_0+\epsilon)\leq m\epsilon_0$. 
In such case, $\Sigma^{a+m\epsilon}_m=\Sigma^0_m$ by Remark \ref{rem-large-m}. Therefore, the components for $m\in \Z_{\geq 0}$ vanishes if $m\geq \frac{2a}{\epsilon_0}$.

For each $m\in \Z_{\geq 0}$, we think of $C^{\dr}_{*-m(d-2)} (\Sigma^{a+m\epsilon}_m,\Sigma^{0}_m)$ as a linear subspace of $C^{<a}_*(\epsilon)$ in a natural way.
For $\delta\in C^{\dr}_{n-d}(S_{\epsilon})$,
we define a degree ($-1$) linear map
\[D_{\delta}\colon C^{<a}_*(\epsilon)\to C^{<a}_{*-1}(\epsilon)\]
so that
for $x\in C^{\dr}_{p-m(d-2)} (\Sigma^{a+m\epsilon}_m,\Sigma^{0}_m)$, 
\[D_{\delta} (x)= \partial x + \sum_{k=1}^m (-1)^{p+1+kd}  f_{k,\delta}(x) \in C_{p-1}^{<a}(\epsilon) .\]
When $m=0$, the RHS is just equal to $\partial x$.
\begin{prop}\label{prop-D-2}
If $\delta\in C^{\dr}_{n-d}(S_{\epsilon})$ satisfies $\partial \delta=0$, then
$D_{\delta}\circ D_{\delta}=0$ holds.
\end{prop}

\begin{proof}Take an arbitrary $x\in C^{\dr}_{p-m(d-2)} (\Sigma^{a+m\epsilon}_m,\Sigma^{0}_m)$.
Since $\partial\circ \partial =0$,
\begin{align*}
D_{\delta}\circ D_{\delta}(x)=&\sum_{k=1}^m \left( (-1)^{p+1+kd}\partial \circ f_{k,\delta}(x) + (-1)^{p+kd}f_{k,\delta}\circ \partial (x) \right) \\
& +\sum_{k'=1}^{m+1}\sum_{k=1}^m (-1)^{(k+k')d-1}f_{k',\delta}\circ f_{k,\delta}(x) .
\end{align*}
Applying (\ref{delta-small-quotient}) for $q=n-d$ and $\xi=\delta$ for which $\partial \delta=0$ holds, we can see that the first summand is equal to 0.
For the second summand, we apply (\ref{delta-commute-quotient}) for $q=n-d$ and $\xi=\delta$. Then
\begin{align*}
&\sum_{k'=1}^{m+1}\sum_{k=1}^m  (-1)^{(k+k')d-1} f_{k',\delta}\circ f_{k,\delta}(x) \\
=&\sum_{1\leq k\leq k' \leq m} \left( (-1)^{(k+k'+1)d-1}f_{k'+1,\delta}\circ f_{k,\delta}(x) +(-1)^{(k+k')d-1} f_{k,\delta}\circ f_{k',\delta}(x) \right) \\
=&0.
\end{align*}
This shows that $D_{\delta}\circ D_{\delta}(x)=0$.
\end{proof}

In summary, for $a\in \R_{>0}$, $\epsilon\in (0, \epsilon_0/(5C_0)]$ and $\delta \in C^{\dr}_{n-d}(S_{\epsilon})$ with $\partial \delta=0$, a chain complex $(C^{<a}_*(\epsilon),D_{\delta})$ is defined. Let $H_*^{<a}(\epsilon_,\delta)$
denote its homology.


The chain complex $(C^{<a}_*(\epsilon),D_{\delta})$ is filtered by subcomplexes $\{ \mathcal{F}^{<a}_{\epsilon,p} \}_{p\in \Z }$ defined by
\begin{align}\label{filtration}
\mathcal{F}^{<a}_{\epsilon,p}\coloneqq \bigoplus_{m\geq -p}^{\infty} C^{\dr}_{*-m(d-2)} (\Sigma^{a+m\epsilon}_{m},\Sigma^{0}_{m}) . 
\end{align}
Let $E^{<a}_{(\epsilon,\delta)} \coloneqq (\{(E^{<a}_{(\epsilon,\delta)})^r_{p,q}\}, \{(d^{<a}_{(\epsilon,\delta)})^r_{p,q}\})$ be the spectral sequence determined by $ \{ \mathcal{F}^{<a}_{\epsilon,p}\}_{p\in \Z}$.
Note that $\mathcal{F}^{<a}_{\epsilon, p}=0$ for $p\leq -\frac{2a}{\epsilon_0}$ and $\mathcal{F}^{<a}_{\epsilon,p}=\mathcal{F}^{<a}_{\epsilon,0}$ for $p\geq 0$, and thus this spectral sequence converges to $H^{<a}_*(\epsilon,\delta)$ in the sense of \cite[Bounded Convergence 5.2.5]{weibel}.
The first page is given by
\begin{align*}
 (E^{<a}_{(\epsilon,\delta)})^1_{-m,q}= \begin{cases} H^{\dr}_{(q-m)-m(d-2)}(\Sigma^{a+m\epsilon}_{m}, \Sigma^{0}_{m})= H^{\dr}_{q-m(d-1)}(\Sigma^{a+m\epsilon}_{m}, \Sigma^{0}_{m}) & \text{ if }m \geq 0,\\
 0 &\text{ if }m<0.\end{cases}
\end{align*}

Let us state an abstract lemma about morphisms  in the category of spectral sequences. This result, which is a refinement of \cite[Comparison Theorem 5.2.12]{weibel}, will be repeatedly used in the rest of this paper.
\begin{lem}\label{lem-spectral}
Let $E=(\{ E^r_{p,q}\} , \{d^r_{p,q}\})$ and $E'=(\{E'^r_{p,q}\}, \{d'^r_{p,q}\})$ be bounded spectral sequences which converge to $H_*$ and $H'_*$ respectively in the sense of \cite[Bounded Convergence 5.2.5]{weibel}.
Let $f= \{f^r_{p,q}\} $ be a morphism from $E$ to $E'$ which is compatible with $\{h_n \colon H_n\to H'_n\}_{n\in \Z}$.
Then, the following assertion hold:
\begin{itemize}
\item Suppose that for some $r_0\geq 1$ and $n_0\in \Z$, $f^{r_0}_{p,q}$ is an isomorphism if $p+q< n_0$, and a surjection if $p+q=n_0$. Then,
$h_n$ is an isomorphism if $n<n_0$ and a surjection if $n=n_0$. In particular, if  $f^{r_0}_{p,q}$ is an isomorphism for every $p,q\in \Z$ for some $r_0\geq 1$, then $h_n$ is an isomorphism for every $n\in \Z$.
\end{itemize}
\end{lem}
\begin{proof}
Suppose that $\{f^{r}_{p,q}\}$ satisfy the condition of the assertion for $r_0\geq 1$ and $n_0\in \Z$. Note that for any $r\geq 1$ and $p,q\in \Z$,
\[ f^{r+1}_{p,q} \colon  E^{r+1}_{p,q}\cong \ker d^r_{p,q}/\Im d^{r}_{p+r, q-r+1}  \to E'^{r+1}_{p,q}\cong \ker d'^r_{p,q}/\Im d'^{r}_{p+r, q-r+1} \]
is induced by $\{ f^r_{p,q}\}_{p,q}$. Therefore, by inductive arguments about $\{f^{r}_{p,q}\}$ on $r=r_0,r_0+1,\dots $, we can prove that $f^{\infty}_{p,q} \colon E^{\infty}_{p,q}\to E'^{\infty}_{p,q}$ is an isomorphism if $p+q< n_0$, and a surjection if $p+q=n_0$. We omit the concluding  argument about $h_n$, since it is parallel to \cite[Comparison Theorem 5.2.12]{weibel}.
\end{proof}

For $\epsilon,\bar{\epsilon}\in (0, \epsilon_0 /(5C_0)]$ with $\epsilon\leq \bar{\epsilon}$, let $j_{\epsilon,\bar{\epsilon}}\colon \Sigma^{a+m\epsilon}_m \to \Sigma^{a+m\bar{\epsilon}}_m $ ($m\in \Z_{\geq 0}$) be the inclusion maps. These maps induce a linear map
\[(j_{\epsilon,\bar{\epsilon}})_* \colon C^{<a}_*(\epsilon) \to C^{<a}_*(\bar{\epsilon}).\]
Moreover, 
for any $\delta\in C^{\dr}_{n-d}(S_{\epsilon})$ with $\partial \delta=0$, $(j_{\epsilon,\bar{\epsilon}})_*$ is a chain map form $(C^{<a}_*(\epsilon),D_{\delta})$ to $(C^{<a}_*(\bar{\epsilon}), D_{ ( i_{\epsilon,\bar{\epsilon}})_*\delta})$ and preserves the filtrations $\{\mathcal{F}^{<a}_{\epsilon,p}\}_{p\in \Z}$ and $\{\mathcal{F}^{<a}_{\bar{\epsilon},p}\}_{p\in \Z}$.

\begin{lem}\label{lem-isom-j}
Suppose that $a\in \R_{> 0}\setminus \mathcal{L}(K)$. Then the induced map on homology
\[(j_{\epsilon,\bar{\epsilon}})_*\colon H^{<a}_*(\epsilon,\delta) \to H^{<a}_*(\bar{\epsilon},( i_{\epsilon,\bar{\epsilon}})_*\delta)\]
is an isomorphism if  $\bar{\epsilon} $ satisfies $[a,a+ \frac{2a}{\epsilon_0}\bar{\epsilon}]\cap \mathcal{L}(K) =\varnothing$.
\end{lem}
\begin{proof}
If $m\geq \frac{2a}{\epsilon_0}$, $\Sigma^{a+m\bar{\epsilon}}_m=\Sigma^{a+m\epsilon}_m=\Sigma^0_m$.
If $0\leq m\leq \frac{2a}{\epsilon_0}$, $[a+m\epsilon,a+m\bar{\epsilon}]\cap\mathcal{L}(K)=\varnothing$ from the condition on $\bar{\epsilon}$.
Thus, Proposition \ref{prop-hmgy-grp} is applied to show that $H^{\dr}_*(\Sigma^{a+m\bar{\epsilon}}_m, \Sigma^{a+m\epsilon}_m)=0$ for all $m\in \Z_{\geq 0}$.
Therefore, the induced map on the $(-m,q)$-term ($m\geq 0$) of the first page
\[(j_{\epsilon,\bar{\epsilon}})_* \colon (E^{<a}_{(\epsilon,\delta)})^1_{-m,q}= H^{\dr}_{q-m(d-1)}(\Sigma^{a+m\epsilon}_{m}, \Sigma^{0}_{m}) \to  (E^{<a}_{(\bar{\epsilon},(i_{\epsilon,\bar{\epsilon}})_*\delta)})^1_{-m,q} =H^{\dr}_{q-m(d-1)}(\Sigma^{a+m\bar{\epsilon}}_{m}, \Sigma^{0}_{m})  \]
is an isomorphism. Now the assertion follows from Lemma \ref{lem-spectral}.
\end{proof}

As we have seen in Example \ref{ex-hmgy-mfd-reg}, $H^{\dr}_*(N_{\epsilon}^{\reg})\cong H^d_c(N_{\epsilon};\R)$. Therefore we can determine a unique homology class $\Th_{\epsilon}\in H^{\dr}_{n-d}(N_{\epsilon}^{\reg})$ which corresponds to the Thom class of $(TK)^{\perp}$ through
the diffeomorphism $\{(x,v)\in (TK)^{\perp}\mid {v}<\epsilon \}\to N_{\epsilon}\colon (x,v)\mapsto \exp_xv$.

The above lemma leads us to define a set of data $(\epsilon,\delta)$ as follows.
\begin{defi}\label{def-class-of-data}

Let $C\geq 1$ be the constant of Lemma \ref{lem-htpy-sev}.
We define $\mathcal{T}_{a}$ for every $a\in \R_{\geq 0}\setminus \mathcal{L}(K)$ to be the set of pairs $(\epsilon,\delta)$ of $\epsilon\in (0, \epsilon_0/(5C^4)]$ and $\delta\in C^{\dr}_{n-d}(S_{\epsilon})$ such that:
\begin{enumerate}
\item $[a,a+\frac{2a}{\epsilon_0}\hat{\epsilon}]\cap \mathcal{L}(K) =\varnothing$ for $\hat{\epsilon}\coloneqq C^3\epsilon$.
\item $\partial \delta=0$ and $(\ev_0)_* [\delta] =\Th_{\epsilon}\in H^{\dr}_{n-d}(N_{\epsilon}^{\reg})$.
\end{enumerate}
\end{defi}

Let $a,b\in \R_{> 0}$ with $a<b$. 
For $\epsilon\in (0,\epsilon_0/(5C_0)]$ and $\delta\in C^{\dr}_{n-d}(S_{\epsilon})$ with $\partial \delta=0$, there exists a chain map $(I^{a,b}_{\epsilon})_*$ from $(C^{<a}_*(\epsilon),D_{\delta})$ to $(C^{<b}_*(\epsilon),D_{\delta})$ induced by inclusion maps $I^{a,b}_{\epsilon} \colon \Sigma^{a+m\epsilon}_m\to \Sigma^{b+m\epsilon}_m$ for all $m\in \Z_{\geq 0}$.
We define a quotient complex
\[(C^{[a,b)}_*(\epsilon) \coloneqq C^{<b}_*(\epsilon) / C^{<a}_*(\epsilon)  , D_{\delta}). \]
Let $H^{[a,b)}_*(\epsilon,\delta)$ denote its homology.
Obviously, there exists a long exact sequence
\begin{align}\label{ex-seq-epsilon-delta}
\xymatrix{ \cdots \ar[r] &H^{<a}_*(\epsilon,\delta) \ar[r]^-{(I^{a,b}_{\epsilon})_*} &  H^{<b}_*(\epsilon,\delta) \ar[r] &  H^{[a,b)}_*(\epsilon,\delta) \ar[r] & H^{<a}_{*-1}(\epsilon,\delta) \ar[r]^-{(I^{a,b}_{\epsilon})_*} & \cdots .
}\end{align}

The next result is a trivial computation from the spectral sequence.

\begin{prop}\label{prop-isom-short-interval}
For $a,b\in \R\setminus \mathcal{L}(K)$ with $a<b$ and $(\epsilon,\delta)\in \mathcal{T}_{a} \cap\mathcal{T}_b$, the following hold:
\begin{itemize}
\item If $[a,b]\cap \mathcal{L}(K) =\varnothing$, then $H^{[a,b)}_*(\epsilon,\delta)=0$.
\item If there exist $c\in \mathcal{L}(K)$ and $m_0\in \Z_{\geq 1}$ such that
$[a,b]\cap \mathcal{L}_m(K)=\begin{cases} \{c\} &\text{ if }m=m_0, \\ \varnothing &\text{ else,}\end{cases}$ then
\[H^{[a,b)}_*(\epsilon,\delta) \cong  H^{\dr}_{*-m_0(d-2)}(\Sigma^{b}_{m_0}, \Sigma^{a}_{m_0}). \]
\end{itemize}
\end{prop}
\begin{proof} Let $E^{[a,b)}_{(\epsilon,\delta)}$ be the spectral sequence determined by a filtration $\{\mathcal{F}^{<b}_{\epsilon,m}/\mathcal{F}^{<a}_{\epsilon,m}\}_{m\in \Z}$.
We apply Proposition \ref{prop-hmgy-grp} to the first page.
For the first case, $(E^{[a,b)}_{(\epsilon,\delta)})^1_{p,q}=0$ for every $p,q\in \Z$, so the assertion is trivial. For the second case, $(E^{[a,b)}_{(\epsilon,\delta)})^1_{p,q}=0$ for every $p\neq -m_0$, so all differentials are the zero map. Therefore, $H^{[a,b)}_q(\epsilon,\delta) \cong (E^{[a,b)}_{(\epsilon,\delta)})^1_{-m_0,q+m_0}$ and the assertion follows from
\[(E^{[a,b)}_{(\epsilon,\delta)})^1_{-m_0,q+m_0}= H^{\dr}_{q-m_0(d-2)} (\Sigma^{b+m_0\epsilon}_{m_0},\Sigma^{a+m_0\epsilon}_{m_0}) \cong H^{\dr}_{q-m_0(d-2)}(\Sigma^{b}_{m_0}, \Sigma^{a}_{m_0}) . \]
Here, the last isomorphism comes from Proposition \ref{prop-hmgy-grp}.
\end{proof}

\subsection{Variants from $[-1,1]$ and $[-1,1]^2$-modeled de Rham chains}\label{subsec-chain-cpx-by-[-1,1]}

In Section \ref{subsec-[-1,1]-model}, we introduced  $[-1,1]$-modeled and $[-1,1]^2$-modeled de Rham chains. In this section, we define chain complexes as Section \ref{subsec-def-of-chain-cpx} by using these types of chains. Their constructions and some of computations are parallel to the former section, so we often omit proofs.

First, we deal with $[-1,1]$-modeled chains.
For $a\in \R_{>0}$ and $\epsilon \in (0, \epsilon_0/(5C_0)]$,
we consider a graded $\R$-vector space
\[\bar{C}^{<a}_*(\epsilon) \coloneqq \bigoplus_{m=0}^{\infty} \bar{C}^{\dr}_{*-m(d-2)} (\Sigma^{a+m\epsilon}_m, \Sigma^{0}_m).\]
For $\bar{\delta}\in \bar{C}^{\dr}_{n-d}(S_{\epsilon})$, we define a degree ($-1$) map $\bar{D}_{\bar{\delta}}\colon \bar{C}^{<a}_*(\epsilon) \to \bar{C}^{<a}_{*-1}(\epsilon)$ by
\[\bar{D}_{\bar{\delta}} (x) \coloneqq \partial x +\sum_{k=1}^m (-1)^{p+1+kd} \bar{f}_{k,\bar{\delta}}(x) \]
for $x\in \bar{C}^{\dr}_{p-m(d-2)} (\Sigma^{a+m\epsilon}_m, \Sigma^{0}_m)$.

\begin{prop}\label{prop-barD-2}
If $\bar{\delta}\in \bar{C}^{\dr}_{n-d}(S_{\epsilon})$ satisfies $\partial \bar{\delta}=0$, then $\bar{D}_{\bar{\delta}} \circ \bar{D}_{\bar{\delta}} =0$ holds.
\end{prop}

This is analogous to Proposition \ref{prop-D-2} and can be deduced from the two equations of Proposition \ref{prop-f-bar-equations}.
From this proposition, for $\bar{\delta}\in \bar{C}^{\dr}_{n-d}(S_{\epsilon})$ with $\partial \delta=0$,
we obtain a chain complex $( \bar{C}^{<a}_*(\epsilon), \bar{D}_{\bar{\delta}}  )$. Let $\bar{H}^{<a}_*(\epsilon,\bar{\delta})$ denote its homology.

Let us consider a relation to the chain complex defined in Section \ref{subsec-def-of-chain-cpx}. The linear maps (\ref{e-pm}) for $X=\Sigma^{a+m\epsilon}_m$ 
\[e_+,e_-\colon \bar{C}^{\dr}_*(\Sigma^{a+m\epsilon}_m)\to C^{\dr}_*(\Sigma^{a+m\epsilon}_m)\ (m\in \Z_{\geq 0})\]
naturally induce linear maps
\[e_{\epsilon,+},e_{\epsilon,-} \colon \bar{C}^{<a}_*(\epsilon)  \to  C^{<a}_*(\epsilon),  \]
and these are chain maps from $ (\bar{C}^{<a}_*(\epsilon),\bar{D}_{\bar{\delta}} )$ to $( C^{<a}_*(\epsilon) ,D_{e_+\bar{\delta}} )$ and $( C^{<a}_*(\epsilon) ,D_{e_-\bar{\delta}} )$ respectively.

We define a filtration $\{ \bar{\mathcal{F} }^{<a}_{\epsilon, p}\}_{p\in \Z}$ by
\[ \bar{\mathcal{F} }^{<a}_{\epsilon, p} \coloneqq \bigoplus_{m\geq -p} \bar{C}^{\dr}_{*-m(d-2)} (\Sigma^{a+m\epsilon}_{m}, \Sigma^{0}_{m}).\]
Let $\bar{E}^{<a}_{(\epsilon,\bar{\delta})}$ be the spectral sequence determined by this filtration.
\begin{lem}\label{lem-e-quasi-isom}
$e_{\epsilon,+}$ and $e_{\epsilon,-}$ are quasi-isomorphisms.
\end{lem}
\begin{proof}
We prove this assertion for only $e_{\epsilon,+}$. The proof for $e_{\epsilon,-}$ is parallel.
Since $e_{\epsilon,+}$ preserves filtrations $\{ \bar{\mathcal{F}}^{<a}_{\epsilon, p}\}_{p\in \Z}$ and  $\{ \mathcal{F}^{<a}_{\epsilon, p}\}_{p\in \Z}$, this induces a map on the first page
$(e_{\epsilon,+})_* \colon (\bar{E}^{<a}_{(\epsilon,\bar{\delta})})^1_{p,q}\to (E^{<a}_{(\epsilon,e_+\bar{\delta})})^1_{p,q} $.
For $p=-m\leq 0$, this coincides with
\[(e_{+})_*  \colon \bar{H}^{\dr}_{q-m(d-1)} ( \Sigma^{a+m\epsilon}_{m}, \Sigma^{0}_{m}) \to H^{\dr}_{q-m(d-1)} ( \Sigma^{a+m\epsilon}_{m}, \Sigma^{0}_{m}) . \]
From Lemma \ref{lem-rest-isom-bar}, this  map is an isomorphism. Now the assertion follows form Lemma \ref{lem-spectral}.
\end{proof}

For $\bar{\epsilon},\hat{\epsilon}\in (0, \epsilon_0/(5C_0)]$ with $\bar{\epsilon}\leq \hat{\epsilon}$, 
the linear map $(j_{\bar{\epsilon},\hat{\epsilon}})_* \colon \bar{C}^{<a}_*(\bar{\epsilon})\to \bar{C}^{<a}_*(\hat{\epsilon})$,
induced by the inclusion maps $j_{\bar{\epsilon},\hat{\epsilon}}\colon \Sigma^{a+m\bar{\epsilon}}_m \to \Sigma^{a+m\hat{\epsilon}}_m$ for all $m\in \Z_{\geq 0}$, is a chain map from $( \bar{C}^{<a}_*(\bar{\epsilon}), \bar{D}_{\bar{\delta}})$ to $(\bar{C}^{<a}_*( \hat{\epsilon}), \bar{D}_{ (i_{\bar{\epsilon},\hat{\epsilon}})_* \bar{\delta} } )$.
\begin{lem}\label{lem-isom-j-bar}
Suppose that $a\in \R_{> 0}\setminus \mathcal{L}(K)$. Then the induced map on homology
\[(j_{\bar{\epsilon},\hat{\epsilon}})_*  \colon \bar{H}^{<a}_*(\bar{\epsilon},\bar{\delta}) \to \bar{H}^{<a}_*(\hat{\epsilon}, (i_{\bar{\epsilon},\hat{\epsilon}})_* \bar{\delta})\]
is an isomorphism if $\hat{\epsilon}$ satisfies 
$[a,a+\frac{2a}{\epsilon_0}\hat{\epsilon}]\cap \mathcal{L}(K) =\varnothing$.
\end{lem}
\begin{proof}
The proof is parallel to that of Lemma \ref{lem-isom-j}.
The chain map $(j_{\bar{\epsilon},\hat{\epsilon}})_*$ preserves the filtrations $\{ \bar{\mathcal{F}}^{<a}_{\bar{\epsilon},p}\}_{p\in \Z}$ and $\{ \bar{\mathcal{F}}^{<a}_{\hat{\epsilon},p}\}_{p\in \Z}$. This induces an isomorphism on the first page since for every $m\in \Z_{\geq 0}$, 
\[\bar{H}_*(\Sigma^{a+m\hat{\epsilon} }_m,\Sigma^{a+m\bar{\epsilon}}_m) =0\]
holds by Proposition \ref{prop-analogy}.
Now the assertion follows from Lemma \ref{lem-spectral}.
\end{proof}
The above lemma leads us to the following definition.
\begin{defi}
Let $C\geq 1$ be the constant of Lemma \ref{lem-htpy-sev}. We define $\bar{\mathcal{T}}_{a}$ for $a\in \R_{> 0}\setminus \mathcal{L}(K)$ to be the set of pairs $(\bar{\epsilon},\bar{\delta})$ of $\bar{\epsilon}\in (0, \epsilon_0/ (5C^3) ]$ and $\bar{\delta}\in \bar{C}^{\dr}_{n-d}(S_{\bar{\epsilon}})$ such that:
\begin{enumerate}
\item $[a,a+\frac{2a}{\epsilon_0}\hat{\epsilon}]\cap \mathcal{L}(K) =\varnothing$ for $\hat{\epsilon}\coloneqq C^2\bar{\epsilon}$.
\item $\partial \bar{\delta}=0$ and $(\ev_0)_*[e_+\bar{\delta}]= \Th_{\bar{\epsilon}} \in H^{\dr}_{n-d}(N_{\bar{\epsilon}}^{\reg})$.
\end{enumerate}
\end{defi}

Next, we deal with $[-1,1]^2$-modeled chains.
For $a\in \R_{>0}$ and $\epsilon\in (0, \epsilon_0 /(5C_0)]$, we consider a graded $\R$-vector space
\[\hat{C}^{<a}_*(\epsilon) \coloneqq \bigoplus_{m=0}^{\infty} \hat{C}^{\dr}_{*-m(d-2)} (\Sigma^{a+m\epsilon}_m, \Sigma^{0}_m). \]
For $\hat{\delta}\in \hat{C}^{\dr}_{n-d}(S_{\epsilon})$, we define a degree ($-1$) map 
$\hat{D}_{\hat{\delta}}\colon \hat{C}^{<a}_*(\epsilon) \to \hat{C}^{<a}_{*-1}(\epsilon)$ by
\[ \hat{D}_{\hat{\delta}}(x) \coloneqq \partial x+\sum_{k=1}^m(-1)^{p+1+kd} \hat{f}_{k,\hat{\delta}}(x) \]
for $x\in \hat{C}^{\dr}_{p-m(d-2)} (\Sigma^{a+m\epsilon}_m, \Sigma^{0}_m)$.
\begin{prop}
If $\hat{\delta}\in \hat{C}^{\dr}_{n-d}(S_{\epsilon})$ satisfies $\partial \hat{\delta}=0$, then
$\hat{D}_{\hat{\delta}} \circ \hat{D}_{\hat{\delta}} =0$ holds.
\end{prop}
This is analogous to Proposition \ref{prop-D-2} and can be deduced from the two equations of Proposition \ref{prop-f-hat-equations}.
From this proposition, for $\hat{\delta}\in \hat{C}^{\dr}_{n-d}(S_{\epsilon})$ with $\partial \hat{\delta}=0$,
we obtain a chain complex $( \hat{C}^{<a}_*(\epsilon), \hat{D}_{\hat{\delta}}  )$. Let $\hat{H}^{<a}_*(\epsilon,\hat{\delta})$ denote its homology.

Let us consider a relation to the chain complex defined by $[-1,1]$-modeled de Rham chains. For $j=1,2$, the linear maps of (\ref{e-j-pm}) for $X=\Sigma^{a+m\epsilon}_m$
\[e^j_+, e^j_- \colon \hat{C}^{\dr}_*(\Sigma^{a+m\epsilon}_m) \to \bar{C}^{\dr}_*(\Sigma^{a+m\epsilon}_m) \ (m\in \Z_{\geq 0}) \]
naturally induce linear maps
\[e^j_{\epsilon,+}, e^j_{\epsilon,-} \colon  \hat{C}^{<a}_*(\epsilon) \to \bar{C}^{<a}_*(\epsilon) , \]
and these are chain maps from $ (\hat{C}^{<a}_*(\epsilon),\hat{D}_{\hat{\delta}} )$ to $(\bar{C}^{<a}_*(\epsilon) ,\bar{D}_{e^j_+\hat{\delta} } )$ and $(\bar{C}^{<a}_*(\epsilon) ,\bar{D}_{e^j_-\hat{\delta} } )$ respectively.

\begin{lem}\label{lem-isom-e-j}
$e^j_{\epsilon,+}$ and $e^j_{\epsilon,-}$ ($j=1,2$) are quasi-isomorphisms.
\end{lem}
\begin{proof}
The proof is parallel to that of Lemma \ref{lem-e-quasi-isom}. This time, we use the spectral sequence determined by a filtration $\{ \hat{\mathcal{F}}^a_{\epsilon, p}\}_{p\in \Z}$, which is defined by
\[ \hat{\mathcal{F}}^a_{\epsilon, p} \coloneqq \bigoplus_{m\geq -p} \hat{C}^{\dr}_{*-m(d-2)} (\Sigma^{a+m\epsilon}_{m}, \Sigma^{0}_{m}).\]
$e^j_{\epsilon,+}, e^j_{\epsilon,-}$ ($j=1,2$) preserve filtrations $\{ \hat{\mathcal{F}}^a_{\epsilon, p}\}_{p\in \Z}$ and $\{ \bar{\mathcal{F}}^a_{\epsilon, p}\}_{p\in \Z}$. By Lemma \ref{lem-rest-isom-hat}, they induce an isomorphism on the first page. Now the assertion follows form Lemma \ref{lem-spectral}.
\end{proof}

\subsection{The limit of $\epsilon \to 0$}\label{subsec-limit}

In this section, we define a transition map
\[k_{(\epsilon',\delta'),(\epsilon,\delta)} \colon H^{<a}_*(\epsilon',\delta') \to H^{<a}_*(\epsilon ,\delta)\]
for every $a\in \R_{> 0}\setminus \mathcal{L}(K)$ and $(\epsilon,\delta),(\epsilon',\delta')\in \mathcal{T}_a$ with $\epsilon'\leq \epsilon$, 
by using $(\bar{\epsilon},\bar{\delta}) \in \bar{\mathcal{T}}_{a}$ satisfying
\begin{align}\label{condition-delta-bar}
\epsilon\leq \bar{\epsilon}, \ e_+\bar{\delta}=(i_{\epsilon,\bar{\epsilon}})_*\delta ,\ e_-\bar{\delta}=(i_{\epsilon',\bar{\epsilon}})_*\delta'.
\end{align}
In fact, $k_{(\epsilon',\delta'),(\epsilon,\delta)}$ is an isomorphism.
We also prove that $( \{H^{<a}_*(\epsilon,\delta)\}_{(\epsilon,\delta)\in \mathcal{T}_{a}}, \{k_{(\epsilon',\delta'),(\epsilon,\delta)} \}_{\epsilon'\leq \epsilon} )$ forms an inverse system.

\subsubsection{Construction of transition maps}

Let us first prove the existence of the above $(\bar{\epsilon},\bar{\delta})$.

\begin{lem}\label{lem-existence-delta-bar}
For $(\epsilon,\delta),(\epsilon',\delta)\in \mathcal{T}_{a}$ with $\epsilon'\leq \epsilon$, there exists $(\bar{\epsilon},\bar{\delta})\in \bar{\mathcal{T} }_{a}$ satisfying (\ref{condition-delta-bar}).
\end{lem}
\begin{proof}
Let us take $\bar{\epsilon} \coloneqq C\epsilon $ for the constant $C$ in Lemma \ref{lem-htpy-sev}, and rewrite $\delta_+\coloneqq (i_{\epsilon,\bar{\epsilon}})_* \delta$ and $\delta_-\coloneqq ( i_{ \epsilon',\bar{\epsilon}})_* \delta'$ for short. Since
\[ (\ev_0)_*[\delta - (i_{\epsilon',\epsilon})_* \delta']= \Th_{\epsilon}-(i_{\epsilon', \epsilon})_*\Th_{\epsilon'}=0\in H^{\dr}_{n-d}(N_{\epsilon}),\]
Proposition \ref{prop-exactness} shows that there exists $\theta \in C^{\dr}_{n-d+1}(S_{\bar{\epsilon}})$ such that
\[\partial \theta= (i_{\epsilon,\bar{\epsilon}} )_*( \delta - (i_{\epsilon',\epsilon})_* \delta' ) = \delta_+ - \delta_-.\]
Let $\kappa\colon \R\to [0,1]$ be a $C^{\infty}$ function such that $\kappa(r)=\begin{cases} 1 & \text{ if }1\leq r, \\ 0 & \text{ if }r\leq -1. \end{cases}$ We take chains $\beta_+,\beta_-\in \bar{C}^{\dr}_0(\{0\})$ defined by
\[\begin{cases}
\beta_+\coloneqq  [\R,\id_{\R},(\id_{\R_{\geq 1}},\id_{\R_{\leq -1}}),\kappa], \\
\beta_-\coloneqq  [\R,\id_{\R},(\id_{\R_{\geq 1}},\id_{\R_{\leq -1}}),1-\kappa].
\end{cases} \]
Now we define $\bar{\delta}$ by
\[\bar{\delta}\coloneqq \beta_+ \times (\bar{i} \delta_+) + \beta_-\times (\bar{i} \delta_-) +( \partial \beta_+)\times (\bar{i} \theta)\in \bar{C}^{\dr}_{n-d}(S_{\bar{\epsilon}}).\]
This satisfies the condition (\ref{condition-delta-bar}). Moreover, $\partial \bar{\delta}=0$ and $(\ev_0)_*[ e_+ \bar{\delta}] =(\ev_0)_*[\delta_+]=\Th_{\bar{\epsilon}}$ hold. 
Now, it is clear that $(\bar{\epsilon},\bar{\delta})=(C\epsilon,\bar{\delta})$ satisfies the two conditions to be an element of $ \bar{\mathcal{T}}_{a}$.
\end{proof}

From Lemma \ref{lem-isom-j} and Lemma \ref{lem-e-quasi-isom}, we can define isomorphisms
\begin{align*}
f_{(\bar{\epsilon},\bar{\delta}),+}&\coloneqq (j_{\epsilon,\bar{\epsilon}})_*^{-1}\circ (e_{\bar{\epsilon},+})_* \colon \bar{H}^{<a}_*(\bar{\epsilon},\bar{\delta}) \to H^{<a}_*(\epsilon,\delta) , \\
f_{(\bar{\epsilon}, \bar{\delta}),-}&\coloneqq (j_{\epsilon',\bar{\epsilon}})_*^{-1} \circ (e_{\bar{\epsilon},-})_* \colon \bar{H}^{<a}_*(\bar{\epsilon},\bar{\delta}) \to H^{<a}_*(\epsilon',\delta') ,
\end{align*}
so that the following diagrams commute:
\[\xymatrix{ 
\bar{H}^{<a}_*(\bar{\epsilon},\bar{\delta}) \ar[d]_{(e_{\bar{\epsilon},+})_*} \ar[r]^{f_{(\bar{\epsilon},\bar{\delta}),+}}& H^{<a}_*(\epsilon,\delta) \ar[d]^{(j_{\epsilon,\bar{\epsilon}})_*} &
\bar{H}^{<a}_*(\bar{\epsilon},\bar{\delta}) \ar[d]_{(e_{\bar{\epsilon},-})_*} \ar[r]^{f_{(\bar{\epsilon},\bar{\delta}),-}} & H^{<a}_*(\epsilon',\delta') \ar[d]^{(j_{\epsilon',\bar{\epsilon}})_*} \\
H^{<a}_*(\bar{\epsilon},e_+\bar{\delta})\ar@{=}[r] &H^{<a}_*(\bar{\epsilon}, (i_{\epsilon,\bar{\epsilon}})_*\delta) , &
H^{<a}_*(\bar{\epsilon},e_-\bar{\delta})\ar@{=}[r] &H^{<a}_*(\bar{\epsilon} ,(i_{\epsilon',\bar{\epsilon}})_*\delta') .
}\]
We define an isomorphism
\begin{align*}
k_{(\bar{\epsilon},\bar{\delta})} \coloneqq f_{(\bar{\epsilon},\bar{\delta}),+} \circ ( f_{(\bar{\epsilon},\bar{\delta}),-} )^{-1} \colon 
H^{<a}_*(\epsilon',\delta') \to H^{<a}_*(\epsilon,\delta).
\end{align*}
Later, we will prove the independence on $(\bar{\epsilon},\bar{\delta})$ (Corollary \ref{cor-indep-k}), and this is the transition map $k_{(\epsilon',\delta), (\epsilon,\delta)}$ we need.

\begin{lem}\label{lem-trivial-delta-bar}
When $(\epsilon',\delta')=(\epsilon,\delta) \in \mathcal{T}_{a}$, we may take $(\epsilon,\bar{i}\delta)\in \bar{\mathcal{T}}_{a}$ as an element satisfying (\ref{condition-delta-bar}). In this case, we have
\[k_{(\epsilon,\bar{i}\delta)} = \id_{H^{<a}_*(\epsilon,\delta)} \]

\end{lem}
\begin{proof}
To prove this assertion, let us introduce a chain map
$\bar{i}_{\epsilon} $ from $ (C^{<a}_*(\epsilon),D_{\delta})$ to $ (\bar{C}^{<a}_*(\epsilon), \bar{D}_{\bar{i}\delta} )$
induced by $\bar{i}\colon C^{\dr}_{*-m(d-2)}(\Sigma^{a+m\epsilon}_m)\to \bar{C}^{\dr}_{*-m(d-2)}(\Sigma^{a+m\epsilon}_m)$ (see (\ref{bar-i})) for all $m\in \Z_{\geq 0}$. This satisfies $e_{\epsilon,+}\circ \bar{i}_{\epsilon}=\id_{C^{<a}_*(\epsilon)} = e_{\epsilon,-}\circ \bar{i}_{\epsilon}$, and thus
\begin{align}\label{e-pm-and-i-bar}
(e_{\epsilon,+})_*=(\bar{i}_{\epsilon})_*^{-1}=(e_{\epsilon,-})_* \colon \bar{H}^{<a}_*(\epsilon,\bar{i}\delta) \to H^{<a}_*(\epsilon,\delta).
\end{align}
Therefore, $k_{(\epsilon,\bar{i}\delta)}= (e_{\epsilon,+})_* \circ (e_{\epsilon,-})_*^{-1}=\id_{H^{<a}_*(\epsilon,\delta)}$.
\end{proof}

\subsubsection{Compositions}

Next, we think about compositions of maps of $\{ k_{(\bar{\epsilon},\bar{\delta})}\}_{(\bar{\epsilon},\bar{\delta})\in \bar{\mathcal{T}}_{a}}$.
For $(\epsilon,\delta),(\epsilon',\delta'),(\epsilon'',\delta'')\in \mathcal{T}_{a}$ with $\epsilon''\leq \epsilon'\leq \epsilon$, suppose that we have chosen $(\bar{\epsilon},\bar{\delta}),(\bar{\epsilon}',\bar{\delta}') ,(\tilde{\epsilon},\tilde{\delta})\in \bar{ \mathcal{T} }_{a}$ satisfying:
\[\begin{array}{rlrl}
e_+\bar{\delta}&=(i_{\epsilon,\bar{\epsilon}})_*\delta, & e_-\bar{\delta}&=(i_{\epsilon',\bar{\epsilon}})_*\delta' , \\
e_+\bar{\delta}'&=(i_{\epsilon',\bar{\epsilon}'})_* \delta', & e_-\bar{\delta}'&=(i_{\epsilon'',\bar{\epsilon}'})_*\delta'' , \\
e_+\tilde{\delta}&=(i_{\epsilon,\tilde{\epsilon}})_*\delta, & e_-\tilde{\delta}&= (i_{\epsilon'',\tilde{\epsilon}} )_*\delta'' .
\end{array}\]
Under this situation, let us first prove the following lemma.

\begin{lem}\label{lem-delta-hat}
There exists $\hat{\epsilon}\in (0,\epsilon_0/ (5C_0)]$ and $\hat{\delta}\in \hat{C}^{\dr}_{n-d}(S_{\hat{\epsilon}})$ such that $\partial \hat{\delta}=0$ and
\[e^1_+\hat{\delta} = (i_{\tilde{\epsilon},\hat{\epsilon}})_* \tilde{\delta},\ e^1_-\hat{\delta}=(i_{\bar{\epsilon}',\hat{\epsilon}})_* \bar{\delta}',\ e^2_+\hat{\delta}=(i_{\bar{\epsilon},\hat{\epsilon}})_* \bar{\delta},\ e^2_-\hat{\delta}=(i_{\epsilon'',\hat{\epsilon}})_* (\bar{i} \delta'' ) .\]
\end{lem}
\begin{proof}
Let us take $\rho \coloneqq C\cdot \max\{\bar{\epsilon}, \bar{\epsilon}',\tilde{\epsilon}\}$, $\hat{\epsilon}\coloneqq C\rho$ and rewrite $\bar{\delta}_+\coloneqq  (i_{\tilde{\epsilon},\rho })_*\tilde{\delta}$ and $\bar{\delta}_-\coloneqq  (i_{\bar{\epsilon}',\rho})_* \bar{\delta}'$ for short. Since
\begin{align*}
(\ev_0)_* \circ (e_+)_* [ ( i_{\tilde{\epsilon}, C^{-1}\rho })_* \tilde{\delta}- ( i_{\bar{\epsilon}', C^{-1}\rho})_* \bar{\delta}'] =& ( i_{\tilde{\epsilon},C^{-1}\rho})_* ( \Th_{\tilde{\epsilon}} ) - ( i_{\bar{\epsilon}',C^{-1}\rho})_* ( \Th_{\bar{\epsilon}'})  \\
=&0 \in H^{\dr}_*(N_{C^{-1}\rho}), 
\end{align*}
Proposition \ref{prop-analogy} shows that there exists $\bar{\theta}_1\in \bar{C}^{\dr}_{n-d+1}(S_{\rho})$ such that $\partial \bar{\theta}_1=\bar{\delta}_+-\bar{\delta}_-$. We define
$\hat{\kappa}^1\colon \R\times \R \to [0,1]\colon (r_1,r_2) \mapsto \kappa(r_1)$,
where $\kappa$ is the function appeared in the proof of Lemma \ref{lem-existence-delta-bar}. Then, we take chains  $\hat{\beta}_+,\hat{\beta}_-\in \hat{C}^{\dr}_0(\{0\})$ define by
\[\begin{cases}
\hat{\beta}^1_+\coloneqq [\R^2,\id_{\R^2}, (\tau^1_+,\tau^1_-),(\tau^2_+,\tau^2_-),\hat{\kappa} ], \\
\hat{\beta}^1_-\coloneqq [\R^2,\id_{\R^2}, (\tau^1_+,\tau^1_-),(\tau^2_+,\tau^2_-),1-\hat{\kappa} ], \end{cases}\]
where $\tau^j_+=\id_{\{r_j\geq 1\}}$ and $\tau^j_-=\id_{\{r_j\leq -1\}}$ for $j=1,2$.
We define
\[\xi \coloneqq \hat{\beta}^1_+\times (\hat{i}^1\bar{\delta}_+) + \hat{\beta}^1_-\times (\hat{i}^1\bar{\delta}_-) +(\partial \hat{\beta}^1_+)\times (\hat{i}^1(\bar{\theta}_1 - \bar{i}e_- \bar{\theta}_1 )) \in \hat{C}^{\dr}_{n-d}(S_{\rho}) . \]
This chain satisfies $\partial \xi=0$ (note that and $e_-(\partial \bar{\theta}_1) =e_-\bar{\delta}_+-e_-\bar{\delta}_- =0$). Moreover,
\[e^1_+\xi = (i_{\tilde{\epsilon},\rho})_*\tilde{\delta},\ e^1_-\xi =(i_{\bar{\epsilon}',\rho})_*\bar{\delta}',\ e^2_-\xi=(i_{\epsilon'',\rho})_* (\bar{i} \delta'' ) \]
hold. The former two equations are easy to check.
The third equation can be checked as follows: Since $e^2_-\circ \hat{i}^1=\bar{i}\circ e_-$ holds, we have
$e^2_- (\hat{i}^1(\bar{\theta}_1 - \bar{i}e_- \bar{\theta}_1 )) =0$ and thus
\[e^2_-\xi = (e^2_-\hat{\beta}^1_+) \times (\bar{i} ( (i_{\epsilon'',\rho })_* \delta'' )) + (e^2_-\hat{\beta}^1_-) \times (\bar{i} ( (i_{\epsilon'',\rho })_*  \delta'') ) =(i_{\epsilon'',\rho})_* (\bar{i} \delta'' ) . \]

Let us write $\underline{\delta}\coloneqq e^2_+ \xi \in \bar{C}^{\dr}_{n-d}(S_{\rho})$ and consider the difference of chains $\underline{\delta} - (i_{\bar{\epsilon},\rho} )_* \bar{\delta} $. We claim that there exists $\bar{\theta}_2\in \bar{C}^{\dr}_{n-d+1} (S_{\hat{\epsilon}})$ such that
\begin{align}\label{condition-theta-bar-2}
\partial \bar{\theta}_2=(i_{\rho ,\hat{\epsilon}})_* \underline{\delta} - (i_{\bar{\epsilon},\hat{\epsilon}} )_*\bar{\delta} ,\  e_{+} \bar{\theta}_2=e_{-} \bar{\theta}_2=0 . 
\end{align}

We prove this claim.
Since $e_+\circ e^2_+=e_+\circ e^1_+$ and $e_-\circ e^2_+= e_+\circ e^1_-$,
we have
\begin{align*}
&(e_+)_* [\underline{\delta} - (i_{\bar{\epsilon},\rho})_* \bar{\delta} ] 
= (e_+\circ e^1_+)_* [\xi]  -(i_{\epsilon,\rho})_*\delta  =0\in H^{\dr}_{n-d}(S_{\rho}), \\
&(e_-)_* [\underline{\delta} - (i_{\bar{\epsilon},\rho})_* \bar{\delta} ] 
= (e_+\circ e^1_-)_* [\xi]  -(i_{\epsilon',\rho})_*\delta'  =0\in H^{\dr}_{n-d}(S_{\rho}).
\end{align*}
From Lemma \ref{lem-rest-isom-bar}, there exists $\bar{\theta}'_2\in \bar{C}^{\dr}_{n-d+1}(S_{\rho})$ such that $\partial \bar{\theta}'_2 = \underline{\delta} - (i_{\bar{\epsilon},\rho})_* \bar{\delta} $ and  $\partial (e_{+}\bar{\theta}'_2)=\partial (e_{-}\bar{\theta}'_2)=0$.
%
Since $(\ev_0)_*[e_{+}\bar{\theta}'_2]$ and $(\ev_0)_*[e_{-}\bar{\theta}'_2]$ are contained in $H^{\dr}_{n-d+1}(N_{\rho}) =\{0\}$, Proposition \ref{prop-analogy} shows that there exist $\varphi_{+},\varphi_- \in C^{\dr}_{n-d+2}(S_{\hat{\epsilon}})$ such that $\partial \varphi_{+}= (i_{\rho,\hat{\epsilon}})_* (e_{+}\bar{\theta}'_2 )$ and $\partial \varphi_{-}=( i_{\rho,\hat{\epsilon}})_* (e_{-}\bar{\theta}'_2 )$. Then, by using $\beta_{+},\beta_-$ in the proof of Lemma \ref{lem-existence-delta-bar}, we define
\[\bar{\theta}_2\coloneqq  (i_{\rho,\hat{\epsilon}})_* \bar{\theta}'_2 - \partial( \beta_+ \times \bar{i} \varphi_+ ) -\partial (\beta_-\times \bar{i} \varphi_-) \in \bar{C}^{\dr}_{n-d+1}(S_{\hat{\epsilon}}), \]
and this chain satisfies (\ref{condition-theta-bar-2}).

We take $\hat{\kappa}^2\colon \R^2\to [0,1]\colon (r_1,r_2)\mapsto \kappa(r_2)$ and
\[\hat{\beta}^2_+\coloneqq [\R^2,\id_{\R^2}, (\tau^1_+,\tau^1_-),(\tau^2_+,\tau^2_-),\hat{\kappa}^2] \in \hat{C}^{\dr}_0(\{0\}).\]
Finally, we define a chain
\[\hat{\delta}\coloneqq (i_{\rho,\hat{\epsilon}})_* \xi - \partial ( \hat{\beta}^2_+\times \hat{i}^2\bar{\theta}_2 ) \in \hat{C}^{\dr}_{n-d}(S_{\hat{\epsilon}}).\]
This satisfies $\partial \hat{\delta}=0$ and the required four equations.
\end{proof}

Lemma \ref{lem-delta-hat} is applied to prove the next proposition.
\begin{prop}\label{prop-comm-k}
$k_{(\bar{\epsilon},\bar{\delta})} \circ k_{(\bar{\epsilon}',\bar{\delta}')} =k_{(\tilde{\epsilon},\tilde{\delta})} \colon H^{<a}_*(\epsilon'',\delta'') \to H^{<a}_*(\epsilon,\delta)$.
\end{prop}
\begin{proof}

Let $(\hat{\epsilon},\hat{\delta})$ be the pair of Lemma \ref{lem-delta-hat}. From Lemma \ref{lem-isom-j-bar} and Lemma \ref{lem-isom-e-j}, we can define isomorphisms
\begin{align*}
f^1_{(\hat{\epsilon},\hat{\delta}),+}&\coloneqq (j_{\bar{\epsilon},\hat{\epsilon}})_*^{-1}\circ (e^1_{\hat{\epsilon},+})_* \colon \hat{H}^{<a}_*(\hat{\epsilon},\hat{\delta}) \to \bar{H}^{<a}_*(\tilde{\epsilon},\tilde{\delta}), \\
f^1_{(\hat{\epsilon}, \hat{\delta}),-}&\coloneqq (j_{\bar{\epsilon}',\hat{\epsilon}})_*^{-1} \circ (e^1_{\hat{\epsilon},-})_* \colon \hat{H}^{<a}_*(\hat{\epsilon},\hat{\delta}) \to \bar{H}^{<a}_*(\bar{\epsilon}',\bar{\delta}'), \\
f^2_{(\hat{\epsilon},\hat{\delta}),+}&\coloneqq (j_{\bar{\epsilon},\hat{\epsilon}})_*^{-1}\circ (e^2_{\hat{\epsilon},+})_* \colon 
\hat{H}^{<a}_*(\hat{\epsilon},\hat{\delta}) \to \bar{H}^{<a}_*(\bar{\epsilon},\bar{\delta}).
\end{align*}

From the definitions of $k_{(\bar{\epsilon},\bar{\delta})}$, $ k_{(\bar{\epsilon}',\bar{\delta}')}$ and $k_{(\tilde{\epsilon},\tilde{\delta})}$, it suffices to show that the following diagram commutes:
\[\xymatrix{
& & \bar{H}^{<a}_*(\tilde{\epsilon},\tilde{\delta})\ar[ddll]_{f_{(\tilde{\epsilon},\tilde{\delta}),-} } \ar[ddrr]^{f_{(\tilde{\epsilon},\tilde{\delta}),+} } & & \\
& & \hat{H}^{<a}_*(\hat{\epsilon},\hat{\delta})\ar[u]_{f^1_{(\hat{\epsilon},\hat{\delta}),+} } \ar[ddl]_{f^1_{(\hat{\epsilon},\hat{\delta}),-} } \ar[ddr]^{f^2_{(\hat{\epsilon},\hat{\delta}),+} } & & \\
H^{<a}_*(\epsilon'',\delta'') & &H^{<a}_*(\epsilon',\delta') & &H^{<a}_*(\epsilon,\delta) \\
& \bar{H}^{<a}_{*}(\bar{\epsilon}',\bar{\delta}') \ar[ul]^{f_{(\bar{\epsilon}',\bar{\delta}'),-} } \ar[ur]_{f_{(\bar{\epsilon}',\bar{\delta}'),+} } & & \bar{H}^{<a}_*(\bar{\epsilon},\bar{\delta}). \ar[ul]^{f_{(\bar{\epsilon},\bar{\delta}),-}}\ar[ur]_{f_{(\bar{\epsilon},\bar{\delta}),+} }& 
}\]
Note that all maps appearing in the diagram are isomorphisms.
We need to prove the following three equations:
\begin{align}\label{composition-equation}
\begin{cases}
f_{(\tilde{\epsilon},\tilde{\delta}),+} \circ f^1_{(\hat{\epsilon},\hat{\delta}),+} = f_{(\bar{\epsilon},\bar{\delta}),+} \circ f^2_{(\hat{\epsilon},\hat{\delta}),+} , \\
f_{(\bar{\epsilon},\bar{\delta}),-} \circ f^2_{(\hat{\epsilon},\hat{\delta}),+} = f_{(\bar{\epsilon}',\bar{\delta}'),+} \circ f^1_{(\hat{\epsilon},\hat{\delta}),-} , \\
f_{(\tilde{\epsilon},\tilde{\delta}),-} \circ f^1_{(\hat{\epsilon},\hat{\delta}),+} = f_{(\bar{\epsilon}',\bar{\delta}'),-} \circ f^1_{(\hat{\epsilon},\hat{\delta}),-}.
\end{cases}
\end{align}
Let us prove the first equation. Returning to the definition of $f_{(\tilde{\epsilon},\tilde{\delta}),+}$ and $f^1_{(\hat{\epsilon},\hat{\delta}),+}$,
\begin{align*}
f_{(\tilde{\epsilon},\tilde{\delta}),+} \circ f^1_{(\hat{\epsilon},\hat{\delta}),+} &= (j_{\epsilon,\tilde{\epsilon}})_*^{-1}\circ (e_{\tilde{\epsilon},+})_* \circ  (j_{\tilde{\epsilon},\hat{\epsilon}})_*^{-1}\circ (e^1_{\hat{\epsilon},+})_* \\
&=(j_{\epsilon,\hat{\epsilon}})_*^{-1}\circ ( e_{\hat{\epsilon},+}\circ  e^1_{\hat{\epsilon},+})_* \\
&=(j_{\epsilon,\hat{\epsilon}})_*^{-1}\circ ( e_{\hat{\epsilon},+}\circ  e^2_{\hat{\epsilon},+})_* \\
&= (j_{\epsilon,\bar{\epsilon}})_*^{-1}\circ (e_{\bar{\epsilon},+})_* \circ  (j_{\bar{\epsilon},\hat{\epsilon}})_*^{-1}\circ (e^2_{\hat{\epsilon},+})_* \\
&=f_{(\bar{\epsilon},\bar{\delta}),+} \circ f^2_{(\hat{\epsilon},\hat{\delta}),+} .
\end{align*}
Here, the second and fourth equality follow from obvious equations
\[(j_{\tilde{\epsilon},\hat{\epsilon}})_* \circ (e_{\tilde{\epsilon},+})_* = (e_{\hat{\epsilon},+})_* \circ (j_{\tilde{\epsilon},\hat{\epsilon}})_*,\ (j_{\bar{\epsilon},\hat{\epsilon}})_* \circ (e_{\bar{\epsilon},+})_*= (e_{\hat{\epsilon},+})_*\circ (j_{\bar{\epsilon},\hat{\epsilon}})_* . \] The third equality follows from
\[e_{\hat{\epsilon},+}\circ  e^1_{\hat{\epsilon},+}= e_{\hat{\epsilon},+}\circ  e^2_{\hat{\epsilon},+},\]
which comes from the relation $e_+\circ e^1_+=e_+\circ e^2_+$ of (\ref{e-and-e-j}).
The second equation of (\ref{composition-equation}) can be proved similarly by applying
\[e_{\hat{\epsilon},-}\circ e^2_{\hat{\epsilon},+}= e_{\hat{\epsilon},+}\circ e^1_{\hat{\epsilon},-},\]
which comes from the relation $e_-\circ e^2_+=e_+\circ e^1_-$ of  (\ref{e-and-e-j}).
To prove the third equation of (\ref{composition-equation}), there is one non-trivial matter: We need to apply
\[(e_{\hat{\epsilon},-}\circ e^1_{\hat{\epsilon},+})_* = (e_{\hat{\epsilon},-}\circ e^1_{\hat{\epsilon},-})_* \colon \hat{H}^{<a}_*(\hat{\epsilon},\hat{\delta}) \to H^{<a}_* (\hat{\epsilon}, (i_{\epsilon'',\hat{\epsilon}})_* \delta''),\]
which does not follow from (\ref{e-and-e-j}) directly.
To check this equation, we consider the following diagram including $\bar{H}^{<a}_* (\hat{\epsilon},( i_{\epsilon'',\hat{\epsilon}})_* (\bar{i}\delta''))$:
\[\xymatrix{
\hat{H}^{<a}_*(\hat{\epsilon},\hat{\delta}) \ar[rr]^{(e^1_{\hat{\epsilon}, +})_*} \ar[dd]_{(e^1_{\hat{\epsilon},-})_*} \ar[rd]^{(e^2_{\hat{\epsilon},-})_*}& & \bar{H}^{<a}_*(\hat{\epsilon}, (i_{\tilde{\epsilon},\hat{\epsilon}} )_*\tilde{\delta}) \ar[dd]^{(e_{\hat{\epsilon},-})_*}\\
& \bar{H}^{<a}_* (\hat{\epsilon}, (i_{\epsilon'',\hat{\epsilon}})_* (\bar{i}\delta''))\ar@<-0ex>  [rd]_{(e_{\hat{\epsilon},-})_*} \ar@<1ex>[rd]^{(e_{\hat{\epsilon},+})_*} & \\
\bar{H}^{<a}_* (\hat{\epsilon}, (i_{\bar{\epsilon}',\hat{\epsilon}})_* \bar{\delta}')\ar[rr]^{(e_{\hat{\epsilon},-})_*} & & H^{<a}_* (\hat{\epsilon}, (i_{\epsilon'',\hat{\epsilon}})_* \delta'') .
}\]
Then
\[ e_{\hat{\epsilon},-} \circ e^1_{\hat{\epsilon},+}=e_{\hat{\epsilon},+}\circ e^2_{\hat{\epsilon},-},\ e_{\hat{\epsilon},-} \circ e^1_{\hat{\epsilon},-} =e_{\hat{\epsilon},-} \circ e^2_{\hat{\epsilon},-}\]
follow directly from the relations $e_-\circ e^1_+=e_+\circ e^2_-$ and $e_-\circ e^1_-=e_-\circ e^2_-$ of (\ref{e-and-e-j}). If we rewrite $ (i_{\epsilon'',\hat{\epsilon}})_* \delta''$ by $\delta^\#$,
the equation (\ref{e-pm-and-i-bar}) shows that
\[(e_{\hat{\epsilon},+})_* = (e_{\hat{\epsilon},-})_* \colon \bar{H}^{<a}_* (\hat{\epsilon}, \bar{i}\delta^\#) \to H^{<a}_* (\hat{\epsilon}, \delta^\#).\]
From the above diagram, we get $(e_{\hat{\epsilon},-})_* \circ (e^1_{\hat{\epsilon},+})_* = (e_{\hat{\epsilon},-} )_*\circ ( e^1_{\hat{\epsilon},-})_*$. This finishes the proof.
\end{proof}

\begin{cor}\label{cor-indep-k}
For $(\epsilon,\delta),(\epsilon',\delta')\in \mathcal{T}_{a}$ with $\epsilon'\leq \epsilon$,
the isomorphism
\[k_{(\bar{\epsilon},\bar{\delta})} \colon H^{<a}_*(\epsilon',\delta') \to H^{<a}_*(\epsilon,\delta)\]
does not depend on the choice of $(\bar{\epsilon},\bar{\delta})\in \bar{\mathcal{T}}_{a}$ satisfying (\ref{condition-delta-bar}).
\end{cor}
\begin{proof}
Let $(\bar{\epsilon}, \bar{\delta})$ and $(\tilde{\epsilon}, \tilde{\delta} )$ be two arbitrary choices from $\bar{\mathcal{T}}_{a}$ satisfying (\ref{condition-delta-bar}).
We apply Proposition \ref{prop-comm-k} to the case where $(\epsilon'',\delta'')=(\epsilon',\delta')$ and $(\bar{\epsilon}', \bar{\delta}')=(\epsilon', \bar{i} \delta' )$.
By Lemma \ref{lem-trivial-delta-bar}, $k_{ (\epsilon', \bar{i} \delta') }$ is equal to the identity map on $H^{<a}_*(\epsilon',\delta')$, so we get an equation
\[k_{(\bar{\epsilon}, \bar{\delta})} \circ \id_{H^{<a}_*(\epsilon',\delta')} = k_{(\tilde{\epsilon}, \tilde{\delta})}.\]
\end{proof}
From this result, 
we may rewrite $k_{(\bar{\epsilon},\bar{\delta})}\colon  H^{<a}_*(\epsilon',\delta') \to H^{<a}_*(\epsilon,\delta)$ by $k_{(\epsilon',\delta'),(\epsilon,\delta)}$. The equations of Lemma \ref{lem-trivial-delta-bar} and Proposition \ref{prop-comm-k} can be rewritten as
\begin{align*}
\begin{cases}
k_{(\epsilon,\delta),(\epsilon,\delta)}=\id_{H^{<a}_*(\epsilon,\delta)}, \\
k_{(\epsilon',\delta'),(\epsilon,\delta)} \circ k_{(\epsilon'',\delta''),(\epsilon',\delta')} = k_{(\epsilon'',\delta''),(\epsilon,\delta)}  .
\end{cases}\end{align*}
Now, $ \mathcal{T}_{a}$ becomes a directed set by the relation $(\epsilon',\delta')\leq (\epsilon,\delta)$ if and only if $\epsilon'\leq \epsilon$.
Then we obtain an inverse system
\[\left( \{ H^{<a}_*(\epsilon,\delta)\}_{(\epsilon,\delta)\in \mathcal{T}_{a}} , \{ k_{(\epsilon',\delta'),(\epsilon,\delta)} \}_{\epsilon' \leq \epsilon } \right) \]
and its inverse limit
\[H^{<a}_*(Q,K )\coloneqq \varprojlim_{\epsilon\to 0} H^{<a}_*(\epsilon,\delta )\]
is defined.

\subsubsection{Spectral sequence}

Lastly, we extend the above discussions to spectral sequences.
For $(\epsilon,\delta),(\epsilon',\delta')\in \mathcal{T}_a$ with $\epsilon'\leq \epsilon$, we take $(\bar{\epsilon},\bar{\delta})\in \bar{\mathcal{T}}_a$ satisfying (\ref{condition-delta-bar}).
Chain maps $(j_{\epsilon,\bar{\epsilon}})_*$, $e_{\bar{\epsilon}_+}$, $e_{\bar{\epsilon},-}$ and $(j_{\epsilon',\bar{\epsilon}})_*$ induce a zig-zag of morphisms of spectral sequences
\begin{align}\label{zig-zag}
\xymatrix{
E^{<a}_{(\epsilon',\delta')} \ar[r]^-{(j_{\epsilon',\bar{\epsilon}})_*} & E^{<a}_{(\bar{\epsilon},(i_{\epsilon',\bar{\epsilon}})_*\delta')} & \bar{E}^{<a}_{(\bar{\epsilon},\bar{\delta})} \ar[l]_-{(e_{\bar{\epsilon}_-})_* }  \ar[r]^-{(e_{\bar{\epsilon},+})_*} & E^{<a}_{(\bar{\epsilon}, (i_{\epsilon,\bar{\epsilon}})_*\delta)}  &  E^{<a}_{(\epsilon,\delta)} \ar[l]_-{(j_{\epsilon,\bar{\epsilon}})_*}.
}\end{align}
All of them are isomorphism. Let $k_{(\epsilon',\delta'),(\epsilon,\delta)}\colon  E^{<a}_{(\epsilon',\delta')}\to E^{<a}_{(\epsilon,\delta)} $ denote the composition of these maps. (The independence on $(\epsilon,\delta)$ can be proved as Corollary \ref{cor-indep-k}.)
\begin{prop}\label{prop-spectral-seq}
There exists a spectral sequence $E^{<a}= ( \{(E^{<a})^r_{p,q}\}$, $\{(d^{<a})^r_{p,q}\} )$ which converges to $H^{<a}_*(Q,K)$ in the sense of \cite[Bounded Convergence 5.2.5]{weibel} such that
\[(E^{<a})^1_{-m,q} = \begin{cases} H^{\dr}_{q-m(d-1)}(\Sigma^a_{m},\Sigma^0_{m}) &\text{ if }m\geq 0, \\
0 &\text{ if }m<0 \end{cases}\]
\end{prop}
\begin{proof} 
 On the first page, the middle two maps of (\ref{zig-zag}) have the form
\[\xymatrix{
H_{q+p(d-1)}(\Sigma^{a-p\bar{\epsilon}}_{-p},\Sigma^0_{-p}) & \bar{H}_{q+p(d-1)}(\Sigma^{a-p\bar{\epsilon}}_{-p},\Sigma^0_{-p}) \ar[l]_-{(e_+)_*} \ar[r]^-{(e_-)_*}& H_{q+p(d-1)}(\Sigma^{a-p\bar{\epsilon}}_{-p},\Sigma^0_{-p}) .
}\]
Since $(e_-)_*=(\bar{i})^{-1}_*=(e_+)_*$, this composition is equal to the identity map. Therefore, $k_{(\epsilon',\delta'),(\epsilon,\delta)}$ is equal to  $(j_{\epsilon',\epsilon})_*\colon H_{q+p(d-1)}(\Sigma^{a-p\epsilon'}_{-p},\Sigma^0_{-p})  \to H_{q+p(d-1)}(\Sigma^{a-p\epsilon}_{-p},\Sigma^0_{-p}) $ on the $(p,q)$-term ($p\leq 0$) of the first page.

By $\{k_{(\epsilon',\delta'),(\epsilon,\delta)}\}_{\epsilon'\leq \epsilon}$,
we define $(E^{<a})^r_{p,q}\coloneqq \varprojlim_{\epsilon\to 0}(E^{<a}_{(\epsilon,\delta)})^r_{p,q}$. Moreover,  \[(d^{<a})^r_{p,q}\colon (E^{<a})^r_{p,q} \to (E^{<a})^r_{p-r,q+r-1}\] is defined to be the map induced by $\{(d^{<a}_{(\epsilon,\delta)})^r_{p,q}\}_{(\epsilon,\delta)\in \mathcal{T}_a}$. Then the $(p,q)$-term of the first page is given by
\[(E^{<a})^1_{p,q} =\varprojlim_{\epsilon\to 0} H_{q+p(d-1)}(\Sigma^{a-p\epsilon}_{-p},\Sigma^0_{-p}) = H_{q+p(d-1)}(\Sigma^a_{-p},\Sigma^0_{-p}),\]
for $p\leq 0$ and $(E^{<a})^1_{p.q}=0$ for $p>0$.
Since $\{k_{(\epsilon',\delta'),(\epsilon,\delta)}\}_{\epsilon'\leq \epsilon}$ consists of isomorphisms,
it is clear that $E^{<a}$ is a spectral sequence which converges to $H^{<a}_*(Q,K)$.
\end{proof}

\subsection{Definition of $H^{\str}_*(Q,K)$}\label{subsec-string-hmgy}

In this section, we define $I^{a,b} \colon H^{<a}_*(Q,K) \to H^{<b}_*(Q,K)$ for $a,b\in \R_{> 0}\setminus \mathcal{L}(K)$ with $a\leq b$ to get a direst system $(\{H^{<a}_*(Q,K)\}_{a\in \R_{> 0}\setminus \mathcal{L}(K)}, \{I^{a,b}\}_{a\leq b} )$. After defining $H^{\str}_*(Q,K)$ as its direct limit, we will give a structure of a unital graded $\R$-algebra.

\subsubsection{The limit of $a\to \infty$}
Let $a,b$ be the above real numbers. For any $(\epsilon,\delta)\in \mathcal{T}_{a} \cap \mathcal{T}_{b}$,
we have considered in Section \ref{subsec-def-of-chain-cpx} a linear map
\[(I^{a,b}_{\epsilon})_* \colon H^{<a}_*(\epsilon,\delta) \to H^{<b}_*(\epsilon,\delta) ,\]
which is induced by the inclusion maps $I^{a,b}_{\epsilon}\colon \Sigma^{a+m\epsilon}_m \to \Sigma^{b+m\epsilon}_m$ for all $m\in Z_{\geq 0}$.

\begin{lem}\label{lem-commute-I-k}
Suppose that $(\bar{\epsilon},\bar{\delta})\in \bar{\mathcal{T}}_{a} \cap  \bar{\mathcal{T}}_{b}$ satisfies (\ref{condition-delta-bar}) for $(\epsilon,\delta),(\epsilon',\delta)\in \mathcal{T}_a\cap \mathcal{T}_b$ with $\epsilon'\leq \epsilon$. Then, the following diagram commutes:
\[\xymatrix{
H^{<a}_*(\epsilon,\delta) \ar[r]^-{ (I^{a,b}_{\epsilon})_* } & H^{<b}_*(\epsilon,\delta) \\
H^{<a}_*(\epsilon',\delta') \ar[u]^-{k_{(\epsilon',\delta'),(\epsilon,\delta)}} \ar[r]^-{ (I^{a,b}_{\epsilon'})_* }& H^{<b}_*(\epsilon',\delta') \ar[u]_-{k_{(\epsilon',\delta'),(\epsilon,\delta)}}.
}\]
\end{lem}
\begin{proof}
$I^{a,b}_{\bar{\epsilon}}$ induces a linear map
$(I^{a,b}_{\bar{\epsilon}})_* \colon  \bar{H}^{<a}_*(\bar{\epsilon},\bar{\delta}) \to \bar{H}^{<b}_*(\bar{\epsilon},\bar{\delta})$. 
It suffices to show that
\begin{align*}
(I^{a,b}_{\epsilon})_* \circ f_{(\bar{\epsilon},\bar{\delta}),+} = f_{(\bar{\epsilon},\bar{\delta}),+} \circ (I^{a,b}_{\bar{\epsilon}})_* &\colon \bar{H}^{<a}_*(\bar{\epsilon},\bar{\delta}) \to H^{<b}_*(\epsilon,\delta) ,\\
(I^{a,b}_{\epsilon'})_* \circ f_{(\bar{\epsilon},\bar{\delta}),-} = f_{(\bar{\epsilon},\bar{\delta}),-} \circ (I^{a,b}_{\bar{\epsilon}})_*
&\colon \bar{H}^{<a}_*(\bar{\epsilon},\bar{\delta}) \to H^{<b}_*(\epsilon',\delta') .
\end{align*}
Let us check the first equation. Since $j_{\epsilon,\bar{\epsilon}}\circ I^{a,b}_{\epsilon}=I^{a,b}_{\bar{\epsilon}}\circ j_{\epsilon,\bar{\epsilon}}\colon \Sigma^{a+m\epsilon}_m\to \Sigma^{b+m\bar{\epsilon}}_m$,
\begin{align*}
(I^{a,b}_{\epsilon})_* \circ f_{(\bar{\epsilon},\bar{\delta}),+} &= (I^{a,b}_{\epsilon})_* \circ (j_{\epsilon,\bar{\epsilon}})_*^{-1}\circ (e_{\bar{\epsilon},+})_* \\
&= (j_{\epsilon,\bar{\epsilon}})_*^{-1} \circ  (I^{a,b}_{\bar{\epsilon}})_* \circ (e_{\bar{\epsilon},+})_*  \\
&=(j_{\epsilon,\bar{\epsilon}})_*^{-1}\circ  (e_{\bar{\epsilon},+})_* \circ   (I^{a,b}_{\bar{\epsilon}})_* \\
&=f_{(\bar{\epsilon},\bar{\delta}),+} \circ (I^{a,b}_{\bar{\epsilon}})_* .
\end{align*}
The second equation can be proved by replacing $\epsilon$ and ``$+$'' in the above computation by $\epsilon'$ and ``$-$''.
\end{proof}

This lemma implies that, after taking the limits of $(\epsilon,\delta) \in \mathcal{T}_{a} \cap \mathcal{T}_{b}$ with $\epsilon\to 0$, we get a linear map
\[I^{a,b} \coloneqq \varprojlim_{\epsilon\to 0}(I_{\epsilon}^{a,b})_* \colon H^{<a}_*(Q,K) \to H^{<b}_*(Q,K).\]
Furthermore, for $b\geq a$ and $(\epsilon,\delta),(\epsilon',\delta')\in \mathcal{T}_a\cap\mathcal{T}_b$ with $\epsilon'\leq \epsilon$, $k_{(\epsilon',\delta'),(\epsilon,\delta)}$ induces an isomorphism from $H^{[a,b)}_*(\epsilon',\delta')$ to $ H^{[a,b)}_*(\epsilon,\delta)$. Thus, we can also define
\[H^{[a,b)}_*(Q,K)\coloneqq \varprojlim_{\epsilon\to 0} H^{[a,b)}_*(\epsilon,\delta ).\]
Note that a long exact sequence
\begin{align}\label{long-ex-seq}
\xymatrix{
\cdots \ar[r] & H^{<a}_*(Q,K) \ar[r]^{I^{a,b}} & H^{<b}_*(Q,K) \ar[r] & H^{[a,b)}_*(Q,K) \ar[r] & H^{<a}_{*-1}(Q,K) \ar[r]^-{I^{a,b}} & \cdots
}
\end{align}is induced by (\ref{ex-seq-epsilon-delta}).

From the direct system $(\{H^{<a}_*(Q,K))\}_{a\in \R_{> 0} \setminus \mathcal{L}(K)}, \{ I^{a,b} \}_{a\leq b})$, we finally define a graded $\R$-vector space
\[H^{\mathrm{string}}_*(Q,K)\coloneqq  \varinjlim_{a\to \infty} H^{<a}_*(Q,K ). \]

\subsubsection{Product structure}\label{subsubsec-product}

Let us see that $H^{\mathrm{string}}_*(Q,K)$  has a  structure of  a unital associative graded $\R$-algebra.
For any $a,a'\in \R_{> 0}\cup \{\infty\}$, $m,m'\in \Z_{\geq 0}$ and $\epsilon\in (0, \epsilon_0/5]$, there is a map
\[\Pi\colon \Sigma^{a+m\epsilon}_m \times \Sigma^{a'+m'\epsilon}_{m'} \to \Sigma^{(a+a')+(m+m')\epsilon}_{m+m'}\colon ( (\gamma_k)_{k=1,\dots ,m}, (\gamma'_l)_{l=1,\dots ,m'} ) \mapsto (\gamma_1,\dots ,\gamma_m,\gamma'_1,\dots ,\gamma'_{m'}).\]
When $m=0$ or $m'=0$, $\Pi$ is identified with the identity map.
We define a linear map 
\begin{align}\label{prod-of-chain}
C^{<a}_p(\epsilon) \otimes C^{<a'}_q(\epsilon) \to C^{<a+a'}_{p+q}(\epsilon)\colon x\otimes y \mapsto x\star y
\end{align}
so that for $x\in C^{\dr}_{p-m(d-2)}(\Sigma^{a+m\epsilon}_m,\Sigma^{0}_m)$ and $y\in C^{\dr}_{q-m'(d-2)}(\Sigma^{a'+m'\epsilon}_{m'},\Sigma^{0}_{m'}) $,
\[x\star y=  (-1)^{mqd} \Pi_*(x\times y)\]
We note that the associative relation $(x\star y)\star z =x\star (y\star z)$ holds.
Suppose that $a,a',a+a'\notin \mathcal{L}(K)$. For any $(\epsilon,\delta)\in \mathcal{T}_{a} \cap \mathcal{T}_{a'}\cap \mathcal{T}_{a+a'}$, the above map is compatible with the differential $D_{\delta}$. Indeed, for $x\in C^{\dr}_{p-m(d-2)}(\Sigma^{a+m\epsilon}_m,\Sigma^{0}_m)$ and $y\in C^{\dr}_{q-m'(d-2)}(\Sigma^{a'+m'\epsilon}_{m'},\Sigma^{0}_{m'}) $,
\begin{align*}
(-1)^{mqd} D_{\delta} (x\star y ) =&  \partial (\Pi_*(x\times y)) + \sum_{k=1}^{m+m'} (-1)^{p+q+1+kd} f_{k,\delta}(\Pi_*(x\times y) )\\
=&\Pi_* (\partial x\times y)  + \sum_{k=1}^m (-1)^{p+q+1+kd+s_0}  \Pi_* (f_{k,\delta}(x)\times y) \\
&+ (-1)^{p-m(d-2)} \Pi_*(x\times \partial y)+ \sum_{l=1}^{m'} (-1)^{p+q+1+(l+m)d} \Pi_* (x\times f_{l,\delta}(y)) \\
=& (-1)^{mqd} ( D_{\delta} (x)) \star y + (-1)^{p+mqd} x \star (D_{\delta} (y)).
\end{align*}
Here, $s_0\coloneqq  ( q-m'(d-2) ) (d+1)$. This computation shows that
\[D_{\delta}(x\star y)= (D_{\delta}(x))\star y + (-1)^p x\star (D_{\delta}(y))\]
holds for $x\in C^{<a}_p(\epsilon)$ and $y\in C^{<a'}_q(\epsilon)$.
Therefore, (\ref{prod-of-chain}) induces a linear map on homology
\[H^{<a}_p(\epsilon,\delta)\otimes H^{<a'}_q(\epsilon,\delta) \to H^{<a+a'}_{p+q}(\epsilon,\delta) \]
for every $(\epsilon,\delta) \in \mathcal{T}_{a} \cap \mathcal{T}_{a'}\cap \mathcal{T}_{a+a'}$.

Likewise, let us define $ x\ \bar{\star} \ y\coloneqq (-1)^{mqd}\Pi_*(x\times y)\in \bar{C}^{<a+a'}_{p+q}(\bar{\epsilon})$ for $x\in \bar{C}^{\dr}_{p-m(d-2)}(\Sigma^{a+m\bar{\epsilon}}_m,\Sigma^{0}_m)$ and $y\in \bar{C}^{\dr}_{q-m'(d-2)}(\Sigma^{a'+m'\bar{\epsilon}}_{m'},\Sigma^{0}_{m'}) $. Then, $e_{\bar{\epsilon},+},e_{\bar{\epsilon},-}$ intertwine the $\star$-operation and the $\bar{\star}$-operation.
This shows that the $\star$-operation is compatible with $\{ k_{(\epsilon',\delta'),(\epsilon,\delta)} \}_{\epsilon'\leq \epsilon}$. Therefore, on the limit of $\epsilon\to 0$, a linear map
\[H^{<a}_p(Q,K)\otimes H^{<a'}_q(Q,K)\to H^{<a+a'}_{p+q}(Q,K)\]
is induced. The commutativity with $\{ I^{a,b} \}_{a\leq b}$ naturally holds. As a result, we get an associative product structure on $H^{\str}_*(Q,K)$. The element $1\in H^{\str}_0(Q,K)$, which comes from
$1\in \R= C^{\dr}_0(\Sigma^a_0,\Sigma^0_0) \subset C^{<a}_0(\epsilon,\delta)$, is the unit of this graded algebra.

\subsubsection{Explicit choice of $(\epsilon,\delta)$}\label{subsub-explicit}

Lastly, let us introduce a condition on $(\epsilon,\delta)\in \mathcal{T}_a$ so that $\delta$ can be written explicitly. The concrete computations in Section \ref{sec-example} and \ref{sec-cord-alg} become easier by choosing $(\epsilon,\delta)$ satisfying this condition.

Suppose that there exists a fixed trivialization $\R^d\times K\to (TK)^{\perp}$ of the normal bundle of $K$ which preserves orientations and fiber metrics. For every $\epsilon\leq \epsilon_0$, let us write $ \mathcal{O}_{\epsilon}\coloneqq  \{w\in \R^{d} \mid |w|<\epsilon/2\}$.
Composing with the exponential map (\ref{tubular-neighborhood}), we obtain a diffeomorphism 
\[h\colon\mathcal{O}_{\epsilon} \times K\to N_{\epsilon}  ,\]
which preserves orientations.
In this case, we say $(\epsilon,\delta)\in \mathcal{T}_a$ is \textit{standard with respect to $h$}, if $\delta\in C^{\dr}_{n-d}(S_{\epsilon})$ has the form
\begin{align}\label{delta-standard}
\delta = [N_{\epsilon}, \psi_{\epsilon},h_*(\nu_{\epsilon}\times 1)] 
\end{align}
satisfying:
\begin{itemize}
\item $\psi_{\epsilon}\colon N_{\epsilon}\to S_{\epsilon}\colon v\to (\sigma^v_j)_{j=1,2}$ is defined by
\[\sigma^v_1\colon [0,\epsilon/2]\to N_{\epsilon}\colon t\mapsto h \left( \frac{\epsilon-2t}{\epsilon}w ,x\right), \ \sigma^v_2\colon [0,\epsilon/2]\to N_{\epsilon}\colon t\mapsto h \left( \frac{2t}{\epsilon}w ,x\right) , \]
for $v= h (w,x)\in N_{\epsilon}$.
\item $h_*(\nu_{\epsilon} \times 1)\in \Omega^{n-d}_c(N_{\epsilon})$ for some $\nu_{\epsilon}\in \Omega^{d}_c( \mathcal{O}_{\epsilon})$ with $\int_{\mathcal{O}_{\epsilon}} \nu_{\epsilon}=1$.
\end{itemize}

Suppose that $(\bar{\epsilon},\bar{\delta})\in \bar{\mathcal{T}}_a$ satisfies (\ref{condition-delta-bar}) for $(\epsilon,\delta),(\epsilon',\delta')\in \mathcal{T}_a$ ($\epsilon'\leq \epsilon$) which are given as above.
Then, we say $(\bar{\epsilon},\bar{\delta})$ is \textit{standard with respect to $h$}, if $\bar{\epsilon}=\epsilon$ and $\bar{\delta}\in \bar{C}^{\dr}_{n-d}(S_{\epsilon})$ has the form
\begin{align}\label{standard-delta-bar}
\bar{\delta}= (-1)^{n} [\R\times N_{\epsilon}, \bar{\psi}_{\epsilon',\epsilon},(\id_{\R_{\geq 1}\times N_{\epsilon}} , \id_{\R_{\leq -1}\times N_{\epsilon}} ), \bar{\eta}_{\epsilon',\epsilon} ] 
\end{align}
such that for some $C^{\infty}$ function $\kappa\colon \R\to [0,1]$ with $\kappa(r)=\begin{cases} 0& \text{ if }r\leq -1, \\1 & \text{ if }r\geq 1, \end{cases}$ the following hold:
\begin{itemize}
\item $\bar{\psi}_{\epsilon',\epsilon}(r,v)=(r,(\sigma^{(r,v)}_j)_{j=1,2})\in \R\times S_{\epsilon}$ is defined by
\[\sigma^{(r,v)}_1 \colon [0,\epsilon_r/2] \to N_{\epsilon} \colon t\mapsto h \left(\frac{\epsilon_r-2t}{\epsilon_r}w, x \right),\ \sigma^{(r,v)}_2 \colon [0,\epsilon_r/2] \to N_{\epsilon} \colon t\mapsto h \left( \frac{2t}{\epsilon_r}w, x \right)\]
for $r\in \R$, $ v=h(w,x)\in N_{\epsilon}$ and $\epsilon_r\coloneqq \kappa(r)\epsilon+(1-\kappa(r))\epsilon'$.
\item 
$(\id_{\R}\times h)^*\bar{\eta}_{\epsilon',\epsilon}= \kappa\times (\nu_{\epsilon}\times 1) + (1-\kappa)\times (\nu_{\epsilon'}\times 1)+(d\kappa)\times ( \theta \times 1)$
for some $\theta\in \Omega^{d-1}_c(\mathcal{O}_{\epsilon})$ satisfying $d\theta= \nu_{\epsilon}-\nu_{\epsilon'}$.
\end{itemize}
In summary, in order to compute $H^{\str}_*(Q,K)$ when a trivialization $h$ is given, we only need to deal with $(\epsilon,\delta)\in \mathcal{T}_a$ and $(\epsilon,\bar{\delta})\in \bar{\mathcal{T}}_a$ which are standard with respect to $h$.

\subsection{Invariance}\label{subsec-invariance}

In this section, we prove the invariance of $H^{\mathrm{string}}_*(Q,K)$ up to isomorphism by changing auxiliary data. More precisely, we consider the dependence of the construction on the following data (See the beginning of Section \ref{sec-sp-of-paths}):
\begin{enumerate}
\item a complete Riemannian metric $g$ on $Q$.
\item a constant $C_0\geq 1$ which bounds the speed of all $\gamma\in \Omega_K(Q)$.
\item a real number $\epsilon_0 >0$ which is the diameter (in the direction of fibers) of a tubular neighborhood $N_{\epsilon_0}$ of $K$.
\item a $C^{\infty}$ function $\mu\colon [0,\frac{3}{2}]\to [0,1]$ which is used to define  $\con_k$.
\end{enumerate}

\begin{notation}
Let $X$ be an arbitrary notation which we have defined in the former sections. As a rule in this section, if its definition depends on some auxiliary data $S$, we rewrite $X$ by $X_S$ when discussing the dependence on $S$.
\end{notation}

\textbf{Independence on} $\boldsymbol{\mu .}$
We choose a $C^{\infty}$ family $\bar{\mu}\coloneqq (\mu_r)_{r\in \R}$ such that each $\mu_r$ satisfy the same condition as $\mu$, and  $\mu_{r}=\begin{cases} \mu_{-1} & \text{ if }r\leq -1, \\ \mu_1 &\text{ if }r\geq 1.\end{cases}$
Then, a map $\con_{k,\mu_r}$ is defined for each $r\in \R$.
For $(\bar{\epsilon},\bar{\delta})\in \mathcal{T}_a$, let us define $\bar{f}_{k,\bar{\delta},\bar{\mu}} \colon \bar{C}^{\dr}_*(\Sigma^{a+m\bar{\epsilon}}_m) \to \bar{C}^{\dr}_*(\Sigma^{a+(m+1)\bar{\epsilon}}_{m+1})$ by
replacing $\con_k$ in the definition of $\Phi_k$ of (\ref{explicit-rep-bar}) by $\con_{k,\mu_{r(u)}}$. We also replace $\bar{f}_{k,\bar{\delta}}$ in the definition of $\bar{D}_{\bar{\delta}}$ by $\bar{f}_{k,\bar{\delta},\bar{\mu}}$ to define a linear map
\[\bar{D}_{\bar{\delta},\bar{\mu}} \colon \bar{C}^{<a}_*(\bar{\epsilon}) \to \bar{C}^{<a}_{*-1}(\bar{\epsilon}) .  \] 
This satisfies $\bar{D}_{\bar{\delta},\bar{\mu}} \circ \bar{D}_{\bar{\delta},\bar{\mu}} =0$, so we get a chain complex $( \bar{C}^{<a}_*(\bar{\epsilon}), \bar{D}_{\bar{\delta},\bar{\mu}} )$.  Let $\bar{H}^{<a}_*(\epsilon, \bar{\delta}, \bar{\mu})$ denote its homology group.

We rewrite $e_{\epsilon,+},e_{\epsilon,-} \colon \bar{C}^{<a}_*(\epsilon) \to C^{<a}_*(\epsilon)$ by $e_{\epsilon,\bar{\mu},+},e_{\epsilon,\bar{\mu},-} $ respectively. They induce
\begin{align*}
 (e_{\epsilon,\bar{\mu},+})_* &\colon \bar{H}^{<a}_*(\epsilon, \bar{\delta}, \bar{\mu}) \to H^{<a}_*(\epsilon,e_{+}\delta)_{\mu_{1}}, \\
 (e_{\epsilon,\bar{\mu},-})_*  &\colon \bar{H}^{<a}_*(\epsilon, \bar{\delta}, \bar{\mu}) \to H^{<a}_*(\epsilon,e_{-}\delta)_{\mu_{-1}}.
 \end{align*}
We can prove as Lemma \ref{lem-e-quasi-isom} that they are isomorphisms. When $(\bar{\epsilon},\bar{\delta}) \in \bar{\mathcal{T}}_{[a,b)}$ satisfies (\ref{condition-delta-bar}), we define
an isomorphism $k_{(\bar{\epsilon},\bar{\delta},\bar{\mu})} \colon H^{<a}_*(\epsilon',\delta')_{\mu_{-1}} \to H^{<a}_*(\epsilon,\delta)_{\mu_1}$ by
\[k_{(\bar{\epsilon},\bar{\delta},\bar{\mu})} \coloneqq ((j_{\epsilon,\bar{\epsilon}})_*^{-1}\circ (e_{\bar{\epsilon},\bar{\mu},+})_*)\circ ((j_{\epsilon',\bar{\epsilon}})_*^{-1}\circ (e_{\bar{\epsilon},\bar{\mu},-})_*)^{-1} . \]
It is proved as Proposition \ref{prop-comm-k} that the two triangles in the following diagram commute:
\[ \xymatrix@C=36pt{    H^{<a}_*(\epsilon,\delta)_{\mu_{-1}} \ar[r]^-{k_{(\epsilon,\bar{i}\delta, \bar{\mu})} } &   H^{<a}_*(\epsilon,\delta)_{\mu_{1}} \\
H^{<a}_*(\epsilon' ,\delta')_{\mu_{-1}} \ar[r]^-{k_{(\epsilon',\bar{i}\delta',\bar{\mu})} } \ar[u]^{k_{ (\epsilon',\delta'),(\epsilon,\delta) }}  \ar[ru]^-{ k_{(\bar{\epsilon},\bar{\delta},\bar{\mu})} }&   H^{<a}_*(\epsilon',\delta')_{\mu_{1}}\ar[u]_{k_{ (\epsilon',\delta'),(\epsilon,\delta) }}   . } \]
Therefore, $\{ k_{(\epsilon,\bar{i}\delta,\bar{\mu})} \}_{(\epsilon,\delta)\in \mathcal{T}_a}$ induces an isomorphism on the limits of $\epsilon \to 0$ 
\[k^a_{\bar{\mu}}\colon H^{<a}_*(Q,K)_{\mu_{-1}} \to H^{<a}_*(Q,K)_{\mu_1} .\]
It is easy to see as Lemma \ref{lem-commute-I-k} that $\{k^a_{\bar{\mu}}\}_{a\in \R_{>0}\setminus \mathcal{L}(K)}$ commute with $\{I^{a,b}\}_{a\leq b}$, so we get an isomorphism from $H^{\str}_*(Q,K)_{\mu_{-1}}$ to $H^{\str}_*(Q,K)_{\mu_1}$. It is also possible to prove as Corollary \ref{cor-indep-k} that this isomorphism does not depend on the choice of $\bar{\mu}$.


\textbf{Independence on} $\boldsymbol{\epsilon_0.}$ For $\epsilon'_0\leq \epsilon_0$, we consider the inclusion maps $ \mathfrak{j}_{\epsilon'_0,\epsilon_0}\colon \Sigma^{a+m\epsilon}_{m,\epsilon'_0} \to \Sigma^{a+m\epsilon}_{m,\epsilon_0}$ for all $m\in \Z_{\geq 0}$. They induce a chain map $(\mathfrak{j}_{\epsilon'_0,\epsilon_0})_*$
from $(C^{<a}_*(\epsilon)_{\epsilon'_0},D_{\delta} )$ to $(C^{<a}_*(\epsilon)_{\epsilon_0} , D_{\delta} )$ for every $a\in \R_{>0}\setminus \mathcal{L}(K)$ and $(\epsilon,\delta)\in \mathcal{T}_{a,\epsilon'_0} (\subset \mathcal{T}_{a,\epsilon_0} )$.
This preserves filtrations $\{\mathcal{F}^{<a}_{\epsilon, p,\epsilon_0}\}_{p\in \Z}$ and $\{\mathcal{F}^{<a}_{\epsilon, p,\epsilon'_0}\}_{p\in \Z}$, so it induces an morphism of the spectral sequences. On the $(-m,q)$-term ($m\geq 0$) of the first page, it is equal to the map
\[ (\mathfrak{j}_{\epsilon'_0,\epsilon_0})_* \colon H^{\dr}_{q-m(d-1)}(\Sigma^{a+m\epsilon}_{m,\epsilon'_0} ,\Sigma^{0}_{m,\epsilon'_0}) \to H^{\dr}_{q-m(d-1)}(\Sigma^{a+m\epsilon}_{m,\epsilon_0},\Sigma^{0}_{m,\epsilon_0}) , \]
which is an isomorphism by Lemma \ref{lem-epsilon0-indep}. Therefore, by Lemma \ref{lem-spectral},
$ (\mathfrak{j}_{\epsilon'_0,\epsilon_0})_* \colon H^{<a}_*(\epsilon,\delta)_{\epsilon'_0}\to H^{<a}_*(\epsilon,\delta)_{\epsilon_0} $
is also an isomorphism.

We can prove  its commutativity with $\{k_{(\epsilon',\delta'), (\epsilon,\delta)}\}_{\epsilon'\leq \epsilon}$ as Lemma \ref{lem-commute-I-k}, so we get an isomorphism on the limit of $\epsilon \to 0$
\[\mathfrak{J}_{\epsilon'_0,\epsilon_0}^a \colon H^{<a}_*(Q,K)_{\epsilon'_0} \to H^{<a}_*(Q,K)_{\epsilon_0} .\]
It holds naturally that $\{\mathfrak{J}_{\epsilon'_0,\epsilon_0}^a\}_{a\in \R_{>0}\setminus \mathcal{L}(K)}$ commutes with $\{I^{a,b}\}_{a\leq b}$. Therefore, on the limit of  $a\to \infty$, we get an isomorphism from $H^{\str}_*(Q,K)_{\epsilon'_0}$ to $H^{\str}_*(Q,K)_{\epsilon_0}$.

\textbf{Independence on} $\boldsymbol{C_0.}$ For $C'_0\geq C_0\geq 1$, we consider the inclusion maps $\mathfrak{j}_{C_0,C'_0} \colon \Sigma^a_{m,C_0}\to \Sigma^a_{m,C'_0}$ for all $m\in \Z_{\geq 0}$.  Parallel to the proof of the independence on $\epsilon_0$,
we apply Lemma \ref{lem-epsilon0-indep} to show that an isomorphism
\[\mathfrak{J}^a_{C_0,C'_0} \colon H^{<a}_*(Q,K)_{C_0} \to H^{<a}_*(Q,K)_{C'_0}\]
is induced. It hold naturally that $\{\mathfrak{J}^a_{C_0,C'_0}  \}_{a\in \R_{>0}\setminus \mathcal{L}(K)}$ commutes with $\{I^{a,b}\}_{a\leq b}$. Therefore, on the limit of  $a\to \infty$, we get an isomorphism from $H^{\str}_*(Q,K)_{C_0}$ to $H^{\str}_*(Q,K)_{C'_0}$.

\textbf{Independence on} $\boldsymbol{g.}$
First, let us introduce an graded algebra $\acute{H}^{\str}_*(Q,K)_g$ which is isomorphic to $H^{\str}_*(Q,K)_{(g,C_0,\epsilon_0)}$, but whose definition does not depend on $\epsilon_0$ and $C_0$. For every $a\in \R_{>0}\setminus \mathcal{L}(K)_g$,
we define the limit of $\epsilon_0\to 0$ and $C_0\to\infty$
\[\acute{H}^{<a}_*(Q,K)_g \coloneqq \varinjlim_{C_0\to \infty} \varprojlim_{\epsilon_0\to 0} H^{<a}_*(Q,K)_{(g,C_0,\epsilon_0)}\]
by $\{\mathfrak{J}^a_{\epsilon'_0,\epsilon_0}\}_{\epsilon'_0\leq \epsilon_0}$ and $\{\mathfrak{J}^a_{C_0,C'_0}\}_{C_0\leq C'_0}$ defined above. Then, $\{I^{a,b}\}_{a\leq b}$ induces a family of maps
\[\{ \acute{I}^{a,b}_g\colon \acute{H}^{<a}_*(Q,K)_g \to \acute{H}^{<b}_*(Q,K)_g \}_{a\leq b},\]
and we can take the limit of $a\to \infty$ to define $\acute{H}^{\str}_*(Q,K)_g \coloneqq \varinjlim_{a\to \infty } \acute{H}^{<a}_*(Q,K)_g$.

Suppose that $g$ and $g'$ are complete Riemannian metrics on $Q$. For $a>0$, there exists a compact subset $Z_a$ which contains the images of all $\gamma\in \bigcup_{C_0\geq 1}\Omega_{K}(Q)_{(g,C_0)}$ with $\len_g \gamma <a$ and the images of all $\gamma \in \bigcup_{C_0\geq 1}\Omega_{K}(Q)_{(g',C_0)}$ with $\len_{g'} \gamma <a$. For any $a\in \R_{>0} \setminus (\mathcal{L}(K)_g\cup \mathcal{L}(K)_{g'})$, there exists a constant $c_a\geq 1$ such that
$ |\cdot |_{g'} \leq c_a |\cdot |_g$ and $|\cdot |_{g}\leq c_a |\cdot|_{g'}$ on $Z_a$. We may additionally assume that $ac_a\notin \mathcal{L}(K)_{g}\cup \mathcal{L}(K)_{g'}$.

Let $a\in \R_{>0} \setminus (\mathcal{L}(K)_g\cup \mathcal{L}(K)_{g'})$. For any $C_0\geq 1$ and $\epsilon_0>0$, we have the inclusion map
\[\mathfrak{j}_{(C_0,\epsilon_0)} \colon \Sigma^a_{m,(g,C_0,\epsilon_0)} \to \Sigma^{ac_a}_{m,(g',C_0c_a,\epsilon_0c_a)} . \] 
%
In addition, let $S^a_{\epsilon, g}$ be a subspace of $S_{\epsilon,g}$ which consists of $(\sigma_i)_{i=1,2}$ satisfying $ |(\sigma_i)'(t)|_g \leq c_a^{-1}$ for $i=1,2$. If $\epsilon$ is sufficiently small, we have the inclusion map
\[\mathfrak{i}_{\epsilon} \colon S^a_{\epsilon,g} \to S_{\epsilon c_a ,g'} .\]
These maps induces a linear map on the homology groups
\begin{align*}
(\mathfrak{j}_{(C_0,\epsilon_0)})_* &\colon H^{<a}_*(\epsilon, \delta)_{(g,C_0, \epsilon_0)} \to H^{<ac_a}_*(\epsilon c_a,(\mathfrak{i}_{\epsilon})_* \delta )_{(g',C_0c_a,\epsilon_0c_a)}
\end{align*}
for $(\epsilon,\delta)\in \mathcal{T}_{a, (g,C_0,\epsilon_0)}$ such that $\delta\in C^{\dr}_{n-d}(S^a_{\epsilon, g})$.
Its commutativity with $\{k_{(\epsilon',\delta'), (\epsilon,\delta)}\}_{\epsilon'\leq \epsilon}$ can be proved as Lemma \ref{lem-commute-I-k}. Let us write the induced map on the limits of $\epsilon \to 0$ by
\begin{align*}
\mathfrak{J}_{(C_0,\epsilon_0)}^a \colon H^{<a}_*(Q,K)_{(g,C_0, \epsilon_0)} \to H^{<ac_a}_*(Q,K)_{(g',C_0c_a, \epsilon_0c_a)} .
\end{align*}
Moreover, it holds naturally that the maps of $\{\mathfrak{J}^a_{(C,\epsilon_0)}\}_{C_0\geq 1,\epsilon_0>0}$ commute with those of $\{\mathfrak{J}^a_{\epsilon'_0,\epsilon_0}\}_{\epsilon'_0\leq \epsilon_0}$ and $\{\mathfrak{J}^a_{C_0,C'_0}\}_{C_0\leq C'_0}$, so we get a map on the limit of $\epsilon_0\to 0$ and $C_0\to \infty$
\begin{align*}
\mathfrak{J}^{a}&\coloneqq \varinjlim_{C_0\to \infty} \varprojlim_{\epsilon_0\to 0}\mathfrak{J}^a_{(C_0,\epsilon_0)}\colon \acute{H}^{<a}_*(Q,K)_{g} \to \acute{H}^{<c_aa}_*(Q,K)_{g'} .
\end{align*}
Lastly, $\{\mathfrak{J}^a\}_{a\in \R_{>0}\setminus (\mathcal{L}(K)_g \cup \mathcal{L}(K)_{g'})}$ is compatible with $\{\acute{I}^{a,b}_g\}_{a\leq b}$, so it induces a map on the limit of $a\to \infty$
\[\mathfrak{J}\colon \acute{H}^{\str}_*(Q,K)_g \to \acute{H}^{\str}_*(Q,K)_{g'}.\]

If we exchange $g$ and $g'$, we can also define $(\mathfrak{J}')^a \colon \acute{H}^{<a}_*(Q,K)_{g'} \to \acute{H}^{<ac_a}_*(Q,K)_{g}$ for $a\in \R_{>0} \setminus (\mathcal{L}(K)_g\cup \mathcal{L}(K)_{g'})$.
For $b\coloneqq ac_a$, we have
\[(\mathfrak{J}')^{b} \circ \mathfrak{J}^a =\acute{I}^{a,bc_b}_g,\ \mathfrak{J}^{b}\circ (\mathfrak{J}')^{a} =\acute{I}^{a,bc_b}_{g'}.\]
Therefore, $\varinjlim_{a\to \infty} (\mathfrak{J}')^a$ is the inverse map of $\mathfrak{J}$. This proves the independence on $g$.

We finish the proof of the independence on auxiliary data.

Finally, let us prove the invariance by changing the orientation of $K$.
Suppose that $K=\sqcup_{\alpha\in A} K_{\alpha}$ for connected components $\{K_{\alpha}\}_{\alpha\in A}$. Then, $N_{\epsilon}=\sqcup_{\alpha\in A} N_{\epsilon,\alpha}$ and $S_{\epsilon}= \sqcup_{\alpha\in A} S_{\epsilon,\alpha}$, where $N_{\epsilon,\alpha}$ is a tubular neighborhood of $K_{\alpha}$ and $S_{\epsilon,\alpha}$ consists of pairs of paths in $N_{\epsilon,\alpha}$. In addition,  for every $\alpha_1,\dots ,\alpha_{2m}\in A$, let $\Sigma^a_{m,(\alpha_1,\dots ,\alpha_{2m})}$ be the subspace of $\Sigma^a_m$ consisting  of $(\gamma_k\colon [0,T_k]\to Q)_{k=1,\dots, m}$ such that $\gamma_k(0)\in K_{\alpha_{2k-1}}$ and $\gamma_k(T_k)\in K_{\alpha_{2k}}$ for $k=1,\dots, m$. Then $C^{\dr}_{n-d}(S_{\epsilon})=\bigoplus_{\alpha\in A} C^{\dr}_{n-d}(S_{\epsilon,\alpha})$ and $C^{\dr}_*(\Sigma^a_m)= \bigoplus_{\alpha_1,\dots ,\alpha_{2m}\in A} C^{\dr}_*(\Sigma^a_{m,(\alpha_1,\dots ,\alpha_{2m})})$.

For any subset $B\subset A$, let $K_B$ be an oriented submanifold obtained from $K$ by reversing the orientations of $\{K_{\alpha}\}_{\alpha\in B}$. For every $\delta =\sum_{\alpha\in A} \delta_{\alpha}\in C^{\dr}_{n-d}(S_{\epsilon})$ ($\delta_{\alpha}$ is a chain in $S_{\epsilon,\alpha}$), let us write $\delta_B\coloneqq \sum_{\alpha\in A\setminus B} \delta_{\alpha} -\sum_{\alpha\in B} \delta_{\alpha}$.
By using a notation
\[s(\alpha_1,\dots, \alpha_{2m})\coloneqq \#\{k\in \{1,\dots ,m\} \mid \alpha_{2k}\in B \},\]
we define a linear map $F^a_{\epsilon}\colon C^{<a}_*(\epsilon)\to C^{<a}_*(\epsilon)$ so that $F^a_{\epsilon}(x)= (-1)^{s(\alpha_1,\dots, \alpha_{2m})} x$ for every $x\in C^{\dr}_{*-m(d-2)}(\Sigma^a_{m,(\alpha_1,\dots ,\alpha_{2m})}, \Sigma^0_{m,(\alpha_1,\dots ,\alpha_{2m})})$ ($m\geq 1$), and $F^a_{\epsilon}(x)=x$ for $x\in C^{\dr}_*(\Sigma^a_0,\Sigma^0_0)$. For every $(\epsilon,\delta)\in \mathcal{T}_a$, $(\epsilon,\delta_B)$ satisfies the conditions of Definition \ref{def-class-of-data} for $(Q,K_B)$, and $F^a_{\epsilon}$ is a chain map from $(C^{<a}_*(\epsilon),D_{\delta})$ to $(C^{<a}_*(\epsilon),D_{\delta_B})$. Moreover, $F^a_{\epsilon}$ is compatible with the $\star$-operation. By taking the limit of $\epsilon\to 0$ and $a\to \infty$, we obtain an isomorphism of graded $\R$-algebras
\[F\colon H^{\str}_*(Q,K)\to H^{\str}_*(Q,K_B).\]
\begin{rem}
Similarly, one can prove the invariance of $H^{\str}_*(Q,K)$ by changing the orientation of $Q$.
\end{rem}

\begin{prop}\label{prop-isotopy-invariance}
$H^{\str}_*(Q,K)$ is invariant by a $C^{\infty}$ isotopy of $K$.
\end{prop}
\begin{proof}
For two oriented compact submanifold $K_0,K_1$ of $Q$,
suppose that there exists a $C^{\infty}$ family of embedding maps $\{f_t\colon K_0\to Q\}_{t\in [0,1]}$ such that $f_0$ is the inclusion map of $K_0$ and $f_1(K_0) = K_1$. 
Then, this isotopy can be extended to an ambient isotopy $\{F_t\}_{t\in [0,1]}$ such that $F_0=\id_Q$ and $F_1(K_0)=K_1$.
Since $F_1$ is an isometry form $(Q,(F_1)^*g)$ to $(Q,g)$, 
it naturally induces an isomorphism
\[ H^{\str}_*(Q,K_0)_{(F_1)^*g}\to H^{\str}_*(Q,K_1)_g .\]
The assertion follows from the independence on the Riemannian metric on $Q$.
\end{proof}

\section{Examples}\label{sec-example}

In this section we determine the algebraic structure of $H^{\str}_*(Q,K)$ for two examples when $Q=\R^{2d-1}$ ($d\geq 2$).
These examples are higher dimensional generalizations of the Hopf link and the unlink in $\R^3$.

The manifold $Q=\R^{2d-1}$ has the standard orientation. Let us use the coordinate $(z_0,z_1,z_2)\in \R^{d-1}\times \R\times \R^{d-1} =  \R^{2d-1}$. The unit sphere $S^{d-1}\subset \R^{d}$ is oriented as the boundary of the unit ball. 
We consider three ways of embedding of unit sphere $S^{d-1}\subset \R^d$ into $\R^{2d-1}$:
\begin{align*}
&S^{d-1} \subset \R^{d}= \R^{d-1}\times \R \to \R^{2d-1}\colon (z_0,z_1)\mapsto (z_0,z_1,0), \\
&S^{d-1} \subset \R^{d}= \R\times \R^{d-1} \to \R^{2d-1}\colon (z_1,z_2)\mapsto (0,z_1+1,z_2), \\
&S^{d-1} \subset \R^{d}= \R^{d-1}\times \R \to \R^{2d-1}\colon (z_0,z_1)\mapsto (z_0,z_1,z_2^*),
\end{align*}
for a fixed vector $z^*_2\in \R^{d-1}\setminus \{0\}$.
Their images are written by $K_0, K_1, K_2$ in order. These submanifolds are oriented so that the diffeomorphisms $S^{d-1}\to K_i$ from the above maps change the sign of orientation by $(-1)^{d-1}$.

As a notation, given a set $\mathcal{S}$ and a map $\mathcal{S}\to \Z\colon s\mapsto |s|$,
let $\mathcal{A}_*(\mathcal{S})$ denote the unital non-commutative graded $\R$-algebra freely generated by $\mathcal{S}$ such that $s\in \mathcal{A}_{|s|}(\mathcal{S})$ for every $s\in \mathcal{S}$.

\subsection{Computation of $H^{\str}_*(\R^{2d-1},K_0\cup K_1)$}\label{subsec-hopf}

Let us define $\mathcal{A}^{\hopf}_*\coloneqq \mathcal{A}_*(\mathcal{C}\cup \mathcal{D}\cup\mathcal{E})$ by the following three sets:
\begin{align*}
\mathcal{C} &\coloneqq \{c^0_{i,j}\}_{i\neq j} \cup \{c^1_{i,i}\}_{i}\cup \{c^1_{i,j},\bar{c}^1_{i,j}\}_{i\neq j}\cup \{c^2_{i,j}\}_{i,j}, \\
\mathcal{D} &\coloneqq \{d^1_{i,i}\}_{i}\cup  \{d^2_{i,j}\}_{i,j}, \\
\mathcal{E} &\coloneqq \{e^1_{i,i}\}_{i}\cup \{e^2_{i,j}\}_{i,j},
\end{align*}
where $i$ and $j$ run over $\{0,1\}$.
The degree of each element is given by
\begin{align*}
 & |c^0_{i,j}|=d-2,\ |c^1_{i,i}|=|c^1_{i,j}|=|\bar{c}^1_{i,j}|=2d-3  \text{ for }i\neq j,\\
 & |c^1_{i,i}|=2d-3,\ |c^2_{i,j}|=3d-4, \\
 & |d^1_{i,i}|=2d-3,\ |d^2_{i,j}|=3d-4, \\
 & |e^1_{i,i}|=2d-4,\ |e^2_{i,j}|=3d-5.
\end{align*}
We define a graded derivation $\partial \colon \mathcal{A}^{\hopf}_* \to \mathcal{A}_{*-1}^{\hopf}$ so that
\begin{align*}
& \partial c^0_{i,j}=0,\ \partial c^1_{i,i}=\partial c^1_{i,j}=\partial \bar{c}^1_{i,j}=0,\ \partial c^2_{i,j}=0, \\ 
& \partial d^1_{i,i}=e^1_{i,i},\ \partial d^2_{i,j}=e^2_{i,j}, \\
& \partial e^1_{i,i} =0,\ \partial e^2_{i,j}=0,
\end{align*}
and they are extended by the Leibniz rule. 
We also define anther graded derivation $F\colon \mathcal{A}^{\hopf}_* \to \mathcal{A}_{*-1}^{\hopf}$ so that
\begin{align*}
& F c^0_{i,j}=F c^1_{i,j}=F \bar{c}^1_{i,j} =0 \text{ for }i\neq j, \\
& F c^1_{0,0}=(-1)^d e^1_{0,0} + (-1)^d c^0_{0,1}c^0_{1,0} ,\ F c^1_{1,1}=(-1)^d e^1_{1,1} +c^0_{1,0}c^0_{0,1}  ,\\
& F c^2_{0,0}= -e^2_{0,0}- (\bar{c}^1_{0,1}c^0_{1,0} + (-1)^{d}c^1_{0,1}c^0_{1,0}) ,\ F c^2_{1,1}= - e^2_{1,1} - ((-1)^d \bar{c}^1_{1,0}c^0_{0,1} + c^1_{1,0}c^0_{0,1}) , \\
& F c^2_{0,1}= -e^2_{0,1},\ F c^2_{1,0}= -e^2_{1,0}, \\
& F d^1_{i,i}=0,\ F d^2_{i,j}=0, \\
& F e^1_{i,i}=0,\ F e^2_{i,j}=0,
\end{align*}
and they are extended by the Leibniz rule. 
It is easy to see that $\partial\circ \partial =0$, $\partial\circ F + F \circ \partial =0$ and $F \circ F =0$ hold. Therefore, we obtain a differential graded $\R$-algebra $(\mathcal{A}^{\hopf}_*,\partial + F)$.
Note that the differential graded $\R$-algebra $(\mathcal{A}^{\hopf}_*,\partial)$ is obtained from $(\mathcal{A}_*(\mathcal{C}), 0)$ by \textit{stabilizations} (See \cite[Definition 3.9]{ens}). Thus, 
\begin{align}\label{stabilization}
\xymatrix{
(\mathcal{A}_*(\mathcal{C}),0) \ar[r]^{i} & (\mathcal{A}^{\hopf}_{*},\partial ) & (\mathcal{A}_*(\mathcal{C}),0) \ar[l]_{\tau} ,
}\end{align}
where $i$ is the inclusion map and $\tau$ is the projection map, are quasi-isomorphisms. For the proof, see \cite[Corollary 3.11]{ens}.

The goal of this section is to prove the next theorem.
\begin{thm}\label{thm-string-hopf}  There exists an isomorphism of graded $\R$-algebras
\[  H_*(\mathcal{A}_*^{\hopf},\partial + F) \cong H^{\str}_*(\R^{2d-1},K_0\cup K_1). \]
\end{thm}
To compute $H^{\str}_*(Q,K_0\cup K_1)$, we fix auxiliary data so that $g$ is the standard Rimmanian metric on $\R^{2d-1}$.
The constant $C_0$ is required to be $C_0 >3$. The other data, $\epsilon_0$ and $\mu$, are not specified.
The proof is divided into three steps.

\textit{Step 1.} 
We first observe de Rham chains in $C^{\dr}_*(\Sigma^{a}_m,\Sigma^0_m)$.
%
We define a map
\[\varphi \colon (K_0\cup K_1)^2 \to \Omega_{K_0\cup K_1}(\R^{2d-1})\]
so that each $(p,p')\in (K_0\cup K_1)^2$ is mapped to a path of a segment
\[\varphi (p,p')\colon [0,1]\to \R^{2d-1}\colon t\mapsto (1-t) p+t p' .\]
We fix two points $p_0\coloneqq (0,1,0) \in K_0$ and $p_1\coloneqq (0,0,0)\in K_1$. Then, we define de Rham chains
\begin{align}\label{def-of-x}
\begin{split}
x^0_{i,j} &\coloneqq [\{(p_i,p_j)\}, \rest{\varphi}{\{(p_i,p_j)\}}, 1]\in C^{\dr}_{0}(\Sigma^{a}_1,\Sigma^0_1)\ (i\neq j), \\
x^1_{i,i} &\coloneqq [\{p_i\}\times K_i , \rest{\varphi }{\{p_i\}\times K_i },1] \in C^{\dr}_{d-1}(\Sigma^{a}_1,\Sigma^0_1), \\
x^1_{i,j} &\coloneqq [\{p_i\}\times K_j , \rest{\varphi }{\{p_i\}\times K_j },1] \in C^{\dr}_{d-1}(\Sigma^{a}_1,\Sigma^0_1)\ (i\neq j), \\
\bar{x}^1_{i,j} &\coloneqq [K_i\times \{p_j\} , \rest{\varphi}{ K_i\times \{p_j\} },1] \in C^{\dr}_{d-1}(\Sigma^{a}_1,\Sigma^0_1)\ (i\neq j), \\
x^2_{i,j}&\coloneqq [K_i\times K_j, \rest{\varphi}{K_i\times K_j}, 1] \in C^{\dr}_{2d-2}(\Sigma^{a}_1,\Sigma^0_1).
\end{split}
\end{align}
Here, $a>3$ and $i,j$ runs over $\{0,1\}$.
Obviously, they are cycle chains for $\partial$. We write the set of these chains by
\[\mathcal{X}\coloneqq  \{x^0_{i,j}\}_{i\neq j} \cup \{x^1_{i,i}\}_{i}\cup \{x^1_{i,j},\bar{x}^1_{i,j}\}_{i\neq j}\cup \{x^2_{i,j}\}_{i,j} . \]
If we define a function $\mathfrak{l}\colon \mathcal{X} \to\R_{>0}$ by
\begin{align*}
&\mathfrak{l}(x^0_{i,j})= \mathfrak{l}(x^1_{i,j}) =\mathfrak{l}(\bar{x}^1_{i,j})=1 \text{ for } i\neq j, \\
&\mathfrak{l}(x^1_{i,i})=\mathfrak{l}(x^2_{i,i})=2, \\
& \mathfrak{l} (x^2_{i,j}) =3 \text{ for } i\neq j , 
\end{align*}
then, each $x\in \mathcal{X}$ satisfies $x\in C^{\dr}_*(\Sigma^{\mathfrak{l}(x)+\epsilon}_1,\Sigma^0_1)$ for any $\epsilon>0$.

For every $a\in \R_{>0}$ and $m\in \Z_{\geq 1}$, let us consider a manifold
\begin{align*}
B^a_m \coloneqq \{(q^0_1,q^1_1,\dots , q^0_m,q^1_m)\in (K_0\cup K_1)^{2m} \mid \sum_{k=1}^m |q^0_k-q^1_k| <a \text{ or }\min_{1\leq k\leq m}  |q^0_k-q^1_k| <\epsilon_0\}. \end{align*}
This is homotopy equivalent to $\Sigma^a_m$ by two  smooth maps
\begin{align*} \pi_m &\colon \Sigma^{a}_m \to B^a_m \colon (\gamma_k\colon [0,T_k]\to \R^{2d-1})_{k=1,\dots ,m}\to (\gamma_1(0),\gamma_1(T_1),\dots ,\gamma_m(0),\gamma_m(T_m)), \\
 i_m &\colon B^a_m \to \Sigma^{a}_m \colon (q^0_1,q^1_1,\dots , q^0_m,q^1_m) \mapsto ( \varphi (q^0_k,q^1_k))_{k=1,\dots ,m} ,\end{align*}
for which $\pi_m\circ i_m=\id_{B^a_m}$ holds and $i_m\circ \pi_m $ is homotopic to $\id_{\Sigma^a_m}$ (c.f. Lemma \ref{lem-diagram-up-to-htpy}).

\begin{notation}In this section,
if $N$ is a submanifold of a manifold $M$, then the inclusion map $N \to M$ is denoted by $\iota_N$.
\end{notation}

\begin{lem}\label{lem-hmgy-class}
Let $M$ be an oriented manifold and $N$ be its open submanifold. Suppose that there exists an approximately smooth function $f\colon M \to \R$ such that $N=f^{-1}((-\infty,0))$. In addition, we assume that $H^{\sing}_*(M,N)$ has a finite dimension. Then there exists an isomorphism between $H^{\sing}_*(M,N)$ and $ H^{\dr}_*(M,N)$ such that for every closed oriented $k$-dimensional submanifold $K$ of $M$, the fundamental class $[K]\in H^{\sing}_k(M,N)$ corresponds to $(-1)^{s(k)}[K,\iota_K ,1]\in  H^{\dr}_k(M,N)$, where $s(k)\coloneqq (k-\dim M)(k-\dim M-1)/2$. 
\end{lem} 
\begin{proof}
We consider the correspondence through the following isomorphisms:
\[  H^{\sing}_*(M,N)\cong H^{\dim M-*}_{c,\dr}(M,N) \cong  H^{\dr}_*(M^{\reg},N^{\reg}) \to  H^{\dr}_*(M,N) .\]
The first isomorphism is defined by the Poincar\'{e} duality. The second isomorphism was given in Example \ref{ex-hmgy-mfd-reg}.
The last isomorphism is induced by $\id_M\colon M^{\reg} \to M$ \cite[Proposition 5.2]{Irie-BV}.
Let us identify the tubular neighborhood $N_K$ of $K$ with the normal bundle of $K$.
Then, $[K]\in  H^{\sing}_k(M,N)$ corresponds to $[\eta]\in H^{\dim M-k}_{c,\dr}(M,N)$, where $\eta\in \Omega_c^{\dim M-k}(M)$ has its support in $N_K$ and represents the Thom class of the normal bundle. Recalling Example \ref{ex-hmgy-mfd-reg}, we can see that this cohomology class corresponds to $(-1)^{s(k)}[M,\id_M,\eta]=(-1)^{s(k)}[N_K,\id_{N_K},\eta] \in H^{\dr}_k(M^{\reg},N^{\reg})$. Let $\pi_{N_K} \colon N_K\to K$ be the bundle projection. Then, as a de Rham chain in $C^{\dr}_k(M,N)$,  $[N_K,\iota_{N_K},\eta]$ is homologous to $[N_K,\iota_K\circ \pi_{N_K},\eta]$ since $\iota_{N_K}\colon N_K\to M$ is homotopic to $\iota_K \circ \pi_{N_K}\colon N_K\to  K\subset M$.
Now the assertion follows since
\[[N_K ,\iota_K\circ \pi_{N_K},\eta]= [K,\iota_K, (\pi_{N_K})_! \eta]=[K,\iota_K,1] \in H^{\dr}_k(M,N). \]
\end{proof}

Let us see that a basis of $H^{\dr}_*(\Sigma^a_m,\Sigma^0_m)$ is given by $\mathcal{X}$ and the $\star$-operation.
First, suppose that $m=1$ and $a>3$. Then $B^a_1= K\times K$. Through the isomorphism $H^{\dr}_*(B^a_1,B^0_1) \cong H^{\sing}_*(B^a_1,B^0_1)$ of Lemma \ref{lem-hmgy-class}, 
$\{ (\pi_1)_* [x]\mid x\in \mathcal{X}\} $ corresponds to the set of singular homology classes
\[ \{ [\{(p_i,p_j)\}]\}_{i\neq j} \cup \{[\{p_i\}\times K_i]\}_{i} \cup \{[\{p_i\}\times K_j]\}_{i\neq j} \cup \{[K_i\times \{p_j\}]\}_{i\neq j} \cup \{[K_i\times K_j]\}_{i,j}, \]
which is a basis of $H^{\sing}_*(K\times K, B^0_1)$. Therefore, $\{[x]\mid x\in \mathcal{X}\}$ is a basis of $H^{\dr}_*(\Sigma^a_1,\Sigma^0_1)$ for $a>3$.

For $a\in (0,3]$, we consider the deformation along the negative gradient vector field of a $C^{\infty}$ function
$E_{i,j}\colon (K_i\cup K_j)^2\to \R\colon (p,p')\to |p-p'|^2$.
If $i=j$, then $\max E_{i,j}=2$, $\min E_{i,j}=0$, and the subset $\{(p,p')\in K_i\times K_i\mid |p-p'|<a\}$ for $a\leq 2$ has $E_{i,i}^{-1}(0)=\{(p,p')\in K_i\times K_i\mid p=p'\}$ (the diagonal) as a deformation retract. If $i\neq j$, then $\max E_{i,j}=3$, $\min E_{i,j}=1$, and the subset $\{(p,p')\in K_i\times K_j\mid |p-p'|<a\}$ for $a\leq 3$ has $E_{i,j}^{-1}(0)=(K_i\times \{p_j\})\cup (\{p_i\}\times K_j)$ (a bouquet) as a deformation retract. From these observations, we can see that
$\{[x] \mid x\in \mathcal{X},\ \mathfrak{l} (x)<a\}$
is a basis of $H^{\dr}_*(\Sigma^a_1,\Sigma^0_1)$.
In general, for any $m\in \Z_{\geq 1}$ and $a\in \R_{>0}$,
\begin{align}\label{basis-Sigma}
 \{ [ x_1\star \dots \star x_m] \mid x_1,\dots,x_m \in \mathcal{X},\ \mathfrak{l}(x_1)+\dots +\mathfrak{l}(x_m) <a\} 
 \end{align}
is a basis of $H^{\dr}_*(\Sigma^a_m,\Sigma^0_m)$.

We fix a trivialization
$h\colon \mathcal{O}_{\epsilon_0} \times (K_0\cup K_1)\to N_{\epsilon_0}$ so that $h(w,p_0) = p_0+(0,w)$ and $h(w,p_1)=p_1 + (-w,0)$ for every $w\in \mathcal{O}_{\epsilon_0}\subset \R^d$. (The orientations of $K_0,K_1$ are chosen so that this map preserves orientations.)

Suppose that $a\in \R_{>0}\setminus \mathcal{L}(K_0\cup K_1)$ ($= \R_{>0}\setminus \Z_{\geq 1}$).
In the following series of three lemmata, we will observe the chains $f_{1,\delta}(x)$ for each $x\in \mathcal{X}$. Hereafter, we assume that $(\epsilon,\delta)\in\mathcal{T}_a$ is standard with respect to $h$ (See subsection \ref{subsub-explicit}).

\begin{lem}\label{lem-f-x-0}For $i\neq j$, we have equations
\[\begin{array}{ll}
f_{1,\delta}(x^0_{i,j})=0, & f_{1,\delta}(x^1_{i,j})=f_{1,\delta}(\bar{x}^1_{i,j})=0 . 
\end{array}
\]
\end{lem}
\begin{proof}
For $(p,p')\in K_i\times K_j$ ($i\neq j$) such that either $p=p_i$ or $p'=p_j$, $\varphi(p,p')$ satisfies the condition (iii) of Lemma \ref{lem-short-length} for $m=k=1$. The equations follow from Lemma \ref{lem-short-length}.
\end{proof}

\begin{lem}\label{lem-f-x-ii}
For $i\in \{0,1\}$, there exist $y^1_{i,i}\in C^{\dr}_{1}(\Sigma^{2+2\epsilon}_2,\Sigma^0_2)$ and $y^2_{i,i}\in C^{\dr}_{d}(\Sigma^{2+2\epsilon}_2,\Sigma^0_2)$ such that
\begin{align}\label{f-x-ii}
\begin{split}
\partial y^1_{0,0} &= f_{1,\delta}(x^1_{0,0}) - x^0_{0,1}\star x^0_{1,0},\\
\partial y^1_{1,1} &= f_{1,\delta} (x^1_{1,1}) - (-1)^d x^0_{1,0}\star x^0_{0,1}, \\
\partial y^2_{0,0} & = f_{1,\delta} (x^2_{0,0} ) - (\bar{x}^1_{0,1}\star x^0_{1,0} +(-1)^{d} x^1_{0,1}\star x^0_{1,0} ) , \\
\partial y^2_{1,1} & = f_{1,\delta} (x^2_{1,1} )- ((-1)^d\bar{x}^1_{1,0}\star x^0_{0,1} +  x^1_{1,0}\star x^0_{0,1} ) , 
\end{split}
\end{align}
and 
\begin{align}\label{f-y-ii}
( f_{1,\delta} +(-1)^d f_{2,\delta} ) (y^1_{i,i})=0,\ ( f_{1,\delta} +(-1)^d f_{2,\delta} ) (y^2_{i,i})=0 .
\end{align}
\end{lem}
\begin{proof}
We only show the existence of $y^2_{0,0}$ and $y^1_{0,0}$. Replacing $(K_0,p_0)$ by $(K_1,p_1)$, $y^2_{1,1}$ and $y^1_{1,1}$ are constructed in a parallel way, except the difference of signs.

Since $(\epsilon,\delta)$ is standard, $\delta$ has the form (\ref{delta-standard}). Using the notations of (\ref{explicit-rep}), we can write $f_{1,\delta}(x^2_{0,0})=  [W_1,\Phi_1,\zeta_1] $, where
\begin{align*}
W_1 &=  \{ ((p,p'),\tau,v)\in (K_0\times K_0) \times \R\times N_{\epsilon} \mid 2\epsilon<\tau<1-2\epsilon,\ (1-\tau)p+\tau p' =v\}, \\
\Phi_1 &\colon W_1\to \Sigma^{2+2\epsilon}_{2} \colon ((p,p'),\tau,v)\mapsto \con_1(\varphi(p,p'), (1,\tau), \psi_{\epsilon}(v)), \\
\zeta_1&\in \Omega^{d}_c(W_1)\colon (\zeta_1)_{((p,p'),\tau,v)}=\rho_{\epsilon}(1,\tau)\cdot (h_*(\nu_{\epsilon}\times 1 ))_v .
\end{align*}
$N_{\epsilon}$ is a disjoint union of $h(\mathcal{O}_{\epsilon}\times K_0)$ and $h(\mathcal{O}_{\epsilon}\times K_1)$. Corresponding to this division, we define $W'_i\coloneqq \{((p,p'),\tau,v)\in W_1\mid v\in h(\mathcal{O}_{\epsilon}\times K_i)\}$ for $i=0,1$. 
If $ ((p,p'),\tau,v) \in W'_0$, $\tau$ satisfies the condition (ii) of Lemma \ref{lem-short-length}, so $[W'_0,\Phi_1,\zeta_1]=0\in C^{\dr}_{d-1}(\Sigma^{2+2\epsilon}_2,\Sigma^0_2)$.
For $((p,p'),\tau,v ) \in W'_1$, $\rho_{\epsilon}(1,\tau)=1$ holds. Moreover, there is an diffeomorphism
\[ I \colon \mathcal{O}_{\epsilon}\times K_0 \to W'_1 \colon (w,p)\mapsto ( (p,p'_{(w,p)}) ,|p-h(w,p_1)|, h(w,p_1)). \]
Here $p'_{(w,p)}\in K_0\setminus \{p\}$ is determined by $h(w,p_1) \in  \Im (\varphi(p,p'_{(w,p)}))$, as described in Figure \ref{figure-lemma1}. The diffeomorphism $I$ preserves orientations and $I^*\zeta_1= \nu_{\epsilon} \times 1$ holds.

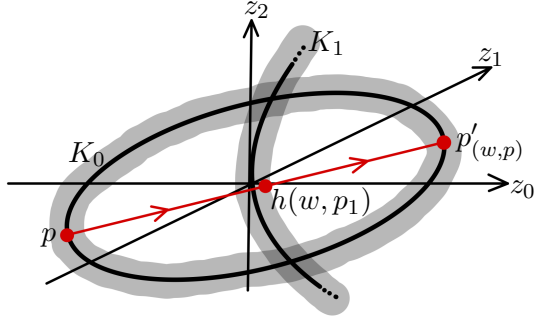
\begin{figure}
\centering
\begin{overpic}[width=7cm]{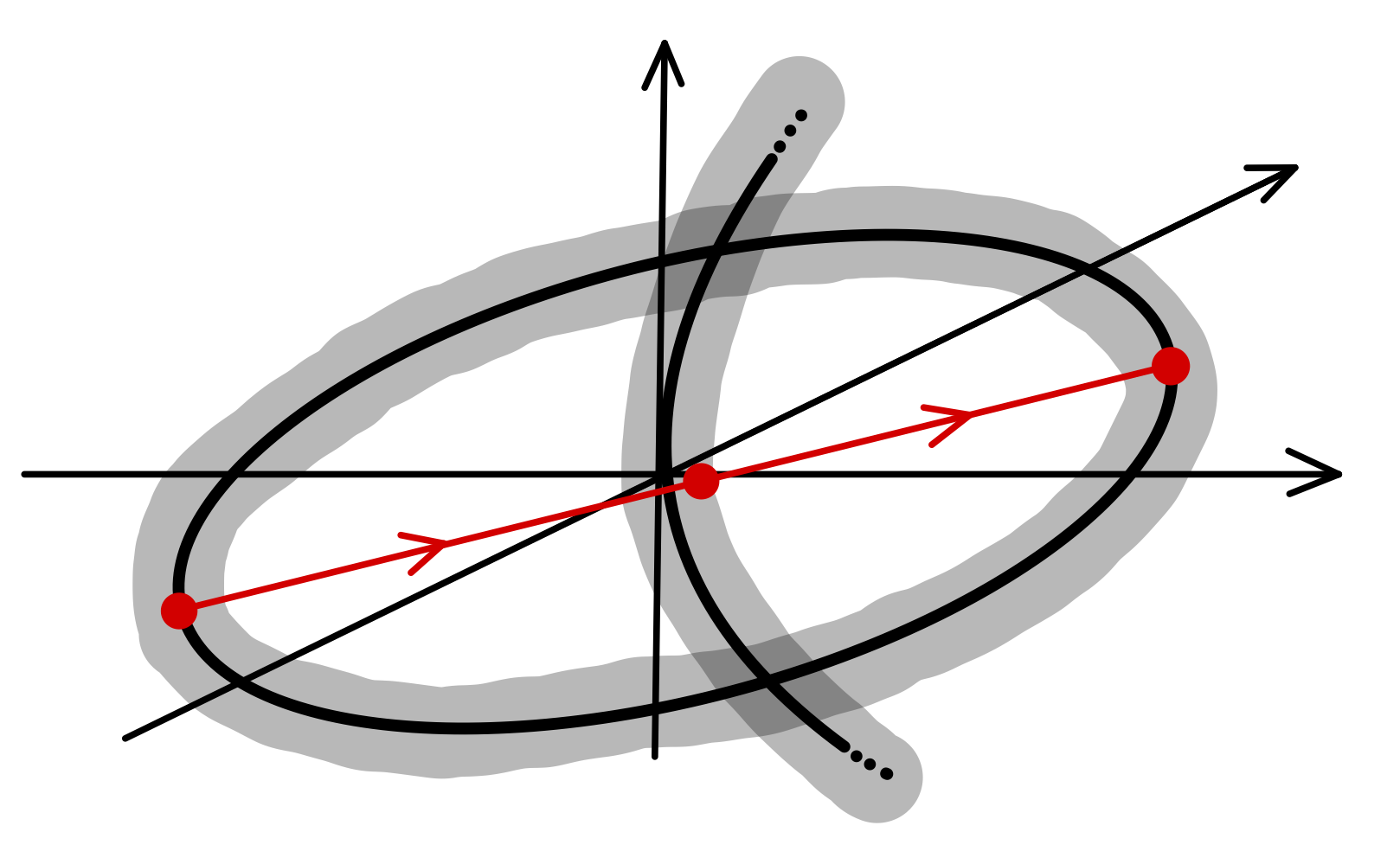}
\put(96,26){$z_0$}
\put(90,51){$z_1$}
\put(46,60){$z_2$}
\put(8,17){$p$}
\put(51,23){$h(w,p_1)$}
\put(86,35){$p'_{(w,p)}$}
\put(13,32){$K_0$}
\put(58,53){$K_1$}
\end{overpic}
\caption{The red path is $\varphi(p,p'_{(w,p)} )$. The gray region is the tubular neighborhood  $N_{\epsilon}$ of $K_0\cup K_1$.}\label{figure-lemma1}
\end{figure}

Likewise, we consider an explicit description of $f_{1,\delta}(x^1_{0,0})$. Then, we have 
\begin{align*}
f_{1,\delta}(x^2_{0,0})&=[\mathcal{O}_{\epsilon}\times K_0, \Phi_1\circ I,  \nu_{\epsilon} \times 1] \in C^{\dr}_{d-1}(\Sigma^{2+2\epsilon}_2,\Sigma^0_2) ,\\
f_{1,\delta}(x^1_{0,0})&=[ \mathcal{O}_{\epsilon}\times \{p_0\},\rest{\Phi_1\circ I}{ \mathcal{O}_{\epsilon}\times \{p_0\}},  \nu_{\epsilon} \times 1 ]\in C^{\dr}_{0}(\Sigma^{2+2\epsilon}_2,\Sigma^0_2) .
\end{align*}
Let us define
$\tilde{ \Phi}_1\colon \R\times (\mathcal{O}_{\epsilon}\times K_0) \to \Sigma^{2+2\epsilon}_2$ as follows: Choose a $C^{\infty}$ function $\kappa\colon \R\to [0,1]$ such that $\kappa(s)=\begin{cases} 0 & \text{ if }s\leq \frac{1}{2}, \\ 1 & \text{ if }s\geq 1. \end{cases}$ 
For $s\geq \frac{1}{2}$, we define $\tilde{\Phi}_1(s,(w,p)\coloneqq \Phi_1\circ I(\kappa(s)\cdot w, p)$. Then the first path (resp. the second path) of $\tilde{\Phi}_1(\frac{1}{2},(w,p))$ is equal to $\varphi(p,p_1)$ (resp. $\varphi(p_1,-p)$) up to   an reparameterization, so we define $\tilde{\Phi}(s,(w,p))$ for $s\leq \frac{1}{2}$ by interpolating the parameterizations so that $\tilde{\Phi}(s,(w,p))=(\varphi(p,p_1), \varphi(p_1,-p))$ for $s\leq 0$. 
Now we define
\begin{align*}
\tilde{y}^2_{0,0} &\coloneqq  (-1)^d [\R\times ( \mathcal{O}_{\epsilon}\times K_0 ), \tilde{\Phi}_1,\chi\times ( \nu_{\epsilon} \times 1) ] \in C^{\dr}_d(\Sigma^{2+2\epsilon}_2,\Sigma^0_2), \\
\tilde{y}^2_{0,0} &\coloneqq (-1)^d [\R\times (\mathcal{O}_{\epsilon}\times \{p_0\} ), \rest{ \tilde{\Phi}_1}{\R\times (\mathcal{O}_{\epsilon}\times \{ p_0\} )},\chi\times ( \nu_{\epsilon} \times 1 ) ] \in C^{\dr}_1(\Sigma^{2+2\epsilon}_2,\Sigma^0_2) , 
\end{align*}
where $\chi \colon \R\to [0,1]$ is a $C^{\infty}$ function with a compact support such that $\chi \equiv 1$ on $[0,1]$.

From the constructions of $\tilde{y}^2_{0,0}$ and $\tilde{y}^1_{0,0}$, we can compute their boundary chains as follows:
Let us introduce two maps $\varphi_0\colon K_0\to K_0\times K_0 \colon p\mapsto (p,-p)$ and $\tilde{i}_2\colon K_0\times K_0 \to \Sigma^{2+2\epsilon}_2\colon (p,p')\mapsto (\varphi(p,p_1), \varphi(p_1,p'))$. Then, we have
\begin{align*}
\partial \tilde{y}^2_{0,0} &=f_{1,\delta}(x^2_{0,0})- (\tilde{i}_2)_*[ \mathcal{O}_{\epsilon}\times K_0, \varphi_0\circ \pr_{K_0}, \nu_{\epsilon}] \\ 
&=f_{1,\delta}(x^2_{0,0})- (\tilde{i}_2)_*[\varphi_0(K_0) , \iota_{ \varphi_0(K_0)} ,1 ], \\
\partial \tilde{y}^1_{0,0} &=f_{1,\delta}(x^1_{0,0})-  (\tilde{i}_2)_*[ \mathcal{O}_{\epsilon}\times \{p_0\}, \varphi_0\circ \pr_{\{p_0\}}, \nu_{\epsilon}\times 1] \\
&=f_{1,\delta}(x^1_{0,0})-  (\tilde{i}_2)_*[\{\varphi_0(p_0)\} ,\iota_{\{\varphi_0(p_0)\}} , 1] .
\end{align*}
Here we used the condition that $\int_{\mathcal{O}_{\epsilon}}\nu_{\epsilon}=1$. Moreover, we can check from the definition that $\tilde{\Phi}_1(s,(w,p))$ satisfies the condition (iii) of Lemma \ref{lem-short-length} for $m=2$ and $k=1,2$. Therefore, $f_{k,\delta}(\tilde{y}^2_{0,0})=0$ and $f_{k,\delta}(\tilde{y}^1_{0,0})=0$ hold for $k=1,2$.

As homology classes in $H^{\sing}_*(K_0\times K_0)$,
\begin{align*}
[\varphi_0(K_0)] &=[K_0\times \{p_0\}] + (-1)^{d}[\{p_0\}\times K_0]\in H^{\sing}_{d-1}(K_0\times K_0), \\
[\{\varphi_0(p_0)\}] &=[\{(p_0,p_0)\}]\in H^{\sing}_0(K_0\times K_0).
\end{align*} Therefore, by Lemma \ref{lem-hmgy-class}, there exist $z^2_{0,0}\in C^{\dr}_{d}(K_0\times K_0)$ and $z^1_{0,0}\in C^{\dr}_{1}(K_0\times K_0)$ such that
\begin{align*}
\partial z^2_{0,0}&=[\varphi_0(K_0), \iota_{ \varphi_0(K_0) } ,1] -\left( [K_0\times \{p_0\}, \iota_{K_0\times \{p_0\}} ,1] + (-1)^d [\{p_0\}\times K_0, \iota_{\{p_0\}\times K_0} ,1] \right) , \\
\partial z^1_{0,0}&= [\{\varphi_0(p_0)\} , \iota_{ \{\varphi_0(p_0)\} } , 1] -[\{(p_0,p_0)\} , \iota_{ \{(p_0,p_0)\} },1] .
\end{align*}
It is clear from the definition of each $x\in \mathcal{X}$ that
\begin{align*}
&(\tilde{i}_2)_*[K_0\times \{p_0\}, \iota_{K_0\times \{p_0\}} ,1] = \bar{x}^1_{0,1}\star x^0_{1,0}  , \ (\tilde{i}_2)_* [\{p_0\}\times K_0, \iota_{\{p_0\}\times K_0} ,1]  = x^1_{0,1}\star x^0_{1,0} ,\\
& (\tilde{i}_2)_* [\{(p_0,p_0)\} , \iota_{ \{(p_0,p_0)\} },1]  = x^0_{0,1}\star x^0_{1,0} .
\end{align*}
Therefore, $y^2_{0,0}\coloneqq \tilde{y}^2_{0,0}+ (\tilde{i}_2)_*z^2_{0,0}\in C^{\dr}_{d}(\Sigma^{2+2\epsilon}_2,\Sigma^0_2)$ and $y^1_{0,0}\coloneqq \tilde{y}^1_{0,0}+  (\tilde{i}_2)_*z^1_{0,0} \in C^{\dr}_{0}(\Sigma^{2+2\epsilon}_2,\Sigma^0_2)$ satisfy the first and the second equations of (\ref{f-x-ii}). Moreover, any paths in the image of $\tilde{i}_2$ satisfies the condition (iii) of Lemma \ref{lem-short-length} for $m=2$ and $k=1,2$. Therefore, $f_{k,\delta}( (\tilde{i}_2)_*z^2_{0,0})=0$ and $f_{k,\delta}( (\tilde{i}_2)_*z^1_{0,0})=0$ hold for $k=1,2$, and thus $y^2_{0,0}$ and $y^1_{0,0}$ satisfy (\ref{f-y-ii}).
\end{proof}

\begin{lem}\label{lem-f-x-ij}
There exist $y^2_{0,1}, y^2_{1,0}\in C^{\dr}_{2d-1}(\Sigma^{3+2\epsilon}_{2},\Sigma^0_2)$ such that 
\begin{align}\label{f-x-ij}
\partial y^2_{0,1} = f_{1,\delta} (x^2_{0,1})  , \ 
\partial y^2_{1,0} =  f_{1,\delta} (x^2_{1,0})  ,
\end{align}
and
\begin{align}\label{f-y-ij}
 \ ( f_{1,\delta} + (-1)^d f_{2,\delta} ) (y^2_{0,1}) =0,\  ( f_{1,\delta} + (-1)^d f_{2,\delta} ) (y^2_{1,0}) =0.
 \end{align}
\end{lem}
\begin{proof}
We only show the existence of $y^2_{0,1}$. Exchanging $K_0$ and $K_1$, $y^2_{1,0}$ is constructed in a parallel way.

Since $(\epsilon,\delta)$ is standard, $\delta$ has the form (\ref{delta-standard}).
Using the notations of (\ref{explicit-rep}), we can write $f_{1,\delta}(x^2_{0,1})=[W_1,\Phi_1,\zeta_1]$, where
\begin{align*}
W_1 &= \{ ( (p,p'), \tau,v) \in (K_0\times K_1)\times \R\times N_{\epsilon}\mid 2\epsilon<\tau <1-2\epsilon, (1-\tau)p+\tau p' =v\}, \\
\Phi_1 &\colon W_1\to \Sigma^{3+2\epsilon}_{2} \colon ((p,p'),\tau,v)\mapsto \con_1(\varphi(p,p'), (1,\tau), \psi_{\epsilon}(v)), \\
\zeta_1&\in \Omega^{d}_c(W_1)\colon (\zeta_1)_{((p,p'),\tau,v)}=\rho_{\epsilon}(1,\tau)\cdot (h_*(\nu_{\epsilon}\times 1))_v .
\end{align*}
If $p\in K_0$ is sufficiently close to $p_0$, then $\varphi(p,p')$ satisfy the condition (iii) of Lemma \ref{lem-short-length} for any $p'\in K_1$. Symmetrically, if $p'\in K_1$ is sufficiently close to $p_1$, then $\varphi(p,p')$ satisfy the same condition for any $p\in K_0$. See Figure \ref{figure-lemma2}. Therefore, For any bump function $b\colon K_0\times K_1 \to \R$ whose support is localized near $(K_0\times \{p_1\})\cup (\{p_0\}\times K_1) $, $[W_1,\Phi_1,\zeta'_1]=0$ holds for
\[\zeta'_1\in \Omega^{d}_c(W_1) \colon (\zeta'_1)_{((p,p'),\tau,v)}=b(p,p')\cdot \rho_{\epsilon}(1,\tau)\cdot (h_*(\nu_{\epsilon}\times 1))_v .\]
We remark that $[W_1,\Phi_1,\zeta'_1]$ is an explicit description (\ref{explicit-rep}) of $f_{1,\delta} ([K_0\times K_1,\rest{\varphi}{K_0\times K_1}, b])$.

\begin{figure}
\centering
\begin{overpic}[width=7cm]{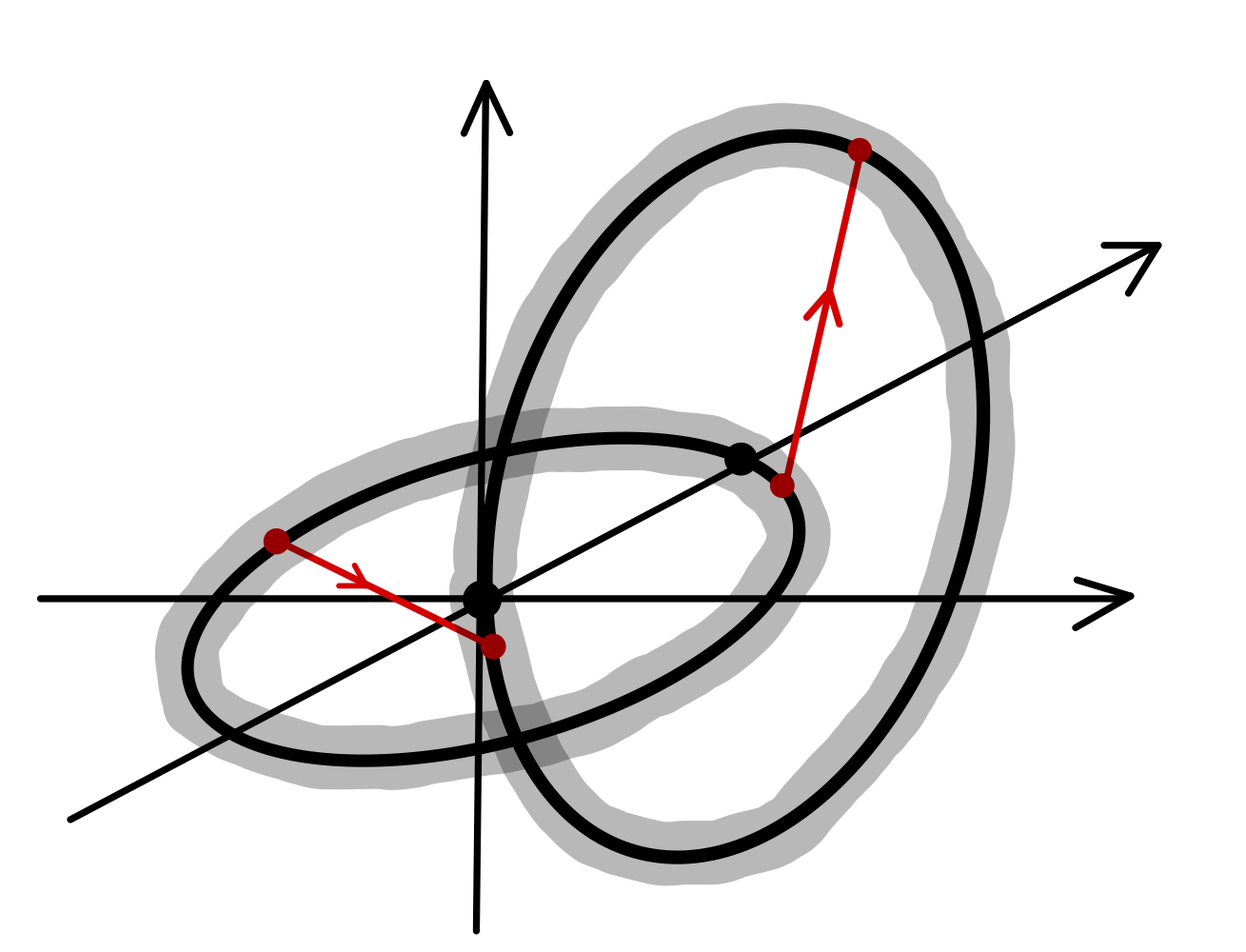}
\put(92,28){$z_0$}
\put(88,61){$z_1$}
\put(36,72){$z_2$}
\put(41,23){$p'$}
\put(19,35){$p$}
\put(65,38){$p$}
\put(70,70){$p'$}
\put(20,11){$K_0$}
\put(69,11){$K_1$}
\end{overpic}
\caption{The red path is $\varphi(p,p')$ when $p$ is close to $p_0$ or $p'$ is close to $p_1$. The gray region is the tubular neighborhood $N_{\epsilon}$ of $K_0\cup K_1$}\label{figure-lemma2}
\end{figure}
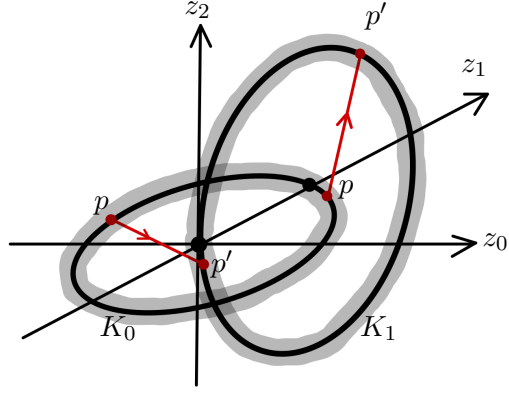

Now we choose $b$ so that it is constant to $1$ on a neighborhood of $(K_0\times \{p_1\})\cup (\{p_0\}\times K_1) $. Then, the above computation shows that
\begin{align*}
f_{1,\delta}(x^2_{0,1}) &= f_{1,\delta} ([K_0\times K_1,\rest{\varphi}{K_0\times K_1}, 1-b]) +f_{1,\delta} ([K_0\times K_1,\rest{\varphi}{K_0\times K_1}, b]) \\
&=  f_{1,\delta} ([ K'_0\times K'_1 ,\rest{\varphi}{K'_0\times K'_1}, 1-\rest{b}{K'_0\times K'_1}]),
\end{align*}
where $K'_i$ ($i=0,1$) is the complement of a small closed ball containing $p_i$. Since $K'_0\times K'_1$ is contractible, there exists a map $R\colon \R\times K'_0\times K'_1 \to K'_0\times K'_1$ such that $R(s,\cdot)=\id_{K'_0\times K'_1}$ for $s\geq 1$ and $R(s,\cdot)$ is constant to some point in $K'_0\times K'_1$ for $s\leq 0$.
Using the function $\chi$ in the proof of Lemma \ref{lem-f-x-ii}, let us define a chain
\[\tilde{x}^2_{0,1} \coloneqq [ \R\times (K'_0\times K'_1), \varphi\circ R, \chi\times  (1-\rest{b}{K'_0\times K'_1} ) ] \in C^{\dr}_{2d-1}(\Sigma^{3+2\epsilon}_{2},\Sigma^0_2) . \]
This chain satisfies
\[\partial \tilde{x}^2_{0,1} = [ K'_0\times K'_1 ,\rest{\varphi}{K'_0\times K'_1}, 1-\rest{b}{K'_0\times K'_1}] +  [ \R\times (K'_0\times K'_1), \varphi\circ R, \chi\times  \rest{db}{K'_0\times K'_1} ) ]  .\] 
Note that $[ K'_0\times K'_1 ,\varphi\circ R(0,\cdot), 1-\rest{b}{K'_0\times K'_1}] =0$ since $\varphi\circ R(0,\cdot)$ is constant.
The second chain of the RHS is mapped by $f_{1,\delta}$ to $0$ since the support of $db$ is localized near $(K_0\times \{p_1\})\cup (\{p_0\}\times K_1) $.

Now we define $y^2_{0,1} \coloneqq f_{1,\delta}(\tilde{x}^2_{0,1})\in C^{\dr}_{2d-1}(\Sigma^{3+2\epsilon}_{2},\Sigma^0_2) $. Then, the first equation of (\ref{f-x-ij}) holds since $\partial y^2_{0,1}= f_{1,\delta}(\partial \tilde{x}^2_{0,1})= f_{1,\delta}(x^2_{0,1})$. The first equation (\ref{f-y-ij}) follows from $(f_{1,\delta}+ (-1)^d f_{2,\delta})\circ f_{1,\delta}=0$ (See the proof of Proposition \ref{prop-D-2}).
\end{proof}

\textit{Step 2.}  
We define a function $\mathfrak{l}\colon \mathcal{C}\cup \mathcal{D}\cup \mathcal{E} \to\R_{> 0}$ by
\begin{align*}
& \mathfrak{l}(c^0_{i,j})= \mathfrak{l}(c^1_{i,j}) =\mathfrak{l}(\bar{c}^1_{i,j})=1 \text{ for } i\neq j ,\ \mathfrak{l}( c^1_{i,i} ) =2, \\
& \mathfrak{l} (c^2_{i,j})=\mathfrak{l}(d^2_{i,j})=\mathfrak{l} (e^2_{i,j}) = \begin{cases} 2 & \text{ if }i=j, \\ 3 & \text{ if }i \neq j . \end{cases}
\end{align*}
For every $a\in \R_{>0}\setminus \mathcal{L}(K_0\cup K_1)$, let $\mathcal{A}_*^{<a}$ be an $\R$-subspace of $\mathcal{A}^{\hopf}_*$ spanned by words of elements in $\mathcal{C}\cup \mathcal{D}\cup \mathcal{E} $ such that the sum of the values of $\mathfrak{l}$ is less than $a$. Then, $(\partial + F) (\mathcal{A}^{<a}_*) \subset \mathcal{A}^{<a}_*$, so we get a subcomplex $(\mathcal{A}^{<a}_*,\partial+F)$.

We continue to use $(\epsilon,\delta)\in \mathcal{T}_a$ which is standard with respect to $h$.
The second step is to construct a chain map from $(\mathcal{A}^{<a}_*,\partial+F)$ to $(C^{<a}_*(\epsilon),D_{\delta})$, and prove that it is a quasi-isomorphism.

Let $y^1_{i,i}$ and $y^2_{i,j}$ be the chains of Lemma \ref{lem-f-x-ii} and \ref{lem-f-x-ij}, which depends on $(\epsilon,\delta)$. We define a linear map $\Phi^{<a}_{(\epsilon,\delta)}\colon \mathcal{A}_*^{<a}\to C^{<a}_*(\epsilon)$ so that
\begin{align}\label{Phi-hopf}
\begin{split}
& \Phi^{<a}_{(\epsilon,\delta)} (c^0_{i,j})\coloneqq x^0_{i,j},\ \Phi^{<a}_{(\epsilon,\delta)} (c^1_{i,j})\coloneqq x^0_{i,j},\ \Phi^{<a}_{(\epsilon,\delta)} (\bar{c}^1_{i,j})\coloneqq \bar{x}^1_{i,j}\ \text{ for }i\neq j, \\
&  \Phi^{<a}_{(\epsilon,\delta)} (c^1_{i,i})\coloneqq x^1_{i,i},\ \ \Phi^{<a}_{(\epsilon,\delta)} (c^2_{i,j})\coloneqq x^2_{i,j}, \\
&\Phi^{<a}_{(\epsilon,\delta)} (d^1_{i,i}) \coloneqq   y^1_{i,i},\ \Phi^{<a}_{(\epsilon,\delta)} (d^2_{i,j}) \coloneqq  y^2_{i,j}, \\
& \Phi^{<a}_{(\epsilon,\delta)} (e^1_{i,i}) \coloneqq \partial y^1_{i,i},\ \Phi^{<a}_{(\epsilon,\delta)} (e^2_{i,j}) \coloneqq \partial y^1_{i,j},
\end{split}
\end{align}
and extend them naturally by the product map on $\mathcal{A}^{\hopf}_*$ and the $\star$-operation.
\begin{prop}\label{prop-chain-map-hopf}
 $\Phi^{<a}_{(\epsilon,\delta)}$ is a chain map from $(\mathcal{A}^{<a}_*,\partial+F)$ to $(C^{<a}_*(\epsilon),D_{\delta})$. 
 \end{prop}
 \begin{proof} This follows immediately from the series of three lemmata in \textit{Step 1}. For each $\xi \in \mathcal{C}\cup \mathcal{D}\cup \mathcal{E}$, $D_{\delta}\circ \Phi^{<a}_{(\epsilon,\delta)}(\xi) = (\partial +F)\circ \Phi^{<a}_{(\epsilon,\delta)}(\xi)$ is proved by:
\begin{itemize}
\item   Lemma \ref{lem-f-x-0} if $\xi=c^0_{i,j}, c^1_{i,j},\bar{c}^1_{i,j}$ ($i\neq j$).
\item  the equations  (\ref{f-x-ii}) if  $\xi=c^1_{i,i},c^2_{i,i}$. 
\item  the equations (\ref{f-x-ij}) if $\xi=c^2_{i,j}$ ($i\neq j$). 
\item the equations (\ref{f-y-ii}) if $\xi =d^1_{i,i},d^2_{i,i},e^1_{i,i},e^2_{i,i}$. 
\item  the equations  (\ref{f-y-ij}) if $\xi=d^2_{i,j},e^2_{i,j}$ ($i\neq j$). 
\end{itemize}
\end{proof}
 Therefore, we have a linear map on the homology groups
\[ (\Phi^{<a}_{(\epsilon,\delta)})_* \colon H_*(\mathcal{A}^{<a}_*,\partial +F)\to H^{<a}_*(\epsilon,\delta).\]

\begin{prop}\label{prop-Phi-hopf-isom}
$(\Phi^{<a}_{(\epsilon,\delta)})_*$ is an isomorphism.
\end{prop}

\begin{proof}
We introduce a function $\mathfrak{m} \colon \mathcal{C}\cup \mathcal{D}\cup \mathcal{E} \to \Z_{\geq 1}$ so that $\mathfrak{m} (\mathcal{C})=\{1\}$ and $\mathfrak{m} (\mathcal{D})=\mathfrak{m}(\mathcal{E})=\{2\}$.
For every $m\in \Z_{\geq 0}$, let $\mathcal{A}_*^{<a}(m)$ be an $\R$-subspace of $\mathcal{A}^{<a}_* $ generated by words of elements of $\mathcal{C}\cup \mathcal{D}\cup \mathcal{E}$ such that the sum of the values of $\mathfrak{m}$ is equal to $m$. (When $m=0$, $\mathcal{A}^{<a}_*(0) \coloneqq \R\cdot 1$.)
Then, the chain complex $(\mathcal{A}^{<a}_*,\partial +F )$ is filtered by subcomplexes $\{\mathcal{G}^{<a}_p\}_{p\in \Z}$ defined by
$\mathcal{G}^{<a}_p\coloneqq \bigoplus_{m\geq -p}\mathcal{A}^{<a}_*(m)$. Let us consider the spectral sequence determined by this filtration. The $(-m,q)$-term ($m\geq 0$) of its first page is given by
\[ H_{q-m} ( \mathcal{A}^{<a}_*(m), \partial ) \cong  H_{q-m}(\mathcal{A}^{<a}_*(m,\mathcal{C}),0)=\mathcal{A}^{<a}_{q-m}(m,\mathcal{C}) .\]
Here, $\mathcal{A}^{<a}_*(m,\mathcal{C}) \coloneqq \mathcal{A}^{<a}_*(m) \cap \mathcal{A}_*(\mathcal{C})$ and the first isomorphism is induced by restricting the quasi-isomorphisms (\ref{stabilization}).
$\Phi^{<a}_{(\epsilon,\delta)}$ preserves the filtrations $\{\mathcal{G}^{<a}_p\}_{p\in \Z}$ and $\{\mathcal{F}^{<a}_{\epsilon,p}\}_{p\in \Z}$. The induced map on the ($-m$)-th column ($m\geq 0$) of the first page has the form
\[(\Phi^{<a}_{(\epsilon,\delta)})_*\colon \mathcal{A}^{<a}_{*-m}(m, \mathcal{C}) \to H^{\dr}_{*-m(d-1)}(\Sigma^{a+m\epsilon}_{m},\Sigma^0_{m}).\]
This map is an isomorphism since the basis $\{c_1\cdots c_{m}\mid c_1,\dots ,c_{m}\in \mathcal{C} ,\ \mathfrak{l}(c_1)+\dots + \mathfrak{l}(c_{m}) <a \}$ of $\mathcal{A}^{<a}_{*-m}(m, \mathcal{C})$ is mapped to the basis (\ref{basis-Sigma}). In the ($-m$)-th column for $m<0$, $(\Phi^{<a}_{(\epsilon,\delta)})_*$ is a map between the zero vector space. The proposition now follows from Lemma \ref{lem-spectral}.
\end{proof}

\textit{Step 3.} The last step is to show that the family of maps
\[ \{ (\Phi^{<a}_{(\epsilon,\delta)})_*\mid a\in \R_{>0}\setminus \mathcal{L}(K),\ (\epsilon,\delta)\in \mathcal{T}_a\text{ is standard with respect to }h\}  \]
induces an isomorphism on the limit of $\epsilon \to 0$ and $a\to \infty$.

\begin{lem} For $(\epsilon,\delta),(\epsilon',\delta')\in \mathcal{T}_a$ ($\epsilon'\leq \epsilon$) which are standard with respect to $h$,
\[k_{(\epsilon',\delta'),(\epsilon,\delta)} \circ (\Phi^{<a}_{(\epsilon',\delta')})_* = (\Phi^{<a}_{(\epsilon,\delta)})_*.\]
\end{lem}
\begin{proof}
We have defined $\Phi^{<a}_{(\epsilon,\delta)}$ by the chains $\{y^1_{i,i},y^2_{i,j}\}$ which depends on $(\epsilon,\delta)$. As a notation, let $\{(y^1_{i,i})',(y^2_{i,j})'\}$ be the corresponding chains constructed from $(\epsilon',\delta')$, by which $\Phi^{<a}_{(\epsilon',\delta')}$ is defined.

We take $(\bar{\epsilon},\bar{\delta})\in \bar{\mathcal{T}}_a$ satisfying \ref{condition-delta-bar} for $(\epsilon,\delta),(\epsilon',\delta')$. We may assume that it is standard with respect to $h$, and thus $\bar{\epsilon}=\epsilon$.
As in Lemma \ref{lem-f-x-0}, $\bar{f}_{1,\bar{\delta}} (\bar{i}(x))=0$ holds for $x=x^0_{i,j}, x^1_{i,j},\bar{x}^1_{i,j}$ ($i\neq j$), where $\bar{i}$ is the map (\ref{bar-i}).
We claim that there exist $[-1,1]$-modeled chains $\overline{y^1_{i,i}}\in \bar{C}^{\dr}_{1}(\Sigma^{2+2\epsilon}_2,\Sigma^0_2)$, $\overline{y^2_{i,i}}\in \bar{C}^{\dr}_{d}(\Sigma^{2+2\epsilon}_2,\Sigma^0_2)$ and $\overline{y^2_{i,j}}\in \bar{C}^{\dr}_{d}(\Sigma^{3+2\epsilon}_2,\Sigma^0_2)$ ($i\neq j$) which satisfy the following equations:
\begin{itemize}
\item the variants of the equations (\ref{f-x-ii}), (\ref{f-y-ii}), (\ref{f-x-ij}) and (\ref{f-y-ij}) determined by replacing $\{ y^1_{i,i} , y^2_{i,j}, f_{k,\delta}, \star\}$  by $\{\overline{y^1_{i,i}},  \overline{y^2_{i,j}}, \bar{f}_{k,\bar{\delta}}, \bar{\star}\}$, and $x\in \mathcal{X}$ by $\bar{i}(x)$.
\item $e_+ \overline{y^1_{i,i}} =y^1_{i,i} $, $e_+\overline{y^2_{i,j}}= y^2_{i,j}$, $e_-\overline{y^1_{i,i}} = (j_{\epsilon',\epsilon})_* (y^1_{i,i})'$, and $e_-\overline{y^2_{i,j}} =  (j_{\epsilon',\epsilon})_* ( y^2_{i,j})'$.
\end{itemize}
This claim is proved by rewriting the proof of Lemma \ref{lem-f-x-ii} and \ref{lem-f-x-ij} for $[-1,1]$-modeled chains. We omit the proof.

We define a linear map $\bar{\Phi}^{<a}_{\epsilon}\colon \mathcal{A}^{<a}_* \to \bar{C}^{<a}_*(\epsilon)$ as in (\ref{Phi-hopf}) by replacing $x\in \mathcal{X}$ by $\bar{i}(x)$ and $\{y^1_{i,i},y^2_{i,j}\}$ by $\{\overline{y^1_{i,i}}, \overline{y^2_{i,j}}\}$, and extend naturally by the product on $\mathcal{A}^{\hopf}_*$ and the $\bar{\star}$-operation. The former equations about $\overline{y^1_{i,i}} ,\overline{y^2_{i,j}}$ ensures that $\bar{\Phi}^{<a}_{\epsilon}$ is a chain map from $( \mathcal{A}^{<a}_*,\partial +F)$ to $(\bar{C}^{<a}_*(\epsilon),\bar{D}_{\bar{\delta}})$, as in Proposition \ref{prop-chain-map-hopf}. The latter equations about $\overline{y^1_{i,i}} ,\overline{y^2_{i,j}}$ ensures the commutativity of the following diagram:
\[\xymatrix@C=36pt{
 & & H^{<a}_*(\epsilon,\delta) \\
H_*(\mathcal{A}^{<a}_*,\partial +F ) \ar[r]^-{ (\bar{\Phi}^{a}_{\epsilon})_* } \ar@/^15pt/[rru]^{(\Phi^a_{(\epsilon,\delta)})_*} \ar@/_15pt/[rrd]^{(\Phi^{a}_{(\epsilon',\delta')})_*}& \bar{H}^{<a}_*(\epsilon,\bar{\delta})\ar[ru]^{f_{(\epsilon,\bar{\delta}),+}} \ar[rd]^{f_{(\epsilon,\bar{\delta}),-}}& \\
& & H^{<a}_*(\epsilon',\delta')\ar[uu]_{k_{(\epsilon',\delta'), (\epsilon,\delta)}}.
}\]
This proves the lemma. \end{proof}
Therefore, the family of maps $\{ (\Phi^{<a}_{(\epsilon,\delta)})_*\mid  (\epsilon,\delta)\in \mathcal{T}_a\text{ is standard with respect to }h\}$ induces an isomorphism on the limit of $\epsilon \to 0$
\[\Phi^{<a}_*\colon H_*(\mathcal{A}^{<a}_*,\partial +F) \to H^{<a}_*(\R^{2d-1},K_0\cup K_1)\]
for every $a\in \R_{>0}\setminus \mathcal{L}(K_0\cup K_1)$. Moreover, $\{\Phi^{<a}_*\}_{a\in \R_{>0}\setminus \mathcal{L}(K_0\cup K_1)}$ is naturally compatible with $\{I^{a,b}\}_{a\leq b}$ and the family of linear maps
\[\{ H_*(\mathcal{A}^{<a}_*,\partial +F) \to  H_*(\mathcal{A}^{<b}_*,\partial +F) \}_{a\leq b}\]
which is induced by the inclusion map $\mathcal{A}^{<a}_*\to \mathcal{A}^{<b}_*$. Therefore, on the limit of $a\to \infty$, we obtain an isomorphism
\[H_*(\mathcal{A}^{\hopf}_*,\partial +F) = \varinjlim_{a\to \infty} H_*(\mathcal{A}^{<a}_*,\partial +F) \to H^{\str}_*(\R^{2d-1},K_0\cup K_1) .\]
This finishes the proof of Theorem \ref{thm-string-hopf}.

\subsection{Computation of $H^{\str}_*(\R^{2d-1},K_0\cup K_2)$}

We define $\mathcal{A}^{\mathrm{unlink}}_*\coloneqq \mathcal{A}_*(\mathcal{C'})$ by the set
\[ \mathcal{C}'\coloneqq \{c^0_{i,j}\}_{i\neq j} \cup \{c^1_{i,i}\}_{i}\cup \{c^1_{i,j},\bar{c}^1_{i,j}\}_{i\neq j}\cup \{c^2_{i,j}\}_{i,j}, \]
where $i$ and $j$ run over $\{0,2\}$. The degree of each element is given by
\begin{align*}
 & |c^0_{i,j}|=d-2,\ |c^1_{i,j}|=|c^1_{i,j}|=|\bar{c}^1_{i,j}|=2d-3 \text{ for } i\neq j ,\\
 & |c^1_{i,i}|=2d-3,\ |c^2_{i,j}|=3d-4.
 \end{align*}
(Obviously, there exists an isomorphism $\mathcal{A}_*(\mathcal{C}) \cong \mathcal{A}^{\mathrm{unlink}}_*$ as graded $\R$-algebras.)
We define a graded derivation $\partial\coloneqq  0\colon \mathcal{A}^{\mathrm{unlink}}_* \to \mathcal{A}^{\mathrm{unlink}}_*$. For a differential graded algebra $(\mathcal{A}^{\mathrm{unlink}}_*,\partial)$, we have the following result.
\begin{thm}\label{thm-string-unlink}
There exists an isomorphism of graded $\R$-algebras
\[H_*(\mathcal{A}^{\mathrm{unlink}}_*,\partial) \cong H^{\str}_*(\R^{2d-1},K_0\cup K_2).\]
\end{thm}
To compute $H^{\str}_*(\R^{2d-1},K_0\cup K_2)$, we fix auxiliary data so that $g$ is the standard Riemannian metric on $\R^{2d-1}$. The constant $C_0$ is required to be $C_0> \sqrt{|z_2^*|^2+4}$. The other data, $\epsilon_0$ and $\mu$, are not specified.
The strategy of the proof is the same as Theorem \ref{thm-string-hopf}, but it is much more simple. We only see the outline of each step.

\textit{Step 1.}
We may assume that $|z^*_2|>2$.
Let us fix points $p_0\coloneqq (0,1,0)\in K_0$ and $p_2\coloneqq (0,1,z^*_2)\in K_2$, and define submanifolds of $(K_0\cup K_2)^2$
\begin{align*}
K_{0,2}&\coloneqq \{(p,p')\in K_0\times K_2\mid p'= p+ (0,0,z^*_2)\}, \\
K_{2,0} &\coloneqq \{(p,p')\in K_2\times K_0\mid p'= p - (0,0,z^*_2)\} .
\end{align*}
Let $\varphi \colon (K_0\cup K_2)^2\to \Omega_{K_0\cup K_2}(\R^{2d-1})$ be the map defined as in Section \ref{subsec-hopf} by replacing $K_1$ by $K_2$.
Then, we define the set of chains
\[\mathcal{X}'\coloneqq  \{x^0_{i,j}\}_{i\neq j} \cup \{x^1_{i,i}\}_{i}\cup \{x^1_{i,j},\bar{x}^1_{i,j}\}_{i\neq j}\cup \{x^2_{i,j}\}_{i,j} , \]
where $i$ and $j$ run over $\{0,2\}$, as follows: 
\begin{align*}
\begin{split}
x^0_{i,j} &\coloneqq [\{(p_i,p_j)\}, \rest{\varphi}{\{(p_i,p_j)\}}, 1]\in C^{\dr}_{0}(\Sigma^{a}_1,\Sigma^0_1)\ (i\neq j), \\
x^1_{i,i} &\coloneqq [\{p_i\}\times K_i , \rest{\varphi }{\{p_i\}\times K_i },1] \in C^{\dr}_{d-1}(\Sigma^{a}_1,\Sigma^0_1), \\
x^1_{i,j} &\coloneqq [K_{i,j} , \rest{\varphi }{K_{i,j} },1] \in C^{\dr}_{d-1}(\Sigma^{a}_1,\Sigma^0_1)\ (i\neq j), \\
\bar{x}^1_{i,j} &\coloneqq [K_i\times \{p_j\} , \rest{\varphi}{ K_i\times \{p_j\} },1] \in C^{\dr}_{d-1}(\Sigma^{a}_1,\Sigma^0_1) \ (i\neq j), \\
x^2_{i,j}&\coloneqq [K_i\times K_j, \rest{\varphi}{K_i\times K_j}, 1] \in C^{\dr}_{2d-2}(\Sigma^{a}_1,\Sigma^0_1).
\end{split}
\end{align*}
Here, $a> \sqrt{ |z^*_2|^2+4 }$. If we define $\mathfrak{l}\colon \mathcal{X}' \to \R_{>0}$ by
\begin{align*}
&\mathfrak{l}(x^0_{i,j}) = \mathfrak{l}(x^1_{i,j}) = |z^*_2| \text{ and }  \mathfrak{l}(\bar{x}^1_{i,j}) =\mathfrak{l}(x^2_{i,j})= \sqrt{ |z^*_2|^2+4 } \text{ for }i\neq j, \\
&\mathfrak{l}(x^1_{i,i})=\mathfrak{l}(x^2_{i,i})=2,
\end{align*}
then, each $x\in \mathcal{X}'$ satisfies $x\in C^{\dr}_*(\Sigma^{\mathrm{l}(x)+\epsilon}_1,\Sigma^0_1)$ for any $\epsilon>0$.
Furthermore, a basis of $H^{\dr}_*(\Sigma^a_m,\Sigma^0_m)$ for $a\in \R_{>0}$ and $m\in \Z_{\geq 1}$ is given by the set homology classes
\[ \{ [x_1\star \dots \star x_m] \mid x_1,\dots,x_m\in \mathcal{X}' ,\ \mathfrak{l}(x_1)+\dots +\mathfrak{l}(x_m) <a\} . \]

The reason of simplicity in this case is the following: For any $x\in \mathcal{X}'$ and $(\epsilon,\delta)\in \mathcal{T}_{\mathfrak{l}(x)}$, the equation
\[f_{1,\delta} (x)=0\in C^{\dr}_*(\Sigma^{\mathfrak{l}(x)+2\epsilon}_2,\Sigma^0_2)\]
holds since the path $\varphi(p,p')$ satisfies the condition (iii) of Lemma \ref{lem-short-length} for any $(p,p')\in (K_0\cup K_2)^2$ (c.f. Lemma \ref{lem-f-x-0}).

\textit{Step 2.} There exists a bijection $\mathcal{C}'\to \mathcal{X}'$ which maps $c^{k}_{i,j}$ to  $x^{k}_{i,j}$ and $\bar{x}^1_{i,j}$ to $\bar{c}^1_{i,j}$ for $k\in \{0,1,2\}$ and $i,j\in \{0,2\}$. Composing $\mathfrak{l}\colon \mathcal{X}'\to \R_{>0}$ with this bijection, a function $\mathfrak{l}\colon \mathcal{C}'\to \R_{>0}$ is defined.
Similar to $\mathcal{A}^{\hopf}_*$,  $\mathcal{A}_*^{\mathrm{unlink}}$ is filtered by subcomplexes $(\mathcal{A}^{<a}_*,\partial )$  for all $a\in \R_{>0}\setminus \mathcal{L}(K_0\cup K_2)$ which is defined by using $\mathfrak{l} \colon \mathcal{C}'\to \R_{>0}$.

Now $\Phi^{<a}_{\epsilon}\colon \mathcal{A}^{<a}_* \to C^{<a}_*(\epsilon)$ is defined so that $c\in \mathcal{C}'$ is mapped to  $x\in \mathcal{X}'$ which corresponds to $c$ via the above bijection, and extend it naturally via the product map on $\mathcal{A}^{\mathrm{unlink}}_*$ and the $\star$-operation. It is clear in this case that $\Phi^{<a}_{\epsilon}$ is a chain map from $(\mathcal{A}^{<a}_*,\partial=0)$ to $(C^{<a}_*(\epsilon), D_{\delta})$. The fact that this map is a quasi-isomorphism is proved by a similar argument as Proposition \ref{prop-Phi-hopf-isom} about spectral sequences. 

\textit{Step 3.} We check that the family of maps $(\Phi^{<a}_{\epsilon})_*$ induces an isomorphism on the limit of $\epsilon \to 0$ and $a\to \infty$. This finishes the proof of Theorem \ref{thm-string-unlink}.

\subsection{A corollary and its potential application}

The next result is a corollary from the above computations
\begin{cor}\label{cor-distinguish} 
As  graded $\R$-algebras,
\[H^{\str}_*(\R^{2d-1},K_0\cup K_1)\ncong H^{\str}_*(\R^{2d-1},K_0\cup K_2).\]
\end{cor}
\begin{proof}
From Theorem \ref{thm-string-hopf} and \ref{thm-string-unlink}, it suffices to show that $H_*(\mathcal{A}^{\hopf}_*,\partial)$ is not isomorphic to $\mathcal{A}^{\mathrm{unlink}}_*$ as a graded $\R$-algebra.
Let us rewrite $(c^0_{0,1}, c^0_{1,0},e^1_{0,0}, e^1_{1,1})$ by $(a_0,a_1,b_0,b_1)$ and $(c^0_{0,2},c^0_{2,0})$ by $(a'_0,a'_1)$. In addition, we define $C_0\coloneqq b_0+ a_0a_1$ and $C_1\coloneqq b_1+ (-1)^d a_1a_0$.

If $d=2$,
$H_0(\mathcal{A}^{\hopf}_*,\partial +F)$ is the (a priori) non-commutative $\R$-algebra generated by $\{a_0,a_1,b_0,b_1\}$ modulo the ideal generated by $\{b_0,b_1,C_0,C_1\}$. This is isomorphic to the commutative algebra $\R[a_0,a_1]/(a_0a_1)$.
On the other hand, $\mathcal{A}^{\mathrm{unlink}}_0$ is the non-commutative algebra freely generated by $\{a'_0,a'_1\}$.
Therefore, $H_0(\mathcal{A}^{\hopf}_*,\partial)\ncong \mathcal{A}^{\mathrm{unlink}}_0$ as $\R$-algebras.

If $d\geq 3$, the lower degree parts are isomorphic as vector spaces. Indeed, for $p\leq 2d-5$,
\[H_p(\mathcal{A}^{\hopf}_*,\partial +F) \cong \mathcal{A}^{\mathrm{unlink}}_p \cong \begin{cases}
\R & \text{ if }p=0, \\
\R a_0\oplus \R a_1 & \text{ if }p=d-2 , \\
0 & \text{ else }.
\end{cases}\]
However, $H_{2d-4}(\mathcal{A}^{\hopf}_*,\partial +F)$ is the $\R$-vector space spanned by $\{a_ia_j\mid i,j\in \{0,1\}\}\cup \{b_0,b_1\}$ modulo the subspace generated by $\{b_0,b_1,C_0,C_1\}$, so its dimension is equal to $2$. On the other hand, $\mathcal{A}^{\mathrm{unlink}}_{2d-4}$ is the $\R$-vector space spanned by $\{a'_ia'_j\mid  i,j\in \{0,1\} \}$, so its dimension is equal to $4$. Therefore,  $H_{2d-4}(\mathcal{A}^{\hopf}_*,\partial)\ncong \mathcal{A}^{\mathrm{unlink}}_{2d-4}$ as $\R$-vector spaces.
\end{proof}

Let us see a potential application of this result. First we determine spin structures on unit conormal bundles
\begin{prop}\label{prop-spin}
Let $Q$ be an $n$-dimensional Riemannian manifold with a fixed spin structure. Then, for every submanifold $K$ in $Q$, we can assign a spin structure on its unit conormal bundle $\Lambda_K$ so that if $K$ is isotopic to $K'$ as a submanifold in $Q$, then $\Lambda_{K}$ is isotopic to $\Lambda_{K'}$  as a Legendrian subamanifold with a spin structure.
\end{prop}
\begin{proof}Let us identify $T^*Q$ with $TQ$ via Riemannian metric. We also identify $Q$ with the zero-section of $TQ$.
Let $L_K$ be the conormal bundle of $K$. Note that the tangent space of $TQ$ at $(q,0)\in Q$ is equal to $T_qQ\oplus  T_qQ$, where the first component is the tangent space of the base space $Q$, and the second component is the tangent space of the fiber $T_qQ$. For every $q\in K$, $T_{(q,0)}(L_K) =T_qK\oplus (T_qK)^{\perp}$. Thus the vector bundle $\rest{ T(L_K)}{K}$ has a spin structure induced by $\rest{TQ}{K}$. Since $K$ is a deformation retract of $L_K$, this spin structure is extended to $T(L_K)$. By using a diffeomorphism $\R_{>0} \times \Lambda_K\to L_K\setminus K \colon (r, (q,p))\mapsto (q,r\cdot p)$, we can determine a spin structure on $T(\Lambda_K)$ so that the spin structure on $\underline{\R}\oplus T\Lambda_K$, induced by the inclusion map $\mathrm{Spin}(2n-1)\to \mathrm{Spin}(2n)$, is equal to the spin structure on  $\rest{ T(L_K)}{\Lambda_K} \cong \underline{\R}\oplus T\Lambda_K$. This spin structure on $\Lambda_K$ for every submanifold $K$ clearly satisfies the condition of this proposition.
\end{proof}

Let us consider the unit conormal bundles of $K_0\cup K_1$ and $K_0\cup K_2$.

\begin{prop}As a $(2d-2)$-dimensional submanifold of $UT^*\R^{2d-1}$ with the spin structure determined by Proposition \ref{prop-spin}, $\Lambda_{K_0\cup K_1}$ is isotopic to $\Lambda_{K_0\cup K_2}$.
\end{prop}
\begin{proof}
For $s\in [0,1]$, we define $K^s_1\coloneqq \{q +(0,2s,0)\in \R^{2d-1} \mid q\in K_1 \}$. We also choose a $C^{\infty}$ function $ [0,1]\to [0,\pi]\colon s\mapsto \theta_s$ so that $\theta_0=\theta_1=0$ and $\theta_{1/2}=\pi/2$, and define $R_s\in \mathrm{SO}(2d-1)$ for $s\in [0,1]$ by
$R_s (v_0,v_1,v_2) \coloneqq ((\cos \theta_s) v_0 - (\sin \theta_s) v_2, v_1, (\sin \theta_s) v_0 + (\cos \theta_s) v_2)$
for every $(v_0,v_1,v_2)\in \R^{d-1}\times \R\times \R^{d-1}$.
We then define an isotopy $(\Lambda_s)_{s\in [0,1]}$ from $\Lambda_{K^0_1}= \Lambda_{K_1}$ to $\Lambda_{K^1_1}$ by
\[\Lambda_s \coloneqq \{(q,p)\in UT^*\R^{2d-1} \mid q\in K_s,\ \rest{p\circ R_s}{T_q K^1_s}=0 \}. \]
$\Lambda_s$ intersects $\Lambda_{K_0}$ if and only if $s=1/2$, and $\Lambda_{1/2}\cap \Lambda_{K_0} = \Lambda_{K_0} \cap UT^*_{p_0}\R^{2d-1}$, where $p_0=(0,1,0)\in \R^{2d-1}$.
We can slightly perturb $(\Lambda_s)_{s\in [0,1]}$ around $s=1/2$ to an isotopy $(\Lambda'_s)_{s\in [0,1]}$ so that $\Lambda'_s$ does not intersect $\Lambda_{K_0}$ for every $s\in [0,1]$. This isotopy is homotopic to an isotopy $(\Lambda_{K_1^s})_{s\in [0,1]}$, which preserves the spin structure of Proposition \ref{prop-spin}. In addition, $K_0\cup K^1_1$ is isotopic to $K_0\cup K_2$ in $\R^{2d-1}$.
Therefore, as a $C^{\infty}$ submanifold with a spin structure, $\Lambda_{K_0}\cup \Lambda_{K_1}$ is  isotopic to  $\Lambda_{K_0}\cup \Lambda_{K_2}$.
\end{proof}

If  Conjecture \ref{conj-intro} in the introduction
is ture, then Corollary \ref{cor-distinguish} can be applied to show that the unit conormal bundle $\Lambda_{K_0\cup K_1}$ is not isotopic to $\Lambda_{K_0\cup K_2}$ as a Legendrian submanifold with a spin structure in $UT^*\R^{2d-1}$, though they are isotopic as $C^{\infty}$ submanifolds with spin structures by the above proposition.

\section{Cord algebra and $H^{\str}_0(Q,K)$}\label{sec-cord-alg}

Throughout this section, we consider the case where the codimension of $K$ is $2$ (i.e. $d=2$) and the normal bundle $(TK)^{\perp}$ is trivial. The purpose  is to show that $H^{\str}_0(Q,K)$ is isomorphic to an isotopy invariant of $K$, called \textit{cord algebra}.

\subsection{Cord algebra and string homology}\label{subsec-cord-alg-and-string-hmgy}

In this section, we refer to \cite{celn, ng} and give a definition of cord algebra and string homology. Note that, in this paper, their coefficients are reduced from original $\Z[\pi_1(\partial N_{\epsilon_0})]$ to $\R$.

We fix a frame of $(TK)^{\perp}$ to give an isomorphism $ \R^2\times K\cong  (TK)^{\perp}$ of vector bundles over $K$ which preserves their fiber metrics and orientations. Combining with the map (\ref{tubular-neighborhood}), we obtain a diffeomorphism
\[h \colon \mathcal{O}_{\epsilon_0}\times K \to N_{\epsilon_0} ,\]
which preserves orientations.
Here, $\mathcal{O}_{\epsilon} =\{w\in \R^2\mid |w|<\epsilon/2 \}$ for every $\epsilon\leq \epsilon_0$ as defined in Subsection \ref{subsub-explicit}.

First, we define an $\R$-algebra $\mathrm{Cord}(Q,K;\R)$. Its relation to the cord algebra defined in \cite{celn, ng} is discussed later in Remark \ref{rem-previous-cord}.
Let us prepare several notations.
We fix $w_0\in \mathcal{O}_{\epsilon_0}\setminus \{0\}$ and define a submanifold disjoint from $K$
\[K'\coloneqq \{ h(w_0,x)\mid x\in K\}\subset N_{\epsilon_0}.\]
For every $x\in K$, we define $c_x\colon[0,1]\to Q\setminus K$ to be the constant path at $h(w_0,x)\in K'$. We also define $m_x\colon [0,1]\to Q\setminus K$ to be a loop in a punctured disk $h ((\mathcal{O}_{\epsilon_0}\setminus\{0\}) \times\{x\})\subset N_{\epsilon_0}\setminus K$ based at $h(w_0,x)\in K'$, whose winding number around $h(0,x)$ is equal to $1$.
In addition, let $\pi_1(Q\setminus K ,K')$ be the set of homotopy classes of continuous paths $\gamma\colon [0,1]\to Q\setminus K$ such that $\gamma(\{0,1\}) \subset K'$. 

\begin{defi}
Let $\mathcal{A}$ be the unital non-commutative $\R$-algebra freely generated by the set $\pi_1(Q\setminus K ,K')$.
We define the two-sided ideal $\mathcal{I}$ generated by the following elements:
\[\begin{cases}
[c_x] , \\
[\gamma_1\cdot \gamma_2]-[\gamma_1\cdot m_x \cdot \gamma_2] -[\gamma_1][\gamma_2], 
\end{cases}\]
for all $x\in K$ and $\gamma_j\colon [0,1] \to Q\setminus K $ ($j=1,2$) such that $\gamma_j(\{0,1\})\subset K'$ and $\gamma_1(1)=h(w_0,x)=\gamma_2(0)$.
Then, we define an $\R$-algebra $\mathrm{Cord}(Q,K;\R)\coloneqq \mathcal{A}/\mathcal{I}$, and call it the \textit{cord algebra} of $(Q,K)$.
\end{defi}

\begin{rem}\label{rem-previous-cord}
When $K$ is $1$-dimensional and connected (i.e. $K$ is an oriented knot in a $3$-manifold $Q$), we fix a base point $* \in K'$. A \textit{cord} is a path $\gamma\colon [0,1]\to Q$ such that $\gamma([0,1])\cap K=\varnothing$ and $\gamma(0),\gamma(1)\in K'\setminus\{*\}$. The notion of cord algebra (or \textit{cord ring}) of knots was defined in, for instance, \cite{celn, ng2, ng}. The most refined one is \cite[Definition 2.6]{celn}, which is defined as a non-commutative algebra over $\Z$ generated by the set of homotopy classes of cords and $\{\lambda^{\pm},\mu^{\pm}\}$, modulo the relations about $\{\lambda^{\pm},\mu^{\pm}\}$ and the ``skein relations''.
If we substitute both $\lambda$ and $\mu$ by $1\in \Z$ and tensor this $\Z$-algebra with $\R$, we obtain an $\R$-algebra isomorphic to $\mathrm{Cord}(Q,K;\R)$.
(The isomorphism is induced by a natural map from the set of homotopy classes of cords to $\pi_1(Q\setminus K,K')$.)

We should also note that in \cite[Definition 2.1]{ng}, the cord algebra over $\Z[H_1(\partial N_{\epsilon_0})]$ was defined when $K$ is a connected codimension $2$ submanifold of an arbitrary manifold and its normal bundle is oriented.
In our setting, we have an isomorphism $(h^{-1})_*\colon H_1(\partial N_{\epsilon_0})\to H_1(S^1\times K)\cong H_1(S^1) \oplus H_1(K)$. There exists a ring homomorphism $\varphi\colon \Z[H_1(\partial N_{\epsilon_0})]\to \R$ determined by $\varphi (h_*([S^1]))=-1$ and  $\varphi(h_*(c))=1$ for every $c\in H_1(K)$. If  the base change of the cord algebra of \cite[Definition 2.1]{ng} is done by $\varphi$, we obtain an $\R$-algebra isomorphic to $\mathrm{Cord}(Q,K;\R)$.
\end{rem}

Next, we refer to \cite[Section 2.1]{celn} and define the string homology which is simplified for our purpose. 
For $m\in \Z_{\geq 1}$, $a\in \R_{\geq 0} \cup \{\infty\}$ and $p=0,1$, we define an $\R$-subspace $C^{\pitchfork}_p(m,a) \subset C^{\sing}_p(\Sigma^a_m)$ consisting of generic singular $p$-chains satisfying  jet transversality conditions.
(Recall that we have fixed a topology of $\Sigma^a_m$ in Subsection \ref{subsubsec-additional})

In the case of $p=0$, $C^{\pitchfork}_0(m,a)$ is generated by $(\gamma_k\colon[0,T_k]\to Q)_{k=1,\dots ,  m}\in \Sigma^a_m$ satisfying the following conditions:
\begin{itemize}
\item[(0a)] $(\gamma_k)'(0),(\gamma_k)'(T_k)\notin TK $ for every $k\in \{1,\dots ,m\}$.
\item[(0b)] $\gamma_k(t)\notin K$ for every $k\in \{1,\dots ,m\}$ and $t\neq 0, T_k$.
\end{itemize}
In the case of $p=1$, $C^{\pitchfork}_1(m,a)$ is generated by
1-parameter families of paths
\[[0,1]\to \Sigma^a_m \colon u\mapsto (\gamma^u_k\colon [0,T^u_k]\to Q)_{k=1,\dots ,m}\]
such that $[0,1]\to \R_{>0}\colon u\mapsto T^u_k$ is a $C^{\infty}$ function and
\[\Gamma_k\colon \{(u,t) \mid 0\leq u\leq 1,\ 0\leq t\leq T^u_k\} \to Q\]
is a $C^{\infty}$ map for every $k\in \{1,\dots ,m\}$, and satisfies the following conditions:
\begin{itemize}
\item[(1a)] $(\gamma^0_k)_{k=1,\dots , m}$ and $(\gamma^1_k)_{k=1,\dots , m}$ satisfy (0a), (0b).
\item[(1b)] $(\gamma^u_k)'(0),(\gamma^u_k)'(T^u_k)\notin TK$ for every $u\in [0,1]$ and $\Gamma^{\interior}_k\coloneqq \rest{\Gamma_k}{ \{(u,t)\mid t \neq 0,T^u_k \} }$ is transverse to $K$ for every $k\in \{ 1,\dots, m\}$.
\item[(1c)] If $(u_*,t_*), (u'_*,t'_*)\in \coprod_{k=1}^m(\Gamma^{\interior}_k)^{-1}(K)$ are distinct points, then $u_* \neq u'_*$ holds.
\end{itemize}
Note that the condition (1b) implies that $(\Gamma^{\interior}_k)^{-1}(K)$ is a finite set.
In addition, we define $C^{\pitchfork}_p(0,a)\coloneqq C^{\sing}_p(\Sigma^a_0)$ for $a\in \R_{\geq 0}\cup \{\infty\}$ and $p=0,1$.

By (1a), the boundary operator of singular chain complex $\partial^{\sing}\colon C^{\sing}_1(\Sigma^a_m) \to C^{\sing}_0(\Sigma^a_m)$ is restricted to the map
\[\partial^{\sing} \colon C^{\pitchfork}_1(m,a) \to C^{\pitchfork}_0(m,a)\colon (\gamma^u_k)^{u\in [0,1]}_{k=1,\dots ,m} \mapsto (\gamma^1_k)_{k=1,\dots , m} -  (\gamma^0_k)_{k=1,\dots , m} . \]
We also define a linear map $f_{k}^{\pitchfork}\colon C^{\pitchfork}_1(m,a)\to  C^{\pitchfork}_0(m+1,a)$ for $m\in \Z_{\geq 1}$ so that for $x=(\gamma^u_k \colon [0,T^u_k]\to Q)^{u\in [0,1]}_{k=1,\dots , m} \in C^{\pitchfork}_1(m,a)$,
\begin{align*}
f_{k}^{\pitchfork}(x) \coloneqq  \sum_{(u_*,t_*)\in (\Gamma^{\interior}_k)^{-1}(K)} \sign (u_*,t_*)\cdot (\gamma^{u_*}_1\dots ,\gamma^{u_*}_{k-1}, \widehat{\gamma^{u_*}_k}^1 , \widehat{\gamma^{u_*}_k}^2,\gamma^{u_*}_{k+1},\dots,\gamma^{u_*}_m) .
\end{align*}
Here
\[\widehat{\gamma_k^{u_*}}^1\coloneqq \rest{\gamma^{u_*}_k}{[0,t_*]}\colon [0,t_*]\to Q ,\ \widehat{\gamma_k^{u_*}}^2\coloneqq \rest{\gamma^{u_*}_k}{[t_*,T^{u_*}_k]}  (\cdot -t_*)\colon [0,T^{u_*}_k-t_*]\to Q ,\]
 and $\sign (u_*,t_*)\in \{\pm 1\}$ is the orientation sign of an open embedding into $\mathcal{O}_{\epsilon_0} \subset \R^2$
\begin{align*}
\Gamma^{\fib}_k\coloneqq \pr_{\R^2}\circ h^{-1} \circ \Gamma_k
\end{align*}
defined on a small neighborhood  of $(u_*,t_*) \in (\Gamma^{\interior}_k)^{-1}(K) \subset (0,1)\times \R_{>0}$.
For convenience, let us define for $p\notin \{0,1\}$, $C^{\pitchfork}_p(m,a) \coloneqq 0$, $\partial^{\sing}\coloneqq 0 \colon C^{\pitchfork}_p(m,a) \to C^{\pitchfork}_{p-1}(m,a)$ and $f^{\pitchfork}_k\coloneqq 0\colon C^{\pitchfork}_p(m,a) \to C^{\pitchfork}_{p-1}(m+1,a)$.

For $a\in \R_{> 0}\cup\{\infty\}$ and $m\in \Z_{\geq 0}$, we define a quotient vector space
\[C^{\pitchfork,<a}_*(m) \coloneqq   C^{\pitchfork}_*(m,a)/C^{\pitchfork}_*(m,0) . \]
Then $\partial^{\sing}$ and  $f^{\pitchfork}_k$ induce linear maps 
\[\partial^{\sing} \colon C^{\pitchfork,<a}_*(m) \to C^{\pitchfork,<a}_{*-1}(m) ,\ f^{\pitchfork}_k\colon C^{\pitchfork,<a}_*(m) \to C^{\pitchfork,<a}_{*-1}(m+1) .\]
Now we define a graded $\R$-vector space
\[C^{\pitchfork,<a}_*\coloneqq \bigoplus_{m=0}^{\infty} C^{\pitchfork,<a}_*(m) .\]
For each $m\in \Z_{\geq 0}$, $C^{\pitchfork,<a}_*(m)$ is considered to be its linear subspace in a natural way.
Then, we define a degree $(-1)$ map $D^{\pitchfork} \colon C^{\pitchfork,<a}_* \to C^{\pitchfork,<a}_{*-1} $ by
\begin{align}\label{D-delta-transverse}
D^{\pitchfork}(x) \coloneqq  \partial^{\sing}x + \sum_{k=1}^m f^{\pitchfork}_k(x) .
\end{align}
for $x\in C^{\pitchfork,<a}_*(m)$. For $x\in C^{\pitchfork, <a}_*(0)$, the RHS is just equal to $\partial^{\sing}x$.
Then we obtain a chain complex $(C^{\pitchfork,<a}_*,D^{\pitchfork})$.
Let $H^{\pitchfork,<a}_*$ denote its homology group.
In addition, for $a,b\in \R_{\geq 0}\cup \{\infty\}$ with $a\leq b$, we define a linear map $J^{a,b}\colon H^{\pitchfork,<a}_* \to H^{\pitchfork,<b}_*$ induced by the inclusion maps $C^{\pitchfork}_*(m,a) \to C^{\pitchfork}_*(m,b) $ for all $m\in \Z_{\geq 0}$.

In this section, we call $H^{\pitchfork,<\infty}_*$ the \textit{string homology} of $(Q,K)$.
Note that the direct limit $\varinjlim_{a\to \infty} H^{\pitchfork,<a}_*$ defined by $\{J^{a,b}\}_{a\leq b}$ is isomorphic to $H^{\pitchfork,<\infty}_0 $.
Furthermore, $H^{\pitchfork,<\infty}_0 $ has a associative product structure induced by $\Pi \colon \Sigma^{\infty}_m\times \Sigma^{\infty}_{m'} \to \Sigma^{\infty}_{m+m'}$. Thus $H^{\pitchfork,<\infty}_0$ is a unital associative $\R$-algebra, whose unit comes from $* \in C^{\pitchfork,<\infty}_0(0)$.

The next proposition is essentially proved in \cite[Proposition 2.9]{celn}.
\begin{prop}\label{prop-cord-to-string}
As an $\R$-algebra, $\mathrm{Cord}(Q,K;\R)$ is isomorphic to $ H_0^{\pitchfork,<\infty}$.
\end{prop}
\begin{proof}
For every homotopy class $z\in \pi_1(Q\setminus K,K')$, we choose a $C^{\infty}$ path $\gamma$ which represents $z$. We then define a path $\bar{\gamma}$ as follows: For $x_0,x_1\in K$ such that $\gamma(i)=h (w_0,x_i)$ for $i\in \{0,1\}$,
\[\bar{\gamma}\colon [0,3]\to Q\colon t\mapsto \begin{cases}
h (t \cdot w_0,x_0) & \text{ if } 0\leq t\leq 1,\\
\gamma(t-1) & \text{ if }1\leq t\leq 2, \\
h ((3-t) \cdot w_0,x_1) & \text{ if } 2\leq t\leq 3.
\end{cases}\]
We modify $\bar{\gamma}$ to $\tilde{\gamma}$ by a reparameterization so that $(\tilde{\gamma})_{k=1}\in \Sigma^{\infty}_1$ and satisfies (0a) and (0b).
Then a homomorphism of unital $\R$-algebras
\[ F \colon \mathcal{A} \to  H_0^{\pitchfork,<\infty}\]
is defined so that $z=[\gamma]$ is mapped to $[(\tilde{\gamma})_{k=1}]$. From the definition of $D^{\pitchfork}$, it can be checked that $F$ is a well-defined surjection and maps the ideal $\mathcal{I}$ into 0. Therefore, we obtain a surjective homomorphism of unital $\R$-algebras $\bar{F} \colon \mathcal{A}/\mathcal{I}\to  H_0^{\pitchfork,<\infty}$.

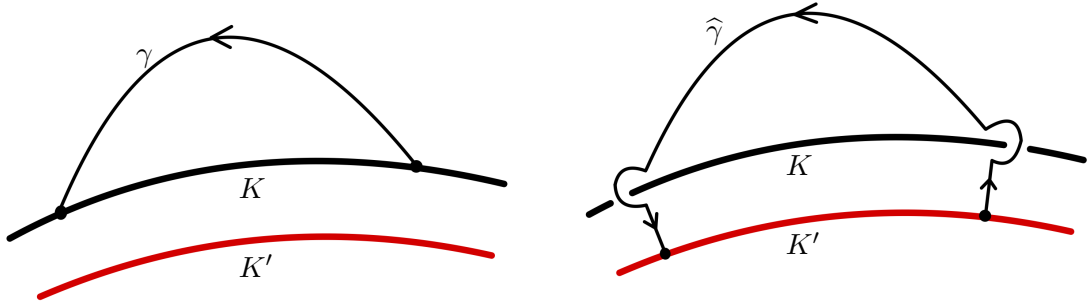
\begin{figure}
\centering
\begin{overpic}[width=15cm]{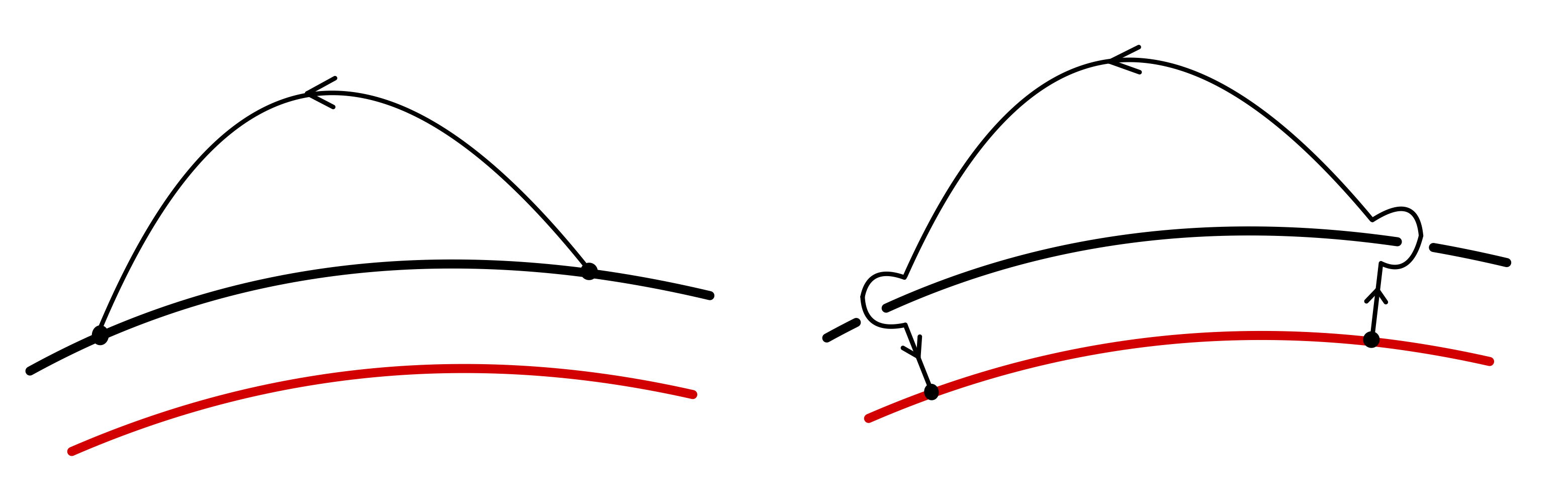}
\put(22,5){$K'$}
\put(22,12){$K$}
\put(13,24){$\gamma$}
\put(70,7){$K'$}
\put(70,14){$K$}
\put(63,26){$\widehat{\gamma}$}
\end{overpic}
\caption{The picture of $\widehat{\gamma}$ defined for $(\gamma)_{k=1}\in \Sigma^{\infty}_1$.}\label{figure-perturb}
\end{figure}

We prove that $\bar{F}$ is injective by describing its inverse map. For $(\gamma\colon [0,T]\to Q)_{k=1} \in \Sigma^{\infty}_1$ satisfying (0a) and (0b), let us define
\[\check{\gamma} \colon [0,1] \to Q\colon t \mapsto \begin{cases}
h ((1-3t)\cdot w_0, \gamma(0)) & \text{ if } 0\leq t\leq \frac{1}{3}, \\
\gamma(3T t-T) & \text{ if }\frac{1}{3} \leq t\leq \frac{2}{3} , \\
h ((3t-2)\cdot w_0,\gamma(T)) & \text{ if } \frac{2}{3}\leq t\leq 1.
\end{cases}\]
As described in Figure \ref{figure-perturb}, we change $\check{\gamma}$ into $\widehat{\gamma}$ by small perturbations inside $N_{\epsilon_0}$ around $t\in \{\frac{1}{3},\frac{2}{3}\}$ so that $\widehat{\gamma}$ does not intersect $K$. We then obtain a homotopy class $[\widehat{\gamma}]\in \pi_1(Q\setminus K,K')$.
Note that for any $\widehat{\gamma}^j$ ($j=1,2$) from two choices of perturbations, there exist $l_0,l_1\in \{0,+1,-1\}$ such that
\[ [\widehat{\gamma}^2]= [ (m_{\gamma(0)})^{l_0} \cdot \widehat{\gamma}^1 \cdot (m_{\gamma(T)})^{l_1} ] \in \pi_1(Q\setminus K, K').\]
Here, $(m_x)^1\coloneqq m_x$, $(m_x)^0\coloneqq c_x$, and $(m_x)^{-1}$ denotes the inverse path of $m_x$. (As a natural extension, $(m_x)^l$ for $l\in \Z$ is defined.)
Thus, $[\widehat{\gamma}^1]=[\widehat{\gamma}^2]$ holds as an element of  $\mathcal{A}/\mathcal{I}$.
If $\gamma\in \Sigma^{0}_1$ (i.e. $\len \gamma <\epsilon_0$), $[\widehat{\gamma}]=[(m_x)^l]\in \pi_1(Q\setminus K,K')$ for some $x\in K$ and $l\in \Z$. In such case, $[\widehat{\gamma}]=0 $ holds as an element of $\mathcal{A}/\mathcal{I}$. Therefore, we have a well-defined linear map
\[ G \colon C^{\pitchfork,<\infty}_0  \to \mathcal{A}/\mathcal{I} \colon \begin{cases} 
C^{\pitchfork,<\infty}(m)\ni (\gamma_k)_{k=1,\dots ,m} &\mapsto  [\widehat{\gamma_1}] \cdots  [\widehat{\gamma_m}]   \ (m\geq 1), \\
C^{\pitchfork,<\infty}(0)\ni 1 &\mapsto \text{the unit} .   \end{cases} \]
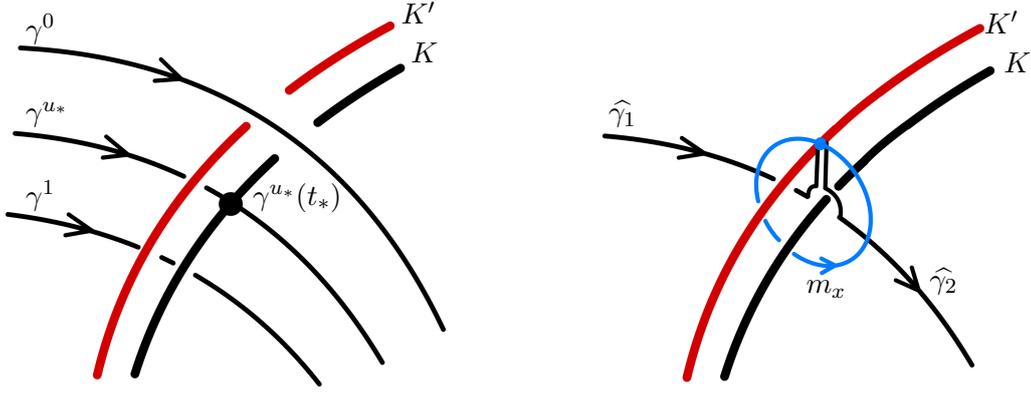
\begin{figure}
\centering
\begin{overpic}[width=13cm]{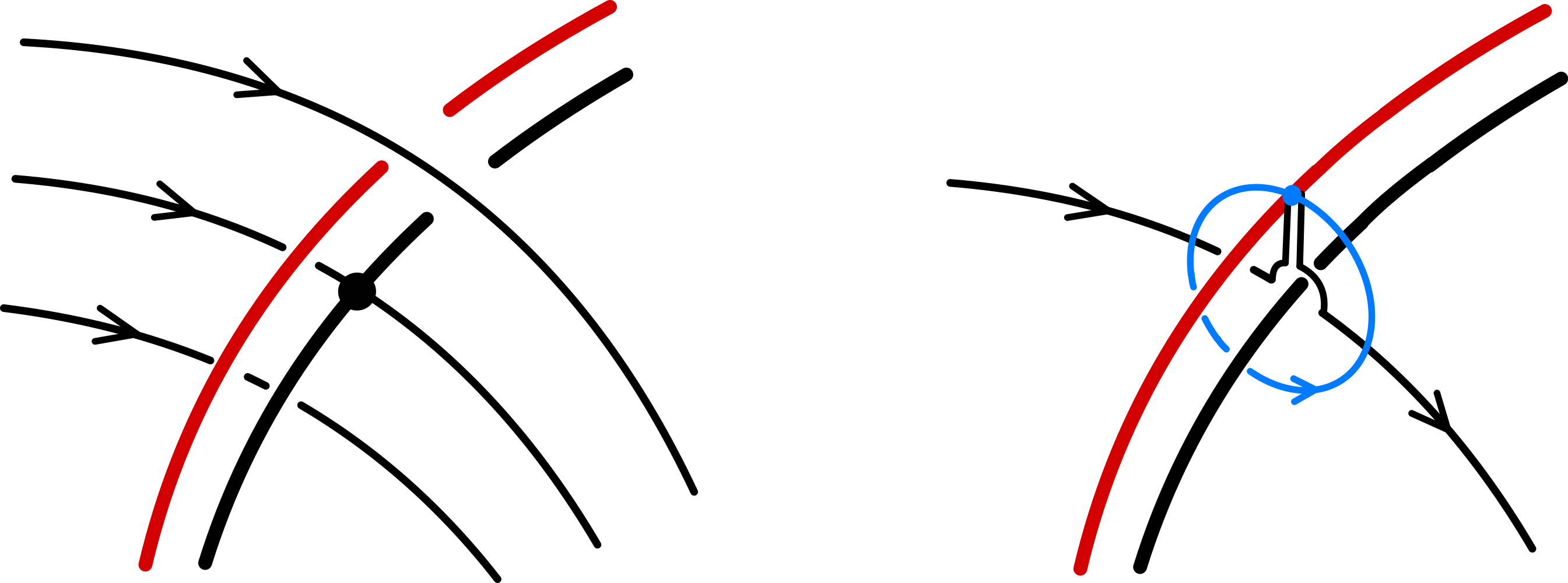}
\put(2,27){$\gamma^{u_*}$}
\put(2,35.5){$\gamma^0$}
\put(2,18.5){$\gamma^1$}
\put(25,18.5){$\gamma^{u_*}(t_*)$}
\put(41,33){$K$}
\put(40,37){$K'$}
\put(61,27){$\widehat{\gamma_1}$}
\put(93.5,10){$\widehat{\gamma_2}$}
\put(81,9.5){$m_x$}
\put(101,32){$K$}
\put(99,36){$K'$}
\end{overpic}
\caption{The LHS describes the $1$-chain $(\gamma^u)^{u\in [0,1]}_{k=1}$ such that $\sign (u_*,t_*)=+1$ at $(u_*,t_*)\in (\Gamma^{\interior}_1)^{-1}(K)$. From the RHS, we can see that $[\widehat{\gamma^0}]= [\widehat{\gamma_1}\cdot \widehat{\gamma_2}]$ and $[\widehat{\gamma^1}]= [\widehat{\gamma_1}\cdot m_x\cdot \widehat{\gamma_2}]$ as elements of $\pi_1(Q\setminus K,K')$.}\label{figure-cord}
\end{figure}

From the transversality condition (1b) together with (1c), it follows that $\Im D^{\pitchfork}$ is mapped into $0$. Indeed, in a simple case, for $(\gamma^u)^{u\in [0,1]}_{k=1}\in C^{\pitchfork}_1(1,\infty)$ such that $(\Gamma^{\interior}_1)^{-1}(K)=\{(u_*,t_*)\}$, we can see that
\[G(D^{\pitchfork}(x)) = \begin{cases} [\widehat{\gamma_1}\cdot m_x \cdot \widehat{\gamma_2}]-[\widehat{\gamma_1}\cdot \widehat{\gamma_2}] + \sign(u_*,t_*) [\widehat{\gamma_1}][\widehat{\gamma_2}] & \text{ if } \sign (u_*,t_*)=+1, \\ 
[\widehat{\gamma_1}\cdot \widehat{\gamma_2}] - [\widehat{\gamma_1}\cdot m_x \cdot \widehat{\gamma_2}] + \sign(u_*,t_*) [\widehat{\gamma_1}][\widehat{\gamma_2}] & \text{ if } \sign (u_*,t_*)=-1 , \end{cases}\]
for $x\coloneqq \gamma^{u_*}(t_*)$, $\gamma_1\coloneqq \rest{\gamma^{u_*}}{[0,t_*]}$ and $\gamma_2\coloneqq \rest{\gamma^{u_*}}{[t_*,T^{u_*}]} (\cdot -t_*)$. Figure \ref{figure-cord} describes the case where $\sign (u_*,t_*)=+1$.
Thus $G(D^{\pitchfork}( (\gamma^u)^{u\in [0,1]}_{k=1} )) =0 \in \mathcal{A}/\mathcal{I}$. The condition (1c) implies that the general case can be reduced to this simple case.
Therefore, we obtain a well-defined linear map $\bar{G}\colon  H_0^{\pitchfork,<\infty} \to  \mathcal{A}/\mathcal{I}$. Finally, for any $[\gamma]\in \pi_1(Q\setminus K,K')$ such that $\gamma(i)=h(w_0,x_i)$ for $i\in \{0,1\}$, there exist $l_0,l_1\in \{0,+1,-1\}$ such that
\[\bar{G}\circ \bar{F} ([\gamma]) =[(m_{x_0})^{l_0}\cdot \gamma \cdot (m_{x_1})^{l_1}],\]
which is equal to $[\gamma]$ in $\mathcal{A}/\mathcal{I}$. This implies that $\bar{F}$ is injective.
\end{proof}

\subsection{Connecting string homology and $H^{\str}_*(Q,K)$}\label{subsec-construction-of-chain-map}

The purpose of this section is to construct a linear map form $H^{\pitchfork,<a}_*$ to $H^{<a}_*(Q,K)$ for all $a\in \R_{> 0}\setminus \mathcal{L}(K)$. On the limit of $a\to \infty$, an $\R$-algebra homomrphism from $ H^{\pitchfork,<\infty}_*$ to $H^{\str}_*(Q,K)$ is defined.
Before constructing this map, we will prepare two things.

\subsubsection{Preliminaries}

First, we define a map $\Psi$ which associates de Rham chains with singular chains.
Let $\kappa,\chi \colon \R\to [0,1]$ be $C^{\infty}$ functions such that:
\begin{itemize}
\item $\kappa (u)=0 $ if $u\leq 0$ and $\kappa(u)= 1$ if $u\geq 1$. In addition, $\kappa'(u)>0$ if $0<u<1$.
\item $\chi\colon \R\to [0,1]$ has a compact support and $\chi (s)=1$ for every $s\in [0,1]$.
\end{itemize}
For $p\in \Z$,  a linear map
\[\Psi\colon C^{\pitchfork}_p(m,a) \to C^{\dr}_p(\Sigma^{a}_m) \]
is defined by
\[\begin{cases}
\Psi((\gamma_k)_{k=1,\dots , m})\coloneqq [\{0\},c_0, 1] & \text{ if }p=0, \\
\Psi( (\gamma^u_k)^{u\in [0,1]}_{k=1,\dots, m} ) \coloneqq [\R, c_1,\chi]& \text{ if }p=1, \\
\Psi=0 & \text{ else, }
\end{cases}\]
where $c_0$ is the constant map to $(\gamma_k)_{k=1,\dots, m} $ and $c_1\colon \R \to \Sigma^{a}_{m} \colon u\mapsto (\gamma^{\kappa(u)}_k)_{k=1,\dots , m} $.
These maps commute with boundary operators, namely,
\[\partial \circ \Psi = \Psi\circ \partial^{\sing} \colon C^{\pitchfork}_p(m,a) \to C^{\dr}_{p-1}(\Sigma^a_m).\]

Next, we define a filtration $\{C^{\pitchfork}_p(m,a,\epsilon)\}_{\epsilon>0}$ of $C^{\pitchfork}_p(m,a)$.
In the cases of $p=0$, $C^{\pitchfork}_0(m,a,\epsilon )$ is an $\R$-subspace generated by $(\gamma_k\colon[0,T_k]\to Q)_{k=1,\dots , m}\in \Sigma^a_m$ satisfying (0a), (0b) and the following condition:
\begin{itemize}
\item[(0c)] There exists $\tau_0 \in (0, \epsilon_0/(5C_0)]$ such that 
$\gamma_k ([\tau_0, T_k-\tau_0])\cap N_{\epsilon}=\varnothing $ for every $k\in \{1,\dots ,m\}$.
\end{itemize}
In the case of $p=1$, $C^{\pitchfork}_1(m,a,\epsilon)$ is an $\R$-subspace generated by $(\gamma^u_k\colon [0,T^u_k]\to Q)^{u\in [0,1]}_{k=1,\dots ,m}$ satisfying (1a), (1b), (1c) and the following conditions:
\begin{itemize}
\item[(1d)] $(\gamma^0_k)_{k=1,\dots ,m}$ and $(\gamma^1_k)_{k=1,\dots , m}$ satisfy (0c).
\item[(1e)] There exist $\tau_0\in (0,\epsilon_0/(5C_0)]$ and an open neighborhood $U_{(u_*,t_*)}$ for each $(u_*,t_*)\in (\Gamma^{\interior}_k)^{-1}(K)$ ($k=1,\dots ,m$) such that
\[(u_*,t_*)\in U_{(u_*,t_*)} \subset \{(u,t)\mid 0<u<1,\ \tau_0<t < T^u_k-\tau_0\} , \]
and the following hold:
\begin{itemize}
\item[(1e-1)] 
$U_{(u_*,t_*)} \subset \{(u,t) \mid |t-t_*|< \tau_0 \}$.
\item[(1e-2)] For any two distinct points $(u_*,t_*), (u'_*,t'_*)\in \coprod_{k=1}^m(\Gamma^{\interior}_k)^{-1}(K)$,
the projections of $U_{(u_*,t_*)},U_{(u'_*,t'_*)} \subset [0,1]\times \R_{>0}$ to $[0,1]$
\[\pr_{[0,1]} (U_{(u_*,t_*)}),\  \pr_{[0,1]}(U_{(u'_*,t'_*)}) \subset [0,1] \]
are disjoint.
\item[(1e-3)] For every $k\in \{ 1,\dots ,m\}$,
\[\left( \rest{\Gamma_k}{ \{(u,t)\mid \tau_0 \leq t \leq T^u_k-\tau_0\} } \right)^{-1}(N_{\epsilon})=\bigcup_{(u_*,t_*)\in (\Gamma^{\interior}_k)^{-1}(K)} U_{(u_*,t_*)}. \]
Moreover, for each $(u_*,t_*)\in (\Gamma^{\interior}_k)^{-1}(K)$,
\[\Gamma^{\fib}_k\colon U_{(u_*,t_*)} \to \mathcal{O}_{\epsilon}\]
is a diffeomorphism. (Recall that $\Gamma^{\fib}_k= \pr_{\R^2} \circ h^{-1}\circ \Gamma_k$.)
\end{itemize}
\end{itemize}
In addition, we define $C^{\pitchfork}_p(m,a,\epsilon)\coloneqq C^{\pitchfork}_p(m,a)$ if $p\notin \{0,1\}$ or $m=0$.


Roughly speaking, (0c) means that $\gamma_k(t)$ is far from $K$ by a distance at least $\epsilon/2$, except when $t$ is close to $\{0,T_k\}$. (1e) means that $\gamma^u_k(t)$ is far from $K$ by a distance at least $\epsilon/2$, except when $t$ is close to $\{0,T^u_k\}$ or when $(u,t)$ is close to some point in $(\Gamma^{\interior}_k)^{-1}(K)$.
Note that when $0<\epsilon'\leq \epsilon$, $C^{\pitchfork}_*(m,a,\epsilon)\subset C^{\pitchfork}_*(m,a,\epsilon')$ holds, and
\begin{align}\label{union-pitchfork-epsilon}
\bigcup_{\epsilon>0} C^{\pitchfork}_*(m,a,\epsilon) = C^{\pitchfork}_*(m,a) .
\end{align}
Moreover, $\partial^{\sing}$ and $f^{\pitchfork}_k$ ($k=1,\dots , m$) are restricted to linear maps
\[\partial^{\sing}\colon C^{\pitchfork}_*(m,a,\epsilon) \to C^{\pitchfork}_{*-1}(m,a,\epsilon),\ f^{\pitchfork}_k\colon C^{\pitchfork}_*(m,a,\epsilon) \to C^{\pitchfork}_{*-1}(m+1,a,\epsilon) .
\]

For every $\epsilon>0$, we define $C^{\pitchfork,<a}_*(m,\epsilon) \coloneqq C^{\pitchfork}_*(m,a,\epsilon)/ C^{\pitchfork}_*(m,0,\epsilon)$ and
\[C^{\pitchfork,<a}_*(\epsilon)\coloneqq \bigoplus_{m=0}^{\infty}  C^{\pitchfork,<a}_*(m,\epsilon) .\]
A linear map $D^{\pitchfork}_{\epsilon} \colon C^{\pitchfork,<a}_*(\epsilon) \to C^{\pitchfork,<a}_{*-1}(\epsilon) $ is defined by the same form as (\ref{D-delta-transverse}).
Then we obtain a chain complex $(C^{\pitchfork,<a}_*(\epsilon),D^{\pitchfork}_{\epsilon})$. Let $H^{\pitchfork,<a}_*(\epsilon)$ denote its homology.
When $0<\epsilon'\leq\epsilon $, the inclusion maps $C^{\pitchfork}_*(m,a,\epsilon)\to C^{\pitchfork}_*(m,a,\epsilon')$ for all $m\in \Z_{\geq 0}$ induce a linear map 
\[l_{\epsilon,\epsilon'} \colon H^{\pitchfork,<a}_*(\epsilon) \to H^{\pitchfork,<a}_*(\epsilon'), \]
and we have a direct system $(\{H^{\pitchfork,<a}_*(\epsilon)\}_{\epsilon>0}, \{ l_{\epsilon,\epsilon'} \}_{\epsilon \geq \epsilon'} )$. From the relation (\ref{union-pitchfork-epsilon}), we have
\[\varinjlim_{\epsilon\to 0} H^{\pitchfork,<a}_*(\epsilon) =  H^{\pitchfork,<a}_*. \]

\subsubsection{Construction of chain map}

With the above preparations, we consider the maps
\[ C^{\pitchfork}_p(m,a,\epsilon) \to C^{\dr}_p(\Sigma^{a+m\epsilon}_m) \colon x\mapsto \Psi(x)\ (m\in \Z_{\geq 0})\]
for $\epsilon\in (0,\epsilon_0/(5C_0)]$.
They induce a linear map from $C^{\pitchfork,<a}_*(\epsilon)$ to $C^{<a}_*(\epsilon)$, but this is not a chain map. In order to fill in the gap, we need to prove the following lemma for $a\in \R_{>0}\setminus \mathcal{L}(K)$.

\begin{lem}\label{lem-o-epsilon}
Suppose that $(\epsilon,\delta)\in \mathcal{T}_a$ is standard with respect to $h$. Then, for $m\in \Z_{\geq 1}$ and $k\in \{1,\dots, m\}$, there exists a linear map
\[o_{k,(\epsilon,\delta)} \colon C^{\pitchfork}_1(m,a,\epsilon)\to C^{\dr}_{1}(\Sigma^{a+(m+1)\epsilon}_{m+1}) \]
such that the following hold for any $x\in C^{\pitchfork}_1(m,a,\epsilon)$:
\begin{itemize}
\item[(i)] $\partial( o_{k,(\epsilon,\delta)}(x) ) - (  f_{k,\delta} \circ \Psi (x)-\Psi\circ f^{\pitchfork}_{k} (x ) ) \in C^{\dr}_0(\Sigma^0_{m+1})$.
\item[(ii)] $f_{l,\delta} (o_{k,(\epsilon,\delta)}(x) )\in C^{\dr}_0(\Sigma^0_{m+2})$ for every $l\in \{1,\dots ,m+1\}$.
\end{itemize}
\end{lem}

\begin{proof}

It  suffices to define $o_{k,(\epsilon,\delta)}(x)$ for $x=(\gamma^u_l)^{u\in [0,1]}_{l=1,\dots , m}$ satisfying (1a),\dots ,(1e).
The proof is divided into three steps: We define de Rham chains in the first two steps. In the last step, $o_{k,(\epsilon,\delta)}(x)$ is defined as the sum of these chains and we check the conditions (i) and (ii). 

\textit{Step 1.} From the definitions of $\Psi$ and $f^{\pitchfork}_k$,
\begin{align*}
 & f_{k,\delta}\circ \Psi(x) - \Psi\circ f^{\pitchfork}_k(x) \\
=&   f_{k,\delta}(\Psi (x)) - \sum_{(u_*,t_*)\in (\Gamma^{\interior}_k)^{-1}(K)}  \sign (u_*,t_*)\cdot [\{0\},c_0(u_*,t_*),1]  ,
\end{align*}
where $c_0(u_*,t_*)$ is a constant map to $ (\gamma^{u_*}_1,\dots, \widehat{\gamma^{u_*}_k}^1 , \widehat{\gamma^{u_*}_k}^2 ,\dots ,\gamma^{u_*}_m)$.
Since $(\epsilon,\delta)$ is standard, it has the form (\ref{delta-standard}).
Using the notations of (\ref{explicit-rep}), we can explicitly write
$f_{k,\delta }( \Psi(x) ) = [W_k,\Phi_k,\zeta_k]$, where
\begin{align*}
W_k&=\{(u,\tau,v)\in \R\times \R\times N_{\epsilon} \mid 2\epsilon <\tau <T^{\mu(u)}_k-2\epsilon,\ \gamma^{\mu(u)}_k(\tau)=\sigma^v_1(0) \} , \\
\Phi_k &\colon W_k\to \Sigma^{a+(m+1)\epsilon}_{m+1}\colon (u,\tau,v)\mapsto \con_k( (\gamma^{\mu(u)}_l)_{l=1,\dots , m}, (T^{\mu(u)}_k,\tau),\psi_{\epsilon}(v) ) , \\
\zeta_k&\in \Omega^2_c(W_k):\ (\zeta_k)_{(u,\tau,v)}=\rho_{\epsilon}(T^{\mu(u)}_k,\tau)\cdot \chi(u)\cdot (h_*(\nu_{\epsilon}\times 1))_v .
\end{align*}
Recall the condition (1e-3) and note that $\sigma^v_1(0)=v$. We define $\bar{W}_k\coloneqq W_k\cap \{\tau_0 <\tau <T^{\mu(u)}_k-\tau_0\}$. Then
\[ \bar{W}_k  \to  \bigcup_{(u_*,t_*)\in  (\Gamma^{\interior}_k)^{-1}(K) } U_{(u_*,t_*)} \colon (u, \tau, v) \mapsto (\mu(u),\tau) \]
is an orientation preserving diffeomorphism. Moreover, $\rho_{\epsilon}(T^{\mu(u)}_k,\tau)\cdot \chi(u) =1 $ for $(u,\tau,v) \in \bar{W}_k$. 
On the other hand, it follows from (1e-1) and Lemma \ref{lem-short-length} that $\Phi_k(u,\tau,v)\in \Sigma^0_{m+1}$ for $(u,\tau,v)\in W_k\setminus \bar{W}_k$.
Therefore, we have
\begin{align*}
&f_{k,\delta }( \Psi(x) ) - \sum_{(u_*,t_*)\in  (\Gamma^{\interior}_k)^{-1}(K) }  [U_{(u_*,t_*)}, \Phi'_{(u_*,t_*)}, \zeta'_{(u_*,t_*)} ] \\
=&f_{k,\delta }( \Psi(x) ) - [\bar{W}_k,\rest{\Phi_k}{\bar{W}_k}, \rest{\zeta_k}{\bar{W}_k}] \in C^{\dr}_0(\Sigma^0_{m+1}) ,
\end{align*}
where
\begin{align*}
\Phi'_{(u_*,t_*)} &\colon U_{(u_*,t_*)} \to \Sigma^{a+(m+1)\epsilon}_{m+1}\colon (u,\tau)\mapsto \con_k( (\gamma^u_l)_{l=1,\dots , m}, (T^u_k,\tau),\psi_{\epsilon}( \gamma_k^u(\tau))) ,\\
\zeta'_{(u_*,t_*)}&= (\rest{\Gamma^{\interior}_k}{U_{(u_*,t_*)}})^*(h_*(\nu_{\epsilon}\times 1))\in \Omega^2_c(U_{(u_*,t_*)}) . 
\end{align*}
As a result, $f_{k,\delta}\circ \Psi(x) - \Psi\circ f^{\pitchfork}_k(x)  $ is equal to the chain
\begin{align}\label{D-Psi-Psi-D}
  \sum_{(u_*,t_*)\in (\Gamma^{\interior}_k)^{-1}(K)} \left(  [U_{(u_*,t_*)}, \Phi'_{(u_*,t_*)},\zeta'_{(u_*,t_*)}]  - \sign (u_*,t_*)\cdot [\{0\},c_0(u_*,t_*),1] \right)  
\end{align}
modulo $C^{\dr}_0(\Sigma^0_{m+1})$.

For each $(u_*,t_*)\in (\Gamma^{\interior}_k)^{-1}(K)$, consider a diffeomorphism
\[\Gamma_k^{\fib} = \pr_{\R^2}\circ h^{-1}\circ \rest{\Gamma^{\interior}_k}{U_{(u_*,t_*)}}\colon U_{(u_*,t_*)}\to \mathcal{O}_{\epsilon} \]
and a scalar multiplication $m_s\colon \mathcal{O}_{\epsilon}  \to \mathcal{O}_{\epsilon}  \colon w\mapsto \kappa(s)\cdot w$ for $s\in \R$.
We define a deformation retraction to $\{ (u_*,t_*)\}$
\begin{align*}
\R \times U_{(u_*,t_*)} \to U_{(u_*,t_*)}  \colon (s,(u,\tau) ) \mapsto (u_s,\tau_s)\coloneqq  (\Gamma_k^{\fib})^{-1} \circ m_s\circ  (\Gamma_k^{\fib})(u,\tau ) .
\end{align*}
Now we define a map
\begin{align*}
\tilde{\Phi}'_{(u_*,t_*)} &\colon \R \times U_{(u_*,t_*)} \to \Sigma^{a+(m+1)\epsilon}_{m+1}\colon ( s, (u,\tau) ) \mapsto \Phi'_{(u_*,t_*)}(u_s,\tau_s) 
\end{align*}
and a de Rham chain
\[o^1_{k} \coloneqq  \sum_{(u_*,t_*)\in (\Gamma^{\interior}_k)^{-1}(K)} [\R \times U_{(u_*,t_*)}, \tilde{\Phi}'_{(u_*,t_*)},  \chi\times \zeta'_{(u_*,t_*)} ] \in C^{\dr}_1(\Sigma^{a+(m+1)\epsilon}_{m+1}).\]
If $s\geq 1$, $\tilde{\Phi}'_{(u_*,t_*)}(s,(u,\tau))= \Phi'_{(u_*,t_*)}(u,\tau)$. If $s\leq 0$, $\tilde{\Phi}'_{(u_*,t_*)}(s,(u,\tau))$ is constant to
\[  (\gamma^{u_*}_1,\dots, \tilde{\gamma^{u_*}_k}^1, \tilde{\gamma^{u_*}_k}^2 ,\dots ,\gamma^{u_*}_m) \eqqcolon c'_0(u_*,t_*),\]
which is defined by (\ref{concatenation}) for $\gamma_k=\gamma^{u_*}_k$, $(T_k,\tau)=(T^{u_*}_k,t_*)$ and $\sigma_i \colon [0,\epsilon/2]\to  \{ \gamma^{u_*}_k(t_*)\} \subset Q$.
Therefore, the boundary chain $\partial o^1_{k}$ is equal to
\begin{align*}
 \sum_{(u_*,t_*)\in (\Gamma^{\interior}_k)^{-1}(K)} \left( [ U_{(u_*,t_*)}, \Phi'_{(u_*,t_*)},\zeta'_{(u_*,t_*)} ]  - \left[ \{0\},c'_0(u_*,t_*),\int_{U_{(u_*,t_*)}} (\Gamma_k)^* (h_*(\nu_{\epsilon}\times 1))\right] \right) .
\end{align*}
Since $\int_{\mathcal{O}_{\epsilon}}\nu_{\epsilon}=1$, we can compute that
\begin{align*}
\int_{U_{(u_*,t_*)}} (\Gamma_k)^*  (h_*(\nu_{\epsilon}\times 1))&= \int_{U_{(u_*,t_*)}} (h^{-1}\circ \Gamma_k)^*(\nu_{\epsilon}\times 1)  \\
&= \int_{U_{(u_*,t_*)}} (\Gamma^{\fib}_k)^*\nu_{\epsilon} \\
&=\sign (u_*,t_*)\in \{\pm 1\}.
\end{align*}
Thus, we have
\begin{align}\label{partial-o-1}
\partial (o^1_{k}) =  \sum_{(u_*,t_*)\in (\Gamma^{\interior}_k)^{-1}(K)} \left( [ U_{(u_*,t_*)}, \Phi'_{(u_*,t_*)},\zeta'_{(u_*,t_*)} ]  - \sign (u_*,t_*) \cdot [ \{0\},c'_0(u_*,t_*),1 ] \right) .
\end{align}

\textit{Step 2.} For each $(u_*,t_*)\in (\Gamma^{\interior}_k)^{-1}(K)$, $c'_0(u_*,t_*)$ coincides with $c_0(u_*,t_*)$ up to reparameterizations of $k$-th and $(k+1)$-th paths.
%
We define by interpolating parameterizations
\[c_1(u_*,t_*)\colon [0,1] \to \Sigma^{a+(m+1)\epsilon}_{m+1} \colon s\mapsto (\gamma^{u_*}_1,\dots, \gamma^{u_*}_{k-1},  \widehat{\gamma^{u_*}_k}^1\circ \mu^1_s , \widehat{\gamma^{u_*}_k}^2 \circ \mu^2_s ,\gamma^{u_*}_{k+1},\dots ,\gamma^{u_*}_m),\]
so that $c_1(u_*,t_*) (0)=c_0(u_*,t_*)$ and $c_1(u_*,t_*) (1) =c'_0(u_*,t_*)$.
Then, we obtain a chain
\[o^2_{ (u_*,t_*)} \coloneqq  [\R, c_1(u_*,t_*)\circ \kappa ,\chi] \in C^{\dr}_1(\Sigma^{a+(m+1)\epsilon}_{m+1}),\]
which satisfies $\partial (o^2_{(u_*,t_*)} ) =  [\{0\},c'_0(u_*,t_*),1] - [\{0\},c_0(u_*,t_*),1] $.

\textit{Step 3.}  We define a chain
\[o_{k,(\epsilon,\delta)} (x) \coloneqq o^1_{k} +  \sum_{(u_*,t_*)\in (\Gamma^{\interior}_k)^{-1}(K))} \sign (u_*,t_*)\cdot  o^2_{(u_*,t_*)}. \]
From (\ref{partial-o-1}), $\partial (o_{k,(\epsilon,\delta)} (x) )$ is equal to the chain of (\ref{D-Psi-Psi-D}). Therefore, $o_{k,\delta} (x) $ satisfies the condition (i).
The condition (ii) can be checked as follows: From the conditions (1e-2) and (1e-4), those paths in $\Phi'_{(u_*,t_*)} (s,(u,\tau))$ and $c_1(u_*,t_*)(s)$ satisfy the condition (iii) of Lemma \ref{lem-short-length}.
Therefore, $f_{l,\delta} (o^1_{k})$ and $f_{l,\delta} (o^2_{(u_*,t_*)})$ belongs to $C^{\dr}_0(\Sigma^0_{m+2})$ for $l=1,\dots ,m+1$. This finishes the proof.
\end{proof}
%
For $(\epsilon,\delta)\in \mathcal{T}_a$ which is standard, we define a linear map $\Phi^{<a}_{(\epsilon,\delta)} \colon C^{\pitchfork,<a}_*(\epsilon ) \to C^{<a}_*(\epsilon) $ so that for $x\in C^{\pitchfork}_p(m,a,\epsilon)$, the equivalence class $[x]\in C^{\pitchfork,<a}_p(m,\epsilon )$ is mapped to
\begin{align*}
\Phi^{<a}_{(\epsilon,\delta)} ( [x] ) = \begin{cases}
[ \Psi(x) ] & \text{ if } p=0 ,\\
[\Psi (x) ] - \sum_{k=1}^m [ o_{k,(\epsilon,\delta)}(x) ] & \text{ if }p=1, \\
0 & \text{ else.}
\end{cases}
\end{align*}
The two properties of $o_{k,(\epsilon,\delta)}$ shows that $\Phi^{<a}_{(\epsilon,\delta)}$ is a chain map from $(C^{\pitchfork,<a}_*(\epsilon), D^{\pitchfork}_{\epsilon})$ to $(C^{<a}_*(\epsilon), D_{\delta})$. 
Therefore, we obtain a linear map on homology
\[ (\Phi^{<a}_{(\epsilon,\delta)})_* \colon H^{\pitchfork,<a}_*(\epsilon) \to H^{<a}_*(\epsilon,\delta). \]

\subsubsection{Commutativity with transition maps}

We need to check the relation of $ \Phi^{<a}_{(\epsilon,\delta)} $ with $\{k_{(\epsilon',\delta),(\epsilon,\delta)}\}_{\epsilon'\leq \epsilon}$ and $\{l_{\epsilon,\epsilon'}\}_{\epsilon \geq \epsilon' }$.
\begin{prop}\label{prop-Phi-commute}
For $(\epsilon,\delta),(\epsilon',\delta) \in \mathcal{T}_{a}$ ($\epsilon' \leq \epsilon$) which are standard with respect to $h$, the following diagram commutes:
\begin{align*}
\xymatrix@C=40pt{ H^{\pitchfork,<a}_*(\epsilon) \ar[r]^-{(\Phi^{<a}_{(\epsilon,\delta)})_*}\ar[d]_{l_{\epsilon,\epsilon'}} & H^{<a}_*(\epsilon,\delta)  \\ H^{\pitchfork,<a}_* (\epsilon') \ar[r]^-{(\Phi_{(\epsilon',\delta')}^{<a})_*} & H^{<a}_* (\epsilon',\delta) \ar[u]_{k_{(\epsilon',\delta'), (\epsilon,\delta)}}  . }
\end{align*}
\end{prop}

To prove this proposition, we return to the definition $k_{(\epsilon',\delta'), (\epsilon,\delta)}= k_{(\bar{\epsilon},\bar{\delta})}$  by $(\bar{\epsilon},\bar{\delta})\in \bar{\mathcal{T}}_{a}$ satisfying (\ref{condition-delta-bar}) for $(\epsilon,\delta),(\epsilon',\delta')$. We require that $(\bar{\epsilon},\bar{\delta})$ is standard with respect to $h$, and thus $\bar{\epsilon}=\epsilon$.

We set $\bar{\Psi}\coloneqq \bar{i}\circ \Psi\colon C^{\pitchfork}_*(m,a) \to \bar{C}^{\dr}_*(\Sigma^a_m)$ for all $m\in \Z_{\geq 0}$.
Again, the induced map from $C^{\pitchfork,<a}_*(\epsilon)$ to $\bar{C}^{<a}_*(\epsilon)$ is not a chain map. To fill in the gap, we need the following lemma.
\begin{lem} For $m\in \Z_{\geq 1}$ and $k\in \{1,\dots ,m\}$,
suppose that we have taken maps $o_{k,(\epsilon,\delta)},o_{k,(\epsilon',\delta')}$ of Lemma \ref{lem-o-epsilon} for $(\epsilon,\delta),(\epsilon',\delta')$. Then, there exists a linear map
\[\bar{o}_{k,(\epsilon,\bar{\delta})}\colon C^{\pitchfork}_1(m,a,\epsilon) \to \bar{C}^{\dr}_1(\Sigma^{a+(m+1)\epsilon}_{m+1})\]
such that the following hold for any $x\in C^{\pitchfork}_1(m,a,\epsilon)$
\begin{itemize}
\item $\partial (\bar{o}_{k,(\epsilon,\bar{\delta})}) - (  \bar{f}_{k,\bar{\delta}} \circ \bar{\Psi}(x) - \bar{\Psi} \circ f^{\pitchfork}_{k} (x)  ) \in \bar{C}^{\dr}_0(\Sigma^0_{m+1})$.
\item $\bar{f}_{l,\bar{\delta}} (\bar{o}_{k,(\epsilon,\bar{\delta})} (x) )\in \bar{C}^{\dr}_0(\Sigma^0_{m+2})$ for every $l\in \{1,\dots ,m+1\}$.
\item $e_+(\bar{o}_{k,(\epsilon,\bar{\delta})} (x)) =o_{k,(\epsilon,\delta)}(x) $ and $e_-(\bar{o}_{k,(\epsilon,\bar{\delta})}(x))=(j_{\epsilon',\epsilon})_*(o_{k,(\epsilon',\delta')}(x))$.
\end{itemize}
\end{lem}
\begin{proof}
We omit the detailed proof. Note that $\bar{\delta}$ has the form (\ref{standard-delta-bar}). For any $x=(\gamma_k)_{k=1,\dots ,m}$, we can compute explicitly that the chain $ \bar{f}_{k,\bar{\delta}} \circ \bar{\Psi}(x) - \bar{\Psi} \circ f^{\pitchfork}_{k} (x)$ is equal to the sum of chains for all $(u_*,t_*)\in (\Gamma^{\interior}_k)^{-1}(K)$
\begin{align*}
 &[\R \times U_{(u_*,t_*)}, \bar{\Phi}'_{(u_*,t_*)},(\id_{\R_{\geq 1}}\times U_{(u_*,t_*)},\id_{\R_{\leq -1}}\times U_{(u_*,t_*)} ), \bar{\zeta}'_{(u_*,t_*)}] \\
  & - \sign (u_*,t_*)\cdot [\R,c_0(u_*,t_*), (\id_{\R_{\geq 1}}, \id_{\R_{\leq -1}}), 1] 
\end{align*}
modulo $\bar{C}^{\dr}_0(\Sigma^0_{m+1})$. Here $\bar{\Phi}'_{(u_*,t_*)}\colon \R\times U_{(u_*,t_*)}\to \R\times \Sigma^{a+(m+1)\epsilon}_{m+1}$ is determined by
\begin{align*}
\bar{\Phi}'_{(u_*,t_*)} (r,(u,\tau))\coloneqq (r, \con_k((\gamma^u_l)_{l=1,\dots ,m}, (T^u_k,\tau), \bar{\psi}_{\epsilon',\epsilon}(r,\gamma^u_k(\tau))) )
\end{align*}
and
$(\bar{\zeta}'_{(u_*,t_*)})\coloneqq (\id_{\R}\times \rest{\Gamma^{\interior}_k}{U_{(u_*,t_*)}} )^*(1\times \bar{\eta}_{\epsilon',\epsilon}) \in \Omega^2(\R\times U_{(u_*,t_*)})$. 
The $[-1,1]$-modeled chain $\bar{o}_{k, (\epsilon,\bar{\delta})}$ is defined in a similar way as $o_{k,(\epsilon,\delta)}$ in Lemma \ref{lem-o-epsilon}, and we can check that this chain satisfies the required three conditions.
\end{proof}

\begin{proof}[Proof of Proposition \ref{prop-Phi-commute}]
We define a linear map $\bar{\Phi}^{<a}_{(\epsilon,\bar{\delta})} \colon  C^{\pitchfork,<a}_*(\epsilon) \to \bar{C}^{<a}_*(\epsilon) $ so that for $x\in C^{\pitchfork,<a}_p(m,a,\epsilon)$,
\[ \bar{\Phi}^{<a}_{(\epsilon,\bar{\delta})}  ([x]) = \begin{cases}
[\bar{\Psi}(x)] & \text{ if }p=0, \\
[ \bar{\Psi}(x)] - \sum_{k=1}^m [ \bar{o}_{k,(\epsilon,\bar{\delta})} (x)] & \text{ if }p=1 , \\
0 & \text{ else.}
\end{cases}\]
The first two properties of $\bar{o}_{k,(\epsilon,\bar{\delta})}$ shows that this is a chain map from $(C^{\pitchfork,<a}_*(\epsilon), D^{\pitchfork}_{\epsilon})$ to $(\bar{C}^{<a}_*(\epsilon), \bar{D}_{\bar{\delta}})$. Therefore, we get a linear map on homology
\[(  \bar{\Phi}^{<a}_{(\epsilon,\bar{\delta})} )_* \colon H^{\pitchfork,<a}_*(\epsilon) \to \bar{H}^{<a}_*(\epsilon, \bar{\delta}).\]
The third property of $\bar{o}_{k,(\epsilon,\bar{\delta})}$ implies that the following diagram commutes: 
\[\xymatrix@C=40pt{ H^{\pitchfork,<a}_*(\epsilon) \ar[r]^-{(\Phi^{<a}_{(\epsilon,\delta)})_*}\ar[dd]_{l_{\epsilon,\epsilon'}} \ar[rd]_{(\bar{\Phi}^{<a}_{(\epsilon,\bar{\delta})})_* }&H^{<a}_* (\epsilon,\delta)\ar[r]^{(j_{\epsilon, \epsilon})_*=\id } &  H^{<a}_* (\epsilon,\delta)  \\
 & \bar{H}^{<a}_* (\epsilon,\bar{\delta})\ar[ur]_{(e_{\epsilon,+})_*}\ar[dr]^{(e_{\epsilon,-})_*} &\\
 H^{\pitchfork,<a}_* (\epsilon') \ar[r]^-{(\Phi_{(\epsilon',\delta')}^{<a})_*} &  H^{<a}_* (\epsilon',\delta')\ar[r]^-{(j_{\epsilon',\epsilon})_*} &  H^{<a}_* (\epsilon, (i_{\epsilon',\epsilon})_*\delta')  . }\]
The proposition is now proved since $k_{(\epsilon,\bar{\delta})} = ((j_{\epsilon,\epsilon})^{-1}_*\circ (e_{\epsilon,+})_*) \circ ((j_{\epsilon',\epsilon})^{-1}_*\circ (e_{\epsilon_-})_*)^{-1}$.
\end{proof}

Let $a\in \R_{> 0} \setminus \mathcal{L}(K)$.
Proposition \ref{prop-Phi-commute} shows that the family of maps $\{(\Phi_{(\epsilon,\delta)}^{<a} )_* \mid (\epsilon,\delta)\in \mathcal{T}_a \text{ is standard with respect to }h\}$ induces a linear map on the limits of $\epsilon \to 0$
\[ \Phi^{<a} \colon H^{\pitchfork,<a}_* =\varinjlim_{\epsilon\to 0} H^{\pitchfork,<a}_*(\epsilon  ) \to H^{<a}_*(Q,K)= \varprojlim_{\epsilon \to 0} H^{<a}_*(\epsilon,\delta) .\]
Naturally, those maps of $\{\Phi^{<a}\}_{a\in \R_{>0}\setminus \mathcal{L}(K)}$ commutes with $\{I^{a,b} \colon H^{<a}_*(Q,K)\to H^{<b}_*(Q,K)\}_{a\leq b}$ and $\{ J^{a,b}\colon H^{\pitchfork,<a}_*\to  H^{\pitchfork,<b}_* \}_{a\leq b}$.
Therefore, on the limit of $a\to \infty$, we have a linear map
\[\Phi \colon H^{\pitchfork,\infty}_* \to H^{\str}_*(Q,K). \]
Moreover, it is straightforward to check that $\Phi$ is a homomorphism of unital $\R$-algebras.
%

\subsection{Proof of isomorphism}\label{subsec-proof-of-isom}
In this section, we prove that for every $a\in \R_{>0}\setminus \mathcal{L}(K)$, $\Phi^{<a}$ is an isomorphism in the $0$-th degree. As an immediate consequence, it is shown that the cord algebra of $(Q,K)$ is isomorphic to $H^{\str}_0(Q,K)$.

For each $m\in \Z_{\geq 0}$, let
$\partial^{\sing}_m \colon C^{\pitchfork,<a}_1(m) \to C^{\pitchfork,<a}_0(m)$ denote the singular boundary operator. 
We also write
\begin{align}\label{Psi-isom-surj}
\begin{split}
\Psi _{0,m}&\colon \coker \partial^{\sing}_m \to H^{\dr}_0(\Sigma^a_m,\Sigma^0_m) \colon [x]\mapsto [\Psi (x)] , \\
\Psi_{1,m}&\colon \ker \partial^{\sing}_m \to H^{\dr}_1(\Sigma^a_m,\Sigma^0_m) \colon x \mapsto [\Psi(x)] .
\end{split}
\end{align}
\begin{lem}\label{lem-Psi-isom-surj}
$\Psi_{0,m}$ is an isomorphism and $\Psi_{1,m}$ is a surjection. 
\end{lem}
\begin{proof}
Naturally, there are two maps
\begin{align*}
\coker \partial^{\sing}_m &\to  H^{\sing}_0(\Sigma^a_m,\Sigma^0_m),  \\
\ker \partial^{\sing}_m &\to  H^{\sing}_1(\Sigma^a_m,\Sigma^0_m) .
\end{align*}
induced by the inclusion maps $C^{\pitchfork}_p(m,a) \to  C^{\sing}_p(\Sigma^a_m)$ for $p=0,1$.
The subset of $\Sigma^a_m$ (resp. $C^0([0,1], \Sigma^a_m)$) consisting of elements satisfying the conditions (0a), (0b) (resp. (1a), (1b), (1c)) is open dense.
This fact implies that the first map is an isomorphism and the second map is a surjection.
Then, we consider the following diagram for $p=0,1$: 
\[\xymatrix{
K_{p,m} \ar[r]^-{\Psi_{p,m}} \ar[d]& H^{\dr}_p(\Sigma^a_m,\Sigma^0_m)\ar[r]^-{(\ref{cong-dr-sing-path})}& \varinjlim_{j\to \infty}H_p^{\dr}(B^a_m(2^j),B^0_m(2^j)) \\
 H^{\sing}_p(\Sigma^a_m,\Sigma^0_m) \ar[rr]^-{(\ref{cong-sing-path})} && \varinjlim_{j\to\infty } H_p^{\sing}(B^a_m(2^j),B^0_m(2^j))\ar[u]  .
}\]
Here, $K_{0,m} \coloneqq  \coker \partial^{\sing}_m$ and $K_{1,m}\coloneqq \ker \partial^{\sing}_m$.
The left vertical map is defined as above. The right vertical map is an isomorphism from Proposition \ref{prop-hmgy-of-mfd}. The horizontal maps are the isomorphisms of (\ref{cong-dr-sing-path}) and (\ref{cong-sing-path}).
The commutativity follows from the definition of the right vertical map. See \cite[Section 4.7]{Irie-BV}. Now the assertion of the lemma follows from the diagram.
\end{proof}

For the chain complexes $(C^{\pitchfork,<a}_*,D^{\pitchfork})$ and $(C^{\pitchfork,<a}_*(\epsilon),D^{\pitchfork}_{\epsilon})$, we define their filtrations $\{\mathcal{H}^{<a}_p\}_{p\in \Z}$ and $\{\mathcal{H}^{<a}_{\epsilon,p}\}_{p\in \Z}$ by 
\[ \mathcal{H}^{<a}_p \coloneqq \bigoplus_{m\geq -p} C^{\pitchfork,<a}_*(m) ,\  \mathcal{H}^{<a}_{\epsilon,p} \coloneqq \bigoplus_{m\geq -p} C^{\pitchfork,<a}_*(m,\epsilon).\]
Let $E^{\pitchfork,<a}$ and $E^{\pitchfork,<a}_{\epsilon}$ be the spectral sequences determined by $\{\mathcal{H}^{<a}_p\}_{p\in \Z}$ and $\{\mathcal{H}^{<a}_{\epsilon,p}\}_{p\in \Z}$ respectively.
Their $(-m,q)$-terms of the first pages are given by
\begin{align*}
(E^{\pitchfork,<a})^1_{-m,q}&= \begin{cases}
\coker \partial^{\sing}_{m} & \text{ if }q-m=0 \text{ and }m\geq 0, \\
\ker \partial^{\sing}_{m}  & \text{ if }q-m=1 \text{ and }m\geq 0, \\
0 & \text{ else,}
\end{cases}\\
(E^{\pitchfork,<a}_{\epsilon})^1_{-m,q}&= \begin{cases}
\coker \partial^{\sing}_{\epsilon,m} & \text{ if }q-m=0 \text{ and }m\geq 0, \\
\ker \partial^{\sing}_{\epsilon,m}  & \text{ if }q-m=1 \text{ and }m\geq 0, \\
0 & \text{ else.}
\end{cases}
\end{align*}
Here $\partial^{\sing}_{\epsilon,m} \colon C^{\pitchfork,<a}_1(m,\epsilon)\to  C^{\pitchfork,<a}_0(m,\epsilon)$ also denotes the singular boundary operator.
If $0<\epsilon'\leq \epsilon$, there exists a morphism $l_{\epsilon,\epsilon'}\colon E^{\pitchfork,<a}_{\epsilon}\to E^{\pitchfork,<a}_{\epsilon'}$ induced by the inclusion maps $C^{\pitchfork}_*(m,a,\epsilon) \to C^{\pitchfork}_*(m,a,\epsilon') $ for all $m\in \Z_{\geq 0}$. Naturally, $\varinjlim_{\epsilon\to 0} E^{\pitchfork,<a}_{\epsilon} \cong E^{\pitchfork,<a}$ holds.

For $(\epsilon,\delta)\in \mathcal{T}_a$ which is standard with respect to $h$, the chain map $\Phi^{<a}_{(\epsilon,\delta)}$ preserves the filtrations  $\{\mathcal{H}^{<a}_{\epsilon,p}\}_{p\in \Z}$ and $\{\mathcal{F}^{<a}_{\epsilon,p}\}_{p\in \Z}$, so it induces a morphism of spectral sequences
\[ (\Phi^{<a}_{(\epsilon,\delta)})_* \colon E^{\pitchfork,<a}_{\epsilon} \to E^{<a}_{(\epsilon,\delta)}. \]
Note that on the $(-m,q)$-term ($m\geq 0$) of first pages, this can be written as follows:
\begin{align}\label{first-page-Phi}
\begin{split}
 (\Phi^{<a}_{(\epsilon,\delta)})_* & \colon  \coker \partial^{\sing}_{\epsilon,m} \to  H^{\dr}_0(\Sigma^{a+m\epsilon}_m,\Sigma^0_m) \colon [x]\to [\Psi(x)] \text{ if } q=m , \\
  (\Phi^{<a}_{(\epsilon,\delta)})_* & \colon \ker \partial^{\sing}_{\epsilon,m} \to H^{\dr}_1(\Sigma^{a+m\epsilon}_m,\Sigma^0_m) \colon x\to [\Psi(x)]  \text{ if } q=m+1. 
\end{split}
\end{align}
 
Recall that we have defined $k_{(\epsilon',\delta'),(\epsilon,\delta)} \colon E^{<a}_{(\epsilon',\delta')}\to E^{<a}_{(\epsilon,\delta)}$ by the composition of the maps of (\ref{zig-zag}).
The next result is a variant of Proposition \ref{prop-Phi-commute} for spectral sequences.
\begin{prop}\label{prop-spectral-commute}
The following diagram commutes:
\begin{align*}
\xymatrix@C=40pt{ E^{\pitchfork,<a}_{\epsilon} \ar[r]^-{(\Phi^{<a}_{(\epsilon,\delta)})_*}\ar[d]_{l_{\epsilon,\epsilon'}} & E^{<a}_{(\epsilon,\delta)}  \\ E^{\pitchfork,<a}_{\epsilon'} \ar[r]^-{(\Phi_{(\epsilon',\delta')}^{<a})_*} & E^{<a}_{(\epsilon',\delta')} \ar[u]_{k_{(\epsilon',\delta'), (\epsilon,\delta)}}  . }
\end{align*}
\end{prop}
This can be proved as Proposition \ref{prop-Phi-commute} by taking $(\bar{\Phi}^{<a}_{(\epsilon,\bar{\delta})})_*\colon E^{\pitchfork,<a}_{\epsilon} \to \bar{E}^{<a}_{(\epsilon,\bar{\delta})}$ into consideration. We omit the proof.

We use the spectral sequence $E^{<a}$ of Proposition \ref{prop-spectral-seq}. The above proposition and (\ref{first-page-Phi}) immediately implies the existence of the following morphism of spectral sequences.
\begin{prop}\label{prop-Phi-spectral}
There exists a morphism of spectral sequences $\Phi^{<a}\colon E^{\pitchfork,<a} \to E^{<a}$ such that on the $(-m,q)$-term ($m\geq 0$) of the first page,
\begin{align*}
 \Phi^{<a} =\Psi_{0,m} &\colon (E^{\pitchfork,<a})^1_{-m,q}= \coker \partial^{\sing}_{m} \to (E^{<a})^1_{-m,q} =H^{\dr}_0(\Sigma^{a}_m,\Sigma^0_m) &\text{ if } q=m, \\
 \Phi^{<a} =\Psi_{1,m}  &\colon (E^{\pitchfork,<a})^1_{-m,q}= \ker \partial^{\sing}_{m} \to (E^{<a})^1_{-m,q} =H^{\dr}_1(\Sigma^{a}_m,\Sigma^0_m) & \text{ if } q=m+1.
\end{align*}
\end{prop}
This property of $\Phi^{<a}\colon E^{\pitchfork,<a} \to E^{<a}$ implies a result on the compatible map $\Phi^{<a}\colon  H^{\pitchfork,<a}_p \to H^{<a}_p(Q,K)$.
\begin{prop}\label{prop-isom-short-interval-string}
$\Phi^{<a} \colon H^{\pitchfork,<a}_p \to H^{<a}_p(Q,K)$
is an isomorphism if $p=0$ and a surjection if $p=1$.
\end{prop}
\begin{proof}
By Lemma \ref{lem-Psi-isom-surj} and Proposition \ref{prop-Phi-spectral}, $\Phi^{<a}\colon  (E^{\pitchfork,<a})^1_{p,q} \to  (E^{<a})^1_{p,q}$ is an isomorphism if $p+q\leq 0$ and a surjection if $p+q=1$. Since $E^{\pitchfork,<a}$ converges to $H^{\pitchfork,<a}_*$ and $E^{<a}$ converges to $H^{<a}_*(Q,K)$, we can apply Lemma \ref{lem-spectral} to prove the above assertion for $\Phi^{<a} \colon H^{\pitchfork,<a}_*\to H^{<a}_*(Q,K)$.
\end{proof}
On their limits of $a\to \infty$, $\{\Phi^{<a}\}_{a\in \R_{>0}\setminus \mathcal{L}(K)}$ induces an isomorphism
\[\Phi\colon H^{\pitchfork,<\infty}_0 \to H^{\str}_0(Q,K).\]
Combining with Proposition \ref{prop-cord-to-string}, we finally obtain the following result.
\begin{thm}\label{thm-isom-cord-string}
As a unital $\R$-algebra, $\mathrm{Cord}(Q,K;\R)$ is isomorphic to $H^{\str}_0(Q,K)$.
\end{thm}


\bibliography{reference.bib}

\begin{thebibliography}{10}

\bibitem{a-s-conormal}
A.~Abbondandolo, A.~Portaluri, and M.~Schwarz.
\newblock The homology of path spaces and {F}loer homology with conormal
  boundary conditions.
\newblock {\em J. Fixed Point Theory Appl.}, 4(2):263--293, 2008.

\bibitem{a-s}
A.~Abbondandolo and M.~Schwarz.
\newblock On the {F}loer homology of cotangent bundles.
\newblock {\em Comm. Pure Appl. Math.}, 59(2):254--316, 2006.

\bibitem{abouzaid}
M.~Abouzaid.
\newblock Symplectic cohomology and {V}iterbo's theorem.
\newblock In {\em Free loop spaces in geometry and topology}, volume~24 of {\em
  IRMA Lect. Math. Theor. Phys.}, pages 271--485. Eur. Math. Soc., Z\"{u}rich,
  2015.

\bibitem{chek}
Y.~Chekanov.
\newblock Differential algebra of {L}egendrian links.
\newblock {\em Invent. Math.}, 150(3):441--483, 2002.

\bibitem{chen}
K.~T. Chen.
\newblock On differentiable spaces.
\newblock In {\em Categories in continuum physics ({B}uffalo, {N}.{Y}., 1982)},
  volume 1174 of {\em Lecture Notes in Math.}, pages 38--42. Springer, Berlin,
  1986.

\bibitem{celn}
K.~Cieliebak, T.~Ekholm, J.~Latschev, and L.~Ng.
\newblock Knot contact homology, string topology, and the cord algebra.
\newblock {\em J. \'{E}c. polytech. Math.}, 4:661--780, 2017.

\bibitem{D-R}
G.~Dimitroglou~Rizell.
\newblock Lifting pseudo-holomorphic polygons to the symplectisation of
  {$P\times\Bbb{R}$} and applications.
\newblock {\em Quantum Topol.}, 7(1):29--105, 2016.

\bibitem{eesR}
T.~Ekholm, J.~Etnyre, and M.~Sullivan.
\newblock The contact homology of {L}egendrian submanifolds in {${\Bbb
  R}^{2n+1}$}.
\newblock {\em J. Differential Geom.}, 71(2):177--305, 2005.

\bibitem{ees-ori}
T.~Ekholm, J.~Etnyre, and M.~Sullivan.
\newblock Orientations in {L}egendrian contact homology and exact {L}agrangian
  immersions.
\newblock {\em Internat. J. Math.}, 16(5):453--532, 2005.

\bibitem{ees}
T.~Ekholm, J.~Etnyre, and M.~Sullivan.
\newblock Legendrian contact homology in {$P\times\Bbb R$}.
\newblock {\em Trans. Amer. Math. Soc.}, 359(7):3301--3335, 2007.

\bibitem{eens}
T.~Ekholm, J.~B. Etnyre, L.~Ng, and M.~G. Sullivan.
\newblock Knot contact homology.
\newblock {\em Geom. Topol.}, 17(2):975--1112, 2013.

\bibitem{eli}
Y.~Eliashberg.
\newblock Invariants in contact topology.
\newblock In {\em Proceedings of the {I}nternational {C}ongress of
  {M}athematicians, {V}ol. {II} ({B}erlin, 1998)}, number Extra Vol. II, pages
  327--338, 1998.

\bibitem{ens}
J.~B. Etnyre, L.~L. Ng, and J.~M. Sabloff.
\newblock Invariants of {L}egendrian knots and coherent orientations.
\newblock {\em J. Symplectic Geom.}, 1(2):321--367, 2002.

\bibitem{fuk}
K.~Fukaya.
\newblock Application of {F}loer homology of {L}angrangian submanifolds to
  symplectic topology.
\newblock In {\em Morse theoretic methods in nonlinear analysis and in
  symplectic topology}, volume 217 of {\em NATO Sci. Ser. II Math. Phys.
  Chem.}, pages 231--276. Springer, Dordrecht, 2006.

\bibitem{Irie-BV}
K.~Irie.
\newblock A chain level {B}atalin-{V}ilkovisky structure in string topology via
  de {R}ham chains.
\newblock {\em Int. Math. Res. Not. IMRN}, (15):4602--4674, 2018.

\bibitem{irie-pseudo}
K.~Irie.
\newblock Chain level loop bracket and pseudo-holomorphic disks.
\newblock {\em J. Topol.}, 13(2):870--938, 2020.

\bibitem{milnor}
J.~Milnor.
\newblock {\em Morse theory}.
\newblock Annals of Mathematics Studies, No. 51. Princeton University Press,
  Princeton, N.J., 1963.
\newblock Based on lecture notes by M. Spivak and R. Wells.

\bibitem{ng2}
L.~Ng.
\newblock Knot and braid invariants from contact homology. {II}.
\newblock {\em Geom. Topol.}, 9:1603--1637, 2005.
\newblock With an appendix by the author and Siddhartha Gadgil.

\bibitem{ng}
L.~Ng.
\newblock Framed knot contact homology.
\newblock {\em Duke Math. J.}, 141(2):365--406, 2008.

\bibitem{viterbo-icm}
C.~Viterbo.
\newblock Generating functions, symplectic geometry, and applications.
\newblock In {\em Proceedings of the {I}nternational {C}ongress of
  {M}athematicians, {V}ol. 1, 2 ({Z}\"{u}rich, 1994)}, pages 537--547.
  Birkh\"{a}user, Basel, 1995.

\bibitem{weibel}
C.~A. Weibel.
\newblock {\em An introduction to homological algebra}, volume~38 of {\em
  Cambridge Studies in Advanced Mathematics}.
\newblock Cambridge University Press, Cambridge, 1994.

\end{thebibliography}



\end{document}